\documentclass[12pt]{amsart}

\usepackage{frontmatter}

\title[Phase-Shifted Nanopterons for Singularly Perturbed KdV]{Phase-Shifted Nanopteron Solutions to a Singularly Perturbed Korteweg--de Vries Equation}

\author{Timothy E.\@ Faver}
\address{Department of Mathematics, Kennesaw State University, 850 Polytechnic Lane, Marietta, GA 30060 USA, {\tt{tfaver1@kennesaw.edu}}}



\begin{document}

\begin{abstract}
We construct nanopteron solutions to the differential equation 
\[
\ep^2u^{(4)}+u''-u+u^2
=0,
\]
which is a singular perturbation of the traveling wave problem $u''-u+u^2=0$ for the Korteweg--de Vries (KdV) equation.
These nanopteron solutions are the superposition of the $\sech^2$-type KdV traveling wave profile, a small localized error, and a periodic ``ripple.''
The ripple is parametrized by its amplitude and its phase shift.
In the process of constructing the nanopteron, we preselect one of these parameters and then solve for the other, and we provide detailed commentary on how the amplitude and phase shift solution processes differ.
Our work revisits several prior constructions of nanopteron solutions for this singularly perturbed KdV equation in light of modern developments of nanopterons for lattice differential equations.
In the context of these contemporary methods, the simple structure of the singularly perturbed KdV equation allows it to serve as an accessible template for nanopteron constructions in more complicated problems.
\end{abstract}

\maketitle

\section{Introduction}

\subsection{Solitary waves and nanopterons}

\subsubsection{The problem}
The deceptively simple appearance of the ordinary differential equation
\begin{equation}\label{eqn: the problem}
\ep^2u^{(4)}+u''-u+u^2
= 0
\end{equation}
belies the rich structure of its solutions and the delicate functional-analytic techniques that construct them.
Superficially, \eqref{eqn: the problem} depends on the superposition of the constant-coefficient linear differential operator $u \mapsto \ep^2u^{(4)}+u''-u$ and the tame quadratic nonlinearity $u \mapsto u^2$.
The linearization is also the singular perturbation of the innocuous operator $u \mapsto u''-u$, and, less trivially, the whole problem \eqref{eqn: the problem} is a singular perturbation of the traveling wave problem $u''-u+u^2=0$ for (one version of) the Korteweg--de Vries (KdV) equation.
All of these perspectives on \eqref{eqn: the problem}, its linear and nonlinear terms, and the limiting situation at $\ep=0$ will contribute meaningfully to our analysis.
The broad question that we, like many before us, study is how the solitary wave solution to $u''-u+u^2=0$ perturbs under the singular perturbation in \eqref{eqn: the problem} when $\ep > 0$.

\subsubsection{The water wave connection}
The problem \eqref{eqn: the problem} arises naturally in the analysis of water waves, and a brief overview of some results for water waves and the connection of \eqref{eqn: the problem} to water waves will provide essential context.
These water waves results fall into different categories based on the parameter regimes for the Froude number $F$, which is a nondimensionalized wave speed, and the Bond number $\beta$, which is a nondimensionalized measurement of surface tension \cite{ebb-flow}.
When $\beta = 0$ and no surface tension is present, proofs of the existence of solitary wave solutions to the Euler equations are due to Lavrentiev \cite{lavrentiev}, Friedrich and Hyers \cite{friedrichs-hyers}, Beale \cite{beale-solitary-existence}, Mielke \cite{mielke-reduction}, and Pego and Sun \cite{pego-sun}.
In the ``large'' surface tension regime for $\beta > 1/3$ with $F < 1$, Amick and Kirchg\"{a}ssner \cite{amick-kirchgassner}, Sachs \cite{sachs}, and Sun and Shen \cite{sun-shen-surface-tension} obtained solitary waves as well.

The situation is very different in the ``small'' surface tension case $0 < \beta < 1/3$, where the problem \eqref{eqn: the problem} emerges, and where solitary waves do not.
Instead, for $0 < \beta < 1/3$ and $F > 1$ with $F-1$ small, the water wave problem is known to have nanopteron, or generalized solitary wave, solutions; the profile of a nanopteron traveling wave is the superposition of an exponentially localized ``core'' term and a small-amplitude far-field periodic ``ripple'' term.
The terminology ``nanopteron'' is due to Boyd \cite{boyd}.
Beale \cite{beale} proved the existence of nanopteron solutions for the water wave problem whose ripple amplitude is small beyond all algebraic orders of $F-1$.

This left open the possibility that the ripple's amplitude was zero, in which case the nanopteron would really be a solitary wave.
However, Sun and Shen \cite{sun-shen-ww-expn-small} and Lombardi \cite{Lombardi-orbits97}, using different techniques, subsequently proved that the nanopteron's ripple amplitude is exponentially small in $F-1$, not just small beyond all algebraic orders.
Sun \cite{sun-ww-nonexistence} later established that the water wave problem with small surface tension has no solitary wave solutions, unlike the large surface tension problem.

The nanopteron's periodic ripple can be parameterized in both amplitude and phase shift.
Beale fixed the (possibly zero) phase shift and, in the process of constructing the nanopteron, determined the amplitude based on the phase shift.
Sun \cite{sun-ww-existence} did the reverse and started by fixing the amplitude with the subsequently determined phase shift depending on amplitude.
This dichotomy of solitary waves for large surface tension against nanopterons for small surface tension, as well as the interaction of the ripple's amplitude and phase shift, are central to existing results for our problem \eqref{eqn: the problem} and the new perspectives that we offer here.

The problem \eqref{eqn: the problem} is connected to water waves first via the fifth-order KdV (fKdV) equation $v_t +vv_x + v_{xxx}+\partial_x^5[v] = 0$, which Hunter and Scheurle \cite[Sec.\@ 2]{hunter-scheurle} derived from the water wave equations assuming $F > 1$ and $0 < \beta < 1/3$ with $1/3-\beta$ small.
Appropriately scaled, the problem \eqref{eqn: the problem} is the traveling wave equation for fKdV.
Hunter and Scheurle then showed that \eqref{eqn: the problem} has a nanopteron solution whose periodic amplitude is $\O(\ep)$, if not smaller.
Additionally, in the small surface tension regime, Amick and Kirchg\"{a}ssner \cite[Sec.\@ 4]{amick-kirchgassner} showed that a nonhomogeneous version of \eqref{eqn: the problem} controls the center manifold behavior of the water wave problem.
Boyd provides lengthy bibliographies on fKdV in \cite[Tbl.\@ 6.1, 10.1, 10.2, 10.3]{boyd} with many entries in addition to our citations here.

\subsubsection{Prior results for \eqref{eqn: the problem}}
The immediate value of the problem \eqref{eqn: the problem}, both historically and currently for us, is its remarkably simple structure, as compared to the water wave equations and related problems of interest.
Given this accessible structure and water wave heritage, the natural questions about \eqref{eqn: the problem} were if it has nanopteron solutions and, if so, just how small the ripple amplitude is, and how the ripple's amplitude and phase shift might be related.
Amick and Toland \cite{amick-toland} first addressed existence by constructing nanopterons, using an approach inspired by Beale's, with ripple amplitude small beyond all algebraic orders of $\ep$.
Like Beale, they fixed a (possibly zero) phase shift for the ripple and then developed amplitude.
Sun and Shen \cite{sun-shen-at-expn-small} sharpened the ripple's estimate to be exponentially small, given a phase shift, and also obtained analyticity of their solutions.

Preceding these results, Amick and McLeod \cite{amick-mcleod} proved that \eqref{eqn: the problem} has no solitary wave solutions, and so the ripple amplitudes for Amick and Toland and Sun and Shen must be nonzero.
Sun \cite{sun-at-phase-shift} gave a detailed formal asymptotic relationship between the ripple amplitude and phase shift, as well as a formal argument for the nonexistence of solutions vanishing at $\pm\infty$, and Eckhaus \cite{eckhaus}, Hammersley and Mazzarino \cite{hammersley-mazzarino}, and Pomeau, Ramani, and Grammaticos \cite{pomeau-et-al} all also considered the nonexistence problem.
The existence results of Amick and Toland and Sun and Shen do not provide lower bounds on the ripple amplitude, which (lacking these nonexistence results) could permit their nanopterons to be solitary waves.
However, the problem \eqref{eqn: the problem} fits into the broad framework of Lombardi \cite[Ch.\@ 7, p.\@ 203]{lombardi}, whose machinery first fixes a nonzero, exponentially small amplitude and then constructs nanopterons with the phase shift depending on the amplitude.

Over three decades since these initial existence and nonexistence results, interest in \eqref{eqn: the problem} remains.
Fodor, Forg\'{a}cs, and Mushtaq numerically computed higher-order corrections to the ripple's exponentially small amplitude in \cite{fodor-et-al-corrections} and revisited in \cite{fodor-et-al} the asymptotics of \cite{hammersley-mazzarino} to explore precise amplitude values for ``asymmetric'' solutions to \eqref{eqn: the problem} defined on $[0,\infty)$.
Baldom\'{a}, Guardia, and Pelinovsky \cite{pelinovsky-et-al-kawahara} included an additional cubic nonlinear term in \eqref{eqn: the problem}, making it the traveling wave equation for the Kawahara equation, and found a countable sequence of values of $\ep$ for which this new problem does have solitary waves.
Deng and Sun \cite{deng-sun-multi-hump-at} outlined a formal construction of ``multi-hump'' solutions to \eqref{eqn: the problem}, which are the superposition of multiple localized cores with families of both far-field ripples and ``linking'' ripples connecting the cores.

\subsubsection{Toward a modern analysis of \eqref{eqn: the problem}}
Considering all of these results on existence, nonexistence, and amplitude-phase relations, can anything new be said about nanopteron solutions to \eqref{eqn: the problem}?
Our contention is ``much.''

\begin{enumerate}[label=$\bullet$]

\item
We provide the first comprehensive treatment of the problem \eqref{eqn: the problem} that simultaneously considers both amplitude and phase shift selection.
Previously, Amick and Toland \cite{amick-toland} and Sun and Shen \cite{sun-shen-at-expn-small} fixed the phase shift and then determined amplitude.
Sun's intent in developing amplitude-phase shift relations in \cite{sun-at-phase-shift} was detailed, formal asymptotics, not rigorous existence proofs.
And Lombardi's treatise \cite{lombardi} does not explicitly treat the problem \eqref{eqn: the problem} but rather subsumes it into a larger class of related problems via a complicated nonlinear normal form change of variables; the price of this powerful generality is that it can be difficult to follow exactly how Lombardi's methods apply to the specific problem \eqref{eqn: the problem}.
We discuss in Section \ref{sec: prior approaches} more precisely how our methods compare to those of our predecessors.
Throughout, we highlight exactly how the amplitude and phase shift selection processes are similar and where they diverge.

\item
We incorporate a number of refinements into the existence proofs that are natural outgrowths of nanopteron developments in the decades since the aforementioned results about \eqref{eqn: the problem}.
We adapt the methods of \cite{faver-wright, hoffman-wright, johnson-wright} to prove our main results, Theorems \ref{thm: main amp} and \ref{thm: main PS}, and incorporate Lombardi's oscillatory integral framework \cite{lombardi} into those methods.
We transparently isolate the role of the Lyapunov--Schmidt decomposition \cite[Thm.\@ I.2.3]{kielhofer} and Crandall--Rabinowitz--Zeidler-style ``bifurcation from a simple eigenvalue'' \cite{crandall-rabinowitz, zeidler} techniques that produce the periodic ripples.
We make a ``refined leading-order limit'' that shrinks the size of the amplitude by an $\O(\ep^2)$ factor from what it would otherwise be.
Unlike Amick and Toland \cite[Sec.\@ 3.3]{amick-toland}, we do not need to separate the zero phase shift case from the nonzero when selecting amplitude in terms of phase shift.
And we streamline the final construction of the nanopteron by viewing a key linear operator as a perturbation of the well-understood linearization of KdV at its solitary wave solution.

\item
Our organization of the results reflect our priorities: we defer the majority of the technical proofs to the appendices so that we can concentrate in the main narrative on the functional-analytic arguments that produce nanopterons.
The narrative illustrates how we use certain key estimates and perturbations in a fairly abstract context; the appendices prove those estimates in leisurely detail.
Our intention is that the overall narrative can serve as an accessible template for nanopteron constructions in a broad context, while the appendices highlight the customized estimates that often appear in those constructions.
\end{enumerate}

\subsubsection{The lattice connection}
Regarding the last point, our desire to revisit \eqref{eqn: the problem} from more modern and unified perspectives arises from an interest in lattice dynamics.
A Fermi--Pasta--Ulam--Tsingou (FPUT) lattice \cite{fput-original, dauxois} is an infinite chain of coupled oscillators that serves as a canonical model for wave propagation in discrete, (possibly) heterogeneous media \cite{chong-kev-book, chong-kev-dynamics, cpkd}.
The dynamical behavior of a lattice and the coherent structures that it supports depend closely on the lattice's material data: the oscillator masses and the coupling potentials \cite{vainchtein-survey, pankov}.
Solutions to appropriate KdV equations are known to be good approximations over long times to solutions to the FPUT equations of motion with suitably scaled initial conditions when the lattice's material data is ``periodic'' in that the masses and potentials repeat after a fixed number of sites \cite{schneider-wayne, gmwz, chirilus-bruckner-etal}.

This KdV approximation has been extended to two classes of FPUT lattices in the following ways.
First, Friesecke and Pego \cite{friesecke-pego1} proved that in the monatomic lattice, in which all masses are the same and all potentials are equal, has solitary traveling wave solutions that are a small perturbation of a KdV $\sech^2$-type solitary wave that emerges in the long wave limit  where the wave speed is a small perturbation of the lattice's ``speed of sound.''
Second, symmetric dimer lattices, in which either the masses alternate and the potentials are the same (the mass dimer), or the potentials alternate and the springs are the same (the spring dimer), have nanopteron traveling wave solutions.
(This switch from solitary waves to nanopterons in lattices when material heterogeneity is introduced is akin to the different results for water waves depending on the value of the Bond number relative to $1/3$.)
In the long wave limit, this is due to Faver and Wright \cite{faver-wright} for the mass dimer and to Faver \cite{faver-spring-dimer} for the spring dimer.
There are also nanopterons in ``material'' limits, in which the dimer is close to a monatomic lattice via a perturbation in the material data.
Hoffman and Wright \cite{hoffman-wright} studied the small mass limit for the mass dimer, where one of the masses vanishes, and the lattice becomes monatomic, and Faver and Hupkes studied the equal mass limit for the mass dimer, in which the two alternating masses \cite{faver-hupkes-equal-mass}.
(Strictly speaking, the equal mass result is a ``micropteron'' \cite{boyd}.)

The nanopteron constructions in \cite{faver-wright, faver-spring-dimer, hoffman-wright, faver-hupkes-equal-mass} are all considerable refinements of the original programs of Beale \cite{beale} and Amick and Toland \cite{amick-toland}; in particular, all of these constructions assume a zero phase shift and determine amplitude.
We make some further technical comparisons of the problem \eqref{eqn: the problem} to the lattice problems in Section \ref{sec: solv cond and nano}.
Subsequent results via spatial dynamics based on Lombardi's technology \cite{faver-hupkes-spatial-dynamics} by Faver and Hupkes and center manifold theory \cite{deng-sun-fput} by Deng and Sun have refined the long wave nanopteron's ripple amplitude to be exponentially small and nonzero.

Despite this progress, numerous open questions about dimer nanopterons remain (to say nothing of how the KdV approximation might persist in lattices with more complicated material data).
It is unclear if solitary waves in dimers definitely do not exist, as for \eqref{eqn: the problem} and the water wave problem with small surface tension.
Perhaps solitary waves do exist, but only at special parameter values, which is the situation in the related mass-in-mass lattice \cite{faver-mim-nanopteron, faver-goodman-wright}.
It is likely that an arbitrary phase shift can be selected before determining amplitude, but the inherently nonlocal structure of the lattice traveling wave problem will complicate interactions with the asymptotic phase-shifted behavior.
While Lombardi's methods can be applied to obtain better amplitude estimates in the long wave problem, they will not (ostensibly) work for the material limits, as Lombardi's framework actually presupposes a long wave scenario \cite[Eqn.\@ (7.4), Sec.\@ 8.1.2.2]{lombardi}.
Finally, although all of the nanopteron constructions in \cite{faver-wright, faver-spring-dimer, hoffman-wright, faver-hupkes-equal-mass, faver-mim-nanopteron} employ many common techniques, it is easy to get lost in the abundance of estimates, notation, and changes of variables that the different underlying problems and limits demand.

Based, therefore, on the current lattice nanopteron progress, and the remaining open problems, we believe that a revisiting of the problem \eqref{eqn: the problem} is propitious, both to give this rich problem a refreshed, modern treatment and to unify the occasionally disparate lattice techniques in a conceptually simple setting for future reference.

\subsection{The main results}\label{sec: main results}
The leading-order term in the localized ``core'' of our nanopteron solutions to the problem \eqref{eqn: the problem} will be the KdV solitary wave profile
\begin{equation}\label{eqn: sigma}
\sigma(z) 
:= \frac{3}{2}\sech^2\left(\frac{z}{2}\right).
\end{equation}
Taking $u = \sigma$ solves the problem \eqref{eqn: the problem} at $\ep = 0$.
That is,
\[
\sigma''-\sigma+\sigma^2
= 0
\]

For $b > 0$, let $\U_b \subseteq \C$ be the open strip
\begin{equation}\label{eqn: Ub}
\U_b
:= \set{z \in \C}{|\im(z)| < b}.
\end{equation}
Then $\sigma$ is analytic on $\U_{\pi}$ and exponentially localized in the sense that for $0 < q < 1$ and $0 < b < \pi$, there is $C > 0$ such that
\[
|\sigma(z)|
\le Ce^{-q|z|}, \ z \in \U_b.
\]

To build the phase shift into our nanopterons, we use the ``phase selector'' \cite[p.\@ 1172]{sun-at-phase-shift} function
\begin{equation}\label{eqn: tau}
\tau(z)
:= \tanh\left(\frac{z}{2}\right).
\end{equation}
This is the same phase selector as Sun and Shen \cite{sun-shen-at-expn-small} and Sun \cite{sun-at-phase-shift} use, but it is different from Amick and Toland's \cite[p.\@ 47]{amick-toland}, which is infinitely differentiable but not analytic.
With
\begin{equation}\label{eqn: tau-infty}
\tau_{\infty} 
:= \frac{2}{\norm{\sigma}_{L^1}}
= 3,
\end{equation}
we have
\begin{equation}\label{eqn: tau prime}
\tau'
= \tau_{\infty}\sigma.
\end{equation}
In our subsequent work, the identity \eqref{eqn: tau prime} allows us to focus on the key relationship between the phase selector and the leading-order localized core in the nanopteron and less on exact formulas.

We now have the notation to state our main results.
We first give the nanopteron construction when the phase shift is preselected and the amplitude is determined from that.
After stating the theorem we discuss some nuances of its language and estimates.
We defer an actual explanation for the nanopteron's existence and the failure of a perturbative construction of solitary waves to Section \ref{sec: solv cond and nano}, when we will have established some more notation and fundamental results.
 
\begin{theorem}[Amplitude selection given phase shift]\label{thm: main amp}
Let $0 < q < 1$, $0 < b < \pi$, and $0 \le \theta < \pi$.
There exists $\ep_{\star} > 0$ such that for $0 < \ep < \ep_{\star}$, there are $\alpha_{\ep} \in \R$ and analytic functions $\Phi_{\ep}$, $\Upsilon_{\ep} \colon \U_b \to \R$ such that 
\begin{equation}\label{eqn: main amp u}
u(z)
= \sigma(z) + \alpha_{\ep}e^{-b/\ep}\Phi_{\ep}(z+\ep\theta\tau(z))+\Upsilon_{\ep}(z)
\end{equation}
solves \eqref{eqn: the problem}.
There is $C > 0$ such that the following also hold for $0 < \ep < \ep_{\star}$.

\begin{enumerate}[label={\bf(\roman*)}]

\item
{}
[Estimates on the periodic ripple and frequency]
The periodic ripple $\Phi_{\ep}$ solves \eqref{eqn: the problem}, and there is $\Omega_{\ep} > 0$ such that $\Phi_{\ep}$ is $2\pi/\Omega_{\ep}$-periodic with
\begin{equation}\label{eqn: main amp per}
|\Phi_{\ep}(x)-\cos(\Omega_{\ep}x)| \le C|\alpha_{\ep}|, \ x \in \R,
\quadword{and}
\left|\Omega_{\ep}-\frac{1}{\ep}\right| \le \ep+C|\alpha_{\ep}|e^{-b/\ep}.
\end{equation}

\item
{}
[Estimates on the periodic amplitude]
The periodic amplitude coefficient $\alpha_{\ep}$ satisfies
\begin{equation}\label{eqn: main amp amp}
|\alpha_{\ep}|
\le \begin{cases}
C\ep^4, \ \theta = 0 \\
C\ep^5, \ 0 < \theta < \pi.
\end{cases}
\end{equation}

\item
{}
[Estimates on the localized remainder]
The exponentially localized remainder $\Upsilon_{\ep}$ satisfies
\begin{equation}\label{eqn: main amp error}
|\Upsilon_{\ep}(x)|
\le C\ep^2e^{-q|x|}, \ x \in \R.
\end{equation}
\end{enumerate}
\end{theorem}

We prove this theorem in Appendix \ref{app: proof of main amp} as the culmination of many intermediate results.
The actual mechanism for nanopteron construction, i.e., why the problem \eqref{eqn: the problem} can have nanopteron solutions (since we know that solitary wave solutions are not possible), is best understood in the context of some more tools, which we develop in Section \ref{sec: linear}.
For now, the following commentary will augment the results in this theorem.

\begin{enumerate}[label=$\bullet$]

\item
Since the frequency $\Omega_{\ep}$ of the periodic ripple is a small perturbation of $1/\ep$, by the estimates in \eqref{eqn: main amp per} and \eqref{eqn: main amp amp}, the factor of $\ep$ on the phase shift $\theta$ in \eqref{eqn: main amp u} ensures that as $\re(z) \to \pm\infty$, the solution $u(z)$ behaves like $\alpha_{\ep}e^{-b/\ep}\cos(\Omega_{\ep}z+\theta)$.
Consequently, we would only need to consider phase shifts in $[0,2\pi)$.

\item
However, we actually consider only phase shifts in $[0,\theta)$ due to a certain ``shift-reflection symmetry'' that the ripples possess, due to which using a phase shift in $[\pi,2\pi)$ effectively converts one ripple into another.
This hinges on the identity $\cos(Z+\pi) = -\cos(Z)$.

\item
The decay of the localized remainder at rate $q$, as made precise in \eqref{eqn: main amp error}, is directly related to the decay of the leading-order term $\sigma$ in the core: the remainder decays slightly more slowly than the core.

\item
The parameter $b$ measures the ``exponential smallness strength'' of the amplitude, and, due to estimates from Lombardi (which we present in Theorem \ref{thm: lombardi} and Corollary \ref{cor: lombardi}), this is closely related to the width of strips on which $\sigma$ is analytic and localized.
Analyticity is a key ingredient for our proofs, following Lombardi, and is much more than a nice additional feature of our solutions.

\item
The estimate \eqref{eqn: main amp amp}, in which the amplitude is one power of $\ep$ smaller for a nonzero phase shift than the zero phase shift, agrees with Sun's asymptotics in \cite[Eqn.\@ (44), (45)]{sun-at-phase-shift}.
While we do not prove this here, and so do not state it in \eqref{eqn: main amp amp}, the nonexistence of solitary wave solutions to \eqref{eqn: the problem} due to Amick and McLeod \cite{amick-mcleod} guarantees that $\alpha_{\ep} \ne 0$.
\end{enumerate}

Now we give the nanopteron construction when the amplitude is preselected to be exponentially small but nonzero and the phase shift is determined from that.
As the language and estimates are largely the same as those of Theorem \ref{thm: main amp}, we do not give any further discussion for this result, and the proof is essentially the same as the work in Appendix \ref{app: proof of main amp}.

\begin{theorem}[Phase shift selection given amplitude]\label{thm: main PS}
Let $0 < q < 1$, $0 < b < \pi$, and $0 < A_0 < A_1 < 1$.
There exists $\ep_{\star} > 0$ such that for $0 < \ep < \ep_{\star}$ and $A_0 < A < A_1$, there are $\theta_{\ep} \in \R$ and analytic functions $\Phi_{\ep}$, $\Upsilon_{\ep} \colon \U_b \to \R$ such that 
\[
u(z)
= \sigma(z) + A\ep^4e^{-b/\ep}\Phi_{\ep}(z+\ep\theta_{\ep}\tau(z))+\Upsilon_{\ep}(z)
\]
solves \eqref{eqn: the problem}.
There is $C > 0$ such that the following also hold for $0 < \ep < \ep_{\star}$.

\begin{enumerate}[label={\bf(\roman*)}]

\item
{}
[Estimates on the periodic profile and frequency]
The periodic profile $\Phi_{\ep}$ solves \eqref{eqn: the problem}, and there is $\Omega_{\ep} > 0$ such that $\Phi_{\ep}$ is $2\pi/\Omega_{\ep}$-periodic with
\[
|\Phi_{\ep}(x)-\cos(\Omega_{\ep}x)| \le C\ep^4, \ x \in \R,
\quadword{and}
\left|\Omega_{\ep}-\frac{1}{\ep}\right| \le \ep+C\ep^4e^{-b/\ep}.
\]

\item
{}
[Estimates on the phase shift]
The phase shift $\theta_{\ep}$ satisfies
\[
|\theta_{\ep}|
\le C\ep.
\]

\item
{}
[Estimates on the localized remainder]
The exponentially localized remainder $\Upsilon_{\ep}$ satisfies
\[
|\Upsilon_{\ep}(x)|
\le C\ep^2e^{-q|x|}, \ x \in \R.
\]
\end{enumerate}
\end{theorem}

\subsection{The role of analyticity}
We discuss the importance of analytic functions for our arguments and then describe the function spaces in which most of our analysis takes place.

\subsubsection{Analyticity and amplitude}\label{sec: an and amp}
One way to ensure that the ripple amplitude is exponentially small is to assume it from the start: posit that the amplitude has the form $\alpha{e}^{-b/\ep}$.
Here $\alpha$ is a ``tuning parameter'' that affords some extra flexibility in the amplitude, while $b > 0$ measures the ``exponential smallness strength'' of the amplitude.
The problem with this approach is that, perhaps unsurprisingly, it introduces exponentially large factors of $e^{b/\ep}$ elsewhere in the analysis (e.g., when solving for $\alpha$).

Fortunately, such factors always come paired as a product with an oscillatory integral of the form, roughly,
\[
\int_{-\infty}^{\infty} f(x)e^{i(b/\ep)x} \dx.
\]
If $f$ is analytic and vanishes exponentially fast as $\re(z) \to \pm \infty$ on the strip $\U_b$ from \eqref{eqn: Ub}, an estimate due to Lombardi (Theorem \ref{thm: lombardi} and Corollary \ref{cor: lombardi} below) shows that such integrals are $\O(e^{-b/\ep})$, and this will cancel out the exponentially large factor $e^{b/\ep}$.

In practice, one component of $f$ will be the core $\sigma$, which is analytic and vanishes exponentially fast on strips of width less than $\pi$.
This suggests restricting the exponential smallness strength $b$ to $0 < b < \pi$.

Another component of $f$ will be the ripple.
Following Lombardi \cite[Ch.\@ 7]{lombardi}, we first construct the ripples on the real line and then extend them to be analytic on a strip of width $b$.
A natural requirement of this extension program is that the amplitude be $\O(e^{-b/\ep})$.
Thus we have the somewhat circular notion that in order to have exponentially small amplitudes, we need analytic ripples, and to have analytic ripples we need exponentially small amplitudes.

\subsubsection{Spaces of localized analytic functions and integrals}\label{sec: function spaces}
We work on the strips $\U_b$ defined in \eqref{eqn: Ub}.
For $f \colon \U_b \to \C$ and $q \in \R$, put
\[
\norm{f}_{q,b}
:= \sup_{|\im(z)| < b} e^{q|\re(z)|}|f(z)|.
\]
We use the Banach space of analytic functions that either vanish exponentially fast ($q > 0$) or can grow exponentially large $(q < 0$) as $\re(z) \to \pm\infty$:
\begin{equation}
\H_{q,b} 
:= \set{f \colon \U_b \to \C }{f \text{ is analytic and } \norm{f}_{q,b} < \infty}.
\end{equation}
More often we remaini in the even space
\begin{equation}
\E_{q,b}
:= \set{f \in \H_{q,b}}{f(z) = f(-z), \ z \in \U_b}.
\end{equation}
We will occasionally need spaces of functions whose higher derivatives also vanish exponentially fast as $\re(z) \to \pm\infty$, so for $f \in \H_{q,b}$, we put
\begin{equation}
\norm{f}_{q,b,r}
:= \sum_{j=0}^r \norm{\partial_z^j[f]}_{q,b}
\quadword{and}
\H_{q,b}^r
:= \set{f \in \H_{q,b}}{\norm{f}_{q,b,r} < \infty}.
\end{equation}
We also take
\begin{equation}
\E_{q,b}^r
:= \H_{q,b}^r \cap \E_{q,b}.
\end{equation}
We prove a number of properties of these spaces in Appendix \ref{app: Hqb}.

The exact thresholds for decay rate $q$ and strip width $b$ that we will use hinge on those for the analyticity and decay of $\sigma$ defined in \eqref{eqn: sigma}.

\begin{lemma}\label{lem: q-sigma b-sigma}
Let $q_{\sigma} := 1$ and $b_{\sigma} := \pi$.
Then $\sigma \in \E_{q,b}$ for $0 < q < q_{\sigma}$ and $0 < b < b_{\sigma}$.
\end{lemma}

We notationally single out $q_{\sigma}$ and $b_{\sigma}$, rather than using $1$ and $\pi$, to emphasize how these thresholds arise in connection with $\sigma$.
Occasionally it will be useful to have several different decay rates in play simultaneously, which we can achieve via the following elementary inequalities: if $0 < q < q_{\sigma}$, then
\begin{equation}\label{eqn: elementary q ineq}
0
< q
< \frac{q+q_{\sigma}}{2}
< \frac{q+3q_{\sigma}}{4}
< q_{\sigma}.
\end{equation}

We frequently employ integrals over infinite horizontal rays.
For $f \in \H_{q,b}$ and $z \in \U_b$, put
\[
\int_{(-\infty,z]} f(w) \dw
:= \int_{-\infty}^{\re(z)} f(s+i\im(z)) \ds 
\]
and
\[
\int_{[z,\infty)} f(w) \dw
:= \int_{\re(z)}^{\infty} f(s+i\im(z)) \ds.
\]
The exponential decay of $f$ ensures that these integrals converge, and the analyticity of $f$ ensures the analyticity of these integrals as functions of $z$ (more precisely, one can check the Cauchy--Riemann equations with the fundamental theorem of calculus and differentiating under the integral).

\subsection{Notation}
We mention a few aspects of our largely standard notation.

\begin{enumerate}[label=$\bullet$]

\item
If $\X$ is a normed space and $f \colon \U \subseteq \C \to \C$ is differentiable, then we write $f' = \partial_z[f]$.

\item
If $\X$ and $\Y$ are normed spaces and $\T \colon \X \to \Y$ is a bounded linear operator, then $\norm{\T}_{\X \to \Y}$ is the operator norm of $\T$ with respect to the norms on $\X$ and $\Y$.

\item
For a function $f \colon \U \subseteq \C \to \C$, we say that a Lipschitz estimate on $f$ is an estimate of the form
\[
|f(z)-f(\grave{z})|
\le C|z-\grave{z}|
\]
for some $C > 0$, valid for all $z$, $\grave{z} \in \U$, and this estimate is a contraction estimate if $0 < C < 1$.

\item
A constant $C$ in an estimate will never depend on the small parameter $\ep$, although $C$ might depend on the analytic strip width $b$ and decay rate $q$ parameters.
\end{enumerate}

\subsection{Outline}
We highlight the content of the remaining sections of the main narrative and some of the appendices.

\begin{enumerate}[label=$\bullet$]

\item
Section \ref{sec: linear} studies the linearization of the problem \eqref{eqn: the problem} and introduces a solvability condition in the form of an oscillatory integral that induces the nanopteron.

\item
Section \ref{sec: per} presents results on periodic solutions to \eqref{eqn: the problem} that yield the nanopteron's ripples.

\item
Section \ref{sec: ref LO limit} makes a ``refined leading-order limit'' that shrinks the size of some estimates.

\item
Section \ref{sec: PS nano} makes the phase-shifted nanopteron ansatz for \eqref{eqn: the problem} and converts the solvability condition into a ``selection mechanism'' for amplitude and phase shift.

\item
Section \ref{sec: amplitude problem} constructs nanopteron solutions to \eqref{eqn: the problem} by fixing the phase shift and then determining the amplitude.
This leads to the proof of Theorem \ref{thm: main amp}.

\item
Section \ref{sec: phase shift problem} constructs nanopteron solutions to \eqref{eqn: the problem} by fixing the amplitude to be small but nonzero and then determining the phase shift.
This leads to the proof of Theorem \ref{thm: main PS}.

\item
Section \ref{sec: prior approaches} provides further comparisons with the methods of Amick and Toland \cite{amick-toland}, Sun and Shen \cite{sun-shen-at-expn-small}, Sun \cite{sun-at-phase-shift}, and Lombardi \cite{lombardi} that are better understood toward the end of the narrative, in the full context of our new developments here.

\item
The appendices prove the many technical estimates that the main narrative uses in broader functional-analytic contexts.
In particular, Appendix \ref{app: periodics} constructs the periodic ripples and proves the results stated in Section \ref{sec: per}; Lemma \ref{lem: ultimate fake quadratic} and Appendix \ref{app: contr ests} contain many prototypical estimates that permeated the prior nanopteron constructions in \cite{faver-wright, faver-spring-dimer, faver-hupkes-equal-mass, faver-mim-nanopteron, johnson-wright}; and Appendix \ref{app: op ests} develops the estimates that guarantee the invertibility of the operator that ultimately controls the nanopteron existence proofs.
\end{enumerate}

\section{The Linear Problem}\label{sec: linear}
The linearization of the problem \eqref{eqn: the problem} at $u=0$ is 
\begin{equation}\label{eqn: B-ep}
\B_{\ep}f
:= \ep^2f^{(4)}+f''-f.
\end{equation}
This operator is absolutely central to our work, as the nature of its range is, from a certain point of view, ultimately responsible for the formation of nanopteron solutions to \eqref{eqn: the problem}.

\subsection{The range of the operator $\B_{\ep}$}
We characterize the range of the operator $\B_{\ep}$ defined in \eqref{eqn: B-ep} and give a formula for the inverse of this operator on its range.
The characteristic polynomial for this differential operator is
\begin{equation}\label{eqn: tB}
\tB_{\ep}(z) 
:= \ep^2z^4+z^2-1.
\end{equation}
By the quadratic formula, this factors as 
\begin{equation}\label{eqn: tB factor}
\tB_{\ep}(z) 
= \ep^2(z^2+\omega_{\ep}^2)(z^2-\mu_{\ep}^2),
\end{equation}
for certain $\omega_{\ep}$, $\mu_{\ep} > 0$, whose essential properties we now pause to describe.

We prove the following results about $\omega_{\ep}$ in Appendix \ref{app: omega-ep}.

\begin{lemma}\label{lem: omega-ep}
For $\ep \ne 0$, the number
\begin{equation}\label{eqn: omega-ep}
\omega_{\ep}
:= \sqrt{\frac{\sqrt{1+4\ep^2}+1}{2\ep^2}},
\end{equation}
which we call the critical frequency for the problem \eqref{eqn: the problem}, satisfies the following.

\begin{enumerate}[label={\bf(\roman*)}, ref={(\roman*)}]

\item
If $\omega > 0$ and $\ep^2(i\omega)^4+(i\omega)^2+1=0$, then $\omega=\omega_{\ep}$.

\item
If $0 < \ep < 1$, then
\begin{equation}\label{eqn: ep-omega-ep bounds}
1 
< \ep\omega_{\ep}
< 2.
\end{equation}

\item\label{part: lem omega 2 - 1}
$2(\ep\omega_{\ep})^2-1 = \sqrt{1+4\ep^2}$.

\item
If $\ep > 0$, then
\begin{equation}\label{eqn: omega-ep est}
0
< \omega_{\ep}-\frac{1}{\ep}
\le \ep
\end{equation}
and
\begin{equation}\label{eqn: omega-ep est2}
|\ep\omega_{\ep}-1|
\le \ep^2.
\end{equation}
\end{enumerate}
\end{lemma}

We prove the following results about $\mu_{\ep}$ in in Appendix \ref{app: mu-ep}.

\begin{lemma}\label{lem: mu-ep}
Let $\ep \ne 0$ and define
\begin{equation}\label{eqn: mu-ep}
\mu_{\ep}
:= \sqrt{\frac{\sqrt{1+4\ep^2}-1}{2\ep^2}}.
\end{equation}
Then $\mu = \mu_{\ep} < 1$ is the unique positive solution to $\ep^2\mu^4+\mu^2+1 = 0$, and 
\begin{equation}\label{eqn: mu-ep lim}
|\mu_{\ep}-1| 
\le \ep^2.
\end{equation}
\end{lemma}

The factorization \eqref{eqn: tB factor} for the characteristic polynomial $\tB_{\ep}$ of the differential operator $\B_{\ep}$ then leads to the factorization
\begin{equation}\label{eqn: B-ep factor}
\B_{\ep}
= \ep^2(\partial_z^2+\omega_{\ep}^2)(\partial_z^2-\mu_{\ep}^2).
\end{equation}
for $\B_{\ep}$.
This will be useful for developing a solvability condition for the linear problem $\B_{\ep}f=g$ and then for inverting $\B_{\ep}$ on its range.

Toward the former, suppose that $\B_{\ep}f = g$ for some $f$, $g \in \H_{q,b}$.
We caution that $\B_{\ep}f \not\in \H_{q,b}$ for $f \in \H_{q,b}$ without additional hypotheses on the derivatives of $f$, although by Lemma \ref{lem: Hqb containment} we do have $\B_{\ep}f \in \H_{q,\grave{b}}$ for any $0 < \grave{b} < b$.
By that lemma, $f$ and all of its derivatives vanish exponentially fast on $\R$.
We may therefore integrate by parts (or use the Fourier transform) and apply the result that $i\omega_{\ep}$ is a root of the characteristic polynomial of $\B_{\ep}$ to conclude that 
\[
\int_{-\infty}^{\infty} g(x)e^{\pm{i}\omega_{\ep}x} \dx
= 0.
\]
This sort of oscillatory integral will pervade our later analysis.

For now, we consider inverting $\B_{\ep}$ on its range.
From the partial fractions expansion
\[
\frac{1}{\tB_{\ep}(z)}
= -\frac{1}{\sqrt{1+4\ep^2}(z^2+\omega_{\ep}^2)}+\frac{1}{\sqrt{1+4\ep^2}(z^2-\mu_{\ep}^2)},
\]
we expect that the inverse of $\B_{\ep}$ on its range is, formally,
\[
\B_{\ep}^{-1}
= -\frac{(\partial_z^2+\omega_{\ep}^2)^{-1}}{\sqrt{1+4\ep^2}}+\frac{(\partial_z^2-\mu_{\ep}^2)^{-1}}{\sqrt{1+4\ep^2}}.
\]

We make this precise as follows.

\begin{theorem}\label{thm: B-ep overview}
For $\omega$, $\mu \in \R$ and $g \in \H_{q,b}$, define
\begin{equation}\label{eqn: V-omega}
(\V_{\omega}g)(z)
:= \int_{[z,\infty)} \sin(\omega(z-w))g(w) \dw
\end{equation}
and
\begin{equation}\label{eqn: L-mu}
(\L_{\mu}g)(z)
:= e^{-\mu{z}}\int_{(-\infty,z]} e^{\mu{w}}g(w) \dw 
+ e^{\mu{z}}\int_{[z,\infty)} e^{-\mu{w}}g(w) \dw.
\end{equation}

\begin{enumerate}[label={\bf(\roman*)}]

\item
If $g \in \H_{q,b}$, then $\L_{\mu}g \in \H_{q,b}$ and if $\mu \ne 0$, then $f = \mu^{-1}\L_{\mu}g$ solves $f''-\mu^2f = g$.

\item
If $g \in \H_{q,b}$ and $\medint_{-\infty}^{\infty} g(x)e^{\pm{i}\omega{x}} \dx = 0$, then $\V_{\omega}g \in \H_{q,b}$, and if $\omega \ne 0$, then $f = \omega^{-1}\V_{\omega}g$ solves $f''+\omega^2f = g$.

\item
If $g \in \H_{q,b}$ and $\medint_{-\infty}^{\infty} g(x)e^{\pm{i}\omega_{\ep}{x}} \dx = 0$, then 
\begin{equation}\label{eqn: B-ep-inv}
f
= -\frac{1}{\omega_{\ep}\sqrt{1+4\ep^2}}\V_{\omega_{\ep}}g
+\frac{1}{\mu_{\ep}\sqrt{1+4\ep^2}}\L_{\mu_{\ep}}g
=: \B_{\ep}^{-1}g
\end{equation}
is the unique solution in $\H_{q,b}$ to $\B_{\ep}f = g$.
Conversely, if $f$, $g \in \H_{q,b}$ with $\B_{\ep}f = g$, then $\medint_{-\infty}^{\infty} g(x)e^{\pm{i}\omega_{\ep}{x}} \dx = 0$.
\end{enumerate}
\end{theorem}

\begin{proof}
Checking that $\omega^{-1}(\V_{\omega}g)''+\omega(\V_{\omega}g) = g$ and $\mu^{-1}(\L_{\mu}g)''-\mu(\L_{\mu}g) = g$ pointwise is a direct calculation that holds for any $g \in \H_{q,b}$.
Likewise, checking that $f$ as defined in \eqref{eqn: B-ep-inv} solves $\B_{\ep}f = g$ is then one more direct calculation using the factorization \eqref{eqn: B-ep factor} and the identity
\[
\ep^2(\omega_{\ep}^2+\mu_{\ep}^2)
= \sqrt{1+4\ep^2},
\]
which follows from the definitions of $\omega_{\ep}$ in \eqref{eqn: omega-ep} and $\mu_{\ep}$ in \eqref{eqn: mu-ep}.

We prove that $\V_{\omega}g$, $\L_{\mu}g \in \H_{q,b}$ in Appendices \ref{app: V-omega} and \ref{app: L-mu}.
We emphasize here that the condition $\medint_{-\infty}^{\infty} g(x)e^{\pm{i}\omega{x}} \dx = 0$ is what guarantees that $\V_{\omega}g$ decays sufficiently fast to ensure membership in $\H_{q,b}$.

The uniqueness claim follows from the elementary theory of constant-coefficient differential operators: if $\B_{\ep}f = 0$ for $f \in \H_{q,b}$, then with $g = f''-\mu_{\ep}^2f$, we have $g''+\omega_{\ep}g = 0$.
Since $f$ is localized, so is $g$, and therefore $g=0$, hence $f=0$.

Finally, we discussed before \eqref{eqn: iota-ep} why the solvability condition $\medint_{-\infty}^{\infty} g(x)e^{\pm{i}\omega_{\ep}{x}} \dx = 0$ holds whenever $g \in \H_{q,b}$ satisfies $\B_{\ep}f = g$ for some $f \in \H_{q,b}$.
\end{proof}

While the singular perturbation prevents $\B_{\ep}$ from converging to $\B_0$ in any meaningful norm, we will show that $\B_{\ep}^{-1}$ effectively converges to $\B_0^{-1} = \L_1 = (\partial_z^2-1)^{-1}$.
We say ``effectively'' because $\L_1$ is defined on a larger function space than $\B_{\ep}^{-1}$, and so some care is needed to phrase this convergence correctly; see the proof of Theorem \ref{thm: KdV perturbation amp} for the exact details.
This resembles a phenomenon found in the lattice nanopteron problems: inverses converge better than the original operators.
Versions of this phenomenon appear in \cite[Cor.\@ 3.4]{friesecke-pego1}, \cite[Lem.\@ A.12]{faver-wright}, and, most important for our work here, \cite[Lem.\@ 5]{johnson-wright}.

\subsection{Oscillatory integrals and the solvability condition}
The problem \eqref{eqn: the problem} respects an even symmetry: if $f \in \E_{q,b}$, then $\B_{\ep}f + f^2 \in \E_{q,b}$.
Consequently, we will only look for even nanopterons.
(In fact, Benilov, Grimshaw, and Kuznetsova \cite{benilov-et-al} proved that \eqref{eqn: the problem} cannot have ``asymmetric'' nanopteron solutions that asymptote to $0$ at $-\infty$ but a nontrivial periodic ripple at $+\infty$.)
Working with even functions reduces the number of solvability conditions to one: if $f \in \E_{q,b}$, then $\medint_{-\infty}^{\infty} f(x)e^{i\omega_{\ep}x} \dx = 0$ if and only if $\medint_{-\infty}^{\infty} f(x)e^{-i\omega_{\ep}x} \dx = 0$.
We therefore introduce the oscillatory integral functional
\begin{equation}\label{eqn: iota-ep}
\iota_{\ep}[f]
:= \int_{-\infty}^{\infty} f(x)e^{i\omega_{\ep}x} \dx
\end{equation}
to control this solvability condition: 
\begin{equation}\label{eqn: intro solv cond}
f, \ g \in \E_{q,b} \text{ with } \B_{\ep}
\quadword{if and only if}
\iota_{\ep}[g] = 0.
\end{equation}

We develop a key estimate on this functional from Lombardi's results on oscillatory integrals \cite[Lem.\@ 2.1.1]{lombardi}.

\begin{theorem}[Lombardi]\label{thm: lombardi}
Let $b$, $q$, $\omega > 0$ and $f \in \H_{q,b}$.
Then
\begin{equation}\label{eqn: Lombardi proto}
\left|\int_{-\infty}^{\infty} f(x)e^{i\omega{x}/\ep} \dx\right|
\le \frac{2}{q}e^{-b\omega/\ep}\norm{f}_{q,b}.
\end{equation}
\end{theorem}

\begin{corollary}\label{cor: lombardi}
Let $q$, $b$, $\ep > 0$.
If $f \in \E_{q,b}$, then
\begin{equation}\label{eqn: Lombardi}
|\iota_{\ep}[f]|
\le \frac{2}{q}e^{-b/\ep}\norm{f}_{q,b}.
\end{equation}
\end{corollary}

\begin{proof}
We use Lombardi's estimate \eqref{eqn: Lombardi proto} with $\omega = \ep\omega_{\ep}$ to estimate
\[
|\iota_{\ep}[f]|
= \left|\int_{-\infty}^{\infty} f(x)e^{i\omega_{\ep}x} \dx\right|
= \left|\int_{-\infty}^{\infty} f(x)e^{i(\ep\omega_{\ep})x/\ep} \dx\right|
\le \frac{2}{q}e^{-b(\ep\omega_{\ep})/\ep}\norm{f}_{q,b}
= \frac{2}{q}e^{-b\omega_{\ep}}\norm{f}_{q,b}
\]
and so
\[
|\iota_{\ep}[f]|
\le \frac{2}{q}e^{-b\omega_{\ep}}\norm{f}_{q,b}.
\]
Since $\omega_{\ep}-1/\ep > 0$ by \eqref{eqn: omega-ep est}, we have
\[
e^{-b\omega_{\ep}}
= e^{-b(\omega_{\ep}-1/\ep)}e^{-b/\ep}
\le e^{-b/\ep},
\]
and this yields \eqref{eqn: Lombardi}.
\end{proof}

\subsection{The solvability condition and nanopterons}\label{sec: solv cond and nano}
It is instructive to ignore the abundant prior results for the problem \eqref{eqn: the problem} and obstinately seek solitary wave solutions.
Perturbing from the $\sech^2$-type solution $\sigma$ from \eqref{eqn: sigma} to the KdV traveling wave equation $\sigma''-\sigma+\sigma^2 =0$, we suppose that $u = \sigma+f$ solves \eqref{eqn: the problem} with $f \in \E_{q,b}$ and find
\begin{equation}\label{eqn: doom}
\ep^2f^{(4)}+f''-f+2\sigma{f}
= -\ep^2\sigma^{(4)} - f^2.
\end{equation}
Put
\begin{equation}\label{eqn: Sigma}
\Sigma{f}
:= 2\sigma{f}.
\end{equation}
Then \eqref{eqn: doom} reads
\begin{equation}\label{eqn: doom2}
\B_{\ep}f+\Sigma{f}
= -\ep^2\sigma^{(4)} - f^2,
\end{equation}
and so $f$ must satisfy
\begin{equation}\label{eqn: doom3}
\iota_{\ep}[\Sigma{f}+\ep^2\sigma^{(4)}+f^2]
= 0.
\end{equation}
That is, the single perturbation term $f$ needs to meet the two equations \eqref{eqn: doom} and \eqref{eqn: doom3}, and so this problem of perturbing from $\sigma$ is overdetermined: it is unlikely that one $f \in \E_{q,b}$ could meet both \eqref{eqn: doom} and \eqref{eqn: doom3}.
This motivates a search for an extra unknown to close the overdetermined problem of \eqref{eqn: doom} and \eqref{eqn: doom3}.

Here is how this search plays out.
We could also interpret the solvability condition \eqref{eqn: intro solv cond} as a statement about the kernel of $\B_{\ep}$ on a periodic function space:
\[
\B_{\ep}\sin(\omega_{\ep}\cdot)
= \B_{\ep}\cos(\omega_{\ep}\cdot)
= 0.
\]
This motivated Beale's nanopteron ansatz in \cite{beale} to extend these exact sinusoidal solutions of the linear problem $\B_{\ep}f = 0$ to periodic solutions $p_{\ep}^a$ of the full ($\ep > 0$) nonlinear problem \eqref{eqn: the problem} parametrized in amplitude $a$, and then combine them with the localized solution $\sigma$ to the $\ep=0$ nonlinear problem along with a localized remainder term $\eta$.
Then one still needs to solve the two equations \eqref{eqn: doom} and \eqref{eqn: doom3} with $f = p_{\ep}^a + \eta$, but now there are two unknowns: the remainder $\eta$ and the amplitude $a$.
When we incorporate a phase shift, we will be able to choose between the amplitude and the phase shift as the second unknown.

\subsection{The solvability condition and KdV}
The problem \eqref{eqn: doom2} combines a singular perturbation with KdV.
 The linearization of the KdV traveling wave equation $u''-u+u^2=0$ at its solitary wave solution $\sigma$ is the operator
\begin{equation}\label{eqn: K0}
\K_0f
:= f''-f+2\sigma{f}
= (\partial_z^2-1)f+\Sigma{f},
\end{equation}
and so the linear operator controlling \eqref{eqn: doom} is
\[
\B_{\ep}+\Sigma
= \ep^2\partial_z^4+\K_0.
\]
That is, the singular perturbation and the KdV linearization appear together in the same equation.

We view this as a conceptual advantage of the original problem \eqref{eqn: the problem} over the dimer lattice problems of interest in \cite{faver-wright, faver-spring-dimer, hoffman-wright, faver-hupkes-equal-mass, faver-mim-nanopteron}.
Due to the alternating material structure of the dimer, those problems naturally involved two traveling wave profiles and thus two-component nonlocal systems.
After the right change of variables, those problems simplified into a component governed by KdV (or the KdV-like solitary wave for the monatomic lattice) and a component with a solvability condition like \eqref{eqn: doom3}.
Here, we only have one component with both KdV and a solvability condition.

This puts us more in line with Johnson and Wright's construction of long wave nanopterons for the Whitham problem \cite{johnson-wright}, which is a nonlocal one-component model for water waves that, coarsely, can be viewed as a nonlocal singular perturbation of KdV.
The Whitham problem also bundles KdV and a solvability condition in one equation.
Our modernized treatment of the problem \eqref{eqn: the problem} owes much to the innovations of Johnson and Wright, and their words \cite[p.\@ 107]{johnson-wright} explain the reason for the nanopteron's existence: ``What occurs is that when $\ep > 0$ the [localized core $\sigma$], through a sort of weak resonance, excites a very small amplitude periodic wave with frequency close to [the critical frequency $\omega_{\ep}$].'' 

\section{The Periodic Problem}\label{sec: per}

\subsection{The periodic problem on $\R$}
We first state and discuss a result on the existence of periodic solutions to \eqref{eqn: the problem} that are defined on $\R$.

\begin{theorem}\label{thm: periodics real}
There exist $\ep_{\per,0}$, $a_{\per,0} > 0$ such that for $0 < \ep < \ep_{\per,0}$ and $|a| < a_{\per,0}$, there is an even periodic function $\phi_{\ep}^a \in \Cal^{\infty}(\R)$ such that taking $u = a\phi_{\ep}^a$ solves \eqref{eqn: the problem}.
This solution has the following properties.

\begin{enumerate}[label={\bf(\roman*)}]

\item
For $0 < \ep < \ep_{\per,0}$ and $|a| < a_{\per,0}$, there are $\omega_{\ep}^a \in \R$ and an even $2\pi$-periodic function $\psi_{\ep}^a \in \Cal^{\infty}(\R)$ such that 
\[
\phi_{\ep}^a(x)
= \cos(\omega_{\ep}^ax) + \psi_{\ep}^a(\omega_{\ep}^ax)
\]
for all $x \in \R$.
The profile $\phi_{\ep}^a$ satisfies the ``shift-reflected'' identity
\begin{equation}\label{eqn: per shift refl}
\phi_{\ep}^a\left(x+\frac{\pi}{\omega_{\ep}^a}\right)
= -\phi_{\ep}^{-a}(x).
\end{equation}

\item
Each $\psi_{\ep}^a$ is a power series in $a$ of the form
\begin{equation}\label{eqn: psi-ep main}
\psi_{\ep}^a(X) 
= \sum_{n=1}^{\infty} a^n\psi_{n,\ep}(X),
\end{equation}
where each coefficient function $\psi_{n,\ep}$ is an even trigonometric polynomial of the form
\begin{equation}\label{eqn: psi-n-ep main}
\psi_{n,\ep}(X) 
= \sum_{k=1}^{n+1} c_{n,k,\ep}\cos(kX), \ c_{n,k,\ep} \in \R.
\end{equation}

\item
The frequency $\omega_{\ep}^a$ has the expansion
\begin{equation}\label{eqn: omega-ep-a xi-ep-a}
\omega_{\ep}^a = \omega_{\ep} + \xi_{\ep}^a,
\end{equation}
where $\omega_{\ep}$ is defined in \eqref{eqn: omega-ep} and the remainder term $\xi_{\ep}^a$ is an even power series in $a$ of the form
\[
\xi_{\ep}^a 
= \sum_{n=1}^{\infty} a^{2n}\xi_{n,\ep}, \ \xi_{n,\ep} \in \R.
\]

\item
The coefficients $\psi_{n,\ep}$ and $\xi_{n,\ep}$ in these power series satisfy 
\begin{equation}\label{eqn: psi-n xi-n threshold series}
\sup_{0 < \ep < \ep_{\per,0}} \sum_{n=1}^{\infty} a_{\per,0}^n\big(\norm{\psi_{n,\ep}}_{L_{\per}^2}+|\xi_{n,\ep}|\big)
< \infty,
\qquad
\norm{\psi}_{L_{\per}^2} := \left(\int_{-\pi}^{\pi} |\psi(x)|^2 \dx\right)^{1/2}.
\end{equation}
\end{enumerate}
\end{theorem}

We prove this theorem in Appendix \ref{app: periodics}.
Specifically, we prove in Appendix \ref{app: periodics abstract exist} an abstract bifurcation theorem that constructs solutions to problems of the form $\F_{\ep}(\phi,\omega) = 0$, where $\phi$ is in a Hilbert space and $\omega$ is real.
This is modeled on Lombardi's proof of a related result \cite[Thm.\@ 4.1.2]{lombardi}; our treatment exposes somewhat more transparently a Lyapunov--Schmidt decomposition undergirding the analysis, and we provide commentary on similarities and differences to Lombardi's approach throughout the proof.
Then in Appendix \ref{app: per real proof} we check the hypotheses of this abstract bifurcation theorem in the context of \eqref{eqn: the problem} by setting $u(x) = \phi(\omega{x})$ where $\phi$ is $2\pi$-periodic and $\omega \in \R$ and obtaining an $\ep$-dependent problem that $\phi$ and $\omega$ jointly solve.
This broad strategy has been used in the prior construction of periodics for nanopteron problems \cite[Sec.\@ 4, App.\@ C]{faver-wright}, \cite[Sec.\@ 5, App.\@ B]{hoffman-wright}, \cite[]{faver-spring-dimer}, \cite[Sec.\@ 3, App.\@ C]{faver-mim-nanopteron}, \cite[Sec.\@ 3, App.\@ C]{faver-hupkes-equal-mass}, all of which employed a quantitative ``bifurcation from a simple eigenvalue'' argument in the style of Crandall and Rabinowitz \cite{crandall-rabinowitz} and Zeidler \cite{zeidler}; of these, \cite[Thm.\@ 6.2]{faver-hupkes-wright} offers the most general perspective.
Our result here is fundamentally different because we want the periodic profiles and frequencies to be analytic in $a$.

Our primary interest in Theorem \ref{thm: periodics real} is that it is the frame on which we will build periodic solutions to \eqref{eqn: the problem} that are analytic on a strip containing $\R$.
Consequently, we only make a few brief comments about this result before proceeding to the extension program.
First, it is notationally and algebraically more convenient to have access to a factor of the amplitude, and so we write the solution as the product $u = a\phi_{\ep}^a$; that is, $u = \phi_{\ep}^a$ does not solve \eqref{eqn: the problem}.
Second, the orthogonality of the profile coefficients $\psi_{n,\ep}$ from \eqref{eqn: psi-n-ep main} to $\cos(\cdot)$ is essential for the construction of these coefficients but not so important in subsequent use.
Third, however, it is essential that $\psi_{n,\ep}$ is a trigonometric polynomial of degree $n+1$ (and no larger), as this controls estimates on $\psi_{n,\ep}$ on strips.
Because we are writing the solutions with a factor of $a$ exposed, there is a mismatch in the series \eqref{eqn: psi-ep main}: the $n$th term has a factor of $a^n$ paired with a trigonometric polynomial of degree $n+1$.

\subsection{The periodic problem on a complex strip}\label{sec: per ext program}
The next step is to extend the periodic solutions $\phi_{\ep}^a$ from Theorem \ref{thm: periodics real} to asymptotically phase-shifted periodics that are analytic on a strip $\U_b$, where $0 < b < b_{\sigma}$ with $b_{\sigma}$ from Lemma \ref{lem: q-sigma b-sigma}.
We state and then discuss this extension result.

\begin{theorem}\label{thm: ultimate varphi theorem}
Let $0 < b < b_{\sigma}$ and $\theta_0 > 0$.
There exist $\ep_{\per}$, $a_{\per} > 0$ such that for $0 < \ep < \ep_{\per}$, $|\alpha| < a_{\per}$, and $|\theta| < \theta_0$, the function
\begin{equation}\label{eqn: varphi ultimate}
\varphi_{\ep,b}^{\alpha,\theta}(z)
:= e^{-b/\ep}\cos(\Omega_{\ep,b}^{\alpha}(\Tsf_{\ep\theta}(z)) + \sum_{n=1}^{\infty} \alpha^ne^{-(n+1)b/\ep}\psi_{n,\ep}(\Omega_{\ep}^{\alpha}\Tsf_{\ep\theta}(z))
\end{equation}
is defined, analytic, and even on $\U_b$, where 
\[
\Omega_{\ep}^{\alpha}
:= \omega_{\ep}^{\alpha{e}^{-b/\ep}}
\]
with $\omega_{\ep}^a$ from \eqref{eqn: omega-ep-a xi-ep-a}, $\psi_{n,\ep}$ from \eqref{eqn: psi-n-ep main}, and 
\begin{equation}\label{eqn: Tsf}
\Tsf_{\ep\theta}(z) 
:= z + \ep\theta\tau(z).
\end{equation}
Taking $u = \alpha\varphi_{\ep,b}^{\alpha,0}$ solves \eqref{eqn: the problem} on $\U_b$, and $\varphi_{\ep,b}^{\alpha,0}(x) = \alpha{e}^{-b/\ep}\phi_{\ep}^{\alpha{e}^{-b/\ep}}(x)$ for $x \in \R$ with $\phi_{\ep}^a$ from Theorem \ref{thm: periodics real}.
For each $r \ge 0$, there is $C_r > 0$ such that the following estimates hold.

\begin{enumerate}[label={\bf(\roman*)}]

\item
{}
[Mapping estimate]
If $0 < \ep < \ep_{\per}$, $|\alpha| < a_{\per}$, $|\theta| < \theta_0$, and $z \in \U_b$, then
\begin{equation}\label{eqn: varphi map ultimate}
\big|\partial_z^r[\varphi_{\ep,b}^{\alpha,0}](\Tsf_{\ep\theta}(z))\big|
\le C_r\ep^{-r}.
\end{equation}

\item
{}
[Lipschitz estimate in amplitude]
If $0 < \ep < \ep_{\per}$, $0 \le |\alpha|$, $|\grave{\alpha}| < a_{\per}$, $|\theta| < \theta_0$, and $z \in \U_b$, then
\begin{equation}\label{eqn: varphi Lip alpha ultimate}
\big|\partial_z^r[\varphi_{\ep,b}^{\alpha,0}](\Tsf_{\ep\theta}(z))-\partial_z^r[\varphi_{\ep,b}^{\grave{\alpha},0}](\Tsf_{\ep\theta}(z))\big|
\le C_r\ep^{-r}(|z|+1)|\alpha-\grave{\alpha}|.
\end{equation}

\item
{}
[Lipschitz estimate in phase shift]
If $0 < \ep < \ep_{\per}$, $|\alpha| < a_{\per}$, $0 \le |\theta|$, $|\grave{\theta}| < \theta_0$, and $z \in \U_b$, then
\begin{equation}\label{eqn: varphi Lip theta ultimate}
\big|\partial_z^r[\varphi_{\ep,b}^{\alpha,0}](\Tsf_{\ep\theta}(z))-\partial_z^r[\varphi_{\ep,b}^{\alpha,0}](\Tsf_{\ep\grave{\theta}}(z))\big|
\le C_r\ep^{-r}|\theta-\grave{\theta}|.
\end{equation}
\end{enumerate}
\end{theorem}

We prove this theorem in Appendix \ref{app: periodics}, specifically beginning in Appendix \ref{app: here is where the per fun begins}, and discuss several of its features here.
First, we have left out a factor of $\alpha$ out of \eqref{eqn: varphi ultimate} so that $u = \alpha\varphi_{\ep,b}^{\alpha,0}$ solves \eqref{eqn: the problem}, not $u = \varphi_{\ep,b}^{\alpha,0}$, just as $u = a\phi_{\ep}^a$ from Theorem \ref{thm: periodics real} solves \eqref{eqn: the problem}, not $u = \phi_{\ep}^a$.
Exposing the factor of $\alpha$ in the exact periodic solution turns out to be much more notationally and algebraically convenient, as noted with the solutions from Theorem \ref{thm: periodics real}.
This does lead to a mismatch when comparing the notation to Theorem \ref{thm: periodics real}, as there $u = \alpha{e}^{-b/\ep}\phi_{\ep}^{\alpha{e}^{-b/\ep}}$ with its explicit factors of $e^{-b/\ep}$ is the periodic solution to \eqref{eqn: the problem}.
There is likewise, again, a mismatch between the power $\alpha^n$ in the series in \eqref{eqn: varphi ultimate} and the degree-$(n+1)$ trigonometric polynomial $\psi_{n,\ep}$.
That polynomial is, however, handily matched by the power $e^{(-n+1)b/\ep}$, and the extra factor of $e^{-b/\ep}$ throughout \eqref{eqn: varphi ultimate} is essential to ensuring that the subsequent mapping and Lipschitz estimates are just $\O(\ep^{-r})$ for the $r$th derivatives, not exponentially large.

Second, the range of the phase shift $\theta$ just needs to be bounded, and we do not place any conditions here on the positive upper bound $\theta_0$; in the full nanopteron problem we will take $\theta_0 = \pi$.
We can always ``turn off' the phase shift by taking $\theta = 0$.
In particular, we have
\[
\varphi_{\ep,b}^{\alpha,\theta}
= \varphi_{\ep,b}^{\alpha,0} \circ \Tsf_{\ep\theta}
\]
and it will be very useful for future work in the full nanopteron problem \eqref{eqn: THE equation} to present estimates for the extension in this composition form.
(Specifically, consider the structure of the terms defined in \eqref{eqn: G-ep-theta}, \eqref{eqn: J-ep-neg1-alpha-theta}, \eqref{eqn: J-ep-0-alpha-theta}, and \eqref{eqn: J-ep-1-alpha-theta}, which together constitute the full problem.)
We will need to understand the behavior of compositions with derivatives like $\partial_z^r[\varphi_{\ep,b}^{\alpha,0}]\circ \Tsf_{\ep\theta}$, but not the more complicated differentiated compositions $\partial_z^r[\varphi_{\ep,b}^{\alpha,\theta}] = \partial_z^r[\varphi_{\ep,b}^{\alpha,0} \circ \Tsf_{\ep\theta}]$.

Third, we emphasize the mapping and Lipschitz estimates \eqref{eqn: varphi map ultimate}, \eqref{eqn: varphi Lip alpha ultimate}, and \eqref{eqn: varphi Lip theta ultimate} in full technical detail here because these are the sources of several inherent challenges for the full nanopteron problem.
The $\O(\ep^{-r})$ factor that accompanies the $r$th derivatives means that we will need to have sufficient factors of $\ep$ available to counterbalance leftover derivatives that linger after the forthcoming phase-shifted nanopteron ansatz \eqref{eqn: nanopteronsatz}, as $\varphi_{\ep,b}^{\alpha,\theta}$ does not solve \eqref{eqn: the problem} exactly for $\theta \ne 0$ but generates rather involved residual terms \eqref{eqn: residual}.

The Lipschitz estimates in amplitude and phase shift \eqref{eqn: varphi Lip alpha ultimate} and \eqref{eqn: varphi Lip theta ultimate} are quite different in that the amplitude estimates are $z$-dependent.
This extra factor of $z$ will interact delicately with localized terms in a number of estimates (part \ref{part: ultimate fake quadratic prep2} of Lemma \ref{lem: ultimate fake quadratic prep}, part \ref{part: ultimate J1 prep 2} of Lemma \ref{lem: ultimate J1 prep}, and part \ref{part: ultimate rhs0 2} of Lemma \ref{lem: ultimate rhs0}), necessitating some manipulation of decay rates for success.
These different Lipschitz estimates should not be terribly surprising, as one can compute the toy difference in amplitude
\[
\cos(\alpha{z})-\cos(\grave{\alpha}z) 
= -(\alpha-\grave{\alpha})z\int_0^1 \sin((1-s)\grave{\alpha}z+s\alpha{z}) \ds
\]
to see the factor of $z$ that is missing from the toy difference in phase shift
\[
\cos(z+\theta)-\cos(z+\grave{\theta}) 
= -(\theta-\grave{\theta})\int_0^1 \sin((1-s)\grave{\theta}+s\theta) \ds.
\]

Last, none of the work in Appendix \ref{app: periodics} leading up to Theorem \ref{thm: ultimate varphi theorem} requires the functions under consideration to be scalar-valued, and so the extension program holds more generally for vector-valued functions with ranges in $\C^n$ or a Banach space.
(That is, in \eqref{eqn: psi-n-ep main}, the coefficients $c_{n,k,\ep}$ need not be real-valued.)
Consequently, these results have the potential to be exported to nanopteron problems in multi-component or infinite-dimensional systems that meet the hypotheses of the abstract bifurcation construction in Theorem \ref{thm: abstract per}.

\section{A Refined Leading-Order Limit}\label{sec: ref LO limit}
Let $0 < q < q_{\sigma}$ and $0 < b < b_{\sigma}$.
The KdV solitary wave profile $\sigma \in \E_{q,b}$ defined in \eqref{eqn: sigma} solves the problem \eqref{eqn: the problem} to $\O(\ep^2)$:
\[
\ep^2\sigma^{(4)}+\sigma''-\sigma+\sigma^2
= \ep^2\sigma^{(4)}.
\]
We can perturb $\sigma$ by some $\upsilon \in \E_{q,b}$ so that taking $u = \sigma + \upsilon$ solves \eqref{eqn: the problem} to $\O(\ep^4)$.
This will allow us to achieve some smaller and sharper estimates later in the analysis, which lead to the $\O(\ep^4)$ estimates on amplitude in Theorems \ref{thm: main amp} and \ref{thm: main PS}; these estimates would be $\O(\ep^2)$ otherwise, which is a result of how $\sigma$ solves \eqref{eqn: the problem} only to $\O(\ep^2)$.

This step is motivated by similar refined leading-order limits in the small mass FPUT and MiM problems \cite{hoffman-wright, faver-mim-nanopteron}, which were, for very technical reasons, really quite necessary to the success of those nanopteron programs.
Here, our nanopteron construction would still be successful without this approach, and we do this more as a novelty than a necessity.

We can find a perturbation $\upsilon$ to $\sigma$ such that $u = \sigma+\upsilon$ continues to solve the KdV part of \eqref{eqn: the problem} exactly and solves \eqref{eqn: the problem} to $\O(\ep^4)$ by demanding
\[
\ep^2(\sigma+\upsilon)^{(4)} + (\sigma+\upsilon)''-(\sigma+\upsilon)+(\sigma+\upsilon)^2
= \ep^2\upsilon^{(4)}
\]
and $\upsilon^{(4)} = \O(\ep^2)$.
Since $\sigma''-\sigma+\sigma^2=0$, this equation for $\upsilon$ is equivalent to
\begin{equation}\label{eqn: what upsilon does}
\upsilon
= -\K_0^{-1}(\ep^2\sigma^{(4)}+\upsilon^2),
\end{equation}
where $\K_0$ is defined in \eqref{eqn: K0} and invertible as an operator from $\E_{q,b}$ to $\E_{q,b}$ by Theorem \ref{thm: K0 inv}.
We will apply the following result to solve this fixed-point problem for $\upsilon$ with a tame quantitative contraction mapping argument.
This has a similar flavor, and proof strategy, to \cite[Lem.\@ C.1]{faver-wright}.

\begin{theorem}\label{thm: tame quant contr map}
Let $\X$ be a Banach space, $\D \subseteq \X$ and $\ep_0$, $C$, $r_0 > 0$.
Suppose that for $0 < \ep < \ep_0$, there are maps $\F_{\ep} \colon \D \to \X$ such that the following hold.

\begin{enumerate}[label={\bf(\roman*)}]

\item
If $0 < \ep < \ep_0$ and $x \in \D$ with $\norm{x}_{\X} < r_0$, then
\begin{equation}\label{eqn: abstract F map}
\norm{\F_{\ep}(x)}_{\X} 
\le C\big(\ep+\norm{x}_{\X}^2\big).
\end{equation}

\item
If $0 < \ep < \ep_0$ and $x$, $\grave{z} \in \D$ with $\norm{x}_{\X}$, $\norm{\grave{x}}_{\X} < r_0$, then
\begin{equation}\label{eqn: abstract F Lip}
\norm{\F_{\ep}(x)-\F_{\ep}(\grave{x})}_{\X} 
\le C\big(\ep+\norm{x}_{\X} + \norm{\grave{x}}_{\X}\big)\norm{x-\grave{x}}_{\X}.
\end{equation}
\end{enumerate}

Then there is $\ep_{\F} > 0$ such that for $0 < \ep < \ep_{\F}$, there is a unique $x_{\ep} \in \X$ such that 
\[
x_{\ep} = \F_{\ep}(x_{\ep})
\quadword{and}
\norm{x_{\ep}}_{\X} \le 2C\ep.
\]
\end{theorem}

\begin{proof}
We show that $\F_{\ep}$ is a contraction on $\set{x \in \X}{\norm{x}_{\X} \le 2C\ep}$ for $\ep$ sufficiently small.
Put $\ep_1 := \min\{\ep_0,r_0/\sqrt{2C}\}$.
By \eqref{eqn: abstract F map}, if $\norm{x}_{\X} \le 2C\ep^2$, then
\[
\norm{\F_{\ep}(x)}_{\X}
\le C(\ep+4C^2\ep^2)
= C(1+4C^2\ep)\ep.
\]
Set $\ep_2 := \min\{\ep_1,1/4C^2\}$ to find that if $\norm{x}_{\X} \le 2C\ep^2$ and $0 < \ep < \ep_1$, then $\norm{\F_{\ep}(x)}_{\X} \le 2C\ep^2$.

Next, by \eqref{eqn: abstract F Lip}, if $\norm{x}_{\X}$, $\norm{\grave{x}}_{\X} \le 2C\ep^2$ with $0 < \ep < \ep_2$, then
\[
\norm{\F_{\ep}(x)-\F_{\ep}(\grave{x})}_{\X} 
\le C\big(\ep+\norm{x}_{\X} + \norm{\grave{x}}_{\X}\big)\norm{x-\grave{x}}_{\X}
\le \big(\ep+4C^2\ep^2\big)\norm{x-\grave{x}}_{\X}.
\]
Set $\ep_{\F} := \min\{\ep_2,1,1/2(1+4C^2)\}$ to conclude that if $\norm{x}_{\X}$, $\norm{\grave{x}}_{\X} \le 2C\ep^2$ with $0 < \ep < \ep_2$, then
\[
\norm{\F_{\ep}(x)-\F_{\ep}(\grave{x})}_{\X} 
\le \frac{1}{2}\norm{x-\grave{x}}_{\X}.
\]

The contraction mapping theorem therefore applies to $\F_{\ep}$ on $\set{x \in \X}{\norm{x}_{\X} \le 2C\ep^2}$ for $0 < \ep < \ep_{\F}$ to produce a unique fixed point $x_{\ep}$ of $\F_{\ep}$ in this ball.
\end{proof}

In the following construction of the leading-order perturbation $\upsilon$, we will pick a decay rate $q \in (0,q_{\sigma})$, where $q_{\sigma}$ is from Lemma \ref{lem: q-sigma b-sigma}, and then run a contraction for $\upsilon$ in a space that decays slightly faster than rate $q$, namely, rate $(q+3q_{\sigma})/4$, as indicated in \eqref{eqn: elementary q ineq}.
Later, when we construct the localized remainder term $\eta$ for the nanopteron, we will assume that $\eta$ decays with rate $q$, and it will be convenient for $\upsilon$ to decay somewhat faster than $\eta$; specifically, this helps with the decay-borrowing estimates in the proof of part \ref{part: ultimate rhs0 2} of Lemma \ref{lem: ultimate rhs0}.

\begin{theorem}\label{thm: refined LO lim}
Let $0 < q < q_{\sigma}$ and $0 < b < b_{\sigma}$.
There are $\ep_{\upsilon}$, $C > 0$ such that for $0 < \ep < \ep_{\upsilon}$, there is $\upsilon_{\ep,q} \in \E_{(q+3q_{\sigma})/4,b}$ such that 
\begin{equation}\label{eqn: refined LO lim}
\ep^2(\sigma+\upsilon_{\ep,q})^{(4)} + (\sigma+\upsilon_{\ep,q})''-(\sigma+\upsilon_{\ep,q})+(\sigma+\upsilon_{\ep,q})^2
= \ep^2\upsilon_{\ep,q}^{(4)}
\end{equation}
and
\begin{equation}\label{eqn: upsilon-ep}
\norm{\upsilon_{\ep,q}}_{(q+3q_{\sigma})/4,b}
\le C\ep^2
\quadword{and}
\norm{\upsilon_{\ep,q}^{(4)}}_{(q+3q_{\sigma})/4,b}
\le C\ep^2.
\end{equation}
\end{theorem}

\begin{proof}
It is straightforward to check that Theorem \ref{thm: tame quant contr map} with $\ep^2$ in place of $\ep$ applies to the family of maps
\[
\F_{\ep}(\upsilon)
:= -\K_0^{-1}(\ep^2\sigma^{(4)}+\upsilon^2)
\]
defined on $\E_{(q+3q_{\sigma})/4,b}$.
This yields the threshold $\ep_{\upsilon} > 0$ and solution $\upsilon_{\ep,q} \in \E_{(q+3q_{\sigma})/4,b}$ to $\upsilon_{\ep,q} = \F_{\ep}(\upsilon_{\ep,q})$ that satisfies $\norm{\upsilon_{\ep,q}}_{(q+3q_{\sigma})/4,b} \le C\ep^2$.
That is, $\upsilon_{\ep,q}$ satisfies \eqref{eqn: what upsilon does}, and therefore $\upsilon_{\ep,q}$ also satisfies \eqref{eqn: refined LO lim}.

For the second estimate in \eqref{eqn: upsilon-ep}, we use \eqref{eqn: what upsilon does} to find
\begin{equation}\label{eqn: upsilon-ep 2nd deriv}
\upsilon_{\ep,q}''
= \upsilon_{\ep,q}-2\sigma\upsilon_{\ep,q}-\upsilon_{\ep,q}^2-\ep^2\sigma^{(4)},
\end{equation}
thus $\upsilon_{\ep,q}'' \in \E_{(q+3q_{\sigma})/4,b}$.
We differentiate this twice to obtain
\begin{equation}\label{eqn: upsilon-ep 4th deriv}
\upsilon_{\ep,q}^{(4)}
= \upsilon_{\ep,q}''-2\sigma''\upsilon_{\ep,q}-4\sigma'\upsilon_{\ep,q}'-2\sigma\upsilon_{\ep,q}''+2(\upsilon_{\ep,q}')^2+2\upsilon_{\ep,q}\upsilon_{\ep,q}''-\ep^2\sigma^{(4)}.
\end{equation}
The formula in \eqref{eqn: upsilon-ep 2nd deriv} gives $\upsilon_{\ep,q}'' = \O(\ep^2)$, so the interpolation estimate \eqref{eqn: interpolation} gives $\upsilon_{\ep,q}' \in \H_{(q+3q_{\sigma})/4,b}$ with
\[
\norm{\upsilon_{\ep,q}'}_{(q+3q_{\sigma})/4,b}
\le C\big(\norm{\upsilon_{\ep,q}}_{(q+3q_{\sigma})/4,b}+\norm{\upsilon_{\ep,q}''}_{(q+3q_{\sigma})/4,b}\big)
\le C\ep^2.
\]
Estimating each of the terms in \eqref{eqn: upsilon-ep 4th deriv} then proves the estimate \eqref{eqn: upsilon-ep} for $\upsilon_{\ep,q}^{(4)}$.
\end{proof}

\section{The Phase-Shifted Nanopteron Problem}\label{sec: PS nano}

\subsection{The phase-shifted nanopteron ansatz}\label{sec: PS nano ansatz intro}
Fix $0 < b < b_{\sigma}$, $0 < q < q_{\sigma}$, and $\theta_0 > 0$.
Let $0 < \ep < \min\{\ep_{\upsilon,q},\ep_{\per}\}$.
We make the phase-shifted nanopteron ansatz 
\begin{equation}\label{eqn: nanopteronsatz}
u
= (\sigma + \upsilon_{\ep,q}) + \alpha\varphi_{\ep,b}^{\alpha,\theta}+\eta
\end{equation}
for the problem \eqref{eqn: the problem}, where the terms in this ansatz are the following.

\begin{enumerate}[label=$\bullet$]

\item
The refined leading order limit $\sigma+\upsilon_{\ep,q}$ is constructed in Theorem \ref{thm: refined LO lim} and satisfies $\sigma$, $\upsilon_{\ep,q} \in \E_{(q+3q_{\sigma})/4,b}$ with $\sigma + \upsilon_{\ep,q}$ solving \eqref{eqn: refined LO lim}.
Also, $\upsilon_{\ep,q}$ is formally $\O(\ep^2)$, as Theorem \ref{thm: refined LO lim} makes precise.
At certain technical points (in particular, the proof of part \ref{part: ultimate rhs0 2} of Lemma \ref{lem: ultimate rhs0}) it will be convenient to exploit the slightly higher decay rate of $\upsilon_{\ep,q}$ relative to $\eta$.

\item
The function $\varphi_{\ep,b}^{\alpha,\theta}$ is defined, analytic, and even on $\U_b$, per Theorem \ref{thm: ultimate varphi theorem}, and taking $u = \alpha\varphi_{\ep,b}^{\alpha,0}$ solves \eqref{eqn: the problem} on $\U_b$.
Here we are restricting $|\alpha| < a_{\per}$ and $|\theta| < \theta_0$.
(We do not yet need to specify $0 < \theta_0 < \pi$, but we soon will.)

\item
The localized remainder $\eta$ satisfies $\eta \in \E_{q,b}$.
\end{enumerate}

Making this ansatz generates a number of ``residual'' terms because taking $u = \alpha\varphi_{\ep,b}^{\alpha,\theta}$ no longer solves $\ep^2u^{(4)}+u''-u+u^2=0$ exactly for $\theta \ne 0$.
Rather, since
\[
\ep^2(\alpha\varphi_{\ep,b}^{\alpha,0})^{(4)}
+ (\alpha\varphi_{\ep,b}^{\alpha,0})''
- \alpha\varphi_{\ep,b}^{\alpha,0}
+ (\alpha\varphi_{\ep,b}^{\alpha,0})^2
= 0
\quadword{and}
\varphi_{\ep,b}^{\alpha,\theta} = \varphi_{\ep,b}^{\alpha,0} \circ \Tsf_{\ep\theta},
\]
with $\Tsf_{\ep\theta}$ defined in \eqref{eqn: Tsf}, we have
\[
\ep^2\big((\alpha\varphi_{\ep,b}^{\alpha,0})^{(4)}\circ\Tsf_{\ep\theta}\big)
+ \big((\alpha\varphi_{\ep,b}^{\alpha,0})''\circ\Tsf_{\ep\theta}\big)
- \alpha\varphi_{\ep,b}^{\alpha,\theta}
+ (\alpha\varphi_{\ep,b}^{\alpha,\theta})^2
= 0.
\]
Then with
\begin{equation}\label{eqn: residual}
\varphi_{\ep,b}^{\alpha,\theta,\res}
:= \ep^2\big[(\alpha\varphi_{\ep,b}^{\alpha,\theta})^{(4)}-\big((\alpha\varphi_{\ep,b}^{\alpha,0})^{(4)}\circ\Tsf_{\theta}\big)\big]
+ \big[(\alpha\varphi_{\ep,b}^{\alpha,\theta})''-\big((\alpha\varphi_{\ep,b}^{\alpha,0})''\circ\Tsf_{\theta}\big)\big],
\end{equation}
we have
\[
\ep^2(\alpha\varphi_{\ep,b}^{\alpha,\theta})^{(4)}+(\alpha\varphi_{\ep,b}^{\alpha,\theta})''-\alpha\varphi_{\ep,b}^{\alpha,\theta}+(\alpha\varphi_{\ep,b}^{\alpha,\theta})^2 
= \varphi_{\ep,b}^{\alpha,\theta,\res}.
\]

For the phase-shifted nanopteron ansatz \eqref{eqn: nanopteronsatz} to solve \eqref{eqn: the problem}, we therefore need
\begin{equation}\label{eqn: THE equation0}
\B_{\ep}\eta 
+ \Sigma\eta
+ 2\alpha\varphi_{\ep,b}^{\alpha,\theta}\sigma
+ \varphi_{\ep,b}^{\alpha,\theta,\res}
= -\ep^2\upsilon_{\ep,q}^{(4)} 
- 2\upsilon_{\ep,q}\eta
-2\alpha\varphi_{\ep,b}^{\alpha,\theta}\upsilon_{\ep,q}
- 2\alpha\varphi_{\ep,b}^{\alpha,\theta}\eta
-\eta^2,
\end{equation}
where $\B_{\ep}$ and $\Sigma$ are defined in \eqref{eqn: B-ep} and \eqref{eqn: Sigma}, respectively.
The terms on the right in \eqref{eqn: THE equation0} are sufficiently small in $\ep$, $\eta$, and $\alpha$ that they will not contribute to the leading-order behavior of the nanopteron problem.
We now study the residual $\varphi_{\ep,b}^{\alpha,\theta,\res}$ to extract its dominating terms.

\subsection{Analysis of the residual $\varphi_{\ep,b}^{\alpha,\theta,\res}$}
Recall from \eqref{eqn: varphi map ultimate} that $(\varphi_{\ep,b}^{\alpha,0})^{(r)}$ is formally $\O(\ep^{-r})$.
We expand $\varphi_{\ep,b}^{\alpha,\theta,\res}$ defined in \eqref{eqn: residual}, and keep track of the formal order in $\ep$.
Throughout, we use the identities $\Tsf_{\ep\theta}' = 1 + \ep\theta\tau'$, valid from the definition of $\Tsf_{\ep\theta}$ in \eqref{eqn: Tsf}, and $\tau' = \tau_{\infty}\sigma$, per \eqref{eqn: tau prime}.
We have
\begin{align*}
(\alpha\varphi_{\ep,b}^{\alpha,\theta})''-\big((\alpha\varphi_{\ep,b}^{\alpha,0})''\circ\Tsf_{\theta}\big)
&= 2\ep\big((\alpha\varphi_{\ep,b}^{\alpha,0})''\circ\Tsf_{\ep\theta}\big)(\theta\tau_{\infty}\sigma\big) & \O(\ep^{-1}) \\
&+ \ep^2\big((\alpha\varphi_{\ep,b}^{\alpha,0})''\circ\Tsf_{\ep\theta}\big)(\theta^2\tau_{\infty}^2\sigma^2) &\O(1) \\
&+ \ep\big((\varphi_{\ep,b}^{\alpha,0})'\circ\Tsf_{\ep\theta}\big)(\theta\tau_{\infty}\sigma') &\O(1)
\end{align*}
and
\begin{align*}
\ep^2\big[(\alpha\varphi_{\ep,b}^{\alpha,\theta})^{(4)}-\big((\alpha\varphi_{\ep,b}^{\alpha,0})^{(4)}\circ\Tsf_{\theta}\big)\big]
&= 4\ep^3\big((\alpha\varphi_{\ep,b}^{\alpha,0})^{(4)}\circ\Tsf_{\ep\theta}\big)(\theta\tau_{\infty}\sigma) &\O(\ep^{-1}) \\
&+\ep^2\big((\alpha\varphi_{\ep,b}^{\alpha,0})^{(4)}\circ\Tsf_{\ep\theta}\big)\big[(1+\ep\theta\tau_{\infty}\sigma)^4-1-4\ep\theta\tau_{\infty}\sigma\big] &\O(1) \\
&+6\ep^3\big((\alpha\varphi_{\ep,b}^{\alpha,0})'''\circ\Tsf_{\ep\theta}\big)(1+\ep\theta\tau_{\infty}\sigma)^2(\theta\tau_{\infty}\sigma') &\O(1) \\
&+3\ep^4\big((\alpha\varphi_{\ep,b}^{\alpha,0})''\circ\Tsf_{\ep\theta}\big)(\theta\tau_{\infty}\sigma')^2 &\O(\ep^2) \\
&+4\ep^3\big((\alpha\varphi_{\ep,b}^{\alpha,0})''\circ\Tsf_{\ep\theta}\big)(1+\ep\theta\tau_{\infty}\sigma)(\theta\tau_{\infty}\sigma'') &\O(\ep) \\
&+\ep^3\big((\alpha\varphi_{\ep,b}^{\alpha,0})'\circ\Tsf_{\ep\theta}\big)(\theta\tau_{\infty}\sigma''') &\O(\ep^2).
\end{align*}

The $\O(\ep^{-1})$ terms from $\varphi_{\ep,b}^{\alpha,\theta,\res}$ will contribute meaningfully to the subsequent work.
Unsurprisingly, the $\O(\ep)$ and $\O(\ep^2)$ terms do not, but we will also be able to ignore, mostly, the $\O(1)$ terms, too.
This is roughly because those $\O(1)$ terms are all $0$ when $\theta=0$, and when $\theta \ne 0$, the $\O(\ep^{-1})$ terms dominate.
To isolate more precisely the leading-order behavior in $\alpha$, we adjust these formal $\O(\ep^{-1})$ and $\O(1)$ terms by rewriting
\[
\big((\alpha\varphi_{\ep,b}^{\alpha,0})^{(r)}\circ\Tsf_{\ep\theta}\big)
= \big((\alpha\varphi_{\ep,b}^{\alpha,0})^{(r)}\circ\Tsf_{\ep\theta}\big)+\bunderbrace{\big[\big(\alpha\varphi_{\ep,b}^{\alpha,0})^{(r)}\circ\Tsf_{\ep\theta}\big)-\big((\alpha\varphi_{\ep,b}^{0,0})^{(r)}\circ\Tsf_{\ep\theta}\big)\big]}{\O(\alpha^2)},
\]
where the formal $\O(\alpha^2)$ behavior is made rigorous in Lemma \ref{lem: ultimate fake quadratic}.

Now we rewrite
\begin{equation}\label{eqn: the periodic part expanded}
2\alpha\varphi_{\ep,b}^{\alpha,\theta}\sigma+\varphi_{\ep,b}^{\alpha,\theta,\res}
= \alpha{e}^{-b/\ep}(\tau_{\infty}\omega_{\ep}\theta+1)\chi_{\ep}^{\theta} 
+\alpha\G_{\ep,b}^{\theta}
-\J_{\ep,b,-1}^{\alpha,\theta}
-\J_{\ep,b,0}^{\alpha,\theta}
-\J_{\ep,b,1}^{\alpha,\theta},
\end{equation}
where
\begin{equation}\label{eqn: chi-ep-theta}
\chi_{\ep}^{\theta} 
:= 2\cos(\omega_{\ep}\Tsf_{\ep\theta}(\cdot))\sigma,
\end{equation}
\begin{equation}\label{eqn: G-ep-theta}
\begin{aligned}
\G_{\ep,b}^{\theta}
&:= \ep\big((\varphi_{\ep,b}^{0,0})'\circ\Tsf_{\ep\theta}\big)(\theta\tau_{\infty}\sigma') \\
&+ \ep^2\big((\varphi_{\ep,b}^{0,0})''\circ\Tsf_{\ep\theta}\big)(\theta^2\tau_{\infty}^2\sigma^2) \\
&+\ep^3\big((\varphi_{\ep,b}^{0,0})'''\circ\Tsf_{\ep\theta}\big)(1+\ep\theta\tau_{\infty}\sigma)^2(6\theta\tau_{\infty}\sigma') \\
&+\ep^2\big((\varphi_{\ep,b}^{0,0})^{(4)}\circ\Tsf_{\ep\theta}\big)\big[(1+\ep\theta\tau_{\infty}\sigma)^4-1-4\ep\theta\tau_{\infty}\sigma\big],
\end{aligned}
\end{equation}
\begin{equation}\label{eqn: J-ep-neg1-alpha-theta}
\begin{aligned}
-\J_{\ep,b,-1}^{\alpha,\theta} 
&:= \ep\alpha\big[\big((\varphi_{\ep,b}^{\alpha,0})''\circ\Tsf_{\ep\theta}\big)-\big((\varphi_{\ep,b}^{0,0})''\circ\Tsf_{\ep\theta}\big)\big](2\theta\tau_{\infty}\sigma) \\
&+ \ep^3\alpha\big[\big((\varphi_{\ep,b}^{\alpha,0})^{(4)}\circ\Tsf_{\ep\theta}\big) -\big((\varphi_{\ep,b}^{0,0})^{(4)}\circ\Tsf_{\ep\theta}\big)\big](4\theta\tau_{\infty}\sigma),
\end{aligned}
\end{equation}
\begin{equation}\label{eqn: J-ep-0-alpha-theta}
\begin{aligned}
-\J_{\ep,b,0}^{\alpha,\theta} 
&:= \alpha\big[(\varphi_{\ep,b}^{\alpha,0}\circ\Tsf_{\ep\theta})-(\varphi_{\ep,b}^{0,0}\circ\Tsf_{\ep\theta})\big](2\sigma) \\
&+\ep\alpha\big[\big((\varphi_{\ep,b}^{\alpha,0})'\circ\Tsf_{\ep\theta}\big)-\big((\varphi_{\ep,b}^{0,0})'\circ\Tsf_{\ep\theta}\big)\big](\theta\tau_{\infty}\sigma') \\
&+\ep^2\alpha\big[\big((\varphi_{\ep,b}^{\alpha,0})''\circ\Tsf_{\ep\theta}\big)-\big((\varphi_{\ep,b}^{0,0})''\circ\Tsf_{\ep\theta}\big)\big](\theta^2\tau_{\infty}^2\sigma^2) \\
&+\ep^3\alpha\big[\big((\varphi_{\ep,b}^{\alpha,0})'''\circ\Tsf_{\ep\theta}\big)-\big(\varphi_{\ep,b}^{0,0})'''\circ\Tsf_{\ep\theta}\big)\big](1+\ep\theta\tau_{\infty}\sigma)^2(6\theta\tau_{\infty}\sigma') \\
&+\ep^2\alpha\big[\big((\varphi_{\ep,b}^{\alpha,0})^{(4)}\circ\Tsf_{\ep\theta}\big)-\big((\varphi_{\ep,b}^{0,0})^{(4)}\circ\Tsf_{\ep\theta}\big)\big]\big[(1+\ep\theta\tau_{\infty}\sigma)^4-1-4\ep\theta\tau_{\infty}\sigma\big],
\end{aligned}
\end{equation}
\begin{equation}\label{eqn: J-ep-1-alpha-theta} 
\begin{aligned}
-\J_{\ep,b,1}^{\alpha,\theta} 
&:= \ep^4\big((\alpha\varphi_{\ep,b}^{\alpha,0})''\circ\Tsf_{\ep\theta}\big)(3\theta\tau_{\infty}\sigma')^2 \\
&+\ep^3\big((\alpha\varphi_{\ep,b}^{\alpha,0})''\circ\Tsf_{\ep\theta}\big)(1+\ep\theta\tau_{\infty}\sigma)(4\theta\tau_{\infty}\sigma'') ) \\
&+\ep^3\big((\alpha\varphi_{\ep,b}^{\alpha,0})'\circ\Tsf_{\ep\theta}\big)(\theta\tau_{\infty}\sigma''') \\
&+(\ep\omega_{\ep})\big(\sqrt{1+4\ep^2}-1\big)\omega_{\ep}\big(\alpha{e}^{-b/\ep}\chi_{\ep}^{\theta}\big)(\theta\tau_{\infty}).
\end{aligned}
\end{equation}

We provide some further commentary on the formal structure and properties of these terms.

\begin{enumerate}[label=$\bullet$]

\item
The term $\alpha{e}^{-b/\ep}(\tau_{\infty}\omega_{\ep}\theta+1)\chi_{\ep}^{\theta}$ primarily controls the leading-order behavior of the phase-shifted nanopteron problem in $\alpha$ and $\theta$.

\item
The term $\alpha\G_{\ep,b}^{\theta}$ is $\O(\alpha\theta)$ in the sense that $|\alpha\G_{\ep,b}^{\theta}| \le C|\alpha||\theta|$ uniformly in $\ep$, $\alpha$, and $\theta$.
We make this precise in Appendix \ref{app: G}.
Consequently, this term plays different roles in our two problems.
In the amplitude selection problem, $\alpha\G_{\ep,b}^{\theta}$ may be large, and so it impacts the leading-order behavior in amplitude $\alpha$; we therefore include it along with $\chi_{\ep}^{\theta}$ in our leading-order analysis.
In the phase shift selection problem, because the selected phase shift will be $\O(\ep)$, and because the prescribed amplitude will be $\O(\ep^4)$, the term $\alpha\G_{\ep,b}^{\theta}$ ends up being $\O(\ep^5)$, which is sufficiently small.

\item
The term $\J_{\ep,b,-1}^{\alpha,\theta}$ is $\O(\ep^{-1}\alpha^2)$ in the sense that $\norm{\J_{\ep,b,-1}^{\alpha,\theta}}_{q,b} \le C\ep^{-1}\alpha^2$ uniformly in $\ep$, $\alpha$, and $\theta$.
This is large from the point of view of $\ep$, but the factor of $\alpha^2$ makes it sufficiently small; also, $\J_{\ep,b,-1}^{\alpha,0} = 0$.
We make this precise in Appendix \ref{app: J-1 J0}.

\item
The term $\J_{\ep,b,0}^{\alpha,\theta}$ is $\O(\alpha^2)$ in the sense that $\norm{\J_{\ep,b,0}^{\alpha,\theta}}_{q,b} \le C\alpha^2$ uniformly in $\ep$, $\alpha$, and $\theta$.
This is sufficiently small.
We make this precise in Appendix \ref{app: J-1 J0}.
At $\theta = 0$, this term collapses to the first term in the sum in \eqref{eqn: J-ep-0-alpha-theta}.

\item
The term $\J_{\ep,b,1}^{\alpha,\theta}$ is $\O(\ep\alpha)$ in the sense that $\norm{\J_{\ep,b,1}^{\alpha,\theta}}_{q,b} \le C\ep\alpha$ uniformly in $\ep$, $\alpha$, and $\theta$.
This is sufficiently small; also, $\J_{\ep,b,1}^{\alpha,0} = 0$.
We make this precise in Appendix \ref{app: J1}.
\end{enumerate}

\subsection{The phase-shifted nanopteron problem revealed}
Using \eqref{eqn: the periodic part expanded}, we rewrite the problem \eqref{eqn: THE equation0} as 
\begin{equation}\label{eqn: THE equation}
\B_{\ep}\eta
+ \Sigma\eta
+ \alpha{e}^{-b/\ep}(\tau_{\infty}\omega_{\ep}\theta+1)\chi_{\ep}^{\theta}
+ \alpha\G_{\ep,b}^{\theta}
= \rhs_{q,b}^{\ep}(\eta,\alpha,\theta).
\end{equation}
Here we have defined
\begin{equation}\label{eqn: rhs-ep}
\rhs_{q,b}^{\ep}(\eta,\alpha,\theta)
:= \rhs_{q,b}^{\ep,0}(\eta,\alpha,\theta)
+\J_{\ep,b,-1}^{\alpha,\theta}
+\J_{\ep,b,0}^{\alpha,\theta}
+\J_{\ep,b,1}^{\alpha,\theta},
\end{equation}
with $\J_{\ep,b,-1}^{\alpha,\theta}$ defined in \eqref{eqn: J-ep-neg1-alpha-theta}, $\J_{\ep,b,0}^{\alpha,\theta}$ in \eqref{eqn: J-ep-0-alpha-theta}, and $\J_{\ep,b,1}^{\alpha,\theta}$ in \eqref{eqn: J-ep-1-alpha-theta}, and now
\begin{equation}\label{eqn: rhs-ep-0}
\rhs_{q,b}^{\ep,0}(\eta,a,\theta)
:= -\ep^2\upsilon_{\ep,q}^{(4)} 
- 2\upsilon_{\ep,q}\eta
-2\alpha\varphi_{\ep,b}^{\alpha,\theta}\upsilon_{\ep,q}
- 2\alpha\varphi_{\ep,b}^{\alpha,\theta}\eta
-\eta^2.
\end{equation}

The problem \eqref{eqn: THE equation} is overdetermined.
Recall that with the functional $\iota_{\ep}$ defined in \eqref{eqn: iota-ep}, we have $\iota_{\ep}[\B_{\ep}\eta] = 0$ for $\eta \in \H_{q,b}$, and so solutions $\eta$, $\alpha$, and $\theta$ to \eqref{eqn: THE equation} must also meet
\begin{equation}\label{eqn: sel mech}
\iota_{\ep}[\Sigma\eta]
+ \alpha{e}^{-b/\ep}(\tau_{\infty}\omega_{\ep}\theta+1)\iota_{\ep}[\chi_{\ep}^{\theta}]
+ \alpha\iota_{\ep}[\G_{\ep,b}^{\theta}]
= \iota_{\ep}[\rhs_{q,b}^{\ep}(\eta,\alpha,\theta)].
\end{equation}
Now, however, we have two equations \eqref{eqn: THE equation} and \eqref{eqn: sel mech} but three unknowns $\eta$, $\alpha$, and $\theta$, which are more than enough for success.
We view \eqref{eqn: sel mech} as the ``selection mechanism'' for amplitude $\alpha$ and phase shift $\theta$ in the sense that this equation tells us how $\eta$, $\alpha$, and $\theta$ must interact if we are to have a solution to our phase-shifted nanopteron problem \eqref{eqn: THE equation}.

\section{The Amplitude Selection Problem}\label{sec: amplitude problem}

\subsection{The problem for $\eta$}
Fix $0 \le \theta < \pi$.
All results and estimates in this section will depend on $\theta$.
We maintain the assumptions on $\ep$, $\alpha$, and $\eta$ from the start of Section \ref{sec: PS nano ansatz intro}.

We slightly rewrite the phase-shifted nanopteron problem from \eqref{eqn: THE equation} as
\begin{equation}\label{eqn: THE problem amp}
\B_{\ep}\eta
+ \Sigma\eta
+ \alpha{e}^{-b/\ep}(\tau_{\infty}\omega_{\ep}\theta+1)\big(\chi_{\ep}^{\theta}+e^{b/\ep}(\tau_{\infty}\omega_{\ep}\theta+1)^{-1}\G_{\ep,b}^{\theta}\big)
= \rhs_{q,b}^{\ep}(\eta,\alpha,\theta).
\end{equation}
With $\theta$ fixed, the term $\alpha\G_{\ep,b}^{\theta}$ is formally $\O(\alpha)$, which makes it too large to be included in $\rhs_{q,b}^{\ep}$, and so we bundle it with $\chi_{\ep}^{\theta}$.
We have factored out $\tau_{\infty}\omega_{\ep}\theta+1$ largely for convenience and to emphasize the leading-order role of $\chi_{\ep}^{\theta}$ in the following analysis; since $\theta \ge 0$, there is no problem dividing by $\tau_{\infty}\omega_{\ep}\theta+1$.

The selection mechanism from \eqref{eqn: sel mech} now reads
\begin{equation}\label{eqn: sel mech amp}
\iota_{\ep}[\Sigma\eta]
+ \alpha{e}^{-b/\ep}(\tau_{\infty}\omega_{\ep}\theta+1)\iota_{\ep}[\chi_{\ep}^{\theta}+e^{b/\ep}(\tau_{\infty}\omega_{\ep}\theta+1)^{-1}\G_{\ep,b}^{\theta}]
= \iota_{\ep}[\rhs_{q,b}^{\ep}(\eta,\alpha,\theta)].
\end{equation}
We show in Appendix \ref{app: iota-ep-chi} that $|\iota_{\ep}[\chi_{\ep}^{\theta}+e^{b/\ep}(\tau_{\infty}\omega_{\ep}\theta+1)^{-1}\G_{\ep,b}^{\theta}]|$ is uniformly bounded below away from $0$: for $0 < \theta_0 < \pi$ and $0 \le \theta < \theta_0$, there is $C > 0$ such that for $\ep > 0$ sufficiently small, we have
\begin{equation}\label{eqn: iota-ep-chi amp}
C
< |\iota_{\ep}[\chi_{\ep}^{\theta}+e^{b/\ep}(\tau_{\infty}\omega_{\ep}\theta+1)^{-1}\G_{\ep,b}^{\theta}]|.
\end{equation}
Consequently, we may divide by $\iota_{\ep}[\chi_{\ep}^{\theta}+e^{b/\ep}(\tau_{\infty}\omega_{\ep}\theta+1)^{-1}\G_{\ep,b}^{\theta}]$ without harming any estimates.

Now we can show that if the selection mechanism \eqref{eqn: sel mech amp} holds, then we can solve \eqref{eqn: THE problem amp} for $\eta$.
The style of this argument is inspired by \cite[Prop.\@ 1]{johnson-wright}.

\begin{lemma}\label{lem: how to solve amp}
Let $0 < q < q_{\sigma}$, $0 < b < b_{\sigma}$, and $0 < \theta_0 < \pi$.
For $\eta$, $f \in \E_{q,b}$, $A \in \R$, $0 \le \theta < \theta_0$, and $\ep > 0$ sufficiently small, we have
\begin{equation}\label{eqn: eta A f}
\B_{\ep}\eta + A\big(\chi_{\ep}^{\theta}+e^{b/\ep}(\tau_{\infty}\omega_{\ep}\theta+1)^{-1}\G_{\ep,b}^{\theta}\big)
= f
\end{equation}
if and only if 
\begin{equation}\label{eqn: what eta A f do}
A = \frac{\iota_{\ep}[f]}{\iota_{\ep}[\chi_{\ep}^{\theta}+e^{b/\ep}(\tau_{\infty}\omega_{\ep}\theta+1)^{-1}\G_{\ep,b}^{\theta}]}
\quadword{and}
\eta = \B_{\ep}^{-1}\P_{\ep,b}(\theta)f,
\end{equation}
where $\B_{\ep}^{-1}$ is defined in \eqref{eqn: B-ep-inv} and
\begin{equation}\label{eqn: P-ep-theta}
\P_{\ep,b}(\theta)f
:= f-\frac{\iota_{\ep}[f]}{\iota_{\ep}[\chi_{\ep}^{\theta}+e^{b/\ep}(\tau_{\infty}\omega_{\ep}\theta+1)^{-1}\G_{\ep,b}^{\theta}]}\big(\chi_{\ep}^{\theta}+e^{b/\ep}(\tau_{\infty}\omega_{\ep}\theta+1)^{-1}\G_{\ep,b}^{\theta}\big).
\end{equation}
\end{lemma}

\begin{proof}
If \eqref{eqn: eta A f} holds, then we may apply $\iota_{\ep}$ and, assuming that $\ep$ is sufficiently small (which Lemma \ref{lem: iota-ep-chi-app} makes precise) divide by $\iota_{\ep}[\chi_{\ep}^{\theta}+e^{b/\ep}(\tau_{\infty}\omega_{\ep}\theta+1)^{-1}\G_{\ep,b}^{\theta}]$, as permitted by \eqref{eqn: iota-ep-chi amp}, to obtain the formula for $A$ in \eqref{eqn: what eta A f do}.
Then we substitute that formula for $A$ into \eqref{eqn: eta A f} and rearrange to obtain $\B_{\ep}\eta = \P_{\ep,b}(\theta)f$, with $\P_{\ep,b}(\theta)$ defined in \eqref{eqn: P-ep-theta}.
It is straightforward to calculate that $\iota_{\ep}[\P_{\ep,b}(\theta)f] = 0$ for any $f$, and so Theorem \ref{thm: B-ep overview} implies the formula for $\eta$ in \eqref{eqn: what eta A f do}.

Conversely, if $A$ and $\eta$ are given by \eqref{eqn: what eta A f do} holds, then it is a direct calculation that
\begin{multline*}
\B_{\ep}\eta
= \P_{\ep,b}(\theta)f
= f-\frac{\iota_{\ep}[f]}{\iota_{\ep}[\chi_{\ep}^{\theta}+e^{b/\ep}(\tau_{\infty}\omega_{\ep}\theta+1)^{-1}\G_{\ep,b}^{\theta}]}\big(\chi_{\ep}^{\theta}+e^{b/\ep}(\tau_{\infty}\omega_{\ep}\theta+1)^{-1}\G_{\ep,b}^{\theta}\big) \\
= f-A\big(\chi_{\ep}^{\theta}+e^{b/\ep}(\tau_{\infty}\omega_{\ep}\theta+1)^{-1}\G_{\ep,b}^{\theta}\big),
\end{multline*}
which gives \eqref{eqn: eta A f}.
\end{proof}

Going forward, suppose that $\alpha$ and $\theta$ have been chosen to make the selection mechanism \eqref{eqn: sel mech amp} true.
We will discuss in Section \ref{sec: amp sel} how to do this for $\alpha$ given $\theta$.
By Lemma \ref{lem: how to solve amp} the problem \eqref{eqn: THE problem amp} is then equivalent to
\begin{equation}\label{eqn: eta pre-fp}
(1+\B_{\ep}^{-1}\P_{\ep,b}(\theta)\Sigma)\eta
= \B_{\ep}^{-1}\P_{\ep,b}(\theta)\rhs_{q,b}^{\ep}(\eta,\alpha,\theta),
\end{equation}
where in the application of Lemma \ref{lem: how to solve amp} we have taken $A = \alpha{e}^{-b/\ep}(\tau_{\infty}\omega_{\ep}\theta+1)$ and $f = \rhs_{q,b}^{\ep}(\eta,\alpha,\theta)-\Sigma\eta$.
We can rearrange \eqref{eqn: eta pre-fp} into a fixed-point problem for $\eta$ by inverting $1+\B_{\ep}^{-1}\P_{\ep,b}(\theta)\Sigma$, an idea that we adapt from \cite[Lem.\@ 5]{johnson-wright}.

Specifically, we show in Theorem \ref{thm: KdV perturbation amp} that $1 + \B_{\ep}^{-1}\P_{\ep,b}(\theta)\Sigma$ is a small perturbation of $1+\L_1\Sigma$, which is also invertible by Theorem \ref{thm: K0 inv}.
We start by rewriting
\[
1+\B_{\ep}^{-1}\P_{\ep,b}(\theta)\Sigma
= (\B_{\ep}^{-1}-\L_1)\P_{\ep,b}(\theta)\Sigma + \L_1(\P_{\ep,b}(\theta)-1)\Sigma.
\]
Here $\L_1 = (\partial_z^2-1)^{-1}$ as defined in \eqref{eqn: L-mu}.

That $\B_{\ep}^{-1}$ is a small perturbation of $\L_1$ is plausible from the definition of $\B_{\ep}^{-1}$ in \eqref{eqn: B-ep-inv}.
What we actually show is that the whole product $(\B_{\ep}^{-1}-\L_1)\P_{\ep,b}(\theta)\Sigma$ is small, and keeping the factor $\P_{\ep,b}(\theta)$ in the analysis is essential.
This makes more precise our earlier claim in Section \ref{thm: B-ep overview} that while the singular perturbation prevents $\B_{\ep}$ from converging to $\B_0$ in the operator norm on $\H_{q,b}$, the inverses are nonetheless better behaved.

It is then somewhat more involved to show that $\L_1(\P_{\ep,b}(\theta)-1)\Sigma$ is small in $\ep$, and the ``obvious'' way of doing this by showing that $\P_{\ep,b}(\theta)$ is a small perturbation of the identity turns out to be wrong (as $\P_{\ep,b}(\theta)$ is not, in fact, such a perturbation).
Nonetheless, once we have established a perturbative relationship between $1 + \B_{\ep}^{-1}\P_{\ep,b}(\theta)\Sigma$ and $1+\L_1\Sigma$, we use the factorization of the invertible KdV linearization $\K_0$ from \eqref{eqn: K0} as
\[
\K_0
= (\partial_z^2-1)(1+\L_1\Sigma)
\]
to conclude the invertibility of $1+\L_1\Sigma$.

The immediate result is our fixed-point problem for $\eta$:
\begin{equation}\label{eqn: Ncal-ep amp}
\eta
= (1+\B_{\ep}^{-1}\P_{\ep,b}(\theta)\Sigma)^{-1}\B_{\ep}^{-1}\P_{\ep,b}(\theta)\rhs_{q,b}^{\ep}(\eta,\alpha,\theta)
=: \Ncal_{q,b}^{\ep,\theta}(\eta,\alpha)
\end{equation}

\subsection{The problem for $\alpha$}\label{sec: amp sel}
The estimate \eqref{eqn: iota-ep-chi amp} allows us to rewrite the selection mechanism \eqref{eqn: sel mech amp} as a preliminary fixed-point equation for $\alpha$:
\[
\alpha
= \frac{e^{b/\ep}\iota_{\ep}[\rhs_{q,b}^{\ep}(\eta,\alpha,\theta)-\Sigma\eta]}{(\tau_{\infty}\omega_{\ep}\theta+1)\iota_{\ep}[\chi_{\ep}^{\theta}+e^{b/\ep}(\tau_{\infty}\omega_{\ep}\theta+1)^{-1}\G_{\ep,b}^{\theta}]}.
\]
However, the term $\Sigma\eta$ here is $\O(1)$, so we replace it with the fixed-point equation for $\eta$ from \eqref{eqn: Ncal-ep amp} to get the actual equation for $\alpha$:
\begin{equation}\label{eqn: A-ep}
\alpha
= \frac{e^{b/\ep}\iota_{\ep}[\rhs_{q,b}^{\ep}(\eta,\alpha,\theta)-\Sigma\Ncal_{q,b}^{\ep,\theta}(\eta,\alpha)]}{(\tau_{\infty}\omega_{\ep}\theta+1)\iota_{\ep}[\chi_{\ep}^{\theta}+e^{b/\ep}(\tau_{\infty}\omega_{\ep}\theta+1)^{-1}\G_{\ep,b}^{\theta}]}
=: (\tau_{\infty}\omega_{\ep}\theta+1)^{-1}\A_{q,b}^{\ep,\theta}(\eta,\alpha).
\end{equation}
For later convenience, we have isolated the factor $(\tau_{\infty}\omega_{\ep}\theta+1)^{-1}$ from the rest of $\A_{q,b}^{\ep,\theta}$.

\subsection{Solving the fixed-point equations for $\eta$ and $\alpha$}
Given $\theta$, the problem \eqref{eqn: THE equation} for $\eta$ and $\alpha$ is therefore equivalent to the pair of fixed-point equations
\begin{equation}\label{eqn: amp sel FP proto}
\eta = \Ncal_{q,b}^{\ep,\theta}(\eta,\alpha)
\quadword{and}
\alpha = (\tau_{\infty}\omega_{\ep}\theta+1)^{-1}\A_{q,b}^{\ep,\theta}(\eta,\alpha).
\end{equation}
We put
\begin{equation}\label{eqn: Xqb}
\X_{q,b}
:= \E_{q,b} \times \R
\quadword{and}
\norm{(\eta,\alpha)}_{q,b}
:= \norm{\eta}_{q,b} + |\alpha|
\end{equation}
and define
\begin{equation}\label{eqn: F-ep-b}
\F_{q,b}^{\ep,\theta}(\eta,\alpha)
:= \big(\Ncal_{q,b}^{\ep,\theta}(\eta,\alpha),(\tau_{\infty}\omega_{\ep}\theta+1)^{-1}\A_{q,b}^{\ep,\theta}(\eta,\alpha)\big)
\end{equation}
to compress \eqref{eqn: amp sel FP proto} as 
\[
(\eta,\alpha) 
= \F_{q,b}^{\ep,\theta}(\eta,\alpha).
\]
We prove that given $\theta$, the map $\F_{q,b}^{\ep,\theta}$ has a fixed point in a suitably small ball in $\X_{q,b}$.

However, $\F_{q,b}^{\ep,\theta}$ ultimately does not have tractable contraction estimates in $\norm{\cdot}_{q,b}$ alone.
This is an artifact of the factor of $z$ that emerges from the Lipschitz estimates in amplitude \eqref{eqn: varphi Lip alpha ultimate} on the periodic solutions and the ``decay-borrowing'' techniques that we need to implement to manage this linearly growing factor in a product that needs to belong to a space of decaying functions.
We discussed this briefly after the statement of Theorem \ref{thm: ultimate varphi theorem}, and these decay-borrowing techniques most prominently appear in the estimates of Lemma \ref{lem: ultimate rhs0}.
Rather, we obtain good contraction estimates if we measure $\F_{q,b}^{\ep,\theta}$ in a space of more slowly decaying functions than its inputs.
This sort of ``mixed space'' quantitative contraction mapping argument has been a technique used in all of the lattice nanopteron constructions \cite{faver-wright, faver-spring-dimer, hoffman-wright, faver-mim-nanopteron, faver-hupkes-equal-mass}.
Theorem \ref{thm: amp FP abstract} generalizes the argument by Faver and Wright \cite[Sec.\@ 6.3, 6.4]{faver-wright}, which was posed in exponentially localized Sobolev spaces, to an abstract Banach space setting.
This approach was refined by Hoffman and Wright \cite[Sec.\@ 8.1]{hoffman-wright} to a Hilbert space setting and then further generalized by Johnson and Wright \cite[Thm.\@ 4]{johnson-wright} to a fixed-point argument posed on reflexive Banach spaces; we do not have this machinery available in the function spaces $\H_{q,b}$, which are neither Hilbert spaces nor reflexive Banach spaces.

Instead, we rely on the following result, which involves three continuously embedded Banach spaces $\Zcal$, $\Y$, and $\X$.
This studies a map on $\X$ that possesses good mapping properties in $\Zcal$, good contraction estimates in $\Y$ relative to inputs from $\Zcal$, and adequate Lipschitz estimates in $\X$ relative to inputs from $\Y$ to produce a fixed point of this map in the intermediate space $\Y$.

\begin{theorem}\label{thm: amp FP abstract}
Let $\X$, $\Y$, and $\Zcal$ be Banach spaces with $\Zcal$ continuously embedded in $\Y$ and $\Y$ continuously embedded in $\X$.
Let $\D \subseteq \X$ be nonempty and suppose that $\F \colon \D \to \X$ is a map with the following properties.

\begin{enumerate}[label={\bf(\roman*)}]

\item
{}
[Mapping property in $\Zcal$]
There is a nonempty set $\D_0 \subseteq \Zcal \cap \D$ such that $\F(\D_0) \subseteq \D_0$ and $\overline{\D_0}^{\Y} \subseteq \D$, where $\overline{\D_0}^{\Y}$ is the closure of $\D_0$ in $\Y$.

\item
{}
[Contraction estimate in $\Y$]
There is $C_1 \in (0,1)$ such that if $z$, $\grave{z} \in \D_0$, then 
\begin{equation}\label{eqn: amp FP abstract contr}
\norm{\F(z)-\F(\grave{z})}_{\Y}
\le C_1\norm{z-\grave{z}}_{\Y}.
\end{equation}

\item
{}
[Lipschitz estimate in $\X$]
There is $C_2 > 0$ such that if $y$, $\grave{y} \in \overline{\D_0}^{\Y}$, then
\begin{equation}\label{eqn: amp FP abstract Lip}
\norm{\F(y)-\F(\grave{y})}_{\X}
\le C_2\norm{y-\grave{y}}_{\X}.
\end{equation}
\end{enumerate}
Then there is $z_{\star} \in \overline{\D_0}^{\Y}$ such that $z_{\star} = \F(z_{\star})$.
If, in addition, $C_2 \in (0,1)$, then this fixed point is unique in $\overline{\D_0}^{\Y}$.
\end{theorem}

\begin{proof}
Since $\D_0$ is nonempty, there is $z_0 \in \D$.
Since $\F(\D_0) \subseteq \D_0$ by the mapping property in $\Zcal$, we may define a sequence $(z_j)$ inductively by $z_{j+1} := \F(z_j)$, $j \ge 0$, to have $z_j \in \D_0$ for all $j$.
Since $\Zcal$ embeds into $\Y$, we have $z_j \in \Y$ for all $j$, and so we may estimate
\[
\norm{z_{j+1}-z_j}_{\Y}
= \norm{\F(z_j)-\F(z_{j-1})}_{\Y}
\le C_1\norm{z_j-z_{j-1}}_{\Y}
\]
by the contraction estimate in $\Y$.
Since $C_1 \in (0,1)$, it follows that the sequence $(z_j)$ is Cauchy in $\Y$, and so there is $z_{\star} \in \Y$ such that $\norm{z_j-z_{\star}}_{\Y} \to 0$ as $j \to \infty$.
Consequently, $z_{\star} \in \overline{\D_0}^{\Y} \subseteq \D$, and so $\F(z_{\star})$ is defined.

We show that $z_{\star} = \F(z_{\star})$.
We begin by bounding
\begin{equation}\label{eqn: contr aux0}
\norm{z_{\star}-\F(z_{\star})}_{\X}
\le \norm{z_{\star}-\F(z_j)}_{\X} + \norm{\F(z_j)-\F(z_{\star})}_{\X}, \ j \ge 0.
\end{equation}
The embedding of $\Y$ into $\X$ gives $C_{\Y \to \X} > 0$ such that $\norm{y}_{\X} \le C_{\Y \to \X}\norm{y}_{\Y}$ for all $y \in \Y$, and so
\begin{equation}\label{eqn: contr aux1}
\norm{z_{\star}-\F(z_j)}_{\X}
\le C_{\Y \to \X}\norm{z_{\star}-\F(z_j)}_{\Y}
= C_{\Y \to \X}\norm{z_{\star}-z_{j+1}}_{\Y}
\to 0 \text{ as } j \to \infty.
\end{equation}
Next, since $z_j$, $z_{\star} \in \overline{\D_0}^{\Y}$, we have
\begin{equation}\label{eqn: contr aux2}
\norm{\F(z_j)-\F(z_{\star})}_{\X}
\le C_2\norm{z_j-z_{\star}}_{\X}
\le C_{\Y \to \X}C_2\norm{z_j-z_{\star}}_{\Y}
\to 0 \text{ as } j \to \infty.
\end{equation}
The first inequality in \eqref{eqn: contr aux2} is the Lipschitz estimate in $\X$, and the second is the embedding of $\Y$ into $\X$.
Combining \eqref{eqn: contr aux0}, \eqref{eqn: contr aux1}, and \eqref{eqn: contr aux2} yields $z_{\star} = \F(z_{\star})$.

Last, for uniqueness, suppose that $z \in \overline{\D_0}^{\Y}$ with $z = \F(z)$ and also $C_2 \in (0,1)$.
Then the Lipschitz estimate in $\X$ implies
\[
\norm{z-z_{\star}}_{\X}
= \norm{\F(z)-\F(z_{\star})}_{\X}
\le C_2\norm{z-z_{\star}}_{\X}. 
\]
Since $0 < C_2 < 1$, we have $\norm{z-z_{\star}}_{\X} = 0$.
\end{proof}

We can apply this theorem to the fixed-point problem $(\eta,\alpha) = \F_{q,b}^{\ep,\theta}(\eta,\alpha)$ because of the following estimates, which we prove in Appendix \ref{app: ultimate amp lemma}.

\begin{lemma}\label{lem: ultimate amp lemma}
Let $0 < b < b_{\sigma}$ and $0 < q < q_{\sigma}$.
Fix $0 \le \theta < \pi$.
There are $C_0$, $\ep_{\star} > 0$ such that the following hold.

\begin{enumerate}[label={\bf(\roman*)}]

\item
{}
[Mapping property]
If $0 < \ep < \ep_{\star}$ and $\norm{(\eta,\alpha)}_{(q+q_{\sigma})/2,b} \le C_0\ep^4$, then 
\begin{equation}\label{eqn: amp sel map est}
\norm{\F_{q,b}^{\ep,\theta}(\eta,\alpha)}_{(q+q_{\sigma})/2,b} 
\le C_0\ep^4.
\end{equation}

\item
{}
[Contraction estimate]
If $0 < \ep < \ep_{\star}$ and $\norm{(\eta,\alpha)}_{(q+q_{\sigma})/2,b}, \norm{(\grave{\eta},\grave{\alpha})}_{(q+q_{\sigma})/2,b} \le C_0\ep^4$, then
\begin{equation}\label{eqn: amp sel contr est}
\norm{\F_{q,b}^{\ep,\theta}(\eta,\alpha)-\F_{q,b}^{\ep,\theta}(\grave{\eta},\grave{\alpha})}_{q,b}
\le \frac{1}{2}\norm{(\eta,\alpha)-(\grave{\eta},\grave{\alpha})}_{q,b}.
\end{equation}

\item
{}
[Lipschitz estimate]
If $0 < \ep < \ep_{\star}$ and $\norm{(\eta,\alpha)}_{q,b}, \norm{(\grave{\eta},\grave{\alpha})}_{q,b}  \le C_0\ep^4$, then
\begin{equation}\label{eqn: amp sel Lip est}
\norm{\F_{q,b}^{\ep,\theta}(\eta,\alpha)-\F_{q,b}^{\ep,\theta}(\grave{\eta},\grave{\alpha})}_{q/2,b}
\le \frac{1}{2}\norm{(\eta,\alpha)-(\grave{\eta},\grave{\alpha})}_{q/2,b}.
\end{equation}
\end{enumerate}
\end{lemma}

\begin{theorem}\label{thm: amp sel}
Let $0 < q < q_{\sigma}$, $0 < b < b_{\sigma}$, and $0 \le \theta < \pi$.
Then there are $\ep_{\star}$, $C > 0$ such that if $0 < \ep < \ep_{\star}$, then there are $\eta_{\ep} \in \E_{q,b}$ and $\alpha_{\ep} \in \R$ such that 
\begin{equation}\label{eqn: amp sel est}
(\eta_{\ep},\alpha_{\ep})
= \F_{q,b}^{\ep,\theta}(\eta_{\ep},\alpha_{\ep}),
\qquad
\norm{\eta_{\ep}}_{q,b} \le C\ep^4,
\quadword{and}
|\alpha_{\ep}| \le C\ep^4.
\end{equation}
If $0 < \theta < \pi$, then $\alpha_{\ep}$ has the smaller estimate
\begin{equation}\label{eqn: extra small amp est}
|\alpha_{\ep}|
< C\ep^5.
\end{equation}
\end{theorem}

\begin{proof}
Fix $0 < \ep < \ep_0$.
We show that Lemma \ref{lem: ultimate amp lemma} allows us to apply Theorem \ref{thm: amp FP abstract} to produce a fixed point of $\F_{q,b}^{\ep,\theta}$ defined in \eqref{eqn: F-ep-b}.
In the notation of that theorem, set $\X = \X_{q/2,b}$, $\Y = \X_{q,b}$, and $\Zcal = \X_{(q+q_{\sigma})/2,b}$, per \eqref{eqn: Xqb}, and take
\[
\D 
= \set{(\eta,\alpha) \in \X_{q/2,b}}{\norm{(\eta,\alpha)}_{q/2,b} \le C_0\ep^4}
\]
and
\[
\D_0 
= \set{(\eta,\alpha) \in \X_{(q+q_{\sigma})/2,b}}{\norm{(\eta,\alpha)}_{(q+q_{\sigma})/2,b} \le C_0\ep^4}.
\]

The embedding estimate \eqref{eqn: embed est} in $\H_{q,b}$ implies that if $0 < q_1 < q_2$ and $(\eta,\alpha) \in \X_{q_2,b}$, then $(\eta,\alpha) \in \X_{q_1,b}$ with $\norm{(\eta,\alpha)}_{q_1,b} \le \norm{(\eta,\alpha)}_{q_2,b}$.
Consequently, $\D_0 \subseteq \D$.
Now suppose that $((\eta_j,\alpha_j))$ is a sequence in $\D_0$ such that $\norm{(\eta_j,\alpha_j)-(\eta,\alpha)}_{q,b} \to 0$ for some $(\eta,\alpha) \in \X_{q,b}$.
Then
\begin{equation}\label{eqn: D0 ests for amp contr}
\norm{(\eta,\alpha)}_{q/2,b}
\le \norm{(\eta,\alpha)}_{q,b}
= \lim_{j \to \infty} \norm{(\eta_j,\alpha_j)}_{q,b}
\le \limsup_{j \to \infty} \norm{(\eta_j,\alpha_j)}_{(q+q_{\sigma})/2,b}
\le C_0\ep^4,
\end{equation}
and so $(\eta,\alpha) \in \D$.
That is, $\overline{\D_0}^{\X_{q,b}} \subseteq \D_0$.

The mapping estimate \eqref{eqn: amp sel map est} confirms $\F_{q,b}^{\ep,\theta}(\D_0) \subseteq \D_0$, the contraction estimate \eqref{eqn: amp sel contr est} confirms \eqref{eqn: amp FP abstract contr} with $C_1=1/2$, and the Lipschitz estimate \eqref{eqn: amp sel Lip est} confirms \eqref{eqn: amp FP abstract Lip} with $C_2=1/2$ as well.
Note that the Lipschitz estimate \eqref{eqn: amp sel Lip est} is valid for $(\eta,\alpha)$, $(\grave{\eta},\grave{\alpha}) \in \overline{\D_0}^{\X_{q,b}}$ by the estimate on $\norm{\cdot}_{q,b}$ from \eqref{eqn: D0 ests for amp contr}.
This produces a unique fixed point $(\eta_{\ep,b},\alpha_{\ep,b}) \in \overline{\D_0}^{\X_{q,b}}$ with $\norm{(\eta_{\ep,b},\alpha_{\ep,b})}_{q,b} \le C_0\ep^4$.

We prove the estimate \eqref{eqn: extra small amp est} in Appendix \ref{app: extra small amp est}; this is effectively a consequence of the bounds on $\A_{q,b}^{\ep,\theta}$ that lead to the mapping estimate \eqref{eqn: amp sel map est} combined with the extra $\O(\ep)$ factor $(\tau_{\infty}\omega_{\ep}\theta+1)^{-1}$ when $\theta > 0$.
\end{proof}

We claim that an additional generalization to these results is possible.
All of the underlying estimates for the amplitude selection problem are uniform in $\theta$ when $\theta$ is restricted to some interval $0 \le \theta < \theta_0 < \pi$, provided that we are not interested in the smaller amplitude estimate \eqref{eqn: extra small amp est} when $\theta = 0$.
In particular, the thresholds $\ep_{\star}$ and $C$ from Theorem \ref{thm: amp sel} can be chosen to depend on $\theta_0$, not the individual $\theta$.
This really hinges on a reinterpretation of Lemma \ref{lem: ultimate amp lemma} for a range of $\theta$ in $[0,\pi)$ bounded away from $\pi$.
Consequently, we can obtain the following result with greater flexibility in $\theta$.

\begin{theorem}\label{thm: amp sel}
Let $0 < q < q_{\sigma}$, $0 < b < b_{\sigma}$, and $0 \le \theta_0 < \pi$.
Then there are $\ep_{\star}$, $C > 0$ such that if $0 < \ep < \ep_{\star}$ and $0 \le \theta_{\ep} < \theta_0$, then there are $\eta_{\ep} \in \E_{q,b}$ and $\alpha_{\ep} \in \R$ such that 
\[
\F_{q,b}^{\ep,\theta_{\ep}}(\eta_{\ep},\alpha_{\ep}) = 0,
\qquad
\norm{\eta_{\ep}}_{q,b} < C\ep^4,
\quadword{and}
|\alpha_{\ep}| < C\ep^4.
\]
\end{theorem}

\section{The Phase Shift Selection Problem}\label{sec: phase shift problem}

\subsection{The problem for $\eta$}
We maintain the assumptions on $\ep$, $\alpha$, $\theta$, and $\eta$ from the start of Section \ref{sec: PS nano ansatz intro}, except now we also assume $\alpha \ne 0$.
We rewrite the phase-shifted nanopteron problem from \eqref{eqn: THE equation} once more as
\begin{equation}\label{eqn: THE problem PS}
\B_{\ep}\eta+\Sigma\eta+\alpha{e}^{-b/\ep}(\tau_{\infty}\omega_{\ep}\theta+1)\chi_{\ep}^{\theta}
= \rhs_{q,b}^{\ep}(\eta,\alpha,\theta)-\alpha\G_{\ep,b}^{\theta}.
\end{equation}
We do this because in our eventual selection of $\theta$, we will find that $\theta = \O(\ep)$ here, and so the term $\alpha\G_{\ep,b}^{\theta}$ will really be $\O(\ep\alpha)$, which is sufficiently small.
This is wholly unlike the situation in the amplitude selection problem, for which we kept $\alpha\G_{\ep,b}^{\theta}$ on the left of \eqref{eqn: THE problem amp}.
The selection mechanism \eqref{eqn: sel mech} now reads
\begin{equation}\label{eqn: sel mech PS}
\iota_{\ep}[\Sigma\eta] + \alpha{e}^{-b/\ep}(\tau_{\infty}\omega_{\ep}\theta+1)\iota_{\ep}[\chi_{\ep}^{\theta}]
= \iota_{\ep}[\rhs_{q,b}^{\ep}(\eta,\alpha,\theta)-\alpha\G_{\ep,b}^{\theta}].
\end{equation}

We show in Appendix \ref{app: iota-ep-chi} that $\iota_{\ep}[\chi_{\ep}^{\theta}]$ is uniformly bounded below away from $0$: for $0 < \theta_0 < \pi$ and $|\theta| < \theta_0$, there is $C > 0$ such that for $\ep > 0$ sufficiently small, we have
\begin{equation}\label{eqn: iota-ep-chi PS}
C
< |\iota_{\ep}[\chi_{\ep}^{\theta}]|.
\end{equation}
This is quite analogous to the estimate \eqref{eqn: iota-ep-chi amp}, except now we allow negative phase shifts but restrict the integrand of $\iota_{\ep}$ to a simpler function.

Now we can show that if the selection mechanism \eqref{eqn: sel mech PS} holds, then we can solve \eqref{eqn: THE problem PS} for $\eta$.
The proof of the following lemma is the same as that of Lemma \ref{lem: how to solve amp}, so we omit it.

\begin{lemma}\label{lem: how to solve PS}
Let $0 < q < q_{\sigma}$, $0 < b < b_{\sigma}$, and $0 < \theta_0 < \pi$.
For $\eta$, $f \in \E_{q,b}$, $A \in \R$, $|\theta| < \theta_0$, and $\ep > 0$ sufficiently small, we have
\[
\B_{\ep}\eta+A\chi_{\ep}^{\theta} 
= f
\]
if and only if 
\[
A
= \frac{\iota_{\ep}[f]}{\iota_{\ep}[\chi_{\ep}^{\theta}]}
\quadword{and}
\eta
= \B_{\ep}^{-1}\Pi_{\ep,b}(\theta)f,
\]
where $\B_{\ep}^{-1}$ is defined in \eqref{eqn: B-ep-inv} and
\begin{equation}\label{eqn: Pi-ep-theta}
\Pi_{\ep,b}(\theta)f
:= f - \frac{\iota_{\ep}[f]}{\iota_{\ep}[\chi_{\ep}^{\theta}]}\chi_{\ep}^{\theta}.
\end{equation}
\end{lemma}

Here $\Pi_{\ep,b}(\theta)$ is a less complicated relative of the projection $\P_{\ep,b}(\theta)$ from \eqref{eqn: P-ep-theta}.
Going forward, suppose that $\alpha$ and $\theta$ have been chosen to make the selection mechanism \eqref{eqn: sel mech PS} true.
We discuss in Section \ref{sec: PS sel} how to do this for $\theta$ given $\alpha$ suitably small, but nonzero.
By Lemma \ref{lem: how to solve PS} with $A = \alpha{e}^{-b/\ep}(\tau_{\infty}\omega_{\ep}\theta+1)$ and $f = \rhs_{q,b}^{\ep}(\eta,\alpha,\theta)-\alpha\G_{\ep,b}^{\theta}$, the problem \eqref{eqn: THE problem PS} is then equivalent to
\[
(1+\B_{\ep}^{-1}\Pi_{\ep,b}(\theta)\Sigma)\eta
= \B_{\ep}^{-1}\Pi_{\ep,b}(\theta)\big(\rhs_{q,b}^{\ep}(\eta,\alpha,\theta)-\alpha\G_{\ep}^{\theta}\big).
\]

As in the amplitude selection problem (and stated precisely as Theorem \ref{thm: KdV perturbation PS}), we may invert $1+\B_{\ep}^{-1}\Pi_{\ep,b}(\theta)$ to obtain our problem for $\eta$:
\begin{equation}\label{eqn: eta eqn PS proto1}
\eta
= (1+\B_{\ep}^{-1}\Pi_{\ep,b}(\theta)\Sigma)^{-1}\B_{\ep}^{-1}\Pi_{\ep,b}(\theta)\big(\rhs_{q,b}^{\ep}(\eta,\alpha,\theta)-\alpha\G_{\ep}^{\theta}\big).
\end{equation}
This closely resembles the fixed-point problem \eqref{eqn: Ncal-ep amp} for $\eta$ in the amplitude selection problem, but we will not stop here.
Rather, we will now solve for $\theta$ in terms of $\eta$ and a suitably chosen amplitude $\theta$ and then return to this equation for $\eta$ to obtain our final fixed-point problem.

\subsection{The problem for $\theta$}\label{sec: PS sel}
Assume $\alpha \ne 0$.
We first rearrange the selection mechanism \eqref{eqn: sel mech PS} into
\begin{equation}\label{eqn: sel mech PS 1}
(\tau_{\infty}\omega_{\ep}\theta)\iota_{\ep}[\chi_{\ep}^{\theta}]
+ \iota_{\ep}[\chi_{\ep}^{\theta}]
= \frac{e^{b/\ep}\iota_{\ep}[\rhs_{q,b}^{\ep}(\eta,\alpha,\theta)-\alpha\G_{\ep,b}^{\theta}-\Sigma\eta]}{\alpha}.
\end{equation}
We show in Appendix \ref{app: iota-ep-chi} that the oscillatory integral $\iota_{\ep}[\chi_{\ep}^{\theta}]$ has the expansion
\[
\iota_{\ep}[\chi_{\ep}^{\theta}]
= \frac{2\sinc((\ep\omega_{\ep})\theta)}{\tau_{\infty}}+\I_{\ep}(\theta),
\]
where $\I_{\ep}(\theta) = \O(e^{-2b/\ep})$ uniformly in $\theta$ (provided that $\theta$ is bounded away from $\pi$).
We substitute this into \eqref{eqn: sel mech PS 1} to obtain
\begin{equation}\label{eqn: sel mech PS 2}
\sin((\ep\omega_{\ep})\theta)
+ \M_{q,b}^{\ep}(\eta,\alpha,\theta)
= 0,
\end{equation}
where
\begin{equation}\label{eqn: M-ep}
\M_{q,b}^{\ep}(\eta,\alpha,\theta)
:= \frac{(\ep\omega_{\ep})\theta\I_{\ep}(\theta)}{2}
+ \frac{\ep\iota_{\ep}[\chi_{\ep}^{\theta}]}{2}
- \frac{\ep{e}^{b/\ep}\iota_{\ep}[\rhs_{q,b}^{\ep}(\eta,\alpha,\theta)-\alpha\G_{\ep,b}^{\theta}-\Sigma\eta]}{2\alpha}.
\end{equation}
We can also show that for fixed $\eta$ and $\alpha$, the map $\M_{q,b}^{\ep}(\eta,\alpha,\cdot)$ is continuous.

If $\eta$ is sufficiently small and $\alpha$ is also sufficiently small---but not too small---then $\M_{q,b}^{\ep}$ will also be small.
The sine term then dominates \eqref{eqn: sel mech PS 2}, and so we can use the intermediate value theorem to obtain a solution $\theta$ to \eqref{eqn: sel mech PS 2} that depends on $\eta$, $\alpha$, and $\ep$.
(Note that $\M_{q,b}^{\ep}$ is real-valued even if $\eta$ is complex-valued, as the integrands in each instance of $\iota_{\ep}$ are even functions.)
We make these claims about $\M_{q,b}^{\ep}$ precise in Appendix \ref{app: M-ep}, which enables us to prove the following result about the phase shift selection in Appendix \ref{app: PS exists}.

\begin{theorem}\label{thm: PS exists}
Let $C_0 > 0$ and $0 < A_0 < A_1 < 1$.
There is $\ep_{\Theta,0} > 0$ such that if $0 < \ep < \ep_{\Theta,0}$, $\eta \in \E_{q,b}$ with $\norm{\eta} \le C_0\ep^4$, and $A_0 < A < A_1$, then there is $\Theta_{q,b}^{\ep,A}(\eta) \in \R$ such that 
\begin{equation}\label{eqn: sel mech PS solved}
\sin((\ep\omega_{\ep})\Theta_{q,b}^{\ep,A}(\eta)) + \M_{q,b}^{\ep}(\eta,A\ep^4,\Theta_{q,b}^{\ep,A}(\eta))
= 0
\end{equation}
and
\begin{equation}\label{eqn: Theta-ep pi bounds}
-\frac{\pi}{4(\ep\omega_{\ep})}
\le \Theta_{q,b}^{\ep,A}(\eta)
\le \frac{\pi}{4(\ep\omega_{\ep})}.
\end{equation}
\end{theorem}

Perhaps surprisingly, we cannot construct this phase shift $\Theta_{q,b}^{\ep,A}(\eta)$ using a quantitative contraction mapping argument, as the map $\theta_{\ep}(\cdot,A)$ has very bad Lipschitz estimates in $\eta$.
These are stated exactly in Lemma \ref{lem: Theta map and Lip}, and we elaborate more on the failure of contracting to construct the phase shift after that lemma, which provides the appropriate context.

We do not rule out here the possibility that $\Theta_{q,b}^{\ep,A}(\eta) = 0$.
In this case, \eqref{eqn: sel mech PS 2} becomes
\begin{equation}\label{eqn: this zero PS is not helpful}
2 + \I_{\ep}(0)
- \frac{e^{b/\ep}\iota_{\ep}[\rhs_{q,b}^{\ep}(\eta,A\ep^4,0)-\Sigma\eta]}{A\ep^4}
= 0.
\end{equation}
Assuming $\norm{\eta}_{q,b} \le C_0\ep^4$, the estimates on $\rhs_{q,b}^{\ep}$ from \eqref{eqn: ultimate R map} show that the term involving $\iota_{\ep}$ in \eqref{eqn: this zero PS is not helpful} is $\O(1)$ in $\ep$, but we cannot determine if its interaction with the concrete term $2$ via \eqref{eqn: this zero PS is not helpful} produces any contradiction.
We discuss in Section \ref{sec: comparison with Lombardi} why Lombardi is able to select a phase shift that is definitely nonzero.

\subsection{Solving the fixed-point equation for $\eta$}
When $\theta = \Theta_{q,b}^{\ep,A}(\eta)$ is given by Theorem \ref{thm: PS exists}, the problem \eqref{eqn: eta eqn PS proto1} for $\eta$, and thus \eqref{eqn: THE problem PS} is equivalent to
\begin{multline}\label{eqn: eta eqn PS}
\eta
= (1+\B_{\ep}^{-1}\Pi_{\ep,b}(\Theta_{q,b}^{\ep,A}(\eta))\Sigma)^{-1}\B_{\ep}^{-1}\Pi_{\ep,b}(\Theta_{q,b}^{\ep,A}(\eta))\big(\rhs_{q,b}^{\ep}(\eta,A\ep^4,\Theta_{q,b}^{\ep,A}(\eta))-A\ep^4\G_{\ep}^{\Theta_{q,b}^{\ep,A}(\eta)}\big) \\
=: \F_{q,b}^{\ep,A}(\eta).
\end{multline}
We then solve this problem for $\eta$ via a routine quantitative contraction mapping argument.
Unlike the situation with the amplitude selection problem, here we are not taking Lipschitz estimates in amplitude, and so we do not need any decay-borrowing techniques that require a mixed space contraction.
See (again) the remarks after Theorem \ref{thm: ultimate varphi theorem} for why the Lipschitz estimates in the phase shift are ultimately nicer than those in amplitude.

Nonetheless, some care is needed to pose this contraction in $\eta$ correctly.
To obtain the map $\theta_{\ep}$ and set up the contraction, we need to take $\norm{\eta}_{q,b} \le C_0\ep^4$ for any $C_0 > 0$, on which various estimates and thresholds depend, but to ensure that $\F_{q,b}^{\ep,A}$ is a contraction on a suitably small ball in $\E_{q,b}$, we need to select $C_0$ appropriately---and this choice of $C_0$ depends on $\F_{q,b}^{\ep,A}$, which depends on $\theta_{\ep}$, which depends on $C_0$.
We resolve this ostensibly circular reasoning in Appendix \ref{app: ultimate PS lemma}, where we prove the following auxiliary result.

\begin{lemma}\label{lem: ultimate PS lemma}
Let $0 < q < q_{\sigma}$, $0 < b < b_{\sigma}$, and $0 < A_0 < A_1 < 1$.
There are $C_0 > 0$ and $\ep_{\Theta} \in (0,\ep_{\Theta,0})$ such that if $\Theta_{q,b}^{\ep,A}(\eta) \in \R$ is the solution to \eqref{eqn: sel mech PS solved} from Theorem \ref{thm: PS exists} for $0 < \ep < \ep_{\Theta}$, $\eta \in \E_{q,b}$ with $\norm{\eta}_{q,b} \le C_0\ep^4$, and $A_0 < A < A_1$, then the following hold for the map $\F_{q,b}^{\ep,A}$ defined in \eqref{eqn: eta eqn PS}.

\begin{enumerate}[label={\bf(\roman*)}]

\item
If $0 < \ep < \ep_{\Theta}$, $\norm{\eta}_{q,b} \le C_0\ep^4$ and $A_0 < A < A_1$, then 
\begin{equation}\label{eqn: eta eqn PS-ep-theta map PS}
\norm{\F_{q,b}^{\ep,A}(\eta)}_{q,b}
\le C_0\ep^4.
\end{equation}

\item
If $0 < \ep < \ep_{\Theta}$, $0 \le \norm{\eta}_{q,b}$, $\norm{\grave{\eta}}_{q,b} < C_0\ep^4$, and $A_0 < A < A_1$, then 
\begin{equation}\label{eqn: eta eqn PS Lip PS}
\norm{\F_{q,b}^{\ep,A}(\eta)-\F_{q,b}^{\ep,A}(\grave{\eta})}_{q,b}
\le \frac{1}{2}\norm{\eta-\grave{\eta}}_{q,b}.
\end{equation}
\end{enumerate}
\end{lemma}

This lemma ensures that $\F_{q,b}^{\ep,A}$ is a contraction on the ball $\set{\eta \in \E_{q,b}}{\norm{\eta}_{q,b} \le C_0\ep^4}$, and so the (ordinary) contraction mapping principle yields the following.

\begin{theorem}\label{thm: PS sel}
Let $0 < q < q_{\sigma}$, $0 < b < b_{\sigma}$, and $0 < A_0 < A_1 < 1$.
Then there are $\ep_{\star}$, $C > 0$ such that if $0 < \ep < \ep_{\star}$ and $A_0 < A < A_1$, there is $\eta_{\ep} \in \E_{q,b}$ such that 
\[
\eta_{\ep}
= \F_{q,b}^{\ep,A}(\eta_{\ep})
\quadword{and}
\norm{\eta_{\ep}}_{q,b}
\le C\ep^4.
\]
\end{theorem}

\section{Further Comparisons with Prior Approaches}\label{sec: prior approaches}

\subsection{Amick and Toland's factorization}
As our work is informed by the lattice nanopteron constructions \cite{faver-wright, faver-spring-dimer, hoffman-wright, faver-hupkes-equal-mass, faver-mim-nanopteron}, which themselves are outgrowths of Amick and Toland's strategy \cite{amick-toland}, unsurprisingly our strategy and Amick and Toland's proceed in the same broad strokes.
However, we note a number of meaningful differences.
First, their periodic construction is quite different and involves an implicit function theorem argument on a rescaled version of the problem \eqref{eqn: the problem}.
This argument produces sufficient regularity of the periodics in the amplitude to enabled them to solve for the amplitude purely as a function of the localized remainder \cite[Thm.\@ 3.3]{amick-toland} from their version of the selection mechanism \eqref{eqn: sel mech}.

Consequently, Amick and Toland did not run a two-component fixed-point argument for amplitude and remainder as we do in Theorem \ref{thm: amp sel}.
Their subsequent construction of the remainder \cite[pp.\@ 61--63]{amick-toland} required the inversion of a somewhat different linear operator from our $1+\B_{\ep}^{-1}\P_{\ep,b}(\theta)\Sigma$ in \eqref{eqn: eta pre-fp}, and the connection to the invertible KdV linearization $\K_0$ is more subtle.
In particular, a small regular perturbation of $\K_0$ (which they call $D_{\ep}$ \cite[p.\@ 48]{amick-toland}) is a regular tool in their estimates.

Motivated by our reading of this perturbation, we conjecture that the following approach would also work for our treatment of the problem \eqref{eqn: the problem}.
Arguably the two most significant operators in our analysis are the singularly perturbed oscillator operator $\ep^2(\partial_z^2+\omega_{\ep}^2)$ and the invertible KdV linearization $\K_0$.
The work would be easier if the linearization $\B_{\ep}+\Sigma$ of the problem \eqref{eqn: the problem} at the KdV solitary wave profile $\sigma$ factored as the product $\K_0(\ep^2\partial_z^2+(\ep\omega_{\ep})^2)$.
Then we could invert $\K_0$ to turn the problem $\B_{\ep}\eta+\Sigma\eta=g$ into $\ep^2(\partial_z^2+\omega_{\ep}^2)\eta = \K_0^{-1}g$.
(Since $\K_0$ is a variable-coefficient operator, it will not commute with $\ep^2(\partial_z^2+\omega_{\ep}^2)$, so with the goal of inverting $\K_0$, we keep it on the left in this fictitious product.)

This factorization, however, does not occur, so we might wonder by how much $\B_{\ep}+\Sigma$ fails to factor in this way and consider the difference 
\[
\B_{\ep}+\Sigma-\K_0(\ep^2\partial_z^2+(\ep\omega_{\ep})^2)
= ((1-(\ep\omega_{\ep})^2)+\ep^2(1-2\sigma))\partial_z^2+((\ep\omega_{\ep})^2-1)(1-2\sigma).
\]
It follows that the problem $(\B_{\ep}+\Sigma)f = g$ is equivalent to
\begin{equation}\label{eqn: factored AT}
\ep^2f''+(\ep\omega_{\ep})f
= \ep^2\K_0^{-1}(1+\Sigma)\partial_z^2f-(1-(\ep\omega_{\ep})^2)f+\K_0^{-1}g.
\end{equation}
We could then run our nanopteron program starting from the phase-shifted nanopteron ansatz \eqref{eqn: nanopteronsatz} by taking $f = \eta$ and $g = \sigma+\upsilon_{\ep,q}+\alpha\varphi_{\ep,b}^{\alpha,\theta}$ in \eqref{eqn: factored AT}.

\subsection{Sun and Shen's decoupling}
Sun and Shen \cite{sun-shen-at-expn-small} deployed a rescaling of the independent variable in \eqref{eqn: the problem} and a nonlinear change of the dependent variable to obtain an equivalent system of the form
\begin{equation}\label{eqn: SS}
\begin{cases}
p_1''-p_1+p_1^2 = \ep^2F_{1,\ep}(p_1,p_2) \\
\ep^2p_2''+(\ep\omega_{\ep})^2p_2 = \ep^2F_{2,\ep}(p_1,p_2),
\end{cases}
\end{equation}
where $F_{1,\ep}$ and $F_{2,\ep}$ contain some differential operators.
We view this as ``KdV coupled to an oscillatory field,'' as the first equation here is plainly the KdV traveling wave problem on the left, and the second is the equation of motion for a harmonic oscillator.
This decouples the KdV component of the problem from the solvability condition, and \eqref{eqn: SS} actually resembles more closely the traveling wave problems for lattice nanopterons than \eqref{eqn: the problem} does, as discussed in Section \ref{sec: solv cond and nano}.
Sun and Shen then rewrote the nanopteron problem for \eqref{eqn: SS} using integral operators to invert the KdV linearization and the harmonic oscillator equation.
We view a construction of nanopterons in an arbitrary ``KdV coupled to an oscillatory field'' system that cannot be reduced back to the one-component problem \eqref{eqn: the problem} as the next natural extension of our modern techniques here, and one that would further enrich understanding of the lattice nanopteron problems.

\subsection{Lombardi's Jost solutions}\label{sec: comparison with Lombardi}
Lombardi's approach \cite[p.\@ 203]{lombardi} would first make the long wave scaling $U(X) = \ep^2u(\ep{X})$ to convert the problem \eqref{eqn: the problem} to the regularly perturbed problem 
\[
U^{(4)}+U''-\ep^2U+\ep^4U^2,
\]
which is then converted to a first-order system of four ordinary differential equations via the usual change of variables.
It can be checked that this system satisfies some particular hypotheses that then guarantee nanopteron solutions, and Haragus and Wahl\'{e}n performed this check in their proofs of \cite[Prop.\@ 2.1, 2.2]{haragus-wahlen}.
In more detail, Lombardi would perform a nonlinear normal form change of variables and then effectively undo the long wave scaling to convert this first-order system into another first-order system of four ordinary differential equations \cite[p.\@ 191, Eqn.\@ (7.5)]{lombardi} that, up to some complicated nonlinear terms, is essentially a four-component version of ``KdV coupled to an oscillatory field'' as in \eqref{eqn: SS}.
That the conversion of \eqref{eqn: the problem} to a form amenable to Lombardi's techniques involves both doing and undoing a long wave scaling is one reason why we are interested in seeing Lombardi's methods applied directly to \eqref{eqn: the problem} with no changes of variables.

Those nonlinear terms resulting from the normal form transformation prevent this system from glibly condensing into a two-component problem exactly like \eqref{eqn: SS}, but its structure is transparent enough that a solvability condition for the range of the linearization at the KdV solitary wave can be achieved \cite[Lem.\@ 7.3.18]{lombardi}.
Here the solvability condition is that an integral of the form
\[
\int_{-\infty}^{\infty} f(x)e^{i(\omega_{\ep}x+\ep\vartheta\tau(x))} \dx
\]
vanishes, where $\vartheta$ is a known nonzero constant independent of $\ep$ \cite[Lem.\@ 7.3.14]{lombardi}.
That is, the solvability condition itself incorporates an asymptotic phase shift, distinct from the ripple's phase shift.

As previously discussed, Lombardi's method then fixes the ripple amplitude to be exponentially small and nonzero and solves for the ripple's phase shift $\theta$.
In the process, the solvability condition's phase shift $\vartheta$ interacts with the to-be-determined ripple phase shift $\theta$ in just the right way to ensure that the selected $\theta$ is nonzero \cite[Prop.\@ 7.3.20]{lombardi}.
This is why Lombardi's selected phase shift is definitely nonzero, whereas ours possibly may be zero.

We expect that Lombardi's approach to the solvability condition could work in our framework as follows.
The linearization of the full problem \eqref{eqn: the problem} at the KdV solitary wave $\sigma$ is
\[
(\B_{\ep}+\Sigma)f
= \ep^2f^{(4)}+f''-f+2\sigma{f},
\]
and we expect that $(\B_{\ep}+\Sigma)f = g$, $f$, $g \in \E_{q,b}$, if and only if 
\begin{equation}\label{eqn: j-ep osc int}
\int_{-\infty}^{\infty} g(x)j_{\ep}(x) \dx
= 0
\end{equation}
for some ``asymptotically sinusoidal'' function $j_{\ep}$, i.e., 
\[
j_{\ep}(z)
= \cos(\omega_{\ep}z+\vartheta_{\ep}\tau(z))+f_{\ep}(z), 
\ f_{\ep} \in \E_{q,b}.
\]
Such a function $j_{\ep}$ could be constructed following the methods of Hoffman and Wright for Jost solutions to a singularly perturbed Schr\"{o}dinger equation \cite[Lem.\@ C.1]{hoffman-wright}.
We would then need to verify an analogue of Lombardi's exponentially small estimate \eqref{eqn: Lombardi} for oscillatory integrals of the form \eqref{eqn: j-ep osc int}.
Very broadly (and optimistically), we would redo the work in Sections \ref{sec: amp sel} and \ref{sec: phase shift problem} by replacing the functional $\iota_{\ep}$ with $g \mapsto \medint_{-\infty}^{\infty} g(x)j_{\ep}(x) \dx$ and keeping $\B_{\ep}+\Sigma$ together and inverting this superposition, not just $\B_{\ep}$.

\subsection{Sun's asymptotics}
The asymptotic relationships between amplitude and phase shift that Sun \cite{sun-at-phase-shift} developed are, up to some changes in notation and scaling, essentially the same as an expansion of the selection mechanism \eqref{eqn: sel mech}.
(In particular, Sun's amplitude $A$ is what we would call $a$, i.e., the ripple amplitude before we preselected it to be exponentially small as $a = \alpha{e}^{-b/\ep}$.)
We prove in Appendix \ref{app: amp PS asymp} that the selection mechanism expands as
\begin{equation}\label{eqn: amp PS asymp}
\alpha\left(\frac{\theta}{\ep}+1\right)\sinc(\theta)+\rho_{\ep}(\alpha,\theta) = 0,
\qquad
|\rho_{\ep}(\alpha,\theta)|
\le C\big(\ep^4+\ep\alpha+\ep^{-1}\alpha^2\big),
\end{equation}
and this is our version of Sun's asymptotics \cite[p.\@ 1164]{sun-at-phase-shift}.

\appendix

\section{The Function Space $\H_{q,b}$}\label{app: Hqb}

Recall from Section \ref{sec: function spaces} that for $b > 0$ and $q \in \R$, $\H_{q,b}$ is the space of analytic functions on the strip $\U_b := \set{z \in \C}{|\im(z)| < b}$ such that 
\[
\norm{f}_{q,b}
:= \sup_{z \in \U_b} e^{q|\re(z)|}|f(z)|
< \infty.
\]
For $q > 0$, functions in $\H_{q,b}$ vanish exponentially fast as $\re(z) \to \pm\infty$, while for $q < 0$, they may grow.

\subsection{General estimates in $\H_{q,b}$}

The following estimates will be very useful, and their proofs are straightforward consequences of the definition of $\norm{\cdot}_{q,b}$.

\begin{lemma}
Let $b$, $q$, $\grave{q} > 0$.

\begin{enumerate}[label={\bf(\roman*)}]

\item{}
[Product estimate]
If $f \in \H_{0,b}$ and $g \in \H_{q,b}$, then $fg \in \H_{q,b}$ with
\begin{equation}\label{eqn: product est}
\norm{fg}_{q,b}
\le \norm{f}_{0,b}\norm{g}_{q,b}.
\end{equation}

\item{}
[Product estimate]
If $f$, $g \in \H_{q,b}$, then $fg \in \H_{q,b}$ with 
\begin{equation}\label{eqn: product est2}
\norm{fg}_{q,b}
\le \norm{f}_{q,b}\norm{g}_{q,b}.
\end{equation}

\item{}
[Decay enhancing]
If $f \in \H_{q,b}$ and $g \in \H_{\grave{q},b}$, then $fg \in \H_{q+\grave{q},b}$ with
\begin{equation}\label{eqn: decay-enhancing est}
\norm{fg}_{q+\grave{q},b}
\le \norm{f}_{q,b}\norm{g}_{\grave{q},b}.
\end{equation}

\item{}
[Decay borrowing]
If $f \in \H_{-\grave{q},b}$, $g \in \H_{q+\grave{q},b}$, then $fg \in \H_{q,b}$ with
\begin{equation}\label{eqn: decay-borrowing est}
\norm{fg}_{q,b}
\le \norm{f}_{q-\grave{q},b}\norm{g}_{\grave{q},b}.
\end{equation}

\item{}
[Embedding]
If $q < \grave{q}$ and $f \in \H_{\grave{q},b}$, then $f \in \H_{q,b}$ with
\begin{equation}\label{eqn: embed est}
\norm{f}_{q,b}
\le \norm{f}_{\grave{q},b}.
\end{equation}
\end{enumerate}
\end{lemma}

The decay-borrowing estimate \eqref{eqn: decay-borrowing est} is, perhaps, the most subtle of the helpful estimates above, and it will be our chief tool in counteracting linearly growing factors that emerge when we consider Lipschitz estimates for the periodic terms in amplitude, which are a consequence of \eqref{eqn: varphi Lip alpha ultimate}.
The estimate \eqref{eqn: decay-borrowing est} is inspired by \cite[Eqn.\@ (A.15)]{faver-wright}, and such an estimate has been used in the other nanopteron constructions in \cite{faver-spring-dimer, hoffman-wright, faver-hupkes-equal-mass, faver-mim-nanopteron, johnson-wright}.

The next estimate shows that if $0 < \grave{b} < b$, then the restriction to $\U_{\grave{b}}$ of the derivative of any function in $\H_{q,b}$ is in $\H_{q,\grave{b}}$.

\begin{lemma}\label{lem: Hqb containment}
Let $0 < \grave{b} < b$, $q > 0$, and $j \ge 1$.
If $f \in \H_{q,b}$, then $f^{(j)} \in \H_{q,\grave{b}}$ with 
\[
\norm{f^{(j)}}_{q,\grave{b}}
\le \frac{2^jj!}{(b-\grave{b})^{j-1}}\norm{f}_{q,b}.
\]
\end{lemma}

\begin{proof}
If $z \in \U_{\grave{b}}$ and $w \in \C$ with $|w-z| \le (b-\grave{b})/2 =: b_0$, then $w \in \U_b$.
Consequently, for $z \in \U_{\grave{b}}$, the Cauchy integral formula gives
\[
f^{(j)}(z)
= \frac{j!}{2\pi{i}}\int_{|w-z| = b_0} \frac{f(w)}{(w-z)^{j+1}} \dw
= \frac{j!}{2\pi{i}}\int_0^{2\pi} \frac{f(z+b_0{e}^{is})}{b_0^j{e}^{ijs}} ib_0 \ds.
\]
We may therefore estimate
\begin{equation}\label{eqn: Hqb b containment aux1}
e^{q|\re(z)|}|f^{(j)}(z)|
\le \frac{j!}{2\pi{b}_0^{j-1}}\int_0^{2\pi} e^{q|\re(z)|}|f(z+b_0{e}^{is})| \ds,
\end{equation}
where 
\begin{equation}\label{eqn: Hqb b containment aux2}
e^{q|\re(z)|}|f(z+b_0{e}^{is})| 
\le e^{q|\re(z)|}e^{-q|\re(z+b_0{e}^{is})|}\norm{f}_{q,b}.
\end{equation}
By the reverse triangle inequality,
\[
|\re(z)|-b_0|\cos(s)|
\le |\re(z)+b_0\cos(s)|
= |\re(z+b_0{e}^{is})|,
\]
and so
\begin{equation}\label{eqn: Hqb b containment aux3}
e^{q|\re(z)|}e^{-q|\re(z+b_0{e}^{is})|}\norm{f}_{q,b}
\le e^{-b_0|\cos(s)|}\norm{f}_{q,b}
\le \norm{f}_{q,b}.
\end{equation}
We combine \eqref{eqn: Hqb b containment aux1}, \eqref{eqn: Hqb b containment aux2}, and \eqref{eqn: Hqb b containment aux3} to conclude
\[
e^{q|\re(z)|}|f^{(j)}(z)|
\le \frac{j!}{b_0^{j-1}}\norm{f}_{q,b}.
\qedhere
\]
\end{proof}

The following interpolation estimate, whose proof is based on \cite[Ex.\@ 5.14]{rudin}, will help us show that knowledge of $f$ and $f''$ for $f \in \H_{q,b}$ can be enough to conclude that $f \in \H_{q,b}^2$.

\begin{lemma}\label{lem: interpolation}
If $f \in \H_{q,b}$ with $f'' \in \H_{q,b}$, then $f' \in \H_{q,b}$ with 
\begin{equation}\label{eqn: interpolation}
\norm{f'}_{q,b}
\le (e^q+1)\norm{f}_{q,b}+\frac{e^q}{q}\norm{f''}_{q,b}.
\end{equation}
\end{lemma}

\begin{proof}
Let $z \in \C$.
By Taylor's theorem,
\[
f(z-1)
= f(z)-f'(z)+\int_0^1 (1-s)f''(z-s) \ds,
\]
and so
\[
f'(z)
= f(z)-f(z-1)-\int_0^1 (1-s)f''(z-s) \ds.
\]
We estimate
\[
|f(z-1)|
\le e^{-q|\re(z)-1|}\norm{f}_{q,b}
\le e^qe^{-q|\re(z)|}\norm{f}_{q,b}.
\]
using the reverse triangle inequality on
\[
|\re(z)| - 1
\le |\re(z)-1|.
\]
Likewise,
\[
|f''(z-s)|
\le e^{-q|\re(z)-s|}\norm{f''}_{q,b}
\le e^{qs}e^{-q|\re(z)|}\norm{f''}_{q,b}.
\]
We obtain
\[
e^{q|\re(z)|}|f'(z)|
\le (e^q+1)\norm{f}_{q,b} +\norm{f''}_{q,b}\int_0^1 e^{qs} \ds
= (e^q+1)\norm{f}_{q,b}+\frac{e^q}{q}\norm{f''}_{q,b}.
\qedhere
\]
\end{proof}

There are two elementary functions whose $\H_{q,b}$-norms it will be convenient to know, or at least estimate, fairly precisely.
First we study the $\H_{q,b}$-norm of the function $f(z) = z$, which plays a major role in the actual application of the decay-borrowing estimate (see the proofs of part \ref{part: ultimate fake quadratic prep2} of Lemma \ref{lem: ultimate fake quadratic prep}, part \ref{part: ultimate J1 prep 2} of Lemma \ref{lem: ultimate J1 prep}, and part \ref{part: ultimate rhs0 2} of Lemma \ref{lem: ultimate rhs0}).

\begin{lemma}\label{lem: Hqb norm of z}
Let $q$, $b > 0$.
Then
\[
\sup_{z \in \U_b} e^{-q|\re(z)|}|z|
\le \frac{1}{eq}+b.
\]
\end{lemma}

\begin{proof}
Put $f(x) := xe^{-qx}$, so $f'(x) = e^{-qx}(1-qx)$.
Then $f$ is increasing on $[0,1/q)$ and decreasing on $(1/q,\infty)$, so the global maximum of $f$ occurs at $x=1/q$, and $f(1/q) = 1/eq$.
Now we estimate
\[
\sup_{z \in \U_b} e^{-q|\re(z)|}|z|
= \sup_{\substack{x \in \R \\ |y| < b}} e^{-q|x|}|x+iy|
\le \sup_{X \ge 0} \big(Xe^{-qX} + be^{-qX}\big)
\le \frac{1}{eq}+b.
\qedhere
\]
\end{proof}

Second, the definition of the complex sine and cosine immediately yields the following estimate.

\begin{lemma}\label{lem: sin cos Hqb}
Let $b$, $\omega > 0$.
Then
\begin{equation}\label{eqn: sin cos Hqb}
\sup_{z \in \U_b} |\cos(\omega{z})|
\le e^{\omega{b}}
\quadword{and}
\sup_{z \in \U_b} |\sin(\omega{z})|
\le e^{\omega{b}}.
\end{equation}
\end{lemma}

Finally, since we work primarily with even functions, the following ``reflection'' estimate is useful.

\begin{lemma}
For $f \in \H_{q,b}$, define
\begin{equation}\label{eqn: reflection}
(Rf)(z)
:= f(-z).
\end{equation}
Then $Rf \in \H_{q,b}$ and $\norm{Rf}_{q,b} = \norm{f}_{q,b}$.
\end{lemma}

\subsection{The compact embedding of $\H_{q+\grave{q},b}^{r+s}$ into $\H_{q,b}^r$}
We will invoke this compact embedding when applying the Fredholm alternative to prove the invertibility of the operator $\K_0 \colon \E_{q,b} \to \E_{q,b}$ defined in \eqref{eqn: K0}; see Theorem \ref{thm: K0 inv}.
The proof is largely classical analysis and resembles the arguments that the Sobolev space $H^{r+1}(I)$ is compactly embedded in $H^r(I)$ for closed, bounded subintervals $I \subseteq \R$.
A related proof for localized Sobolev spaces on the line is given in \cite[App.\@ C.3.4]{faver-dissertation}.

\begin{theorem}\label{thm: compact embedding}
Let $\grave{q}$, $b > 0$ and $s \ge 1$.
Then $\H_{q+\grave{q},b}^{r+s}$ is compactly embedded in $\H_{q,b}^r$.
\end{theorem}

\begin{proof}
We first prove this for $s=1$.
Let $(f_j)$ be a bounded sequence in $\H_{q+\grave{q},b}^1$.
Put
\[
M
:= \sup_{j \ge 1} \norm{f_j}_{q+\grave{q},b,1}
= \sup_{j \ge 1} \big(\norm{f_j}_{q+\grave{q},b}+\norm{f_j'}_{q+\grave{q},b}\big).
\]
The sequence $(f_j)$ is therefore uniformly bounded and uniformly Lipschitz on $\U_b$.
By a result of elementary classical analysis, each function $f_j$ can therefore be extended to a continuous function $\tilde{f}_j$ on the closed strip
\[
\overline{\U}_b
= \set{z \in \C}{|\im(z)| \le b},
\]
and the extensions are still uniformly bounded and uniformly Lipschitz with the same constant $M$ that the sequence $(f_j)$ possesses.

For $N \ge 1$, let $\U_{b,N}$ be the ``open box''
\[
\U_{b,N}
:= \set{z \in \C}{|\re(z)| < N, \ |\im(z)| < b}
\]
and let $\overline{\U}_{b,N}$ be its closure.
Then the sequence $(\tilde{f}_j)$ is uniformly bounded and uniformly Lipschitz on the compact set $\overline{\U}_{n,N}$, thus equicontinuous on this set.
The Arzela--Ascoli theorem provides a subsequence $(\tilde{f}_{g_1(j)})$  of $(f_j)$ that converges uniformly to a function $\tilde{f}_{1,\infty}$ on $\overline{\U}_{b,1}$.
Here $g_1 \colon \N \to \N$ is some strictly increasing indexing function.

We then repeat this reasoning inductively to construct a family of subsequences $(\tilde{f}_{g_N(j)})$ of $(\tilde{f}_j)$ such that $(\tilde{f}_{g_{N+1}(j)})$ is a subsequence of $(\tilde{f}_{g_N(j)})$ and $(\tilde{f}_{g_N(j)})$ converges uniformly on $\overline{\U}_{b,N}$ to some function $\tilde{f}_{N,\infty}$ on $\overline{\U}_{b,N}$, which, by uniform convergence, is analytic on $\overline{\U}_{b,N}$.
Moreover, by uniqueness of the limit function, $\tilde{f}_{N+1,\infty}(z) = \tilde{f}_{N,\infty}(z)$ if $|\re(z)| \le N$, and so we may define a function $f \colon \U_b \to \C$ by 
\[
f(z)
:= f_{N,\infty}(z), \
|\re(z)| < N.
\]
Certainly $f$ is analytic on each box $\overline{\U}_{b,N}$ and therefore on $\U_b$.

We claim that $f \in \H_{q+\grave{q},b}$ and that a subsequence of $(f_j)$ converges to $f$ in $\norm{\cdot}_{q,b}$.
Note that $f$ decays at a faster rate than the space in which the convergence takes place; this will be important shortly.
First we estimate $e^{(q+\grave{q})|\re(z)|}|f(z)|$.
Fix $z \in \U_b$ and take $N \ge 1$ such that $|\re(z)| \le N$, so $f(z) = f_{\infty,N}(z)$.
Then
\[
e^{(q+\grave{q})|\re(z)|}|f(z)|
= e^{(q+\grave{q})|\re(z)|}|f_{\infty,N}(z)|
= \lim_{j \to \infty} e^{(q+\grave{q})|\re(z)|}|f_{g_N(j)}(z)|
\le \limsup_{j \to \infty} \norm{f_{g_N(j)}}_{q+\grave{q},b}
\le M,
\]
and so $f \in \H_{q+\grave{q},b}$ with $\norm{f}_{q+\grave{q},b} \le M$.

For the subsequence, since $(f_{g_N(j)})$ converges to $f_{\infty,N}$ uniformly on $\overline{\U}_{b,N}$ for each $N \ge 1$, there is an integer $j_N \ge 1$ such that if $j \ge j_N$, then
\begin{equation}\label{eqn: what jN does}
\sup_{\substack{z \in \overline{\U}_{b,N} \\ j \ge j_N}} |f_{g_N(j)}(z)-f_{g_N(j_N)}(z)|
< \frac{1}{Ne^{\grave{q}N}},
\end{equation}
and we may assume $j_{N+1} > j_N$.
Put $h \colon \N \to \N \colon N \mapsto g_N(j_N)$.
We claim that $(f_{h(N)})$ is a subsequence of $(f_j)$ that converges to $f$ in $\H_{q+\grave{q},b}$.
First, since any indexing function $g_N$ is strictly increasing, we have $g_N(j_N) < g_N(j_{N+1})$.
Second, since $(f_{g_{N+1}}(j))$ is a subsequence of $(f_{g_N(j)})$, we have $g_N(j_{N+1}) \le g_{N+1}(j_{N+1})$.
So, $h$ is strictly increasing, and therefore $(f_{h(N)})$ is a subsequence of $(f_j)$.

Last, we prove the convergence of $(f_{h(N)})$ to $f$ in $\H_{q,b}$.
Let $\ep > 0$ and choose $N_0$ so large that 
\[
\min\left\{\frac{1}{N_0}, 2Me^{-\grave{q}N_0}\right\}
< \ep.
\]
It is here that we are using the hypothesis that $\grave{q} > 0$.

Assume $N \ge N_0$ and fix $z \in \U_b$.
If $|\re(z)| < N$, then $f(z) = f_{\infty,N}(z)$, and so, by \eqref{eqn: what jN does},
\[
e^{q|\re(z)|}|f_{h(N)}(z)-f(z)|
= e^{q|\re(z)|}|f_{h(N)}(z)-f_{\infty,N}(z)|
< \frac{e^{\grave{q}N}}{Ne^{\grave{q}N}}
= \frac{1}{N}
\le \frac{1}{N_0}
< \ep.
\]
If $|\re(z)| \ge N \ge N_0$, then
\begin{multline*}
e^{q|\re(z)|}|f_{h(N)}(z)-f(z)|
\le e^{-\grave{q}|\re(z)|}\big(e^{(q+\grave{q})|\re(z)|}|f_{h(N)}(z)|+e^{(q+\grave{q})|\re(z)|}|f(z)|\big) \\
\le e^{-\grave{q}N_0}\big(\norm{f_{h(N)}}_{q+\grave{q},b}+\norm{f}_{q+\grave{q},b}\big)
\le 2Me^{-\grave{q}N_0}
< \ep.
\end{multline*}

This completes the proof that $\H_{q+\grave{q},b}^1$ is compactly embedded in $\H_{q,b}$ for $\grave{q} > 0$.
Next, we show that $\H_{q+\grave{q},b}^{r+1}$ is compactly embedded in $\H_{q,b}^r$ for $\grave{q} > 0$ by induction on $r$.
The base case is $r=0$, and we just did that above.
Assuming that $\H_{q+\grave{q},b}^{r+1}$ is compactly embedded in $\H_{q,b}$ for some $\grave{q} > 0$ and $r \ge 0$, let $(f_j)$ be a bounded sequence in $\H_{q+\grave{q},b}^{r+2}$.
Then $(f_j)$ is also bounded in $\H_{q+\grave{q},b}^r$ and so has a subsequence $(f_{g_1(j)})$ that converges in $\H_{q,b}$ to some $f \in \H_{q,b}$.
Since $f_j \in \H_{q+\grave{q},b}^{r+2}$, we have $f_{g_1(j)}^{(r+1)} \in \H_{q+\grave{q},b}^1$ with $\norm{f_{g_1(j)}^{(r+1)}}_{q+\grave{q},b,1} \le \norm{f_{g_1(j)}}_{q+\grave{q},b,r+2}$.
That is, $(f_{g_1(j)}^{(r+1)})$ is a bounded sequence in $\H_{q+\grave{q},b}^1$, which is compactly embedded in $\H_{q,b}$.
Consequently, there is a subsequence $(f_{g_2(j)}^{(r+1)})$ of $(f_{g_1(j)}^{(r+1)})$ that converges to some $h \in \H_{q,b}$.
Since $(f_{g_2(j)})$ converges to $f \in \H_{q,b}$, by uniform convergence we must have $h = f^{(r+1)}$.
Thus
\[
\norm{f_{g_2(j)}-f}_{q,b,r+1}
= \norm{f_{g_2(j)}-f}_{q,b,r}+\norm{f_{g_2(j)}^{(r+1)}-f^{(r+1)}}_{q,b}
\to 0 \text{ as } j \to \infty.
\]

Finally, we show that $\H_{q+\grave{q},b}^{r+s}$ is compactly embedded in $\H_{q,b}^r$ when $r \ge 0$, $\grave{q} > 0$, and $s \ge 1$ by induction on $s$.
The case $s=1$ is the compact embedding of $\H_{q+\grave{q},b}^{r+1}$ into $\H_{q,b}^r$.
Assume that $\H_{q+\grave{q},b}^{r+s}$ is compactly embedded in $\H_{q,b}^r$ for some $\grave{q} > 0$, $s \ge 1$, and all $r \ge 0$.
Then $\H_{q+\grave{q},b}^{r+(s+1)} = \H_{q+\grave{q},b}^{(r+s)+1}$, and $\H_{q+\grave{q},b}^{(r+s)+1}$ is compactly embedded in $\H_{q+\grave{q},b}^{r+s}$ by the previous work, and $\H_{q+\grave{q},b}^{r+s}$ is compactly embedded in $\H_{q,b}^r$ by the induction hypothesis.
\end{proof}

\section{Estimates on $\omega_{\ep}$ and $\mu_{\ep}$}\label{app: numbers}

\subsection{Estimates on $\omega_{\ep}$}\label{app: omega-ep}
We prove Lemma \ref{lem: omega-ep}.

\begin{enumerate}[label={\bf(\roman*)}]

\item
This follows from the quadratic formula.

\item
First,
\[
\ep\omega_{\ep}
= \sqrt{\frac{\sqrt{1+4\ep^2}+1}{2}}.
\]
Then, for $\ep > 0$,
\[
\ep\omega_{\ep}
> \sqrt{\frac{1+1}{2}}
= 1,
\]
and if $\ep < 1$, then
\[
\ep\omega_{\ep}
< \sqrt{\frac{1+\sqrt{5}}{2}}
< \sqrt{\frac{1+7}{2}}
= 2,
\]
where we have used $\sqrt{5} < 7$.

\item
We have
\[
2(\ep\omega_{\ep})^2-1
= 2\ep^2\frac{\sqrt{1+4\ep^2}+1}{2\ep^2}-1
= \sqrt{1+4\ep^2}+1-1
= \sqrt{1+4\ep^2}.
\]

\item
For $\omega > 1$, we have
\begin{equation}\label{eqn: difference of squares}
\sqrt{\omega}-1
= \frac{\omega-1}{\sqrt{\omega}+1}
\le \omega-1.
\end{equation}
This gives
\begin{equation}\label{eqn: omega est aux}
\omega_{\ep}-\frac{1}{\ep}
= \frac{1}{\ep}\sqrt{\frac{\sqrt{1+4\ep^2}+1}{2}}-\frac{1}{\ep}
\le \frac{\sqrt{1+4\ep^2}-1}{2\ep},
\end{equation}
where
\begin{equation}\label{eqn: useful omega-ep integral est}
\sqrt{1+4\ep^2}-1
= \sqrt{1+4\ep^2}-\sqrt{1+0}
= \int_0^{4\ep^2} \frac{ds}{2\sqrt{1+s}}
\le 2\ep^2.
\end{equation}
Combining this with \eqref{eqn: omega est aux} gives \eqref{eqn: omega-ep est}, and \eqref{eqn: omega-ep est2} is a consequence of that.
From \eqref{eqn: omega est aux} we also see that $\omega_{\ep}-1/\ep > 0$.
\end{enumerate}

\subsection{Estimates on $\mu_{\ep}$}\label{app: mu-ep}
We prove Lemma \ref{lem: mu-ep}.
That $\mu = \mu_{\ep}$ is the unique positive solution to $\ep^2\mu^4+\mu^2+1 = 0$ follows from the quadratic formula.
To show that $\mu_{\ep} < 1$, we use the mean value theorem to estimate
\[
\frac{\sqrt{1+4\ep^2}-1}{2\ep^2}
= \frac{\sqrt{1+2(2\ep^2)}-\sqrt{1+(2\cdot0)}}{2\ep^2-0}
= \frac{1}{\sqrt{1+2X_{\ep}}}
\le 1
\]
for some $X_{\ep} \in (0,2\ep^2)$.

To prove the estimate \eqref{eqn: mu-ep lim}, we first use the ``difference of squares'' estimate \eqref{eqn: difference of squares} to bound
\[
|\mu_{\ep}-1|
\le \left|\frac{\sqrt{1+4\ep^2}-1}{2\ep^2}-1\right|
= \frac{|\sqrt{1+4\ep^2}-1-2\ep^2}{2\ep^2}.
\]
Then for $X > 0$, we Taylor-expand
\begin{equation}\label{eqn: sqrt Taylor}
\sqrt{1+X}
= 1+\frac{X}{2}-\frac{X^2}{4}\int_0^1 (1-s)(1+sX)^{-3/2} \ds
\end{equation}
to find
\[
\frac{|\sqrt{1+4\ep^2}-1-2\ep^2|}{2\ep^2}
= 2\ep^2\left|\int_0^1 (1-s)(1+4s\ep^2)^{-3/2} \ds\right|
\le 2\ep^2\int_0^1 (1-s) \ds
= \ep^2.
\]

\section{Periodic Solutions}\label{app: periodics}

\subsection{The abstract existence theorem for periodics}\label{app: periodics abstract exist}
The following theorem is our primary tool for constructing periodic solutions to \eqref{eqn: the problem} that are defined on $\R$.
Its proof is an adaptation of Lombardi's techniques in \cite[App.\@ 4.A.3.1, 4.A.3.2]{lombardi}.
Unlike Lombardi, who focused on a particular differential equation in a periodic function space, we work in a more general Hilbert space framework.
However, for simplicity, we choose a much more transparent quadratic nonlinearity than Lombardi used; this is sufficient for the nonlinearity in \eqref{eqn: the problem} and allows us to focus on the role of the linearization without being consumed by notation.

We recall that if $\X$ and $\Y$ are normed spaces, then $\b(\X,\Y)$ is the space of bounded linear operators from $\X$ to $\Y$.

\begin{theorem}\label{thm: abstract per}
Let $\{\X^r\}_{r \ge 0}$ be a family of Hilbert spaces such that $\X^{r+s}$ is continuously embedded in $\X^r$ for each $s \ge 0$.
Denote the inner product on $\X^r$ by $\ip{\cdot}{\cdot}_r$ and the norm that it induces on $\X^r$ by $\norm{\cdot}_r$.
Suppose that there exist $\ep_0$, $r_0 > 0$ such that for $0 < \ep < \ep_0$, there are an analytic map $\T_{\ep} \colon \R \to \b(\X^{r_0},\X^0)$ and a bilinear map $\nl_{\ep} \colon \X^{r_0} \times \X^{r_0} \to \X^0$ such that the following hold.

\begin{enumerate}[label={\bf(\roman*)}]

\item
{}
[One-dimensional (co)kernel]
There exist $\omega_{\ep} \in \R$ and $\nu_{\ep} \in \X^{r_0}$, $\nu_{\ep}^* \in \X^0$ such that 
\begin{equation}\label{eqn: ker coker}
\ker(\T_{\ep}(\omega_{\ep})) = \spn(\nu_{\ep})
\quadword{and}
\ker(\T_{\ep}(\omega_{\ep})^*) = \spn(\nu_{\ep}^*),
\end{equation}
where $\T_{\ep}(\omega_{\ep})^* \colon \X^0 \to \X^{r_0}$ is the adjoint of $\T_{\ep}(\omega_{\ep})$, satisfying
\[
\ip{\T_{\ep}(\omega_{\ep})\phi}{\eta}_0
= \ip{\phi}{\T_{\ep}(\omega_{\ep})^*\eta}_{r_0},
\ \phi \in \X^{r_0}, \ \eta \in \X^0.
\]

\item
{}
[Uniform transversality]
The first derivative $\T_{\ep}'(\omega_{\ep})$ satisfies
\begin{equation}\label{eqn: uniform transversality}
\inf_{0 < \ep < \ep_0} |\ip{\T_{\ep}'(\omega_{\ep})\nu_{\ep}}{\nu_{\ep}^*}_0|
> 0.
\end{equation}

\item
{}
[Uniform coercivity]
There is $C > 0$ such that if $0 < \ep < \ep_0$ and $\psi \in \X^{r_0}$, $\eta \in \X^0$ with
\[
\T_{\ep}(\omega_{\ep})\psi = \eta
\quadword{and}
\ip{\psi}{\nu_{\ep}}_0 = 0,
\]
then
\begin{equation}\label{eqn: coercive}
\norm{\psi}_{r_0}
\le C\norm{\eta}_0.
\end{equation}

\item
{}
[Uniform operator norms]
For each $j \ge 2$, the higher derivatives $\T_{\ep}^{(j)}(\omega_{\ep})$ satisfy
\begin{equation}\label{eqn: per op norms1}
\sup_{0 < \ep < \ep_0} \norm{\T_{\ep}^{(j)}(\omega_{\ep})}_{\X^{r_0} \to \X^0}
< \infty,
\end{equation}
and the nonlinearity $\nl_{\ep}$ satisfies
\begin{equation}\label{eqn: per op norms2}
\sup_{0 < \ep < \ep_0} \norm{\nl_{\ep}}_{\X^{r_0} \times \X^{r_0} \to \X^0}
< \infty.
\end{equation}
There is $C > 0$ such that 
\begin{equation}\label{eqn: per op norms3}
\sup_{0 < \ep < \ep_0}
\left\|\frac{\ip{\eta}{\nu_{\ep}^*}}{\ip{\T_{\ep}'(\omega_{\ep})\nu_{\ep}}{\nu_{\ep}^*}}\T_{\ep}'(\omega_{\ep})\nu_{\ep}\right\|_0
\le C\norm{\eta}_0
\end{equation}
for all $\eta \in \X^0$.
\end{enumerate}

Then there exist $a_{\per,0} > 0$ and, for $0 < \ep < \ep_0$, sequences $(\xi_n^{\ep})$ in $\R$ and $(\psi_n^{\ep})$ in $\X^{r_0}$ with the following properties.

\begin{enumerate}[label={\bf(\roman*)}, ref={(\roman*)}]

\item
For $|a| \le a_{\per,0}$ and $0 < \ep < \ep_0$, the series $\medsum_{n=1}^{\infty} a^n\xi_n^{\ep}$ and $\medsum_{n=1}^{\infty} a^n\psi_n^{\ep}$ converge absolutely in $\R$ and $\X^{r_0}$, respectively, with both convergences uniform in $\ep$:
\begin{equation}\label{eqn: unif est on periodics}
\sup_{0 < \ep < \ep_0} \sum_{n=1}^{\infty} |a_{\per,0}|^n\big(|\xi_n^{\ep}| + \norm{\psi_n^{\ep}}_{r_0}\big)
< \infty.
\end{equation}

\item\label{part: power series soln}
Putting 
\begin{equation}\label{eqn: periodic ansatz}
\phi = a\nu_{\ep} + a\sum_{n=1}^{\infty} a^n\psi_n^{\ep}
\quadword{and}
\omega = \omega_{\ep} + \sum_{n=1}^{\infty} a^n\xi_n^{\ep}
\end{equation}
solves 
\begin{equation}\label{eqn: abstract per prob}
\T_{\ep}(\omega)\phi + \nl_{\ep}(\phi,\phi)
= 0.
\end{equation}

\item\label{part: per solns unique}
The sequences $(\xi_n^{\ep})$ and $(\psi_n^{\ep})$ are unique in that if $(\xi_n)$ is a sequence in $\R$ and $(\psi_n)$ is a sequence in $\X^{r_0}$ with $\ip{\psi_n}{\nu_{\ep}}_0 = 0$ for all $n$, and if putting $\phi = a\nu_{\ep} + a\medsum_{n=1}^{\infty} a^n\psi_n$ and $\omega = \omega_{\ep} + \medsum_{n=1}^{\infty} a^n\xi_n^{\ep}$ solves \eqref{eqn: abstract per prob} for $|a| \le a_{\per,0}$, then $\xi_n = \xi_n^{\ep}$ and $\psi_n = \psi_n^{\ep}$ for all $n$.
\end{enumerate}
\end{theorem}

\begin{proof}
We first make the ansatz \eqref{eqn: periodic ansatz} for the problem \eqref{eqn: abstract per prob} and determine an equation that the coefficients $\xi_n$ and $\psi_n$ must satisfy.
Then we solve this equation with a Lyapunov--Schmidt decomposition.
Finally, with the resulting formulas for $\xi_n$ and $\psi_n$, we prove that the series $\medsum_{n=1}^{\infty} a^n\xi_n^{\ep}$ and $\medsum_{n=1}^{\infty} a^n\psi_n^{\ep}$ converge.

\begin{enumerate}[label={\bf\arabic*.}]

\item
{\it{Formal development of the coefficients.}}
Under the perturbation ansatz $\omega = \omega_{\ep}+\xi$ and $\phi = a(\nu_{\ep}+\psi)$, where $a$, $\xi \in \R$ and $\psi \in \X^{r_0}$ with $\ip{\psi}{\nu_{\ep}}_0 = 0$, we want
\begin{equation}\label{eqn: per formal}
\T_{\ep}(\omega_{\ep}+\xi)\psi + a\nl_{\ep}(\nu_{\ep}+\psi,\nu_{\ep}+\psi)
= 0.
\end{equation}
Then we make the power series ansatz
\[
\xi = \sum_{n=1}^{\infty} a^n\xi_n
\quadword{and}
\psi = \sum_{n=1}^{\infty} a^n\psi_n.
\]
Here $\xi_n \in \R$, and we continue to assume $\psi_n \in \X^{r_0}$ with $\ip{\psi_n}{\nu_{\ep}}_0 = 0$.
The nonlinearity expands easily as 
\begin{equation}\label{eqn: per formal nl}
a\nl_{\ep}(\nu_{\ep}+\psi,\nu_{\ep}+\psi)
= a\nl_{\ep}(\nu_{\ep},\nu_{\ep})
+ \sum_{n=2}^{\infty} a^n\big(2\nl_{\ep}(\nu_{\ep},\psi_{n-1})\big)
+ \sum_{n=3}^{\infty} a^n\left(\sum_{j=1}^{n-2} \nl_{\ep}(\psi_j,\psi_{n-1-j})\right).
\end{equation}

The linear term $\T_{\ep}(\omega_{\ep}+\xi)\psi$ is more complicated, and we develop it carefully.
The operators $\T_{\ep}(\omega_{\ep})$ and $\T_{\ep}'(\omega_{\ep})$ will control the subsequent Lyapunov--Schmidt analysis, so we expose them first:
\begin{equation}\label{eqn: per formal T nu}
\T_{\ep}(\omega_{\ep}+\xi)\nu_{\ep}
= \sum_{n=0}^{\infty} \frac{\xi^n}{n!}\T_{\ep}^{(n)}(\omega_{\ep})\nu_{\ep}
= \xi\T_{\ep}'(\omega_{\ep})\nu_{\ep} + \sum_{n=2}^{\infty} \frac{\xi^n}{n!}\T_{\ep}^{(n)}(\omega_{\ep})\nu_{\ep},
\end{equation}
since $\T_{\ep}(\omega_{\ep})\nu_{\ep} = 0$, and
\begin{equation}\label{eqn: per formal T psi}
\T_{\ep}(\omega_{\ep}+\xi)\psi
= \sum_{n=0}^{\infty} \frac{\xi^n}{n!}\T_{\ep}^{(n)}(\omega_{\ep})\psi
= \T_{\ep}(\omega_{\ep})\psi + \xi\T_{\ep}'(\omega_{\ep})\psi + \sum_{n=2}^{\infty} \frac{\xi^n}{n!}\T_{\ep}^{(n)}(\omega_{\ep})\psi.
\end{equation}

We begin with \eqref{eqn: per formal T nu} and expand first
\begin{equation}\label{eqn: per formal T nu 1}
\xi\T_{\ep}'(\omega_{\ep})\nu_{\ep}
= \T_{\ep}'(\omega_{\ep})(\xi\nu_{\ep})
= \T_{\ep}'(\omega_{\ep})\sum_{n=1}^{\infty} a^n\xi_n\nu_{\ep}
= \sum_{n=1}^{\infty} a^n(\xi_n\T_{\ep}'(\omega_{\ep})\nu_{\ep}).
\end{equation}
For the higher-order terms in \eqref{eqn: per formal T nu}, we put $\xi_0 = 0$ and expand
\begin{equation}\label{eqn: Xi}
\left(\sum_{j=0}^{\infty} \xi_ja^j\right)^n
= \sum_{j=0}^{\infty} \Xi_{n,j}a^j,
\qquad
\Xi_{n,j} 
:= \sum_{\substack{0 \le k_1,\ldots,k_n \le j \\ k_1+\cdots+k_n=j}} \xi_{k_1}\cdots\xi_{k_n}
= \sum_{\substack{1 \le k_1,\ldots,k_n \le j \\ k_1+\cdots+k_n=j}} \xi_{k_1}\cdots\xi_{k_n}.
\end{equation}
Because the indices in the product in the sum defining $\Xi_{n,j}$ are positive and sum to $j$, each index is strictly less than $j$, and so $\Xi_{n,j}$ depends only on $\xi_1,\ldots,\xi_{j-1}$.

If $j < n$, then no $n$ positive integers $k_1,\ldots,k_n$ can sum to $j$, so $\Xi_{n,j} = 0$ for $j < n$.
Thus
\begin{equation}\label{eqn: xi exp}
\left(\sum_{j=0}^{\infty} \xi_ja^j\right)^n
= \sum_{j=n}^{\infty} \Xi_{n,j}a^j,
\end{equation}
and so
\begin{equation}\label{eqn: per formal T nu 2}
\sum_{n=2}^{\infty} \frac{\xi^n}{n!}\T_{\ep}^{(n)}(\omega_{\ep})\nu_{\ep}
= \sum_{n=2}^{\infty}\frac{1}{n!}\T_{\ep}^{(n)}(\omega_{\ep})\left(\sum_{j=n}^{\infty} \Xi_{n,j}a^j\right)\nu_{\ep}
= \sum_{n=2}^{\infty} a^n\left(\sum_{j=2}^n \frac{\Xi_{j,n}}{j!}\T_{\ep}^{(j)}(\omega_{\ep})\nu_{\ep}\right),
\end{equation}
where the last equality follows from the identity
\begin{equation}\label{eqn: series cov id}
\sum_{n=2}^{\infty} \sum_{j=n}^{\infty} c_{n,j}a^j
= \sum_{n=2}^{\infty} \left(\sum_{j=2}^n c_{j,n}\right)a^n.
\end{equation}
Thus
\begin{equation}\label{eqn: per formal T nu final}
\T_{\ep}(\omega_{\ep}+\xi)\nu_{\ep}
= \sum_{n=1}^{\infty} a^n(\xi_n\T_{\ep}'(\omega_{\ep})\nu_{\ep}) + \sum_{n=2}^{\infty} a^n\left(\sum_{j=2}^n \frac{\Xi_{j,n}}{j!}\T_{\ep}^{(j)}(\omega_{\ep})\nu_{\ep}\right).
\end{equation}

More complicated is \eqref{eqn: per formal T psi}, the first two terms of which we expand as
\begin{equation}\label{eqn: per formal T0}
\T_{\ep}(\omega_{\ep})\psi
= \T_{\ep}(\omega_{\ep})\sum_{n=1}^{\infty} a^n\psi_n
= \sum_{n=1}^{\infty} a^n\T_{\ep}(\omega_{\ep})\psi_n
\end{equation}
and
\begin{multline}\label{eqn: per formal T1}
\xi\T_{\ep}'(\omega_{\ep})\psi
= \T_{\ep}'(\omega_{\ep})(\xi\psi)
= \T_{\ep}'(\omega_{\ep})\left(\sum_{n=1}^{\infty} a^n\xi_n\right)\left(\sum_{n=1}^{\infty} a^n\psi_n\right)
= \sum_{n=1}^{\infty} a^{n+1}\left(\sum_{j=1}^n \xi_j\T_{\ep}'(\omega_{\ep})\psi_{n+1-j}\right) \\
= \sum_{n=2}^{\infty} a^n\left(\sum_{j=1}^{n-1} \xi_j\T_{\ep}'(\omega_{\ep})\psi_{n-j}\right)
\end{multline}
by the Cauchy product formula
\begin{equation}\label{eqn: CP}
\left(\sum_{n=m}^{\infty} b_n\right)\left(\sum_{n=1}^{\infty} c_n\right)
= \sum_{n=m}^{\infty} \sum_{j=m}^n b_jc_{n+1-j}.
\end{equation}
We expand the higher-order terms in \eqref{eqn: per formal T psi} as
\begin{equation}\label{eqn: per formal T2 0}
\sum_{n=2}^{\infty} \frac{\xi^n}{n!}\T_{\ep}^{(n)}(\omega_{\ep})\psi
= \sum_{n=2}^{\infty}\frac{1}{n!}\T_{\ep}^{(n)}(\omega_{\ep})\left(\sum_{j=n}^{\infty} \Xi_{n,j}a^j\right)\psi
= \sum_{n=2}^{\infty} a^n\left(\sum_{j=2}^n \frac{\Xi_{j,n}}{j!}\T_{\ep}^{(j)}(\omega_{\ep})\right)\psi,
\end{equation}
where the last equality follows from \eqref{eqn: series cov id}.
We apply the Cauchy product formula \eqref{eqn: CP} again to conclude
\begin{multline}\label{eqn: per formal T2 1}
\sum_{n=2}^{\infty} a^n\left(\sum_{j=2}^n \frac{\Xi_{j,n}}{j!}\T_{\ep}^{(j)}(\omega_{\ep})\right)\psi
= \left(\sum_{n=2}^{\infty} \left(a^n\sum_{j=2}^n \frac{\Xi_{j,n}}{j!}\T_{\ep}^{(j)}(\omega_{\ep})\right)\right)\left(\sum_{n=1}^{\infty} a^n\psi_n\right) \\
= \sum_{n=2}^{\infty} a^{n+1}\sum_{j=2}^n \sum_{k=2}^j \frac{\Xi_{k,j}}{k!}\T_{\ep}^{(k)}(\omega_{\ep})\psi_{n+1-j}
= \sum_{n=3}^{\infty} a^n\sum_{j=2}^{n-1} \sum_{k=2}^j \frac{\Xi_{k,j}}{k!}\T_{\ep}^{(k)}(\omega_{\ep})\psi_{n-j}.
\end{multline}

We combine \eqref{eqn: per formal T0}, \eqref{eqn: per formal T1}, \eqref{eqn: per formal T2 0}, and \eqref{eqn: per formal T2 1} into \eqref{eqn: per formal T psi} and combine these results with the easier linear expansion \eqref{eqn: per formal T nu final} and the quadratic expansion \eqref{eqn: per formal nl} to expand the problem \eqref{eqn: per formal} as
\begin{multline}\label{eqn: what the periodic coeff do}
a\big(\xi_1\T_{\ep}'(\omega_{\ep})\nu_{\ep} + \T_{\ep}(\omega_{\ep})\psi_1 + \nl_{\ep}(\nu_{\ep},\nu_{\ep})\big) \\
+ \sum_{n=2}^{\infty} a^n\left(\xi_n\T_{\ep}'(\omega_{\ep})\nu_{\ep}
+\T_{\ep}(\omega_{\ep})\psi_n
+ \sum_{j=2}^n \frac{\Xi_{j,n}}{j!}\T_{\ep}^{(j)}(\omega_{\ep})\nu_{\ep}
+ \sum_{j=1}^{n-1} \xi_j\T_{\ep}(\omega_{\ep})\psi_{n-j}
+ 2\nl_{\ep}(\nu_{\ep},\psi_{n-1})
\right) \\
+ \sum_{n=3}^{\infty} a^n\left(
\sum_{j=2}^{n-1} \sum_{k=2}^j \frac{\Xi_{k,j}}{k!}\T_{\ep}^{(k)}(\omega_{\ep})\psi_{n-j}+\sum_{j=1}^{n-2} \nl_{\ep}(\psi_j,\psi_{n-1-j})
\right)
= 0.
\end{multline}
Let
\begin{equation}\label{eqn: Gamma1}
\Gamma_{1,\ep}
:= \nl_{\ep}(\nu_{\ep},\nu_{\ep}).
\end{equation}
Recall that $\Xi_{k,j}$ depends only on $\xi_1,\ldots,\xi_{j-1}$ and $\Xi_{j,n}$ depends only on $\xi_1,\ldots,\xi_{n-1}$.
Consequently, we may define
\begin{multline}\label{eqn: Gamma-n}
\Gamma_{n,\ep}(\xi_1,\ldots,\xi_{n-1},\psi_1,\ldots,\psi_{n-1})
:= \sum_{j=2}^n \frac{\Xi_{j,n}}{j!}\T_{\ep}^{(j)}(\omega_{\ep})\nu_{\ep}
+ \sum_{j=1}^{n-1} \xi_j\T_{\ep}(\omega_{\ep})\psi_{n-j} 
+ 2\nl_{\ep}(\nu_{\ep},\psi_{n-1}) \\
+ \sum_{j=2}^{n-1} \sum_{k=2}^j \frac{\Xi_{k,j}}{k!}\T_{\ep}^{(k)}(\omega_{\ep})\psi_{n-j}+\sum_{j=1}^{n-2} \nl_{\ep}(\psi_j,\psi_{n-1-j})
\end{multline}
to conclude that \eqref{eqn: what the periodic coeff do} is equivalent to
\begin{equation}\label{eqn: per pre-LS}
\xi_n\T_{\ep}'(\omega_{\ep})\nu_{\ep}
+\T_{\ep}(\omega_{\ep})\psi_n
+\Gamma_{n,\ep}
= 0
\end{equation}
for $n \ge 1$.
(In general, we will suppress the dependence of $\Gamma_{n,\ep}$ on its arguments.)
We will solve \eqref{eqn: per pre-LS} with a Lyapunov--Schmidt decomposition, and this will prove the uniqueness result in part \ref{part: per solns unique}.

\item
{\it{The Lyapunov--Schmidt decomposition.}}
Put 
\begin{equation}\label{eqn: Pi-ep}
\varpi_{\ep}\eta 
:= \ip{\eta}{\nu_{\ep}^*}_0\nu_{\ep}^*.
\end{equation}
Then \eqref{eqn: per pre-LS} is equivalent to
\begin{equation}\label{eqn: per LS}
\begin{cases}
(1-\varpi_{\ep})\T_{\ep}(\omega_{\ep})\psi_n  + \xi_n(1-\varpi_{\ep})\T_{\ep}'(\omega_{\ep})\nu_{\ep} + (1-\varpi_{\ep})\Gamma_{n,\ep} = 0 \\
\ip{\T_{\ep}(\omega_{\ep})\psi_n}{\nu_{\ep}^*}_0  + \xi_n\ip{\T_{\ep}'(\omega_{\ep})\nu_{\ep}}{\nu_{\ep}^*} _0+ \ip{\Gamma_{n,\ep}}{\nu_{\ep}^*}_0 = 0. 
\end{cases}
\end{equation}

The assumption \eqref{eqn: ker coker} gives
\begin{equation}\label{eqn: ker coker conseq}
\ip{\T_{\ep}(\omega_{\ep})\psi}{\nu_{\ep}^*}_0
= \ip{\psi}{\T_{\ep}(\omega_{\ep})^*\nu_{\ep}^*}_{r_0}
= 0
\end{equation}
for any $\psi \in \X^r$, which simplifies the decomposition \eqref{eqn: per LS} to
\begin{subnumcases}{}
\T_{\ep}(\omega_{\ep})\psi_n  + \xi_n(1-\varpi_{\ep})\T_{\ep}'(\omega_{\ep})\nu_{\ep} + (1-\varpi_{\ep})\Gamma_{n,\ep} = 0 \label{eqn: per LS ID} \\
\xi_n\ip{\T_{\ep}'(\omega_{\ep})\nu_{\ep}}{\nu_{\ep}^*} _0+ \ip{\Gamma_{n,\ep}}{\nu_{\ep}^*}_0 = 0. \label{eqn: per LS FD}
\end{subnumcases}
The uniform transversality estimate \eqref{eqn: uniform transversality} gives $\ip{\T_{\ep}'(\omega_{\ep})\nu_{\ep}}{\nu_{\ep}^*}_0 \ne 0$, and so we can solve the finite-dimensional equation \eqref{eqn: per LS FD} for $\xi_n$ as 
\begin{equation}\label{eqn: xij}
\xi_n
= -\frac{\ip{\Gamma_{n,\ep}}{\nu_{\ep}^*}_0}{\ip{\T_{\ep}'(\omega_{\ep})\nu_{\ep}}{\nu_{\ep}^*}_0}.
\end{equation}
We emphasize that, based on the definitions of $\Gamma_n$ in \eqref{eqn: Gamma1} and \eqref{eqn: Gamma-n}, $\xi_n$ depends only on the previously determined $\xi_1,\ldots,\xi_{n-1}$ and $\psi_1,\ldots,\psi_{n-1}$.

Now we solve the infinite-dimensional equation \eqref{eqn: per LS ID}.
Put
\[
\X_{\ep}^r
:= \set{\psi \in \X^r}{\ip{\psi}{\nu_{\ep}}_0 = 0}
\quadword{and}
\Y_{\ep}^r 
:= \set{\eta \in \X^r}{\ip{\eta}{\nu_{\ep}^*}_0 = 0}.
\]
By \eqref{eqn: ker coker conseq}, $\T_{\ep}(\omega_{\ep})\psi \in \Y_{\ep}^0$ for all $\psi \in \X_{\ep}^r$.
Then the restriction $\restr{\T_{\ep}(\omega_{\ep})}{\X_{\ep}^{r_0}} \colon \X_{\ep}^{r_0} \to \Y_{\ep}^0$ has trivial kernel and cokernel by \eqref{eqn: ker coker} and closed range by the uniform coercivity hypothesis, and so this restriction is invertible.
For $\eta \in \Y_{\ep}^0$, let $\T_{\ep}(\omega_{\ep})^{-1}\eta$ be the vector $\psi \in \X_{\ep}^{r_0}$ such that $\T_{\ep}(\omega_{\ep})\psi=\eta$.
In particular, the coercive estimate \eqref{eqn: coercive} reads
\begin{equation}\label{eqn: coercive app}
\norm{\T_{\ep}(\omega_{\ep})^{-1}\eta}_{r_0}
\le C\norm{\eta}_0.
\end{equation}
We conclude that
\begin{equation}\label{eqn: psij pre}
\psi_n
= \T_{\ep}(\omega_{\ep})^{-1}(1-\varpi_{\ep})(-\xi_n\T_{\ep}'(\omega_{\ep})\nu_{\ep}-\Gamma_{n,\ep})
\end{equation}
solves the infinite-dimensional equation \eqref{eqn: per LS ID}.

We can simplify this expression for $\psi_n$ slightly using the expression \eqref{eqn: xij} for $\xi_n$.
Put
\begin{equation}\label{eqn: P-ep}
\P_{\ep}\eta
:= \frac{\ip{\eta}{\nu_{\ep}^*}}{\ip{\T_{\ep}'(\omega_{\ep})\nu_{\ep}}{\nu_{\ep}^*}}\T_{\ep}'(\omega_{\ep})\nu_{\ep} -\eta,
\end{equation}
which gives $\varpi_{\ep}\P_{\ep}\eta = 0$ for all $\eta$.
Then
\[
-\xi_n\T_{\ep}'(\omega_{\ep})\nu_{\ep}-\Gamma_{n,\ep} \\
= \frac{\ip{\Gamma_{n,\ep}}{\nu_{\ep}^*}}{\ip{\T_{\ep}'(\omega_{\ep})\nu_{\ep}}{\nu_{\ep}^*}}\T_{\ep}'(\omega_{\ep})\nu_{\ep}-\Gamma_{n,\ep} \\
= \P_{\ep}\Gamma_{n,\ep},
\]
and so \eqref{eqn: psij pre} becomes
\begin{equation}\label{eqn: psij}
\psi_n
= \T_{\ep}(\omega_{\ep})^{-1}(1-\varpi_{\ep})\P_{\ep}\Gamma_{n,\ep}
= \T_{\ep}(\omega_{\ep})^{-1}\P_{\ep}\Gamma_{n,\ep}.
\end{equation}
Again we point out that, based on the definitions of $\Gamma_n$ in \eqref{eqn: Gamma1} and \eqref{eqn: Gamma-n}, $\psi_n$ depends only on the previously determined $\xi_1,\ldots,\xi_{n-1}$ and $\psi_1,\ldots,\psi_{n-1}$.

Now that we have determined what the coefficients in the series solutions to \eqref{eqn: abstract per prob} should be, we rigorously prove that these formulas do yield convergent series solutions.

\item
{\it{Formal development of a dominating power series.}}
To summarize, from the formulas \eqref{eqn: xij} for $\xi_n$ and \eqref{eqn: psij} for $\psi_n$, put
\begin{equation}\label{eqn: xi-1-ep psi-1-ep}
\xi_1^{\ep}
:= -\frac{\ip{\nl_{\ep}(\nu_{\ep},\nu_{\ep})}{\nu_{\ep}^*}_0}{\ip{\T_{\ep}'(\omega_{\ep})\nu_{\ep}}{\nu_{\ep}^*}_0}
\quadword{and}
\psi_1^{\ep} 
:= \T_{\ep}(\omega_{\ep})^{-1}\P_{\ep}\nl_{\ep}(\nu_{\ep},\nu_{\ep})
\end{equation}
and, for $j \ge 2$,
\begin{equation}\label{eqn: xi-n-ep}
\xi_n^{\ep}
= -\frac{\ip{\Gamma_{n,\ep}(\xi_1^{\ep},\ldots,\xi_{n-1}^{\ep},\psi_1^{\ep},\ldots,\psi_{n-1}^{\ep})}{\nu_{\ep}^*}_0}{\ip{\T_{\ep}'(\omega_{\ep})\nu_{\ep}}{\nu_{\ep}^*}_0}
\end{equation}
and
\begin{equation}\label{eqn: psi-n-ep}
\psi_n^{\ep}
= \T_{\ep}(\omega_{\ep})^{-1}\P_{\ep}\Gamma_{n,\ep}(\xi_1^{\ep},\ldots,\xi_{n-1}^{\ep},\psi_1^{\ep},\ldots,\psi_{n-1}^{\ep}).
\end{equation}
The formula \eqref{eqn: xi-1-ep psi-1-ep} for $\psi_1^{\ep}$ agrees exactly with Lombardi's formula (4.16) in \cite[Lem.\@ 4.A.10]{lombardi}, which is the analogous part of Lombardi's development.

Now let $(\gamma_n)$ be any sequence in $\R$.
Motivated by the definition of $\Xi_{n,j}$ in \eqref{eqn: Xi}, we slightly abuse notation and set
\[
\Xi_{n,j}(\gamma_1,\ldots,\gamma_{j-1})
:= \sum_{\substack{1 \le k_1,\ldots,k_n \le j \\ k_1+\cdots+k_n=j}} \gamma_{k_1}\cdots\gamma_{k_n}.
\]
For $n \ge 2$, define
\begin{multline}\label{eqn: Mj}
M_n(\gamma_1,\ldots,\gamma_{n-1})
:= \gamma_{n-1} 
+ \sum_{j=1}^{n-1} \gamma_j\gamma_{n-j} 
+ \sum_{j=1}^{n-2} \gamma_j\gamma_{n-1-j}
+ \sum_{j=2}^n \Xi_{j,n}(\gamma_1,\ldots,\gamma_{n-1}) \\
+ \sum_{j=2}^{n-1} \sum_{k=2}^j \Xi_{k,j}(\gamma_1,\ldots,\gamma_{j-1})\gamma_{n-j}
\end{multline}

Then the formula \eqref{eqn: Gamma-n} for $\Gamma_{n,\ep}$ and the uniform operator norms \eqref{eqn: per op norms1}, \eqref{eqn: per op norms2}, and \eqref{eqn: per op norms3} provide $C > 0$ such that for any sequences $(\xi_n)$ in $\R$ and $(\psi_n)$ in $\X^{r_0}$, for $j \ge 2$ we have
\[
\norm{\Gamma_n(\xi_1,\ldots,\xi_{n-1},\psi_1,\ldots,\psi_{n-1})}_0
\le CM_n(|\xi_1|+\norm{\psi_1}_{r_0},\ldots,|\xi_{n-1}|+\norm{\psi_{n-1}}_{r_0}).
\]
Consequently, the formulas \eqref{eqn: xi-n-ep} and \eqref{eqn: psi-n-ep} for the sequences $(\xi_n^{\ep})$ and $(\psi_n^{\ep})$ of coefficients constructed above, the uniform transversality estimate \eqref{eqn: uniform transversality}, the uniform coercivity estimate \eqref{eqn: coercive app}, and the continuous embedding of $\X^{r_0}$ into $\X^0$ imply that 
\begin{equation}\label{eqn: sup per j=1}
\gamma_1
:= \sup_{0 < \ep < \ep_0} |\xi_1^{\ep}| + \norm{\psi_1^{\ep}}_0
< \infty
\end{equation}
and provide $C_{\star} > 0$ such that
\begin{equation}\label{eqn: xi psib intermediate est}
|\xi_n^{\ep}| + \norm{\psi_n^{\ep}}_{r_0}
\le C_{\star}M_n(|\xi_1^{\ep}|+\norm{\psi_1^{\ep}}_{r_0},\ldots,|\xi_{n-1}^{\ep}|+\norm{\psi_{n-1}^{\ep}}_{r_0}), \ j \ge 2.
\end{equation}

Now define a particular sequence $(\gamma_n)$ of real numbers by taking $\gamma_1$ from \eqref{eqn: sup per j=1} and, for $j \ge 2$, setting
\begin{equation}\label{eqn: gamma-j}
\gamma_n
:= C_{\star}M_n(\gamma_1,\ldots,\gamma_{n-1}).
\end{equation}
Since $M_n$ is increasing in each of its arguments, it follows from \eqref{eqn: xi psib intermediate est} inductively that
\begin{equation}\label{eqn: gamma-j est}
\sup_{0 < \ep < \ep_0} |\xi_n^{\ep}| + \norm{\psi_n^{\ep}}_{r_0}
\le \gamma_n,
\end{equation}
and so to prove the convergence of the power series $\medsum_{n=1}^{\infty} a^n\xi_n^{\ep}$ in $\R$ and $\medsum_{n=1}^{\infty} a^n\psi_n^{\ep}$ in $\X^{r_0}$, it suffices to prove that the formal series $\medsum_{n=1}^{\infty} a^n\gamma_n$ converges on some interval $(-a_0,a_0)$.

To gain some intuition, first suppose that this series $\medsum_{n=1}^{\infty} a^n\gamma_n$ does converge on an interval $(-a_0,a_0)$, and call the sum $\Scal(a)$.
Then the definition of $\gamma_n$ in \eqref{eqn: gamma-j}, the definition of $M_n$ in \eqref{eqn: Mj}, and the Cauchy product formula imply that $\Scal$ must satisfy
\begin{equation}\label{eqn: per what S formally does}
\Scal(a)
= \sum_{n=1}^{\infty} a^n\gamma_n
= a\gamma_1+\sum_{n=2}^{\infty} M_n(\gamma_1,\ldots,\gamma_{n-1})a^n
= a\gamma_1+C_{\star}\left(a(\Scal(a)^2+\Scal(a))+\sum_{j=2}^{\infty} \Scal(a)^j\right)
\end{equation}
Since $\Scal(a)$ is analytic with $\Scal(0) = 0$, it is reasonable to demand further that the geometric series $\medsum_{j=0}^{\infty} \Scal(a)^j$ converge, and so from \eqref{eqn: per what S formally does}, we want $\Scal$ to satisfy
\begin{equation}\label{eqn: per what S really formally does}
\Scal(a)
= a\gamma_1+C_{\star}\left(a(\Scal(a)^2+\Scal(a))+\frac{1}{1-\Scal(a)} - (1+\Scal(a))\right).
\end{equation}
We use this equation to prove that the formal series $\medsum_{n=1}^{\infty} a^n\gamma_n$ actually does converge on a nontrivial interval.
This equation \eqref{eqn: per what S really formally does} is analogous to Lombardi's (4.18) in \cite[App.\@ 4.A.3.2]{lombardi}.

\item
{\it{Convergence proofs.}}
Motivated by \eqref{eqn: per what S really formally does}, we define
\[
F(S,a)
:= a\gamma_1+C_{\star}\left(a(S^2+S)+\frac{1}{1-S} - (1+S)\right)-S.
\]
Then $F$ is analytic on $\R^2$ with $F(0,0) = 0$ and $F_S(0,0) = -1 \ne 0$.
The analytic implicit function theorem then provides $a_0 > 0$ and an analytic function $\Scal \colon (-a_0,a_0) \to \R$ such that $F(\Scal(a),a) = 0$ for all $|a| < a_0$.

Since $F(0,0) = 0$, we have $\Scal(0) = 0$, and we implicitly differentiate $F(\Scal(a),a) = 0$ with respect to $a$ to obtain $\Scal'(0) = \gamma_1$.
Now write $\Scal(a) = a\gamma_1+\medsum_{n=2}^{\infty} \beta_na^n$ for some coefficients $\beta_n \in \R$.
The identity $F(\Scal(a),a) = 0$ rearranges, with more Cauchy product formula calculations, to
\[
a\gamma_1+\sum_{n=2}^{\infty} \beta_na^n
= \Scal(a)
= a\gamma_1+C_{\star}\left(a(\Scal(a)^2+\Scal(a))+\sum_{j=2}^{\infty} \Scal(a)^j\right)
= a\gamma_1+\sum_{n=2}^{\infty} C_{\star}M_n(\beta_1,\ldots,\beta_{n-1}).
\]
Thus $\beta_n = C_{\star}M_n(\beta_1,\ldots,\beta_{n-1})$ for $j \ge 2$, and since $\beta_1 = \gamma_1$, it follows inductively that $\beta_n = \gamma_n$ for all $j$.
The series $\medsum_{n=1}^{\infty} \gamma_na^n$ therefore converges (absolutely) on $(-a_0,a_0)$, which proves the convergence of $\medsum_{n=1}^{\infty} a^n\xi_n^{\ep}$ and $\medsum_{n=1}^{\infty} a^n\psi_n^{\ep}$ in $\X^{r_0}$ on this interval.

For the uniform estimate \eqref{eqn: unif est on periodics}, we set $a_{\per,0} := a_0/2$.
Then for $|a| \le a_{\per,0}$, we have
\[
\sum_{n=1}^{\infty} |a_{\per,0}|^n\big(|\xi_n^{\ep}|+\norm{\psi_n^{\ep}}_{r_0}\big)
\le \sum_{n=1}^{\infty} |a_{\per,0}|^n\gamma_n
< \infty
\]
by \eqref{eqn: gamma-j est} and the (absolute) convergence of $S(a_0/2)$.
\qedhere
\end{enumerate}
\end{proof}

\subsection{Periodic Sobolev spaces}
We will apply Theorem \ref{thm: abstract per} on a family of Hilbert spaces that are periodic Sobolev spaces, and we need the following standard notation and results \cite[Sec.\@ 8.1]{kress} and \cite[Sec.\@ 7.2]{hunter-nachtergaele}.

\begin{enumerate}[label=$\bullet$]

\item
Let 
\[
\Cal_{\per}^{\infty}
:= \set{f \in \Cal^{\infty}(\R)}{f(x) = f(x+2\pi), \ x \in \R}
\]
and, for $f \in \Cal_{\per}^{\infty}$, define
\[
\ip{f}{g}_{L_{\per}^2}
:= \int_{-\pi}^{\pi} f(x)\overline{g(x)} \dx.
\]
Denote by $L_{\per}^2$ the completion of $\Cal_{\per}^{\infty}$ in the topology induced by the norm 
\[
\norm{f}_{L_{\per}^2} 
:= \sqrt{\ip{f}{f}_{L_{\per}^2}}.
\]

\item
For $f \in L_{\per}^2$ and $k \in \Z$, define
\[
\hat{f}(k)
:= \frac{1}{2\pi}\int_{-\pi}^{\pi} f(x)e^{-ikx} \dx.
\]
Then $f = 0$ if and only if $\hat{f}(k) = 0$ for all $k \in \Z$.

\item
For $r \ge 0$, let
\[
H_{\per}^r
:= \set{f \in L_{\per}^2}{\norm{f}_{H_{\per}^r} < \infty},
\]
where $\ip{\cdot}{\cdot}_{H_{\per}^r}$ is the norm induced by the inner product 
\[
\ip{f}{g}_{H_{\per}^r}
:= \sum_{k=-\infty}^{\infty} (1+k^2)^r\hat{f}(k)\overline{\hat{g}(k)}.
\]
Each $H_{\per}^r$ is a Hilbert space and $H_{\per}^0 = L_{\per}^2$ with $\ip{f}{g}_{H_{\per}^0} = \ip{f}{g}_{L_{\per}^2}$.

\item
An operator $\T \colon H_{\per}^r \to H_{\per}^s$ is a Fourier multiplier with symbol $\tT \in L^{\infty}(\R)$ if
\[
\hat{\T{f}}(k)
= \tT(k)\hat{f}(k)
\]
for all $f \in H_{\per}^r$ and $k \in \Z$.
Any Fourier multiplier $\T \colon H_{\per}^r \to H_{\per}^s$ with symbol $\tT$ has an adjoint $\T^* \colon H_{\per}^s \to H_{\per}^r$, and this adjoint is given by
\begin{equation}\label{eqn: FM adj}
\hat{\T^*g}(k)
= \frac{\overline{\tT(k)}}{(1+k^2)^{r-s}}\hat{g}(k), \ k \in \Z.
\end{equation}
Here $\T^*$ satisfies
\[
\ip{\T{f}}{g}_{H_{\per}^s}
= \ip{f}{\T^*g}_{H_{\per}^r}, \ f \in H_{\per}^r, \ g \in H_{\per}^s.
\]
\end{enumerate}

\subsection{The proof of Theorem \ref{thm: periodics real}}\label{app: per real proof}
We first reformulate the problem \eqref{eqn: the problem} on periodic Sobolev spaces to fit the framework of Theorem \ref{thm: abstract per}.
Then we check the hypotheses of that theorem.
Finally, we show that the coefficients $\psi_{n,\ep}$ in this context are trigonometric polynomials of degree $n+1$.

\begin{enumerate}[label={\bf\arabic*.}]
\setcounter{enumi}{-1}

\item
{\it{Reformulation.}}
We make the ansatz $u(x) = \phi(\omega{x})$, where $\phi \in E_{\per}^4$ and $\omega \in \R$, to convert the problem \eqref{eqn: the problem} into
\begin{equation}\label{eqn: per problem}
\ep^2\omega^4\phi^{(4)}+\omega^2\phi''-\phi + \phi^2
= 0.
\end{equation}
Put
\begin{equation}\label{eqn: per problem actual T}
\T_{\ep}(\omega)
:= \ep^2\omega^4\partial_X^4 + \omega^2\partial_X^2-1
\quadword{and}
\nl_{\ep}(\phi_1,\phi_2)
:= \phi_1\phi_2,
\end{equation}
so that \eqref{eqn: per problem} is
\[
\T_{\ep}(\omega)\phi + \nl_{\ep}(\phi,\phi)
= 0.
\]
We will use the critical frequency $\omega_{\ep}$ from \eqref{eqn: omega-ep}, and we set
\[
\nu_{\ep}(X)
= \nu_{\ep}^*(X)
:= \cos(X).
\]
Here $\nu_{\ep}$ and $\nu_{\ep}^*$ are both the same and independent of $\ep$, and $\nl_{\ep}$ is also independent of $\ep$.
We work on the spaces 
\[
\X^r 
:= E_{\per}^r
= \set{f \in H_{\per}^r}{f(x) = f(-x), \ x \in \R}
\]
and set $r_0 = 4$.

\item
{\it{One-dimensional (co)kernel.}}
First suppose that $\T_{\ep}(\omega_{\ep})\phi = 0$ for some $\phi \in E_{\per}^4$.
The $k$th Fourier coefficient of $\T_{\ep}(\omega_{\ep})\phi$ then satisfies
\[
\tT_{\ep}(k)\hat{\phi}(k)
= 0,
\qquad
\tT_{\ep}(k) := \ep^2(\omega_{\ep}k)^4-(\omega_{\ep}k)^2-1,
\]
for all $k \in \Z$.
Note that $\tT_{\ep}(k) = \tB_{\ep}(i\omega_{\ep}k)$ with $\tB_{\ep}$ defined in \eqref{eqn: tB}.

For $k \ne \pm1$, the equation $\tT_{\ep}(k) = 0$ has no solutions by Lemma \ref{lem: omega-ep}, and so $\hat{\phi}(k) = 0$ for $k \ne \pm 1$.
Since $\phi$ is even, we have $\hat{\phi}(-1) = \hat{\phi}(1)$, and so
\[
\phi(X)
= \hat{\phi}(-1)e^{-iX}+\hat{\phi}(1)e^{iX}
= \hat{\phi}(1)\big(e^{-iX}+e^{iX}\big)
= 2\hat{\phi}(1)\cos(X)
= 2\hat{\phi}(1)\nu_{\ep}(X).
\]

Likewise, if $\T_{\ep}(\omega_{\ep})^*\eta = 0$ for some $\eta \in E_{\per}^0$, then by \eqref{eqn: FM adj} we have
\[
\frac{\overline{\tT_{\ep}(k)}}{(1+k^2)^4}\hat{\eta}(k)
= 0, \ k \in \Z.
\]
Here $\tT_{\ep}$ is real-valued, so the same argument as above shows $\hat{\eta}(k) = 0$ for $k \ne \pm1$, and since $\eta$ is even, we conclude that $\eta$ is a multiple of $\nu_{\ep}$.

\item
{\it{Uniform transversality.}}
From \eqref{eqn: per problem actual T}, we compute
\begin{multline}\label{eqn: per prob T derivs}
\T_{\ep}'(\omega) = 4\ep^2\omega^3\partial_X^4 + 2\omega\partial_X^2,
\qquad
\T_{\ep}''(\omega) = 12\ep^2\omega^2\partial_X^4+2\partial_X^2,
\qquad
\T_{\ep}'''(\omega) = 24\ep^2\omega\partial_X^4, \\
\qquad
\T_{\ep}^{(4)}(\omega) = 48\ep^2\partial_X^4,
\quadword{and}
\T_{\ep}^{(j)}(\omega) = 0,\ j \ge 5.
\end{multline}
We calculate
\[
\T_{\ep}'(\omega_{\ep})\nu_{\ep}
= 2\omega_{\ep}^2(2\ep(\ep\omega_{\ep})-1)\nu_{\ep}
\]
and then
\[
\ip{\T_{\ep}'(\omega_{\ep})\nu_{\ep}}{\nu_{\ep}^*}_{L_{\per}^2}
= 2\omega_{\ep}^2(2\ep(\ep\omega_{\ep})-1)\norm{\nu_{\ep}}_{L_{\per}^2}^2
= 2\pi\omega_{\ep}^2(2\ep(\ep\omega_{\ep})-1)
\]
By Lemma \ref{lem: omega-ep}, there is $\ep_1 > 0$ such that $|2\ep(\ep\omega_{\ep})-1| \ge 1/2$, and so by the lower bound \eqref{eqn: ep-omega-ep bounds} on $\omega_{\ep}$, we have
\begin{equation}\label{eqn: actual unif transv}
\frac{\pi}{\ep^2}
\le |\ip{\T_{\ep}'(\omega_{\ep})\nu_{\ep}}{\nu_{\ep}^*}_{L_{\per}^2}|.
\end{equation}
This $\ep$-dependent lower bound will be useful later.

\item
{\it{Uniform coercivity.}}
Suppose that $\T_{\ep}(\omega_{\ep})\psi = \eta$ for $\eta \in E_{\per}^0$ and $\psi \in E_{\per}^4$ with $\ip{\psi}{\nu_{\ep}}_0 = 0$.
This orthogonality condition and the evenness of $\psi$ immediately give $\hat{\psi}(\pm1) = 0$.
For $k \ne 0$, we have $\tT_{\ep}(k) \ne 0$, as discussed above, thus
\[
\hat{\psi}(k)
= \frac{\hat{\eta}(k)}{\tT_{\ep}(k)}
= \frac{\hat{\eta}(k)}{\tB_{\ep}(i\omega_{\ep}k)}
\]
We use the factorization \eqref{eqn: tB factor} of $\tB_{\ep}$ to rewrite
\[
\tT_{\ep}(k)
= (\ep\omega_{\ep})^2\omega_{\ep}^2(k^2-1)\left(k^2+\frac{\mu_{\ep}^2}{\omega_{\ep}^2}\right).
\]

At $k = 0$, we therefore have
\[
|\tT_{\ep}(0)|
= (\ep\omega_{\ep})^2\mu_{\ep}^2
\]
and for $|k| \ge 2$ we have
\[
|\tT_{\ep}(k)|
\ge (\ep\omega_{\ep})^2\omega_{\ep}^2\left(1+\frac{\mu_{\ep}^2}{\omega_{\ep}^2}\right).
\]
Lemmas \ref{lem: omega-ep} and \ref{lem: mu-ep} therefore provide $\ep_2 \in (0,\ep_1)$ and $C > 0$ such that $|\tT_{\ep}(k)| \ge C^{-1}$ for $0 < \ep < \ep_2$.
This yields
\[
|\hat{\psi}(k)|
= \left|\frac{\hat{\eta}(k)}{\tT_{\ep}(k)}\right|
\le C|\hat{\eta}(k)|
\]
for $k \in \Z\setminus\{\pm1\}$ and $0 < \ep < \ep_2$.

\item
{\it{Uniform operator norms.}}
The estimate \eqref{eqn: ep-omega-ep bounds} proves the uniform operator norms for $\T_{\ep}^{(j)}(\omega_{\ep})$, $j = 2$, $3$, $4$, and the estimate for $\nl_{\ep}$ is obvious.
For \eqref{eqn: per op norms3}, we use \eqref{eqn: actual unif transv} to bound
\[
\left|\frac{\ip{\eta}{\nu_{\ep}^*}_{L_{\per}^2}}{\ip{\T_{\ep}'(\omega_{\ep})\nu_{\ep}}{\nu_{\ep}^*}_{L_{\per}^2}}\T_{\ep}'(\omega_{\ep})\nu_{\ep}\right|
\le \frac{\ep^2}{\pi}|\ip{\eta}{\nu_{\ep}^*}_{L_{\per}^2}|\norm{2\omega_{\ep}^2(2\ep(\ep\omega_{\ep})-1)\nu_{\ep}}_{L_{\per}^2}.
\]
Lemma \ref{lem: omega-ep} then provides $\ep_3 \in (0,\ep_2)$ and $C > 0$ such that 
\[
\frac{\ep^2}{\pi}|\ip{\eta}{\nu_{\ep}^*}_{L_{\per}^2}|\norm{2\omega_{\ep}^2(2\ep(\ep\omega_{\ep})-1)\nu_{\ep}}_{L_{\per}^2}
\le C\ep\norm{\eta}_{L_{\per}^2}.
\]
\end{enumerate}

We have now checked all of the hypotheses of Theorem \ref{thm: abstract per}. 
Abbreviate
\[
\F_{\ep}(\phi,\omega)
:= \T_{\ep}(\omega)\phi + \nl_{\ep}(\phi,\phi)
\]
with $\T_{\ep}$ and $\nl_{\ep}$ from \eqref{eqn: per problem actual T}.
Theorem \ref{thm: periodics real} yields sequences $(\xi_{n,\ep})$ in $\R$ and $(\psi_{n,\ep})$ in $E_{\per}^2$ such that $\F(a\nu_{\ep}+a\psi_{\ep}^a,\omega_{\ep}^a) = 0$, where
\begin{equation}\label{eqn: per real series for app}
\psi_{\ep}^a(X) = \sum_{n=1}^{\infty} a^n\psi_{n,\ep}(X)
\quadword{and}
\omega_{\ep}^a = \omega_{\ep} + \sum_{n=1}^{\infty} a^n\xi_{n,\ep}.
\end{equation}
These sequences are unique in the sense that if $(\xi_n)$ is a sequence in $\R$ and $(\psi_n)$ is a sequence in $E_{\per}^2$ with $\ip{\psi_n}{\nu_{\ep}}_{L_{\per}^2} = 0$ for all $n$, and if 
\[
\F_{\ep}\left(a\nu_{\ep}+a\sum_{n=1}^{\infty} a^n\psi_n,\omega_{\ep}+\sum_{n=1}^{\infty} a^n\xi_n\right)
= 0
\]
for $|a| \le a_{\per,0}$, then $\xi_n = \xi_{n,\ep}$ and $\psi_n = \psi_{n,\ep}$.

We still need to prove the ``shift-reflected'' identity \eqref{eqn: per shift refl}, that each coefficient $\psi_{n,\ep}$ is a trigonometric polynomial of degree $n+1$, and the evenness of $\omega_{\ep}^a$ in $a$.
We obtain the third result as part of the proof of the first.

\begin{enumerate}[label={\bf\arabic*.}, resume]

\item
{\it{The ``shift-reflected'' identity \eqref{eqn: per shift refl}.}}
For $d \in \R$, let $S^d$ be the ``shift-by-$d$'' operator: given $f \colon \R \to \C$, put $(S^df)(x) := f(x+d)$.
We need three preliminary results about this shift operator.
First, if $f$ is even and $2\pi$-periodic, so is $S^{-\pi}f$.
Second, if $\ip{f}{\cos(\cdot)}_{L_{\per}^2} = 0$, then by substitution
\[
\ip{S^{-\pi}f}{\cos(\cdot)}_{L_{\per}^2}
= \ip{f}{S^{\pi}\cos(\cdot)}_{L_{\per}^2}
= -\ip{f}{\cos(\cdot)}_{L_{\per}^2}
= 0.
\]
Third, $S^d\F_{\ep}(\phi,\omega) = \F_{\ep}(S^d\phi,\omega)$ for any $d$, $\phi$, and $\omega$.

Then
\begin{equation}\label{eqn: a total fake}
0
= S^{\pi}\F_{\ep}(a\nu_{\ep}+a\psi_{\ep}^a,\omega_{\ep}^a)
= \F_{\ep}(aS^{\pi}\nu_{\ep}+aS^{\pi}\psi_{\ep}^a,\omega_{\ep}^a).
\end{equation}
Since
\[
(S^{\pi}\nu_{\ep})(x)
= \cos(x+\pi)
= -\cos(x)
= -\nu_{\ep}(x),
\]
we can rewrite \eqref{eqn: a total fake} using the series from \eqref{eqn: per real series for app} as
\begin{equation}\label{eqn: another total fake}
\F_{\ep}\left(
(-a)\nu_{\ep}+(-a)\sum_{n=1}^{\infty} (-a)^n\big((-1)^{n+1}S^{\pi}\psi_{n,\ep}\big), \omega_{\ep}+\sum_{n=1}^{\infty} (-a)^n\big((-1)^n\xi_{n,\ep}\big)
\right)
= 0.
\end{equation}

Since $\psi_{n,\ep} \in E_{\per}^2$ with $\ip{\psi_{n,\ep}}{\cos(\cdot)}_{L_{\per}^2} = 0$, the same is true for $(-1)^{n+1}S^{\pi}\psi_{n,\ep}$.
Then since \eqref{eqn: another total fake} is true for $|a| \le a_{\per,0}$, the uniqueness result then guarantees
\[
(-1)^{n+1}S^{\pi}\psi_{n,\ep} = \psi_{n,\ep}
\quadword{and}
(-1)^n\xi_{n,\ep} = \xi_{n,\ep}
\]
The latter identity gives $\xi_{2n+1,\ep} = 0$ for all $n$, which proves the evenness of $\omega_{\ep}^a$ in $a$.
The former gives $S^{\pi}\psi_{\ep}^a = -\psi_{\ep}^{-a}$, from which the ``shift-reflected'' identity \eqref{eqn: per shift refl} follows by a direct calculation.

\item
{\it{The degree of $\psi_{n,\ep}$.}}
First we need a number of careful observations about the proof of Theorem \ref{thm: abstract per} and the concrete situation at hand.

\begin{enumerate}[label=$\bullet$]

\item
The function $\nu_{\ep}(X) = \cos(X)$ is a trigonometric polynomial of degree $1$ for all $\ep$.

\item
If $\psi_1$ is a trigonometric polynomial of degree $n$ and $\psi_2$ is a trigonometric polynomial of degree $m$, then $\nl_{\ep}(\psi_1,\psi_2) = \psi_1\psi_2$ is a trigonometric polynomial of degree $n+m$.

\item
If $\psi$ is a trigonometric polynomial of degree $n$, then $\T_{\ep}^{(j)}(\omega_{\ep})\psi$ and $\P_{\ep}\psi$, with $\P_{\ep}$ defined in \eqref{eqn: P-ep}, are trigonometric polynomials of degree $n$.

\item
If $\xi_1,\ldots,\xi_{n-1} \in \R$ and $\psi_k$ is a trigonometric polynomial of degree $k$ for $1 \le k \le n-1$, then $\Gamma_{n,\ep}(\xi_1,\ldots,\xi_{n-1},\psi_1,\ldots,\psi_{n-1})$ defined in \eqref{eqn: Gamma-n} is a trigonometric polynomial of degree $n$.
More precisely, all of the terms in $\Gamma_{n,\ep}$ here are trigonometric polynomials of degree $n-1$ or less except for the degree-$n$ term $\nl_{\ep}(\nu_{\ep},\psi_{n-1})$.

\item
Suppose that $\eta$ is a trigonometric polynomial of degree $n$ and $\psi \in H_{\per}^4$ satisfies $\T_{\ep}(\omega_{\ep})\psi = \eta$, so $\tT_{\ep}(k)\hat{\psi}(k) = \hat{\eta}(k)$.
For $|k| \ge n+1$, we have $\hat{\eta}(k) = 0$.
Since $\tT_{\ep}(k) \ne 0$ for $|k| \ge 2$, we conclude $\hat{\psi}(k) = 0$ for $|k| \ge n+1$, and so $\psi$ is a trigonometric polynomial of degree $n$.
Consequently, if $\eta$ is a trigonometric polynomial of degree $n$, then so is $\T_{\ep}(\omega_{\ep})^{-1}\eta$ when this function is defined.
\end{enumerate}

We then prove that $\psi_{n,\ep}$ is a trigonometric polynomial of degree $n+1$ by induction on $n$.
For $n=1$, we know that $\nl_{\ep}(\nu_{\ep},\nu_{\ep})$ is a trigonometric polynomial of degree $2$, and therefore so is $\psi_{1,\ep}$ by its definition in \eqref{eqn: xi-1-ep psi-1-ep} and the observations above.
The induction step then follows from the observations above and the formula \eqref{eqn: psi-n-ep} for $\psi_{n,\ep}$.
Lombardi performs an analogous induction argument for \cite[Lem.\@ 4.A.12]{lombardi}.
\end{enumerate}

\subsection{Preliminary estimates for the extension program}\label{app: here is where the per fun begins}
We first recall some notation.
Theorem \ref{thm: periodics real} provides a family of functions
\[
\phi_{\ep}^a(x)
= a\cos(\omega_{\ep}^ax) + a\psi_{\ep}^a(\omega_{\ep}^ax)
\]
for $0 < \ep < \ep_{\per,0}$ and $|a| < a_{\per,0}$ such that 
\[
\omega_{\ep}^a = \omega_{\ep} + \xi_{\ep}^a,
\qquad
\xi_{\ep}^a = \sum_{n=1}^{\infty} a^n\xi_{n,\ep},
\]
where $\omega_{\ep}$ is defined in \eqref{eqn: omega-ep}, and 
\begin{equation}\label{eqn: psi app aux}
\psi_{\ep}^a(X) 
= \sum_{n=1}^{\infty} a^n\psi_{n,\ep}(X),
\qquad
\psi_{n,\ep}(X) 
= \sum_{k=1}^{n+1} c_{n,k,\ep}\cos(kX)
\end{equation}
From \eqref{eqn: psi-n xi-n threshold series}, there is $C_{\per} > 0$ such that 
\begin{equation}\label{eqn: per coeff 0 est}
\norm{\psi_{n,\ep}}_{L_{\per}^2} < \frac{C_{\per}}{a_{\per,0}^n}
\quadword{and}
|\xi_{n,\ep}| < \frac{C_{\per}}{a_{\per,0}^n}
\end{equation}
for all $0 < \ep < \ep_{\per,0}$ and $|a| < a_{\per,0}$.

Our goal here is to extend $\phi_{\ep}^a$ to be analytic with an asymptotic phase shift on the strip $\U_b = \set{z \in \C}{|\im(z)| < b}$ for $0 < b < b_{\sigma}$.
We begin with some relatively straightforward estimates and then give an overview of our strategy for an idealized case in Appendix \ref{app: per strat}, after which we address some natural, but incorrect, approaches in Appendix \ref{app: per wrong}.
Then we give the full extension starting in Appendix \ref{app: ext est}.
Throughout, we keep irritatingly close track of the exact constants and thresholds that arise in our estimates and how they depend on various parameters.

We first prove some straightforward estimates for the frequency $\omega_{\ep}^a$.

\begin{lemma}\label{lem: ultimate omega lemma}
There exist $C_{\omega} > 0$, $0 < \ep_{\per,1} \le \ep_{\per,0}$, and $0 \le a_{\per,1} < a_{\per,0}$ such that the following hold for all $0 < \ep < \ep_{\per,1}$ and $0 \le |a|$, $|\grave{a}| < a_{\per,1}$.

\begin{enumerate}[label={\bf(\roman*)}, ref={(\roman*)}]

\item\label{part: xi-ep-a Lip}
$|\xi_{\ep}^a-\xi_{\ep}^{\grave{a}}| 
\le C_{\omega}|a-\grave{a}|$.

\item\label{part: xi-ep-a bound}
$|\xi_{\ep}^a| 
\le C_{\omega}|a|$.

\item\label{part: omega ep vs 1/ep}
$\left|\omega_{\ep}^a - \frac{1}{\ep}\right| 
\le \ep + C_{\omega}|a|$

\item\label{part: omega-ep-a bound}
$\frac{C_{\omega}}{2}\ep^{-1} 
< \omega_{\ep}^a 
<  C_{\omega}\ep^{-1}$.

\item\label{part: omega-ep-a Lip}
$|(\omega_{\ep}^a)^r-(\omega_{\ep}^{\grave{a}})^r| 
\le C_{\omega}r\ep^{1-r}|a-\grave{a}|$ for $r \ge 0$.

\item\label{part: ep omega-ep-a minus 1}
$|\ep\omega_{\ep}^a-1| \le C_{\omega}\ep$.
\end{enumerate}
\end{lemma}

\begin{proof}

\begin{enumerate}[label={\bf(\roman*)}]

\item
The power series 
\[
\xi_{\ep}^{(\cdot)}
\colon (-a_{\per,0},a_{\per,0}) \to \R
\colon a \mapsto \sum_{n=1}^{\infty} a^n\xi_{n,\ep}
\]
converges on $(-a_{\per,0},a_{\per,0})$, so it is differentiable on this interval, and so the derivative series 
\[
\partial_a[\xi_{\ep}^{(\cdot)}](a) = \sum_{n=1}^{\infty} na^{n-1}\xi_{n,\ep}
\]
also converges on $(-a_{\per,0},a_{\per,0})$.
Now we restrict to $a \in (-a_{\per,0}/2,a_{\per,0}/2)$.
We use \eqref{eqn: per coeff 0 est} to estimate
\[
|na^{n-1}\xi_{n,\ep}|
\le n|a|^{n-1}\frac{C_{\per}}{a_{\per,0}^n}
\le \frac{C_{\per}n|a_{\per,0}/2|^{n-1}}{a_{\per,0}^n}
= \frac{C_{\per}}{2a_{\per,0}}\left(\frac{n}{2^n}\right).
\]
The series 
\[
\sum_{n=1}^{\infty} \frac{n}{2^n}
\]
converges by the ratio test.
This shows that 
\[
C_{\omega}
:= \sup_{\substack{0 < \ep < \ep_{\per,0} \\ |a| < a_{\per,0}/2}} 
|\partial_a[\xi_{\ep}^{(\cdot)}](a)| 
< \infty,
\]
which in turn gives the desired Lipschitz estimate for $\xi_{\ep}^{(\cdot)}$.

\item
Take $\grave{a} = 0$ in part \ref{part: xi-ep-a Lip}.

\item
This follows by rewriting
\[
\omega_{\ep}^a-\frac{1}{\ep}
= \left(\omega_{\ep}-\frac{1}{\ep}\right)+\xi_{\ep}^a
\]
and using the estimate \eqref{eqn: omega-ep est} for $\omega_{\ep}-1/\ep$ and part \ref{part: xi-ep-a bound} to control $\xi_{\ep}^a$.

\item
This follows by rewriting
\[
\omega_{\ep}^a
= \left[\left(\omega_{\ep}-\frac{1}{\ep}\right) + \xi_{\ep}^a\right]+\frac{1}{\ep}
\]
and using the estimate \eqref{eqn: omega-ep est} for $\omega_{\ep}-1/\ep$ and part \ref{part: xi-ep-a bound} to control $\xi_{\ep}^a$.

\item
First we use the mean value estimate
\[
|(\omega_{\ep}^a)^r-(\omega_{\ep}^{\grave{a}})^r|
\le r\left(\sup_{0 \le s \le 1} |(1-s)\omega_{\ep}^{\grave{a}}+s\omega_{\ep}^a|^{r-1}\right)|\omega_{\ep}^a-\omega_{\ep}^{\grave{a}}|.
\]
We bound
\[
|(1-s)\omega_{\ep}^{\grave{a}}+s\omega_{\ep}^a|
\le (1-s)|\omega_{\ep}^{\grave{a}}+s|\omega_{\ep}^a|
\le C_{\omega}\ep^{-1}
\]
by part \ref{part: omega-ep-a bound} and
\[
|\omega_{\ep}^a-\omega_{\ep}^{\grave{a}}|
= |\xi_{\ep}^a-\xi_{\ep}^{\grave{a}}|
\le C_{\omega}|a-\grave{a}|
\]
by part \ref{part: xi-ep-a Lip}, from which the desired estimate on $|(\omega_{\ep}^a)^r-(\omega_{\ep}^{\grave{a}})^r|$ follows.

\item
We rewrite
\[
\ep\omega_{\ep}^a-1
= \ep(\omega_{\ep}^a-\omega_{\ep})+(\ep\omega_{\ep}-1).
\]
We estimate
\[
|\omega_{\ep}^a-\omega_{\ep}|
= |\omega_{\ep}^a-\omega_{\ep}^0|
\le C_{\omega}|a|
\]
by part \ref{part: omega-ep-a Lip} and 
\[
|\ep\omega_{\ep}-1|
\le C\ep^2
\]
by \eqref{eqn: omega-ep est2}.
The desired estimate follows from these two bounds.
\qedhere
\end{enumerate}
\end{proof}

Next, we develop some further control the manifestation of the asymptotic phase shift via the function $\tau$ from \eqref{eqn: tau}.
Recall that
\[
\Tsf_{\vartheta}(z)
= z+\vartheta\tau(z)
\]
for $\vartheta \in \R$, with
\begin{equation}\label{eqn: tau-b}
\tau_b 
:= \sup_{z \in \U_b} |\tau(z)|
< \infty
\end{equation}
for $0 < b < b_{\sigma}$.

\begin{lemma}
Let $0 < b < b_{\sigma}$ and $\theta_0 > 0$.
There exist $\ep_{\per,2}$, $a_{\per,2} > 0$ such that if $0 < \ep < \ep_{\per,2}$, $|\theta| \le \theta_0$, $|a| \le a_{\per,2}$, and $z \in \U_b$, then
\begin{equation}\label{eqn: Tsf est global}
|\Tsf_{\ep\theta}(z)| 
\le |z| + 1
\end{equation}
and
\begin{equation}\label{eqn: Tsf est imag}
\omega_{\ep}^a|\im(\Tsf_{\ep\theta}(z))| 
\le \omega_{\ep}^ab + C_{\omega}\theta_0\tau_b. 
\end{equation}
\end{lemma}

\begin{proof}
Since 
\[
|\Tsf_{\ep\theta}(z)|
\le |z|+\ep\theta|\tau(z)|
\le |z|+\ep\theta_0\tau_b,
\]
we can take 
\[
\ep_{\per,2} 
:= \min\left\{\frac{1}{\theta_0\tau_b},\ep_{\per,1}\right\}
\] 
to obtain \eqref{eqn: Tsf est global}.
For \eqref{eqn: Tsf est imag}, we have $|\im(z)| < b$ when $z \in \U_b$, so taking $a_{\per,2} = a_{\per,1}$ implies
\[
|\omega_{\ep}^a\im(\Tsf_{\ep\theta}(z))| 
= \omega_{\ep}^a|\im(z+\ep\theta\tau(z))|
\le \omega_{\ep}^ab+(\ep\omega_{\ep}^a)\theta_0\tau_b
\le \omega_{\ep}^ab + C_{\omega}\theta_0\tau_b
\]
by part \ref{part: omega-ep-a bound} of Lemma \ref{lem: ultimate omega lemma}.
\end{proof}

Now we update the frequency to depend on an exponentially small amplitude.
The following is a direct consequence of Lemma \ref{lem: ultimate omega lemma}.

\begin{theorem}
Let $0 < b < b_{\sigma}$ and $\theta_0 > 0$.
For $0 < \ep < \ep_{\per,2}$ and $|\alpha| \le a_{\per,2}$, the quantity
\begin{equation}\label{eqn: Omega-ep-alpha}
\Omega_{\ep,b}^{\alpha}
:= \omega_{\ep}^{\alpha{e}^{-b/\ep}}
\end{equation}
is defined and satisfies the following.

\begin{enumerate}[label={\bf(\roman*)}]

\item
$\Omega_{\ep}^0 = \omega_{\ep}$.

\item
Take $C_{\omega} > 0$ from Lemma \ref{lem: ultimate omega lemma}.
Then
\begin{equation}\label{eqn: Omega-ep-alpha map1}
|\Omega_{\ep,b}^{\alpha}| 
\le C_{\omega}\ep^{-1},
\end{equation}
\begin{equation}\label{eqn: Omega-ep-alpha map2}
\left|\Omega_{\ep,b}^{\alpha} - \frac{1}{\ep}\right| 
\le \ep + C_{\omega}|\alpha|e^{-b/\ep},
\end{equation}
and
\begin{equation}\label{eqn: Omega-ep-alpha map3}
\frac{C_{\omega}}{2}\ep^{-1} 
< \Omega_{\ep,b}^{\alpha}
< C_{\omega}\ep^{-1}
\end{equation}
for $0 < \ep < \ep_{\per,2}$ and $|\alpha| \le a_{\per,2}$.

\item
For each integer $r \ge 1$, 
\begin{equation}\label{eqn: Omega-ep-Lip alpha}
|(\Omega_{\ep,b}^{\alpha})^r-(\Omega_{\ep,b}^{\grave{\alpha}})^r|
\le C_{\omega}r\ep^{1-r}e^{-b/\ep}|\alpha-\grave{\alpha}|
\end{equation}
for $0 < \ep < \ep_{\per,2}$ and $|\alpha|$, $|\grave{\alpha}| \le a_{\per,2}$.

\item
If $0 < \ep < \ep_{\per,2}$, $|\theta| < \theta_0$, $|\alpha| < a_{\per,2}$, and $z \in \U_b$, then
\begin{equation}\label{eqn: Tsf est imag alpha}
\Omega_{\ep,b}^{\alpha}|\im(\Tsf_{\ep\theta}(z))|
\le \Omega_{\ep,b}^{\alpha}b + C_{\omega}\theta_0\tau_b.
\end{equation}
\end{enumerate}
\end{theorem}

Perhaps surprisingly, the exponentially small factor $e^{-b/\ep}$ in \eqref{eqn: Omega-ep-Lip alpha} will not be very helpful.
Often we will use \eqref{eqn: Omega-ep-Lip alpha} to estimate only one term in a sum of several terms; at least one other term in that sum will not have a factor of $e^{-b/\ep}$, which prevents the bound on the whole sum from being exponentially small.

The estimate \eqref{eqn: Omega-ep-alpha map2} shows that the frequency $\Omega_{\ep}^{\alpha}$ is close to $1/\ep$ in terms of both $\ep$ and $\alpha$.
It will be convenient to have both a more precise, and a more general, version of this.

\begin{lemma}\label{lem: omega-ep-ep-alpha expn}
Let $0 < b < b_{\sigma}$ and $\theta_0 > 0$.
There are $\ep_{\per,3}$, $a_{\per,3} > 0$ such that if $0 < \ep < \ep_{\per,3}$, $0 \le |\alpha|$, $|\grave{\alpha}| < a_{\per,3}$, and $0 \le s \le 1$, then
\begin{equation}\label{eqn: favorite expn est}
b\left|\left((1-s)\Omega_{\ep,b}^{\grave{\alpha}}+s\Omega_{\ep}^{\alpha}\right)-\frac{1}{\ep}\right|
< 1.
\end{equation}
\end{lemma}

\begin{proof}
Let 
\[
\ep_{\per,3} := \min\left\{\ep_{\per,2}, \frac{1}{2b}\right\}
\quadword{and}
a_{\per,3} := \min\left\{a_{\per,2}, \frac{1}{2C_{\omega}b}\right\}
\]
and assume $0 < \ep < \ep_{\per,3}$ and $0 \le |\alpha| < a_{\per,3}$.
Then $\Omega_{\ep,b}^{\alpha}$ and $\Omega_{\ep,b}^{\grave{\alpha}}$ are defined, and by \eqref{eqn: Omega-ep-alpha map2}, we have
\[
\left|(1-s)\Omega_{\ep,b}^{\grave{\alpha}}+s\Omega_{\ep}^{\alpha}-\frac{1}{\ep}\right|
\le (1-s)\left|\Omega_{\ep,b}^{\grave{\alpha}}-\frac{1}{\ep}\right| + s\left|\Omega_{\ep,b}^{\alpha}-\frac{1}{\ep}\right|
\le \ep + C_{\omega}|\alpha|{e}^{-b/\ep}
< \frac{1}{b},
\]
from which the estimate \eqref{eqn: favorite expn est} follows.
\end{proof}

These are all the preliminary results that we need to begin the extension of the periodics $\phi_{\ep}^a$.
We next give an overview of the various technical steps that we will take.

\subsection{A model strategy for the extension program}\label{app: per strat}
We will prove many separate estimates in our quest for Theorem \ref{thm: ultimate varphi theorem}, but they all follow the same general pattern.
Here is an illustrative situation that will serve as a template.

Suppose that we are in the special case in which each term $\psi_{n,\ep}$ in the series \eqref{eqn: psi app aux} is not merely an even $(n+1)$st degree trigonometric polynomial but a genuine sinusoid: 
\begin{equation}\label{eqn: easy psi}
\psi_{n,\ep}(Z) 
= c_{n,\ep}\cos((n+1)Z).
\end{equation}
Fix $\theta_0 > 0$ and $0 < b < b_{\sigma}$.
Our goal is to find $C_n > 0$ such that 
\[
\sum_{n=1}^{\infty} C_n < \infty
\quadword{and}
\big|\alpha^ne^{-(n+1)b/\ep}\psi_{n,\ep}(\Omega_{\ep,b}^{\alpha}\Tsf_{\ep\theta}(z))\big|
\le C_n|\alpha|
\]
for $\ep$ and $\alpha$ sufficiently small, $|\theta| < \theta_0$, and $z \in \U_b$.
This will prove the absolute, uniform convergence (and thus analyticity) of the series
\[
\sum_{n=1}^{\infty} \alpha^ne^{-(n+1)b/\ep}\psi_{n,\ep}(\Omega_{\ep,b}^{\alpha}\Tsf_{\ep\theta}(z))
\]
on $\U_b$ with a uniform $\O(\alpha)$ estimate as well.

Because $|\cos(Z)| \le e^{|\im(Z)|}$ for all $Z \in \C$ and because $|c_{n,\ep}| \le C_{\per}a_{\per,0}^{-n}$ by \eqref{eqn: per coeff 0 est}, we have
\[
|\psi_{n,\ep}(Z)|
\le \frac{C_{\per}e^{(n+1)|\im(Z)|}}{a_{\per,0}^n}.
\]
Take $Z = \omega_{\ep}^a\Tsf_{\ep\theta}(z)$ with $z \in \U_b$ and $|\theta| < \theta_0$.
We use the estimate 
\[
|\im(\omega_{\ep}^a\Tsf_{\ep\theta}(z))| 
\le \omega_{\ep}^ab + C_{\omega}\theta_0\tau_b
\]
from \eqref{eqn: Tsf est imag} to bound
\[
|\psi_{n,\ep}(\omega_{\ep}^a\Tsf_{\ep\theta}(z))|
\le \frac{C_{\per}e^{(n+1)\omega_{\ep}^ab}e^{(n+1)C_{\omega}\theta_0\tau_b}}{a_{\per,0}^n}
= \big(C_{\per}e^{C_{\omega}\theta_0\tau_b}\big)\left(\frac{e^{C_{\omega}\theta_0\tau_b}}{a_{\per,0}}\right)^ne^{(n+1)\omega_{\ep}^ab}.
\]
Because $\omega_{\ep}^a$ is $\O(\ep^{-1})$ by part \ref{part: omega-ep-a bound} of Lemma \ref{lem: ultimate omega lemma}, this confirms our intention to take $a = \alpha{e}^{-b/\ep}$.
As noted in Section \ref{sec: an and amp}, this illustrates the somewhat circular process of extending the periodics to a strip so that we can use an exponentially small amplitude and needing the amplitude to be exponentially small for the extension to work.

Setting $a = \alpha{e}^{-b/\ep}$, this yields
\begin{equation}\label{eqn: per ext rem aux1}
\big|\alpha^ne^{-(n+1)b/\ep}\psi_{n,\ep}(\Omega_{\ep,b}^{\alpha}\Tsf_{\ep\theta}(z))\big|
\le \big(C_{\per}e^{C_{\omega}\theta_0\tau_b}\big)\left(\frac{e^{C_{\omega}\theta_0\tau_b}|\alpha|}{a_{\per,0}}\right)^ne^{(n+1)(\Omega_{\ep,b}^{\alpha}-1/\ep)b}.
\end{equation}
From \eqref{eqn: favorite expn est} with $s=1$, we control the factor $e^{(n+1)(\Omega_{\ep,b}^{\alpha}-1/\ep)b}$ by
\[
\exp\left((n+1)\left(\Omega_{\ep,b}^{\alpha}-\frac{1}{\ep}\right)b\right)
< e^{n+1},
\]
so that \eqref{eqn: per ext rem aux1} becomes
\begin{equation}\label{eqn: per ext rem aux2}
\big|\alpha^ne^{-(n+1)b/\ep}\psi_{n,\ep}(\Omega_{\ep,b}^{\alpha}\Tsf_{\ep\theta}(z))\big|
\le \big(C_{\per}e^{C_{\omega}\theta_0\tau_b+1}\big)\left(\frac{e^{C_{\omega}\theta_0\tau_b+1}|\alpha|}{a_{\per,0}}\right)^n.
\end{equation}
This suggests restricting $|\alpha|$ to 
\[
|\alpha|
< \frac{a_{\per,0}}{2e^{C_{\omega}\theta_0\tau_b+1}}.
\]

The following trick will be helpful later: break off $n-1$ factors of $|\alpha|$ and make the sacrifice of having a slightly more complicated constant factor to convert \eqref{eqn: per ext rem aux2} into
\begin{multline*}
\big|\alpha^ne^{-(n+1)b/\ep}\psi_{n,\ep}(\Omega_{\ep,b}^{\alpha}\Tsf_{\ep\theta}(z))\big|
\le \frac{C_{\per}e^{2(C_{\omega}\theta_0\tau_b+1)}}{a_{\per,0}}\left(\frac{e^{C_{\omega}\theta_0\tau_b+1}|\alpha|}{a_{\per,0}}\right)^{n-1}|\alpha| \\
\le \left(\frac{C_{\per}e^{2(C_{\omega}\theta_0\tau_b+1)}}{a_{\per,0}}\right)\left(\frac{1}{2^{n-1}}\right)|\alpha|.
\end{multline*}
This proves the absolute, uniform convergence of the series
\[
\sum_{n=1}^{\infty} \alpha^ne^{-(n+1)b/\ep}\psi_{n,\ep}(\Omega_{\ep,b}^{\alpha}\Tsf_{\ep\theta}(z))
\]
with the estimate
\[
\left|\sum_{n=1}^{\infty} \alpha^ne^{-(n+1)b/\ep}\psi_{n,\ep}(\Omega_{\ep,b}^{\alpha}\Tsf_{\ep\theta}(z))\right|
\le C|\alpha|
\]
for $C$ independent of $\alpha$, $\ep$, $z$, and $\theta$, under the oversimplification \eqref{eqn: easy psi} on $\psi_{n,\ep}$.
We note that the mismatch of $n$ powers of $\alpha$ against the $(n+1)$st-degree trigonometric polynomial $\psi_{n,\ep}$ is not so important, as this polynomial is paired with $n+1$ powers of $e^{-b/\ep}$, which absorb the polynomial's exponentially large behavior on $\U_b$.

This overview has much in common with Lombardi's periodic extension in \cite[Sec.\@ 7.3.1]{lombardi}.
However, our subsequent program is broader and more ambitious, as we also need Lipschitz estimates in $\alpha$ and $\theta$, and for both these Lipschitz estimates and the prior mapping estimate, we really need results with $\psi_{n,\ep}$ replaced by an arbitrary derivative $\partial_Z^r[\psi_{n,\ep}]$.
Nonetheless, all of the techniques are largely similar to the special case above.

\subsection{Alternative, incorrect approaches to the extension program}\label{app: per wrong}
The model strategy in Appendix \ref{app: per strat} can give us further insight into why simpler approaches to the periodic extension problem will not work and why we must build the phase shift into the extension from the very start.
First, it might be natural to try to extend the periodic solutions $\phi_{\ep}^a$ from Theorem \ref{thm: periodics real} to solutions $\varphi_{\ep,b}^{\alpha}$ on $\U_b$, with
\begin{equation}\label{eqn: how a per ext extends}
\varphi_{\ep,b}^{\alpha}(x) 
= \phi_{\ep}^{\alpha{e}^{-b/\ep}}(x), \ x \in \R.
\end{equation}
Then we could try to define $\varphi_{\ep,b}^{\alpha,\theta} := \varphi_{\ep,b}^{\alpha} \circ \Tsf_{\ep\theta}$.
This will immediately fail since possibly $|\im(\Tsf_{\ep\theta}(z)| > b$ for $z \in \U_b$, no matter how small $\ep$ is chosen, and therefore the composition $\varphi_{\ep,b}^{\alpha} \circ \Tsf_{\ep\theta}$ need not be defined.

Since the problem with this first approach is that perhaps $|\im(\Tsf_{\ep\theta}(z)| > b$ for $z \in \U_b$, we might try to define an extension $\varphi_{\ep,b}^{\alpha}$ on some larger $b$-dependent strip first, say, $\U_{(b+b_{\sigma})/2}$. 
Then, given $\theta$, we could restrict $\ep$ so that $|\im(\Tsf_{\ep\theta})(z)| < (b+b_{\sigma})/2$ for $z \in \U_b$, so $\varphi_{\ep,b}^{\alpha} \circ \Tsf_{\ep\theta}$ would be defined and analytic on $\U_b$.
The problem here, as illustrated by Appendix \ref{app: per strat}, is that we would want to use $e^{-(b+b_{\sigma})/2\ep}$, not $e^{-b/\ep}$, in \eqref{eqn: how a per ext extends}.

But then when we work out the selection mechanism \eqref{eqn: sel mech} after making the nanopteron ansatz, we would have factors of $e^{-(b+b_{\sigma})/2\ep}$, not $e^{-b/\ep}$.
In the amplitude selection problem, this would pair a factor of $e^{(b+b_{\sigma})/2\ep}$ with the oscillatory integral functional $\iota_{\ep}$, see \eqref{eqn: A-ep}.
The integrand of $\iota_{\ep}$ is analytic only on $\U_b$, and so that oscillatory integral is only $\O(e^{-b/\ep})$, which is not small enough to counteract $e^{(b+b_{\sigma})/2\ep}$.
The same flaw would appear in the phase shift selection problem in the map $\M_{q,b}$, see \eqref{eqn: M-ep}.

\subsection{Additional estimates for the extension program}\label{app: ext est}
The following lemma contains the much more technical deployment of the strategy outlined in Appendix \ref{app: per strat}, and the estimates here are the key to the periodic extensions.

\begin{lemma}\label{lem: ultimate psi lemma}
Let $0 < b < b_{\sigma}$ and $\theta_0 > 0$.
The coefficients $\psi_{n,\ep}$ from \eqref{eqn: psi app aux} satisfy the following.

\begin{enumerate}[label={\bf(\roman*)}, ref={(\roman*)}]

\item\label{part: ultimate psi lemma 1}
If $0 < \ep < \ep_{\per,3}$, $n \ge 1$, $r \ge 0$, and $Z \in \C$, then
\begin{equation}\label{eqn: psi-n-ep Z est1}
|\partial_Z^r[\psi_{n,\ep}](Z)|
\le C_{\per}(n+1)^{r+1}\left(\frac{e^{(n+1)|\im(Z)|}}{a_{\per,0}^n}\right).
\end{equation}

\item
If $0 < \ep < \ep_{\per,3}$, $n \ge 1$, $r \ge 0$, $|\theta| < \theta_0$, $|a| < a_{\per,3}$, and $z \in \U_b$, then
\begin{equation}\label{eqn: psi-n-ep Z est2}
\big|\partial_Z^r[\psi_{n,\ep}](\omega_{\ep}^{a}\Tsf_{\ep\theta}(z))\big|
\le C_{\per}e^{C_{\omega}\theta_0\tau_b}(n+1)^{r+1}\left(\frac{e^{C_{\omega}\theta_0\tau_b}}{a_{\per,0}}\right)^ne^{(n+1)\omega_{\ep}^ab}.
\end{equation}

\item
There exist $\ep_{\per,4}$, $a_{\per,4} > 0$ such that if $0 < \ep < \ep_{\per,4}$, $n \ge 1$, $r \ge 0$, $|\theta| < \theta_0$, $0 \le |\alpha|$, $|\grave{\alpha}| \le a_{\per,4}$, and $z \in \U_b$, then
\begin{equation}\label{eqn: psi-n-ep Lip1}
|\alpha^n-\grave{\alpha}^n|
e^{-(n+1)b/\ep}
\partial_Z^r[\psi_{n,\ep}](\Omega_{\ep,b}^{\alpha}\Tsf_{\ep\theta}(z))|
\le C_{\per}\left(\frac{e^{2(C_{\omega}\theta_0\tau_b+1)}}{a_{\per,0}}\right)\left(\frac{n(n+1)^{r+1}}{2^{n-1}}\right)|\alpha-\grave{\alpha}|.
\end{equation}

\item
If $0 < \ep < \ep_{\per,4}$, $n \ge 1$, $r \ge 0$, $|\theta| < \theta_0$, $0 \le |\alpha|$, $|\grave{\alpha}| \le a_{\per,4}$, and $z \in \U_b$, then
\begin{multline}\label{eqn: psi-n-ep Lip2}
|\grave{\alpha}|^n
e^{-nb/\ep}
\big|\partial_Z^r[\psi_{n,\ep}](\Omega_{\ep,b}^{\alpha}\Tsf_{\ep\theta}(z))-\partial_Z^r[\psi_{n,\ep}](\Omega_{\ep,b}^{\grave{\alpha}}\Tsf_{\ep\theta}(z))\big| \\
\le C_{\omega}C_{\per}e^{C_{\omega}\theta_0\tau_b+1}\left(\frac{(n+1)^{r+2}}{2^n}\right)
(|z|+1)
e^{-b/\ep}
|\alpha-\grave{\alpha}|.
\end{multline}

\item
If $0 < \ep < \ep_{\per,4}$, $n \ge 1$, $r \ge 0$, $0 \le |\theta|$, $|\grave{\theta}| < \theta_0$, $|\alpha| \le a_{\per,4}$, and $z \in \U_b$, then
\begin{multline}\label{eqn: psi-n-ep Lip3}
|\alpha|^n
e^{-(n+1)b/\ep}
\big|\partial_Z^r[\psi_{n,\ep}](\Omega_{\ep,b}^{\alpha}\Tsf_{\ep\theta}(z))-\partial_Z^r[\psi_{n,\ep}](\Omega_{\ep,b}^{\alpha}\Tsf_{\ep\grave{\theta}}(z))\big| \\
\le C_{\omega}C_{\per}e^{C_{\omega}\theta_0\tau_b+1}\theta_0\tau_b\left(\frac{(n+1)^{r+2}}{2^n}\right)
|\alpha|
|\theta-\grave{\theta}|.
\end{multline}

\end{enumerate}
\end{lemma}

\begin{proof}

\begin{enumerate}[label={\bf(\roman*)}]

\item
For $Z \in \C$, we use the definition of $\psi_{n,\ep}$ in \eqref{eqn: psi app aux} to estimate
\[
|\partial_Z^r[\psi_{n,\ep}](Z)|
\le \sum_{k=1}^{n+1} |c_{n,k,\ep}|k^r|e^{ikZ}|
\le (n+1)^re^{(n+1)|\im(Z)|}\sum_{k=1}^{n+1} |c_{n,k,\ep}|.
\]
From the equivalence of the $\ell^1$- and $\ell^2$-norms on $\R^{n+1}$, we have
\[
\sum_{k=1}^{n+1} |c_{n,k,\ep}|
\le \sqrt{n+1}\left(\sum_{k=1}^{n+1} |c_{n,k,\ep}|^2\right)^{1/2}.
\]
Last, by Fourier analysis and the expansion \eqref{eqn: psi app aux} for $\psi_{n,\ep}$, we have
\[
\left(\sum_{k=1}^{n+1} |c_{n,k,\ep}|^2\right)^{1/2}
= \norm{\psi_{n,\ep}}_{L_{\per}^2}
\]
Combining these bounds with \eqref{eqn: per coeff 0 est} and, for simplicity, estimating 
\[
(n+1)^r\sqrt{n+1} 
\le (n+1)^{r+1}, 
\]
yields \eqref{eqn: psi-n-ep Z est1}.

\item
This follows from \eqref{eqn: psi-n-ep Z est1} with $Z = \omega_{\ep}^{a}\Tsf_{\ep\theta}(z)$ and the estimate on $|\im(\omega_{\ep}^a\Tsf_{\ep\theta}(z))|$ from \eqref{eqn: Tsf est imag}.

\item
Put 
\begin{equation}\label{eqn: a-per-4 ep-per-4}
a_{\per,4} := \min\left\{\frac{e^{-(C_{\omega}\theta_0\tau_b+1)}a_{\per,0}}{2},a_{\per,3}\right\}
\quadword{and}
\ep_{\per,4} := \ep_{\per,3}.
\end{equation}
Taking $0 \le |\alpha|, \ |\grave{\alpha}| \le a_{\per,4}$ and $0 < \ep < \ep_{\per,4}$ ensures that $\Omega_{\ep,b}^{\alpha}$ and $\Omega_{\ep,b}^{\grave{\alpha}}$ are defined.
By \eqref{eqn: psi-n-ep Z est2}, we have
\begin{equation}\label{eqn: psi-n-ep Lip aux1}
|\alpha^n-\grave{\alpha}^n|
e^{-(n+1)b/\ep}
\big|\partial_Z^r[\psi_{n,\ep}](\Omega_{\ep,b}^{\alpha}\Tsf_{\ep\theta}(z))\big| 
\le C_{\per}e^{C_{\omega}\theta_0\tau_b}(n+1)^{r+1}\left(\frac{e^{C_{\omega}\theta_0\tau_b}}{a_{\per,0}}\right)^n|\alpha^n-\grave{\alpha}^n|
e^{(n+1)(\Omega_{\ep,b}^{\alpha}-1/\ep)b}.
\end{equation}
By \eqref{eqn: favorite expn est} with $s=1$, we have
\begin{equation}\label{eqn: psi-n-ep Lip aux2}
\exp\left((n+1)\left(\Omega_{\ep,b}^{\alpha}-\frac{1}{\ep}\right)b\right)
\le e^{n+1}.
\end{equation}

Next, when $|\alpha|$, $|\grave{\alpha}| \le a_{\per,4}$, the mean value theorem gives
\[
|\alpha^n-\grave{\alpha}^n|
\le \left(\sup_{0 \le s \le 1} n\big|(1-s)|\grave{\alpha}|+s|\alpha|\big|^{n-1}\right)|\alpha-\grave{\alpha}|
\le n
a_{\per,4}^{n-1}
|\alpha-\grave{\alpha}|.
\]
Then
\begin{equation}\label{eqn: psi-n-ep Lip aux3}
\left(\frac{e^{C_{\omega}\theta_0\tau_b}}{a_{\per,0}}\right)^n|\alpha^n-\grave{\alpha}^n|
\le \left(\frac{e^{C_{\omega}\theta_0\tau_b}}{a_{\per,0}}\right)^nna_{\per,4}^{n-1}|\alpha-\grave{\alpha}|
\le \left(\frac{e^{C_{\omega}\theta_0\tau_b}}{a_{\per,0}}\right)\frac{n}{(2e)^{n-1}}|\alpha-\grave{\alpha}|.
\end{equation}
Updating the estimate in \eqref{eqn: psi-n-ep Lip aux1} with \eqref{eqn: psi-n-ep Lip aux2} and \eqref{eqn: psi-n-ep Lip aux3} yields \eqref{eqn: psi-n-ep Lip1}.

\item
We use the mean value theorem to estimate
\begin{multline*}
|\grave{\alpha}|^n
e^{-(n+1)b/\ep}
\big|\partial_Z^r[\psi_{n,\ep}](\Omega_{\ep,b}^{\alpha}\Tsf_{\ep\theta}(z))-\partial_Z^r[\psi_{n,\ep}](\Omega_{\ep,b}^{\grave{\alpha}}\Tsf_{\ep\theta}(z))\big| \\
\le \left(\sup_{0 \le s \le 1} |\grave{\alpha}|^ne^{-(n+1)b/\ep}\big|\partial_Z^{r+1}[\psi_{n,\ep}]\big(((1-s)\Omega_{\ep,b}^{\grave{\alpha}}+s\Omega_{\ep,b}^{\alpha})\Tsf_{\ep\theta}(z)\big)\big|\right)
|\Omega_{\ep,b}^{\alpha}-\Omega_{\ep,b}^{\grave{\alpha}}||\Tsf_{\ep\theta}(z)|.
\end{multline*}
We study the factors here separately and abbreviate
\[
\Delta
:= (1-s)\Omega_{\ep,b}^{\grave{\alpha}}+s\Omega_{\ep,b}^{\alpha}.
\]

First, by \eqref{eqn: psi-n-ep Z est1},
\[
\big|\partial_Z^{r+1}[\psi_{n,\ep}](\Delta\Tsf_{\ep\theta}(z))\big| 
\le C_{\per}\frac{(n+1)^{r+2}}{a_{\per,0}^n}
\exp\big((n+1)\Delta|\im(\Tsf_{\ep\theta}(z))|\big).
\]
Then
\begin{equation}\label{eqn: psi-n-ep Lip aux5}
e^{-(n+1)b/\ep}\big|\partial_Z^{r+1}[\psi_{n,\ep}](\Delta\Tsf_{\ep\theta}(z))\big| 
\le C_{\per}\frac{(n+1)^{r+2}}{a_{\per,0}^n}
\exp\left((n+1)\left(\Delta|\im(\Tsf_{\ep\theta}(z))|-\frac{b}{\ep}\right)\right).
\end{equation}

We use the estimates \eqref{eqn: Tsf est imag alpha} and \eqref{eqn: favorite expn est} to bound 
\[
\Delta|\im(\Tsf_{\ep\theta}(z))| - \frac{b}{\ep}
\le \left(\Delta-\frac{1}{\ep}\right)b+C_{\omega}\theta_0\tau_b
\le 1+C_{\omega}\theta_0\tau_b.
\]
Thus
\begin{equation}\label{eqn: useful exp Lip alpha est for future}
\exp\left((n+1)\left(\Delta|\im(\Tsf_{\ep\theta}(z))|-\frac{b}{\ep}\right)\right)
\le e^{(n+1)(C_{\omega}\theta_0\tau_b+1)},
\end{equation}
which will be useful later, too.
With $|\grave{\alpha}| \le a_{\per,4}$ from \eqref{eqn: a-per-4 ep-per-4}, this turns \eqref{eqn: psi-n-ep Lip aux5} into
\begin{multline*}
|\grave{\alpha}|^n
e^{-(n+1)b/\ep}
\big|\partial_Z^{r+1}[\psi_{n,\ep}](\Delta\Tsf_{\ep\theta}(z))\big|
\le C_{\per}e^{C_{\omega}\theta_0\tau_b+1}(n+1)^{r+2}\left(\frac{a_{\per,4}e^{C_{\omega}\theta_0\tau_b+1}}{a_{\per,0}}\right)^n \\
\le C_{\per}e^{C_{\omega}\theta_0\tau_b+1}\left(\frac{(n+1)^{r+2}}{2^n}\right).
\end{multline*}
All together,
\begin{equation}\label{eqn: psi-n-ep Lip aux6}
\sup_{0 \le s \le 1} |\grave{\alpha}|^ne^{-(n+1)b/\ep}\big|\partial_Z^{r+1}[\psi_{n,\ep}]\big(((1-s)\Omega_{\ep,b}^{\grave{\alpha}}+s\Omega_{\ep,b}^{\alpha})\Tsf_{\ep\theta}(z)\big)\big|
\le C_{\per}e^{C_{\omega}\theta_0\tau_b+1}\left(\frac{(n+1)^{r+2}}{2^n}\right).
\end{equation}

We combine \eqref{eqn: psi-n-ep Lip aux6} with the bounds
\[
|\Omega_{\ep,b}^{\alpha}-\Omega_{\ep,b}^{\grave{\alpha}}|
\le C_{\omega}
e^{-b/\ep}
|\alpha-\grave{\alpha}|
\]
from \eqref{eqn: Omega-ep-Lip alpha} and
\[
|\Tsf_{\ep\theta}(z)|
\le |z|+1
\]
from \eqref{eqn: Tsf est global} to obtain \eqref{eqn: psi-n-ep Lip2}.

\item
We use the mean value theorem to estimate
\begin{multline}\label{eqn: psi-n-ep Lip aux7}
|\alpha|^n
e^{-(n+1)b/\ep}
\big|\partial_Z^r[\psi_{n,\ep}](\Omega_{\ep,b}^{\alpha}\Tsf_{\ep\theta}(z))-\partial_Z^r[\psi_{n,\ep}](\Omega_{\ep,b}^{\alpha}\Tsf_{\ep\grave{\theta}}(z))\big| \\
\le \left(\sup_{0 \le s \le 1} |\alpha|^{n-1}e^{-(n+1)b/\ep}\big|\partial_Z^{r+1}[\psi_{n,\ep}]\big(\Omega_{\ep,b}^{\alpha}((1-s)\Tsf_{\ep\theta}(z)+s\Tsf_{\ep\grave{\theta}}(z))\big)\big|\right)
\Omega_{\ep,b}^{\alpha}
|\alpha|
|\Tsf_{\ep\theta}(z)-\Tsf_{\ep\grave{\theta}}(z)|.
\end{multline} 
Here we are using the trick of breaking off one factor of $\alpha$ to retain for the final estimate.
Abbreviate
\[
\Delta
:= (1-s)\Tsf_{\ep\theta}(z)+s\Tsf_{\ep\grave{\theta}}(z).
\]

We first use \eqref{eqn: psi-n-ep Z est1} to estimate
\begin{equation}\label{eqn: psi-n-ep Lip aux10}
|\alpha|^{n-1}
e^{-(n+1)b/\ep}
\big|\partial_Z^{r+1}[\psi_{n,\ep}](\Omega_{\ep,b}^{\alpha}\Delta)\big| 
\le C_{\per}a_{\per,4}^{n-1}\frac{(n+1)^{r+2}}{a_{\per,0}^n}\exp\left((n+1)\left(\Omega_{\ep,b}^{\alpha}|\im(\Delta)|-\frac{b}{\ep}\right)\right).
\end{equation}
Then we use \eqref{eqn: Tsf est imag alpha} and \eqref{eqn: favorite expn est} to estimate
\[
\Omega_{\ep,b}^{\alpha}|\im(\Delta)|-\frac{b}{\ep}
\le \left(\Omega_{\ep,b}^{\alpha}-\frac{1}{\ep}\right)b+C_{\omega}\theta_0\tau_b
\le 1+C_{\omega}\theta_0\tau_b,
\]
and so
\[
\exp\left((n+1)\left(\Omega_{\ep,b}^{\alpha}|\im(\Delta)|-\frac{b}{\ep}\right)\right)
\le e^{(n+1)(1+C_{\omega}\theta_0\tau_b)}.
\]
This turns \eqref{eqn: psi-n-ep Lip aux10} into
\begin{multline}\label{eqn: psi-n-ep Lip aux11}
|\alpha|^{n-1}
e^{-(n+1)b/\ep}
\big|\partial_Z^{r+1}[\psi_{n,\ep}](\Omega_{\ep,b}^{\alpha}\Delta)\big| 
\le \frac{C_{\per}e^{2(C_{\omega}\theta_0\tau_b+1)}}{a_{\per,0}}(n+1)^{r+2} \left(\frac{e^{C_{\omega}\theta_0\tau_b+1}a_{\per,4}}{a_{\per,0}}\right)^{n-1} \\
\le \frac{C_{\per}e^{2(C_{\omega}\theta_0\tau_b+1)}}{a_{\per,0}}\left(\frac{(n+1)^{r+2}}{2^{n-1}}\right)
\end{multline}
by definition of $a_{\per,4}$ in \eqref{eqn: a-per-4 ep-per-4}.

Next, \eqref{eqn: Omega-ep-alpha map1} gives
\begin{equation}\label{eqn: psi-n-ep Lip aux12}
\Omega_{\ep,b}^{\alpha}|\Tsf_{\ep\theta}(z)-\Tsf_{\ep\grave{\theta}}(z)|
\le C_{\omega}\ep^{-1}\big|(z+\ep\theta\tau(z))-(z+\ep\grave{\theta}\tau(z))\big|
\le C_{\omega}\theta_0\tau_b|\theta-\grave{\theta}|.
\end{equation}
Combining \eqref{eqn: psi-n-ep Lip aux7}, \eqref{eqn: psi-n-ep Lip aux11}, and \eqref{eqn: psi-n-ep Lip aux12} yields \eqref{eqn: psi-n-ep Lip3}.
\qedhere
\end{enumerate}
\end{proof}

\subsection{Construction of the periodic extension}

Now we are ready to construct the extended remainder term for the periodics.

\begin{theorem}\label{thm: ultimate psi theorem}
Let $0 < b < b_{\sigma}$ and $\theta_0 > 0$.
For $0 < \ep < \ep_{\per,4}$, $|\theta| < \theta_0$, and $|\alpha| < a_{\per,4}$, the function
\begin{equation}\label{eqn: psi-ep-b-alphalpha-theta}
\psi_{\ep,b}^{\alpha,\theta}(z) 
:= \sum_{n=1}^{\infty} \alpha^ne^{-(n+1)b/\ep}\psi_{n,\ep}(\Omega_{\ep,b}^{\alpha}\Tsf_{\ep\theta}(z))
\end{equation}
is defined and analytic on $\U_b$.
For $r \ge 0$, there is $C_r > 0$ such that the following hold.

\begin{enumerate}[label={\bf(\roman*)}]

\item
{}
[Mapping estimate]
If $0 < \ep < \ep_{\per,4}$, $r \ge 0$, $|\alpha| < a_{\per,4}$, $|\theta| < \theta_0$, and $z \in \U_b$, then
\begin{equation}\label{eqn: psi map ultimate}
|\partial_z^r[\psi_{\ep,b}^{\alpha,0}](\Tsf_{\ep\theta}(z))|
\le C_r\ep^{-r}|\alpha|.
\end{equation}

\item
{}
[Lipschitz estimate in amplitude]
If $0 < \ep < \ep_{\per,4}$, $r \ge 0$, $0 \le |\alpha|$, $|\grave{\alpha}| < a_{\per,4}$, $|\theta| < \theta_0$, and $z \in \U_b$, then
\begin{equation}\label{eqn: psi Lip alpha ultimate}
\big|\partial_z^r[\psi_{\ep,b}^{\alpha,0}](\Tsf_{\ep\theta}(z))-\partial_z^r[\psi_{\ep,b}^{\grave{\alpha},0}](\Tsf_{\ep\theta}(z))\big|
\le C_r\ep^{-r}(|z|+1)|\alpha-\grave{\alpha}|.
\end{equation}

\item
{}
[Lipschitz estimate in phase shift]
If $0 < \ep < \ep_{\per,4}$, $r \ge 0$, $|\alpha| < a_{\per,4}$, $0 \le |\theta|$, $|\grave{\theta}| < \theta_0$, and $z \in \U_b$, then
\begin{equation}\label{eqn: psi Lip theta ultimate}
\big|\partial_z^r[\psi_{\ep,b}^{\alpha,0}](\Tsf_{\ep\theta}(z))-\partial_z^r[\psi_{\ep,b}^{\alpha,0}](\Tsf_{\ep\grave{\theta}}(z))\big|
\le C_r\ep^{-r}|\alpha||\theta-\grave{\theta}|.
\end{equation}
\end{enumerate}
\end{theorem}

\begin{proof}
First we show that $\psi_{\ep,b}^{\alpha,\theta}$ is indeed defined and analytic.
Take $\grave{\alpha} = 0$ in \eqref{eqn: psi-n-ep Lip1} to find
\[
|\alpha^ne^{-(n+1)b/\ep}\partial_Z^r[\psi_{n,\ep}](\Omega_{\ep,b}^{\alpha}\Tsf_{\ep\theta}(z))|
\le C_{\per}\left(\frac{e^{2C_{\omega}\theta_0\tau_b+1}}{a_{\per,0}}\right)\left(\frac{n(n+1)^{r+1}}{2^{n-1}}\right)|\alpha|.
\]
The series
\begin{equation}\label{eqn: the r-dependent series}
\sum_{n=1}^{\infty} \frac{n(n+1)^{r+1}}{2^{n-1}}
\end{equation}
converges by the ratio test, and so the series
\[
\sum_{n=1}^{\infty} \alpha^ne^{-(n+1)b/\ep}\partial_Z^r[\psi_{n,\ep}](\Omega_{\ep,b}^{\alpha}\Tsf_{\ep\theta}(z))
\]
converges on $\U_b$ to an analytic function.
In particular, taking $r=0$, we see that the function $\psi_{\ep,b}^{\alpha,\theta}$ is defined and analytic on $\U_b$.

Now we prove the mapping and Lipschitz estimates.

\begin{enumerate}[label={\bf(\roman*)}]

\item
Taking $\theta=0$ and differentiating \eqref{eqn: psi-ep-b-alphalpha-theta} term-by-term, we have
\[
\partial_z^r[\psi_{\ep,b}^{\alpha,0}](Z)
= (\Omega_{\ep,b}^{\alpha})^r\sum_{n=1}^{\infty} \alpha^n{e}^{-(n+1)b/\ep}\partial_Z[\psi_{n,\ep}](\Omega_{\ep,b}^{\alpha}Z).
\]
Then
\begin{equation}\label{eqn: psi series aux1}
\partial_z^r[\psi_{\ep,b}^{\alpha,0}](\Tsf_{\ep\theta}(z))
= (\Omega_{\ep,b}^{\alpha})^r\sum_{n=1}^{\infty} \alpha^n{e}^{-(n+1)b/\ep}\partial_Z[\psi_{n,\ep}](\Omega_{\ep,b}^{\alpha}\Tsf_{\ep\theta}(z)).
\end{equation}
Estimate $(\Omega_{\ep,b}^{\alpha})^r$ with \eqref{eqn: Omega-ep-alpha map1}, use \eqref{eqn: psi-n-ep Lip1} with $\grave{\alpha} = 0$ to estimate the $n$th term of the series \eqref{eqn: psi series aux1}, and invoke the convergence of the $r$-dependent series \eqref{eqn: the r-dependent series} to obtain \eqref{eqn: psi map ultimate}.

\item
We rewrite 
\[
\partial_z^r[\psi_{\ep,b}^{\alpha,0}](\Tsf_{\ep\theta}(z))-\partial_z^r[\psi_{\ep,b}^{\grave{\alpha},0}](\Tsf_{\ep\theta}(z))
= \Delta_1 + \Delta_2 + \Delta_3,
\]
where
\[
\Delta_1
:= ((\Omega_{\ep,b}^{\alpha})^r-(\Omega_{\ep,b}^{\grave{\alpha}})^r)
\sum_{n=1}^{\infty} \alpha^ne^{-(n+1)b/\ep}\partial_Z^r[\psi_{n,\ep}](\Omega_{\ep,b}^{\alpha}\Tsf_{\ep\theta}(z)),
\]
\[
\Delta_2
:= (\Omega_{\ep,b}^{\grave{\alpha}})^r
\sum_{n=1}^{\infty} (\alpha^n-\grave{\alpha}^n)e^{-(n+1)b/\ep}\partial_Z^r[\psi_{n,\ep}](\Omega_{\ep,b}^{\alpha}\Tsf_{\ep\theta}(z)),
\]
and
\[
\Delta_3
:= (\Omega_{\ep,b}^{\grave{\alpha}})^r
\sum_{n=1}^{\infty} \grave{\alpha}^ne^{-(n+1)b/\ep}\big(\partial_Z^r[\psi_{n,\ep}](\Omega_{\ep,b}^{\alpha}\Tsf_{\ep\theta}(z))-\partial_Z^r[\psi_{n,\ep}](\Omega_{\ep,b}^{\grave{\alpha}}\Tsf_{\ep\theta}(z))\big).
\]
We estimate each of these terms separately.

First, we recognize
\[
\Delta_1
= ((\Omega_{\ep,b}^{\alpha})^r-(\Omega_{\ep,b}^{\grave{\alpha}})^r)
(\Omega_{\ep,b}^{\alpha})^{-r}
\partial_z[\psi_{\ep,b}^{\alpha,0}](\Tsf_{\ep\theta}(z)).
\]
We estimate the difference $(\Omega_{\ep,b}^{\alpha})^r-(\Omega_{\ep,b}^{\grave{\alpha}})^r$ by \eqref{eqn: Omega-ep-Lip alpha}, the factor $(\Omega_{\ep,b}^{\alpha})^{-r}$ by \eqref{eqn: Omega-ep-alpha map1}, and the factor $\partial_z[\psi_{\ep,b}^{\alpha,0}](\Tsf_{\ep\theta}(z))$ by \eqref{eqn: psi map ultimate} to bound
\[
\Delta_1
\le (C_{\omega}r\ep^{1-r}e^{-b/\ep}|\alpha-\grave{\alpha}|)(C_{\omega}\ep^r)(C_r\ep^{-r})|\alpha|
= C_r\ep^{1-r}e^{-b/\ep}|\alpha||\alpha-\grave{\alpha}|.
\]
This estimate for $\Delta_1$ with its small factors of $\ep$, $e^{-b/\ep}$, and $\alpha$ may look more promising than it is, as we will not obtain such additional small factors in the estimates on $\Delta_2$.

For $\Delta_2$, we use the estimate \eqref{eqn: Omega-ep-alpha map1} on $(\Omega_{\ep,b}^{\grave{\alpha}})^r$, the estimate \eqref{eqn: psi-n-ep Lip1} on the $n$th term of the series in $\Delta_2$, and the convergence of the $r$-dependent series \eqref{eqn: the r-dependent series} to bound
\[
|\Delta_2|
\le C_r\ep^{-r}|\alpha-\grave{\alpha}|.
\]
As predicted, there are no small factors of $\ep$, $|\alpha|$, or $e^{-b/\ep}$ here, unlike for $\Delta_1$.

For $\Delta_3$, we use the estimate \eqref{eqn: Omega-ep-alpha map1} on $(\Omega_{\ep,b}^{\grave{\alpha}})^r$, the estimate \eqref{eqn: psi-n-ep Lip2} on the $n$th term of the series in $\Delta_3$, and the convergence of the $r$-dependent series
\begin{equation}\label{eqn: the other series}
\sum_{n=1}^{\infty} \frac{(n+1)^{r+2}}{2^n}
\end{equation}
to bound
\[
|\Delta_3|
\le C_r\ep^{-r}(|z|+1)e^{-b/\ep}|\alpha-\grave{\alpha}|.
\]

Combining these estimates on $\Delta_1$, $\Delta_2$, and $\Delta_3$ provides the desired estimate \eqref{eqn: psi Lip alpha ultimate}.
We emphasize that the term $\Delta_2$ prevents the factor of $e^{-b/\ep}$ for $\Delta_1$ and $\Delta_3$ from appearing in the final estimate, while the term $\Delta_3$ introduces the linearly growing factor $|z|+1$.

\item
We use \eqref{eqn: psi series aux1} to rewrite
\begin{multline}\label{eqn: psi series aux2}
\partial_z^r[\psi_{\ep,b}^{\alpha,0}](\Tsf_{\ep\theta}(z))-\partial_z^r[\psi_{\ep,b}^{\alpha,0}](\Tsf_{\ep\grave{\theta}}(z)) \\
= (\Omega_{\ep,b}^{\alpha})^r
\sum_{n=1}^{\infty} \alpha^ne^{-(n+1)b/\ep}\big(\partial_Z[\psi_{n,\ep}](\Omega_{\ep,b}^{\alpha}\Tsf_{\ep\theta}(z))-\partial_Z[\psi_{n,\ep}](\Omega_{\ep,b}^{\alpha}\Tsf_{\ep\grave{\theta}}(z))\big).
\end{multline}
The desired estimate \eqref{eqn: psi Lip theta ultimate} follows from the estimate \eqref{eqn: Omega-ep-alpha map1} on $(\Omega_{\ep,b}^{\grave{\alpha}})^r$, the estimate \eqref{eqn: psi-n-ep Lip3} on the $n$th term of the series in \eqref{eqn: psi series aux2}, and the convergence of the $r$-dependent series \eqref{eqn: the other series}.
\qedhere
\end{enumerate}
\end{proof}

Theorem \ref{thm: ultimate varphi theorem} is now a direct consequence of Theorem \ref{thm: ultimate psi theorem}; in particular, take $\ep_{\per} = \ep_{\per,4}$ and $a_{\per} = a_{\per,4}$.

\subsection{Related estimates}
We conclude with some related estimates that are needed in the full nanopteron problem---specifically, to control the terms in \eqref{eqn: J-ep-neg1-alpha-theta} and \eqref{eqn: J-ep-0-alpha-theta}.

\begin{lemma}\label{lem: ultimate fake quadratic}
Let $0 < b < b_{\sigma}$ and $\theta_0 > 0$.
For $0 < \ep < \ep_{\per}$ and $|\alpha| < a_{\per}$, let
\begin{equation}\label{eqn: delta}
\delta_{\ep,b}^{\alpha}
:= \varphi_{\ep,b}^{\alpha,0}-\varphi_{\ep,b}^{0,0}
\end{equation}
with $\varphi_{\ep,b}^{\alpha,\theta}$ defined in \eqref{eqn: varphi ultimate}.
For each $r \ge 0$, there is $C_r > 0$ such that the following hold.

\begin{enumerate}[label={\bf(\roman*)}]

\item
If $0 < \ep < \ep_{\per}$, $|\alpha| < a_{\per}$, $|\theta| < \theta_0$, and $z \in \U_b$, then 
\begin{equation}\label{eqn: fake quadratic map}
|\partial_z^r[\delta_{\ep,b}^{\alpha}](\Tsf_{\ep\theta}(z))|
\le C_r\ep^{-r}(|z|+1)|\alpha|.
\end{equation}

\item
If $0 < \ep < \ep_{\per}$, $0 \le |\alpha|$, $|\grave{\alpha}| < a_{\per}$, $|\theta| < \theta_0$, and $z \in \U_b$, then 
\begin{equation}\label{eqn: fake quadratic Lip alpha}
\big|\partial_z^r[\delta_{\ep,b}^{\alpha}](\Tsf_{\ep\theta}(z))-\partial_z^r[\delta_{\ep}^{\grave{\alpha},0}](\Tsf_{\ep\theta}(z))\big|
\le C_r\ep^{-r}(|z|+1)|\alpha-\grave{\alpha}|.
\end{equation}

\item
If $0 < \ep < \ep_{\per}$, $|\alpha| < a_{\per}$, $0 \le |\theta|$, $|\grave{\theta}| < \theta_0$, and $z \in \U_b$, then 
\begin{equation}\label{eqn: fake quadratic Lip theta}
\big|\partial_z^r[\delta_{\ep,b}^{\alpha}](\Tsf_{\ep\theta}(z))-\partial_z^r[\delta_{\ep,b}^{\alpha}](\Tsf_{\ep\grave{\theta}}(z))\big|
\le C_r\ep^{-r}|\alpha|(|z|+1)|\theta-\grave{\theta}|.
\end{equation}
\end{enumerate}
\end{lemma}

\begin{proof}
By definition of $\varphi_{\ep,b}^{\alpha,\theta}$ in \eqref{eqn: varphi ultimate}, we have
\[
\delta_{\ep,b}^{\alpha}(z)
= \varphi_{\ep,b}^{\alpha,0}(z) - \varphi_{\ep,b}^{0,0}(z)
= e^{-b/\ep}
\cos(\Omega_{\ep,b}^{\alpha}z)
-e^{-b/\ep}
\cos(\omega_{\ep}z)
+ \psi_{\ep,b}^{\alpha,0}(z)
\]
and so
\[
\partial_z^r[\delta_{\ep,b}^{\alpha}](\Tsf_{\ep\theta}(z))
= e^{-b/\ep}
(\Omega_{\ep,b}^{\alpha})^r
\partial_Z^r[\cos(\cdot)](\Omega_{\ep,b}^{\alpha}\Tsf_{\ep\theta}(z))
-e^{-b/\ep}
\omega_{\ep}^r
\partial_Z^r[\cos(\cdot)]\big(\omega_{\ep}\Tsf_{\ep\theta}(z))
+ \partial_z^r[\psi_{\ep,b}^{\alpha,0}](\Tsf_{\ep\theta}(z)\big).
\]
We regularly use these expansions below.

\begin{enumerate}[label={\bf(\roman*)}]

\item
Since $\delta_{\ep,b}^0 = 0$, this estimate is a consequence of \eqref{eqn: fake quadratic Lip alpha} with $\grave{\alpha} = 0$.

\item
We first rewrite
\[
\partial_z^r[\delta_{\ep,b}^{\alpha}](\Tsf_{\ep\theta}(z))
-\partial_z^r[\delta_{\ep}^{\grave{\alpha}}](\Tsf_{\ep\theta}(z))
= \Delta_1 + \Delta_2,
\]
where
\[
\Delta_1
:= e^{-b/\ep}
\big((\Omega_{\ep,b}^{\alpha})^r
\partial_Z^r[\cos(\cdot)](\Omega_{\ep,b}^{\alpha}\Tsf_{\ep\theta}(z))
-(\Omega_{\ep,b}^{\grave{\alpha}})^r
\partial_Z^r[\cos(\cdot)](\Omega_{\ep,b}^{\grave{\alpha}}\Tsf_{\ep\theta}(z))\big),
\]
and
\[
\Delta_2
:= e^{-b/\ep}
\partial_Z^r[\psi_{\ep,b}^{\alpha,0}](\Tsf_{\ep\theta}(z))
-\partial_Z^r[\psi_{\ep,b}^{\grave{\alpha},0}](\Tsf_{\ep\theta}(z)).
\]
We then immediately obtain from \eqref{eqn: psi Lip alpha ultimate} that 
\begin{equation}\label{eqn: fake quadratic Lip alpha aux1}
|\Delta_2|
\le C_r
\ep^{-r}
(|z|+1)
|\alpha-\grave{\alpha}|.
\end{equation}

For $\Delta_1$, we further rewrite 
\[
\Delta_1
= \Delta_{11}+\Delta_{12}, 
\]
where
\[
\Delta_{11}(z)
:= e^{-b/\ep}
\big((\Omega_{\ep,b}^{\alpha})^r-(\Omega_{\ep,b}^{\grave{\alpha}})^r)
\partial_Z^r[\cos(\cdot)](\Omega_{\ep,b}^{\alpha}\Tsf_{\ep\theta}(z)\big)
\]
and
\[
\Delta_{12}(z)
:= e^{-b/\ep}
(\Omega_{\ep,b}^{\grave{\alpha}})^r
\big(\partial_Z^r[\cos(\cdot)](\Omega_{\ep,b}^{\alpha}\Tsf_{\ep\theta}(z))-\partial_Z^r[\cos(\cdot)](\Omega_{\ep,b}^{\grave{\alpha}}\Tsf_{\ep\theta}(z))\big)
\]
For $\Delta_{11}$, we use \eqref{eqn: Omega-ep-Lip alpha} to estimate
\[
|\Delta_{11}(z)|
\le C_{\omega}
r
\ep^{1-r}
|\alpha-\grave{\alpha}|
e^{-b/\ep}
|\sup_{z \in \U_b} e^{-b/\ep}|\partial_Z^r[\cos(\cdot)](\Omega_{\ep,b}^{\alpha}\Tsf_{\ep\theta}(z))|.
\]
We bound
\[
|\partial_Z^r[\cos(\cdot)](\Omega_{\ep,b}^{\alpha}\Tsf_{\ep\theta}(z))|
\le e^{\Omega_{\ep,b}^{\alpha}|\im(\Tsf_{\ep\theta}(z))|}
\le e^{\Omega_{\ep,b}^{\alpha}b+C_{\omega}\theta_0\tau_b}
\]
by \eqref{eqn: Tsf est imag alpha}, so
\[
e^{-b/\ep}|\partial_Z^r[\cos(\cdot)](\Omega_{\ep,b}^{\alpha}\Tsf_{\ep\theta}(z))|
\le e^{(\Omega_{\ep,b}^{\alpha}-1/\ep)b+C_{\omega}\theta_0\tau_b}
\le e^{C_{\omega}\theta_0\tau_b+1}
\]
by \eqref{eqn: favorite expn est}.
Then
\begin{equation}\label{eqn: fake quadratic Lip alpha aux2}
|\Delta_{11}(z)|
\le C_{\omega}
e^{C_{\omega}\theta_0\tau_b+1}
r
\ep^{1-r}
e^{-b/\ep}
|\alpha-\grave{\alpha}|.
\end{equation}

For $\Delta_{12}$, the mean value theorem gives
\[
|\Delta_{12}(z)|
\le C_{\omega}
\ep^{-r}
|\Omega_{\ep,b}^{\alpha}-\Omega_{\ep,b}^{\grave{\alpha}}|
|\Tsf_{\ep\theta}(z)|
\sup_{0 \le s \le 1} e^{-b/\ep}\big|\partial_Z^{r+1}[\cos(\cdot)]\big(((1-s)\Omega_{\ep,b}^{\alpha}+s\Omega_{\ep,b}^{\grave{\alpha}})\Tsf_{\ep\theta}(z)\big)\big|.
\]
We have
\begin{multline*}
e^{-b/\ep}\big|\partial_Z^{r+1}[\cos(\cdot)]\big(((1-s)\Omega_{\ep,b}^{\alpha}+s\Omega_{\ep,b}^{\grave{\alpha}})\Tsf_{\ep\theta}(z)\big)\big|
\le \exp\left(((1-s)\Omega_{\ep,b}^{\grave{\alpha}}+s\Omega_{\ep,b}^{\alpha})|\im(\Tsf_{\ep\theta}(z))|-\frac{b}{\ep}\right) \\
\le e^{C_{\omega}\theta_0\tau_b+1}
\end{multline*}
by \eqref{eqn: useful exp Lip alpha est for future} with $n=0$, and then we use \eqref{eqn: Tsf est global} and \eqref{eqn: Omega-ep-Lip alpha} to obtain
\begin{equation}\label{eqn: fake quadratic Lip alpha aux3}
|\Delta_{12}(z)|
\le C_{\omega}^2e^{C_{\omega}\theta_0\tau_b+1}r\ep^{1-r}e^{-b/\ep}(1+|z|)|\alpha-\grave{\alpha}|.
\end{equation}

We combine \eqref{eqn: fake quadratic Lip alpha aux2} and \eqref{eqn: fake quadratic Lip alpha aux3} to conclude
\[
|\Delta_1(z)|
\le C_r\ep^{1-r}e^{-b/\ep}(|z|+1)|\alpha-\grave{\alpha}|,
\]
and this together with \eqref{eqn: fake quadratic Lip alpha aux1} gives \eqref{eqn: fake quadratic Lip alpha}.
We remark that the estimate \eqref{eqn: fake quadratic Lip alpha aux1} on $\Delta_2$ prevents us from retaining these more recently introduced factors of $e^{-b/\ep}$ in our final estimate.

\item
We first use the definition of $\delta_{\ep,b}^{\alpha}$ in \eqref{eqn: delta} and $\varphi_{\ep,b}^{\alpha,\theta}$ in \eqref{eqn: varphi ultimate} to rewrite
\[
\partial_z^r[\delta_{\ep,b}^{\alpha}](\Tsf_{\ep\theta}(z))-\partial_z^r[\delta_{\ep,b}^{\alpha}](\Tsf_{\ep\grave{\theta}}(z))
= \Delta_{11} + \Delta_{12} + \Delta_2
\]
where
\[
\Delta_{11}(z)
:= e^{-b/\ep}(\Omega_{\ep,b}^{\alpha})^r\big(\partial_Z^r[\cos(\cdot)](\Omega_{\ep,b}^{\alpha}\Tsf_{\ep\theta}(z)) - \partial_Z^r[\cos(\cdot)](\Omega_{\ep,b}^{\alpha}\Tsf_{\ep\grave{\theta}}(z))\big),
\]
\[
\Delta_{12}(z)
:= -e^{-b/\ep}\omega_{\ep}^r\big(\partial_Z^r[\cos(\cdot)](\omega_{\ep}\Tsf_{\ep\theta}(z))-\partial_Z^r[\cos(\cdot)](\omega_{\ep}\Tsf_{\ep\grave{\theta}}(z))\big),
\]
and
\[
\Delta_2(z)
:= \partial_z^r[\psi_{\ep,b}^{\alpha,0}](\Tsf_{\ep\theta}(z))-\partial_z^r[\psi_{\ep,b}^{\alpha,0}](\Tsf_{\ep\grave{\theta}}(z)).
\]
We then immediately obtain from \eqref{eqn: psi Lip theta ultimate} that 
\begin{equation}\label{eqn: fake quadratic Lip theta aux1}
|\Delta_2(z)|
\le C_r\ep^{-r}|\alpha||\theta-\grave{\theta}|.
\end{equation}

Next, we rewrite
\[
\Delta_{11}(z)
= (\Omega_{\ep,b}^{\alpha})^{r+1}
\ep
(\theta-\grave{\theta})
\tau(z)
\int_0^1 e^{-b/\ep}\partial_Z^{r+1}[\cos(\cdot)]\big(\Omega_{\ep,b}^{\alpha}((1-s)\Tsf_{\ep\grave{\theta}}(z)+s\Tsf_{\ep\theta}(z))\big) \ds
\]
and
\[
\Delta_{12}(z)
= -\omega_{\ep}^{r+1}
\ep
(\theta-\grave{\theta})
\tau(z)
\int_0^1 e^{-b/\ep}\partial_Z^{r+1}[\cos(\cdot)]\big(\omega_{\ep}((1-s)\Tsf_{\ep\grave{\theta}}(z)+s\Tsf_{\ep\theta}(z))\big) \ds
\]
so that 
\[
\Delta_{11}(z)+\Delta_{12}(z)
= \Delta_{13}(z)+\Delta_{14}(z)
\]
where
\[
\Delta_{13}(z)
:= ((\Omega_{\ep,b}^{\alpha})^{r+1}-\omega_{\ep}^{r+1})
\ep
(\theta-\grave{\theta})
\tau(z)
\int_0^1 e^{-b/\ep}\partial_Z^{r+1}[\cos(\cdot)]\big(\Omega_{\ep,b}^{\alpha}((1-s)\Tsf_{\ep\grave{\theta}}(z)+s\Tsf_{\ep\theta}(z))\big) \ds
\]
and
\[
\Delta_{14}(z)
:= \omega_{\ep}^{r+1}
\ep
(\theta-\grave{\theta})
\tau(z)
\int_0^1 e^{-b/\ep}\big(\partial_Z^{r+1}[\cos(\cdot)](\grave{Z}_s)-\partial_Z^{r+1}[\cos(\cdot)](Z_s)\big) \ds,
\]
with
\[
Z_s := \Omega_{\ep,b}^{\alpha}((1-s)\Tsf_{\ep\grave{\theta}}(z)+s\Tsf_{\ep\theta}(z))
\quadword{and}
\grave{Z}_s := \omega_{\ep}((1-s)\Tsf_{\ep\grave{\theta}}(z)+s\Tsf_{\ep\theta}(z)).
\]

We use \eqref{eqn: Tsf est imag alpha} and \eqref{eqn: favorite expn est} to estimate the integrand in $\Delta_{13}$ by
\begin{multline*}
|e^{-b/\ep}\partial_Z^{r+1}[\cos(\cdot)](\Omega_{\ep,b}^{\alpha}((1-s)\Tsf_{\ep\grave{\theta}}(z)+s\Tsf_{\ep\theta}(z)))|
\le \exp\left(\Omega_{\ep,b}^{\alpha}\big|\im\big(((1-s)\Tsf_{\ep\grave{\theta}}(z)+s\Tsf_{\ep\theta}(z))\big)\big|-\frac{b}{\ep}\right) \\
\le \exp\left(b\left(\Omega_{\ep,b}^{\alpha}-\frac{1}{\ep}\right)\right)e^{C_{\omega}\theta_0\tau_b}
\le e^{C_{\omega}\theta_0\tau_b+1}.
\end{multline*}
Then 
\begin{equation}\label{eqn: fake quadratic Lip theta aux2}
|\Delta_{13}(z)|
\le \big(C_{\omega}(r+1)\ep^{1-(r+1)}e^{-b/\ep}|\alpha|\big)\ep|\theta-\grave{\theta}|\tau_be^{C_{\omega}\theta_0\tau_b+1}
\end{equation}
by \eqref{eqn: Omega-ep-Lip alpha}.

For $\Delta_{14}$, we estimate the integrand by
\begin{multline*}
|e^{-b/\ep}(\partial_Z^{r+1}[\cos(\cdot)](\grave{Z}_s)-\partial_Z^{r+1}[\cos(\cdot)](Z_s))|
\le \left(\sup_{0 \le t \le 1} e^{-b/\ep}\big|\partial_Z^{r+2}[\cos(\cdot)]\big((1-t)Z_s+t\grave{Z}_s\big)\big|\right)|Z_s-\grave{Z}_s| \\
\le \left(\sup_{0 \le t \le 1} \exp\left(\big|\im((1-t)Z_s+t\grave{Z}_s)\big|-\frac{b}{\ep}\right)\right)|Z_s-\grave{Z}_s|
\le |Z_s-\grave{Z}_s|,
\end{multline*}
where by \eqref{eqn: Tsf est imag alpha} and \eqref{eqn: favorite expn est} we have bounded the supremum by $1$.
Then we estimate
\[
|Z_s-\grave{Z}_s|
= |\Omega_{\ep,b}^{\alpha}-\omega_{\ep}||(1-s)\Tsf_{\ep\grave{\theta}}(z)+s\Tsf_{\ep\theta}(z)|
\le C_{\omega}e^{-b/\ep}|\alpha|\tau_b
\]
to conclude
\begin{equation}\label{eqn: fake quadratic Lip theta aux3}
|\Delta_{14}(z)|
\le C_{\omega}^2\ep^{-r}|\theta-\grave{\theta}|e^{-b/\ep}|\alpha|\tau_b^2.
\end{equation}

We combine \eqref{eqn: fake quadratic Lip theta aux2} and \eqref{eqn: fake quadratic Lip theta aux3} to conclude
\[
|\Delta_{11}(z)+\Delta_{12}(z)| 
= |\Delta_{13}(z)+\Delta_{14}(z)|
\le C_r\ep^{-r}e^{-b/\ep}|\alpha||\theta-\grave{\theta}|,
\]
and this together with \eqref{eqn: fake quadratic Lip theta aux1} gives \eqref{eqn: fake quadratic Lip theta}.
We remark that, once again, the estimate \eqref{eqn: fake quadratic Lip theta aux1} on $\Delta_2$ prevents us from retaining these more recently introduced factors of $e^{-b/\ep}$ in our final estimate.
\qedhere
\end{enumerate}
\end{proof}

\section{Estimates on Terms for the Contraction Arguments}\label{app: contr ests}

\subsection{Estimates on $\G_{\ep,b}^{\theta}$}\label{app: G}
We develop estimates for the term $\G_{\ep,b}^{\theta}$ defined in \eqref{eqn: G-ep-theta}.
This term is a linear combination of terms of the form
\begin{equation}\label{eqn: G-ep-theta terms}
\ep^r(\partial_z^r[\varphi_{\ep,b}^{0,0}]\circ\Tsf_{\ep\theta})\theta^s\varsigma
\end{equation}
where $r$, $s \ge 0$, and $\varsigma \in \H_{q,b}$, and so we develop the following estimates.

\begin{lemma}\label{lem: ultimate G prep}
Let $\theta_0$, $q > 0$ and $0 < b < b_{\sigma}$.
Let $\varsigma \in \H_{q,b}$ and $s \ge 0$.
For each $r \ge 0$, there is $C_r > 0$ such that the following hold.

\begin{enumerate}[label={\bf(\roman*)}]

\item
{}
[Mapping estimate]
If $0 < \ep < \ep_{\per}$ and $|\theta| < \theta_0$, then
\begin{equation}\label{eqn: G map prep}
\norm{(\partial_z^r[\varphi_{\ep,b}^{0,0}]\circ\Tsf_{\ep\theta})\theta^s\varsigma}_{q,b}
\le C_r\ep^{-r}.
\end{equation}

\item
{}
[Lipschitz estimate in phase shift]
If $0 < \ep < \ep_{\per}$ and $0 \le |\theta|$, $|\grave{\theta}| < \theta_0$, then
\begin{equation}\label{eqn: G Lip theta prep}
\norm{(\partial_z^r[\varphi_{\ep,b}^{0,0}]\circ\Tsf_{\ep\theta})\theta^s\varsigma-(\partial_z^r[\varphi_{\ep,b}^{0,0}]\circ\Tsf_{\ep\grave{\theta}})\grave{\theta}^s\varsigma}_{q,b}
\le C_r\ep^{-r}|\theta-\grave{\theta}|.
\end{equation}
\end{enumerate}
\end{lemma}

\begin{proof}

\begin{enumerate}[label={\bf(\roman*)}]

\item
This follows from the product estimate \eqref{eqn: product est} and the mapping estimate \eqref{eqn: varphi map ultimate}.

\item
Rewrite
\[
(\partial_z^r[\varphi_{\ep,b}^{0,0}]\circ\Tsf_{\ep\theta})\theta^s\varsigma-(\partial_z^r[\varphi_{\ep,b}^{0,0}]\circ\Tsf_{\ep\grave{\theta}})\grave{\theta}^s\varsigma
= \Delta_1+\Delta_2
\]
where
\[
\Delta_1
:= (\partial_z^r[\varphi_{\ep,b}^{0,0}]\circ\Tsf_{\ep\theta})\theta^s\varsigma-(\partial_z^r[\varphi_{\ep,b}^{0,0}]\circ\Tsf_{\ep\grave{\theta}})\theta^s\varsigma
\]
and
\[
\Delta_2
:=
(\partial_z^r[\varphi_{\ep,b}^{0,0}]\circ\Tsf_{\ep\grave{\theta}})(\theta^s-\grave{\theta}^s)\varsigma.
\]
We estimate
\[
\norm{\Delta_1}_{q,b}
\le C_r\ep^{-r}|\alpha||\theta-\grave{\theta}|\theta_0^s|\norm{\varsigma}_{q,b}
\]
using the product estimate \eqref{eqn: product est} and the Lipschitz estimate \eqref{eqn: varphi Lip theta ultimate} and
\[
\norm{\Delta_2}_{q,b}
\le C_r\ep^{-r}|\theta-\grave{\theta}|s\theta_0^{s-1}\norm{\varsigma}_{q,b}
\]
by the product estimate and the mapping estimate \eqref{eqn: varphi map ultimate}.
\qedhere
\end{enumerate}
\end{proof}

Since $\G_{\ep,b}^{\theta}$ is a linear combination of terms of the form \eqref{eqn: G-ep-theta terms}, this lemma gives the following estimates.

\begin{lemma}\label{lem: ultimate G}
Let $0 < b < b_{\sigma}$, $0 < q < q_{\sigma}$, and $\theta_0 > 0$.
There exists $C > 0$ such that the following hold.

\begin{enumerate}[label={\bf(\roman*)}]

\item
{}
[Mapping estimate]
If $0 < \ep < \ep_{\per}$ and $0 \le \theta < \theta_0$, then
\begin{equation}\label{eqn: ultimate G map}
\norm{\G_{\ep,b}^{\theta}}_{q,b}
\le C.
\end{equation}

\item
{}
[Lipschitz estimate in phase shift]
If $0 < \ep < \ep_{\per}$ and $0 \le \theta$, $\grave{\theta} < \theta_0$, then
\begin{equation}\label{eqn: ultimate G Lip theta}
\norm{\G_{\ep,b}^{\theta}-\G_{\ep,b}^{\grave{\theta}}}_{q,b}
\le C|\theta-\grave{\theta}|.
\end{equation}

\end{enumerate}
\end{lemma}

The Lipschitz estimate on $\G_{\ep,b}^{\theta}$ and the identity $\G_{\ep,b}^0 = 0$ give another useful estimate when $\theta$ is nonnegative.

\begin{lemma}
Let $0 < b < b_{\sigma}$, $0 < q < q_{\sigma}$, and $\theta_0 > 0$.
There exists $C > 0$ such that if $0 < \ep < \ep_{\per}$ and $0 \le \theta < \theta_0$, then
\begin{equation}\label{eqn: useful G Lip}
(\tau_{\infty}\omega_{\ep}\theta+1)^{-1}\norm{\G_{\ep,b}^{\theta}}_{q,b}
\le C\ep.
\end{equation}
\end{lemma}

\begin{proof}
If $\theta = 0$, then $\G_{\ep,b}^0 = 0$, and so the inequality follows at once.
For $\theta > 0$, we have $\tau_{\infty}\omega_{\ep}\theta+1 > \omega_{\ep}\theta > 0$, and so
\[
\frac{1}{\tau_{\infty}\omega_{\ep}\theta+1}\norm{\G_{\ep,b}^{\theta}}_{q,b}
\le \omega_{\ep}^{-1}\frac{\norm{\G_{\ep,b}^{\theta}}_{q,b}}{\theta}
\le C\ep
\]
by the Lipschitz estimate \eqref{eqn: ultimate G Lip theta} on $\G_{\ep,b}^{\theta}$ with $\grave{\theta} = 0$.
\end{proof}

\subsection{Estimates on $\chi_{\ep}^{\theta}$}\label{app: chi-ep-theta}
The key function $\chi_{\ep}^{\theta}$ defined in \eqref{eqn: chi-ep-theta} is, by itself, exponentially large in $\ep$.
This will not, however, be a problem at all, as all appearances of $\chi_{\ep}^{\theta}$ ultimately come paired with the exponentially small factor $e^{-b/\ep}$ or with the oscillatory integral functional $\iota_{\ep}$, which enjoys Lombardi's exponentially small estimate.

\begin{lemma}\label{lem: chi-ep-theta}
Let $0 < b < b_{\sigma}$, $0 < q < q_{\sigma}$, and $\ep_0$, $\theta_0 > 0$.
There is $C > 0$ such that if $0 < \ep < \ep_{\per}$ and $|\theta| < \theta_0$, then 
\begin{equation}\label{eqn: chi-ep-theta norm solo}
\norm{\chi_{\ep}^{\theta}}_{q,b}
\le Ce^{b/\ep}.
\end{equation}
If, in addition, $0 \le \theta < \theta_0$, then
\begin{equation}\label{eqn: chi-ep-theta norm}
\norm{\chi_{\ep}^{\theta}+e^{b/\ep}(\tau_{\infty}\omega_{\ep}\theta+1)^{-1}\G_{\ep,b}^{\theta}}_{q,b}
\le Ce^{b/\ep}.
\end{equation}
\end{lemma}

\begin{proof}
We first use the product estimate \eqref{eqn: product est} in $\H_{q,b}$ to bound
\[
\norm{\chi_{\ep}^{\theta}}_{q,b}
\norm{\cos(\omega_{\ep}\Tsf_{\ep\theta}(\cdot))\sigma}_{q,b}
\le \norm{\cos(\omega_{\ep}\Tsf_{\ep\theta}(\cdot))}_{0,b}\norm{\sigma}_{q,b},
\]
and then for $z \in \U_b$ we estimate
\begin{equation}\label{eqn: chi norm aux}
|\cos(\omega_{\ep}\Tsf_{\ep\theta}(z))|
\le e^{|\im(\omega_{\ep}\Tsf_{\ep\theta}(z))|}
\le Ce^{b\omega_{\ep}}
\le Ce^{b/\ep}
\end{equation}
by \eqref{eqn: omega-ep est}.
This proves \eqref{eqn: chi-ep-theta norm solo} for $|\theta| < \theta_0$.

Now assume $\theta \ge 0$.
If $\theta = 0$, then since $\G_{\ep}^0 = 0$, we immediately have \eqref{eqn: chi-ep-theta norm}.
Otherwise, for $\theta > 0$, we use the triangle inequality and the Lipschitz estimate \eqref{eqn: useful G Lip} on $\G_{\ep,b}^{\theta}$ (and the positivity of $\theta$ and $\tau_{\infty}\omega_{\ep}\theta+1$) to bound
\[
\norm{\chi_{\ep}^{\theta}+e^{b/\ep}(\tau_{\infty}\omega_{\ep}\theta+1)^{-1}\G_{\ep,b}^{\theta}}_{q,b}
\le Ce^{b/\ep}+e^{b/\ep}\frac{\norm{\G_{\ep,b}^{\theta}}_{q,b}}{\theta}
\le Ce^{b/\ep}.
\qedhere
\]
\end{proof}

We will also need a Lipschitz estimate on $\chi_{\ep}^{\theta}$ (by itself) in $\theta$.

\begin{lemma}\label{lem: chi-ep-theta Lip theta}
Let $0 < b < b_{\sigma}$, $0 < q < q_{\sigma}$, and $\ep_0$, $\theta_0 > 0$.
There is $C > 0$ such that if $0 < \ep < \ep_{\per}$ and $|\theta| < \theta_0$, then 
\begin{equation}\label{eqn: chi-ep-theta Lip theta}
\norm{\chi_{\ep}^{\theta}-\chi_{\ep}^{\grave{\theta}}}_{q,b}
\le Ce^{b/\ep}|\theta-\grave{\theta}|.
\end{equation}
\end{lemma}

\begin{proof}
We use the definition of $\chi_{\ep}^{\theta}$ in \eqref{eqn: chi-ep-theta} and the product estimate \eqref{eqn: product est} in $\H_{q,b}$ to bound
\[
\norm{\chi_{\ep}^{\theta}-\chi_{\ep}^{\grave{\theta}}}_{q,b}
\le 2\norm{\cos(\omega_{\ep}\Tsf_{\ep\theta}(\cdot))-\cos(\omega_{\ep}\Tsf_{\ep\grave{\theta}}(\cdot))}_{0,b}\norm{\sigma}_{q,b}.
\]

For $z \in \U_b$, we have
\[
\big|\cos(\omega_{\ep}\Tsf_{\ep\theta}(z))-\cos(\omega_{\ep}\Tsf_{\ep\grave{\theta}}(z))\big|
\le (\ep\omega_{\ep})|\theta-\grave{\theta}||\tau(z)|\left(\sup_{0 \le s \le 1} \big|\sin\big(\omega_{\ep}z+(\ep\omega_{\ep})\tau(z)((1-s)\theta+s\grave{\theta})\big)\big|\right)
\]
We use the estimate \eqref{eqn: ep-omega-ep bounds} on $\ep\omega_{\ep}$ to control the supremum as 
\[
\big|\sin\big(\omega_{\ep}z+(\ep\omega_{\ep})\tau(z)((1-s)\theta+s\grave{\theta})\big)\big|
\le \exp\big(\big|\im[\omega_{\ep}z+(\ep\omega_{\ep})\tau(z)((1-s)\theta+s\grave{\theta})]\big|\big)
\le Ce^{b\omega_{\ep}},
\]
thus
\[
\norm{\chi_{\ep}^{\theta}-\chi_{\ep}^{\grave{\theta}}}_{q,b}
\le Ce^{b\omega_{\ep}}|\theta-\grave{\theta}|
\le Ce^{b/\ep}|\theta-\grave{\theta}|,
\]
where we have used \eqref{eqn: omega-ep est} to introduce the simpler factor of $e^{b/\ep}$.
\end{proof}

\subsection{Estimates on $\iota_{\ep}[\chi_{\ep}^{\theta}]$}\label{app: iota-ep-chi}
We develop estimates for the oscillatory integral $\iota_{\ep}[\chi_{\ep}^{\theta}]$ and colleagues, with $\iota_{\ep}$ defined in \eqref{eqn: iota-ep} and $\chi_{\ep}^{\theta}$ defined in \eqref{eqn: chi-ep-theta}.

\begin{lemma}\label{lem: iota-ep-chi-app}
Let $0 < b < b_{\sigma}$ and $0 < \theta_0 < \pi$.
There are $C$, $\ep_{\chi} > 0$ such that if $0 < \ep < \ep_{\chi}$ and $0 \le \theta < \theta_0$, then there is $\I_{\ep}(\theta) \in \R$, defined precisely in \eqref{eqn: I-ep-theta}, such that
\begin{equation}\label{eqn: iota-chi-ep-theta exp}
\iota_{\ep}[\chi_{\ep}^{\theta}]
= \frac{2\sinc((\ep\omega_{\ep})\theta)}{\tau_{\infty}}+\I_{\ep}(\theta).
\end{equation}
The following also hold.

\begin{enumerate}[label={\bf(\roman*)}]

\item
If $0 < \ep < \ep_{\chi}$ and $|\theta|  < \theta_0$, then
\begin{equation}\label{eqn: I-ep-theta map}
|\I_{\ep}(\theta)|
\le Ce^{-2b/\ep}.
\end{equation}

\item
If $0 < \ep < \ep_{\chi}$ and $|\theta| < \theta_0$, then
\begin{equation}\label{eqn: osc int lower bound1}
0
< \frac{\sinc(\theta_0)}{2\tau_{\infty}}
< \iota_{\ep}[\chi_{\ep}^{\theta}],
\end{equation}
and if $0 \le \theta < \theta_0$, then
\begin{equation}\label{eqn: osc int lower bound2}
0
< \frac{\sinc(\theta_0)}{2\tau_{\infty}}
< \iota_{\ep}[\chi_{\ep}^{\theta}+e^{b/\ep}(\tau_{\infty}\omega_{\ep}\theta+1)^{-1}\G_{\ep,b}^{\theta}].
\end{equation}

\item
If $0 < \ep < \ep_{\chi}$ and $|\theta|$, $|\grave{\theta}| < \theta_0$, then
\begin{equation}\label{eqn: I-ep Lip}
|\I_{\ep}(\theta)-\I_{\ep}(\grave{\theta})|
\le Ce^{-2b/\ep}|\theta-\grave{\theta}|.
\end{equation}
\end{enumerate}
\end{lemma}

\begin{proof}
First we prove the expansion \eqref{eqn: iota-chi-ep-theta exp}.
We need to calculate
\[
\iota_{\ep}[\chi_{\ep}^{\theta}]
= \int_{-\infty}^{\infty} \cos(\omega_{\ep}x+(\ep\omega_{\ep})\theta\tau(x))\sigma(x)e^{i\omega_{\ep}x} \dx.
\]
Rewriting
\[
\cos(\omega_{\ep}x+(\ep\omega_{\ep})\theta\tau(x))e^{i\omega_{\ep}x}
= \frac{e^{-i(\ep\omega_{\ep})\theta\tau(x)}}{2} + \frac{e^{i(\ep\omega_{\ep})\theta\tau(x)}e^{2i\omega_{\ep}x}}{2}
\]
gives
\begin{equation}\label{eqn: iota calc aux}
\iota_{\ep}[\cos(\omega_{\ep}\Tsf_{\ep\theta}(\cdot))\sigma]
= \frac{1}{2}\int_{-\infty}^{\infty} e^{-i(\ep\omega_{\ep})\theta\tau(x)}\sigma(x)\dx + \I_{\ep}(\theta),
\end{equation}
where
\begin{equation}\label{eqn: I-ep-theta}
\I_{\ep}(\theta)
:= \frac{1}{2}\int_{-\infty}^{\infty} [e^{i(\ep\omega_{\ep})\theta\tau(x)}\sigma(x)]e^{2i\omega_{\ep}x} \dx.
\end{equation}

Assume first that $\theta \ne 0$.
Since $\tau' = \tau_{\infty}\sigma$ and $\tst{\lim_{x \to \pm\infty} \tau(x) = \pm1}$ with $\tau_{\infty} > 0$, we have
\begin{multline*}
\int_{-\infty}^{\infty} e^{-i(\ep\omega_{\ep})\theta\tau(x)}\sigma(x)\dx
= \frac{1}{\tau_{\infty}}\int_{-\infty}^{\infty} e^{-i(\ep\omega_{\ep})\theta\tau(x)}\tau'(x)\dx
= \frac{1}{\tau_{\infty}}\int_{-1}^1 e^{-i(\ep\omega_{\ep})\theta{s}} \ds \\
= \frac{e^{i(\ep\omega_{\ep})\theta}-e^{-i(\ep\omega_{\ep})\theta}}{i(\ep\omega_{\ep})\tau_{\infty}\theta} 
= \frac{2\sinc((\ep\omega_{\ep})\theta)}{\tau_{\infty}}.
\end{multline*}
This identity is still valid for $\theta = 0$, since $\medint_{-\infty}^{\infty} \sigma(x) \dx = 2/\tau_{\infty}$ by \eqref{eqn: tau-infty}.
Combining this with \eqref{eqn: iota calc aux} gives the formula \eqref{eqn: iota-chi-ep-theta exp}.

Now we prove the estimates.

\begin{enumerate}[label={\bf(\roman*)}]

\item
{\it{Proof of the mapping estimate \eqref{eqn: I-ep-theta map} for the remainder $\I_{\ep}$.}}
Put $f_{\ep}^{\theta}(z) := e^{-i(\ep\omega_{\ep})\theta\tau(z)}\sigma(z)$, so $f_{\ep}^{\theta} \in \E_{q,b}$ since $0 < q < q_{\sigma}$ and $0 < b < b_{\sigma}$.
With $|\im(\tau(z))| \le \tau_b$ from \eqref{eqn: tau-b} and the bounds \eqref{eqn: ep-omega-ep bounds}, we estimate
\begin{equation}\label{eqn: aux est for I-ep-theta map}
|e^{i(\ep\omega_{\ep})\theta\tau(z)}|
= e^{\re[i(\ep\omega_{\ep})\theta\tau(z)]}
= e^{-(\ep\omega_{\ep})\theta\im[\tau(z)]}
\le e^{(\ep\omega_{\ep})|\theta\im[\tau(z)]|}
\le e^{2\theta_0\tan(b)}
\end{equation}
and thus
\[
\norm{f_{\ep}^{\theta}}_{q,b}
\le e^{2\theta_0\tan(b)}.
\]
Lombardi's estimate \eqref{eqn: Lombardi} then yields
\[
|\I_{\ep}(\theta)|
\le \frac{2}{q}e^{2\theta_0\tan(b)}e^{-2\omega_{\ep}b}.
\]

We simplify this further by rewriting
\begin{equation}\label{eqn: further simplifying I-theta}
e^{-2\omega_{\ep}b}
= e^{-2b(\omega_{\ep}-1/\ep)}e^{-2b/\ep}
\le e^{2b\ep}e^{-2b/\ep}
\le e^{1-2b/\ep},
\end{equation}
where we have used \eqref{eqn: omega-ep est} to control $\omega_{\ep}-1/\ep$ and now restricted $0 < \ep < 1/2b$.

\item
{\it{Proof of the lower bound \eqref{eqn: osc int lower bound1} for $\iota_{\ep}[\chi_{\ep}^{\theta}]$ and \eqref{eqn: osc int lower bound2} for $\iota_{\ep}[(\tau_{\infty}\omega_{\ep}\theta+1)\chi_{\ep}^{\theta}+\G_{\ep,b}^{\theta}]$.}}
First we estimate $\iota_{\ep}[\chi_{\ep}^{\theta}]$.
We use the formula \eqref{eqn: iota-chi-ep-theta exp} to rewrite
\[
\iota_{\ep}[\chi_{\ep}^{\theta}]
= \frac{2\sinc(\theta)}{\tau_{\infty}}
+ \frac{2\big(\sinc((\ep\omega_{\ep})\theta)-\sinc(\theta)\big)}{\tau_{\infty}}
+ \I_{\ep}(\theta).
\]
We use the Lipschitz estimate $|\sinc(X)-\sinc(Y)| \le |X-Y|$ to bound
\begin{equation}\label{eqn: useful sinc Lip}
|\sinc((\ep\omega_{\ep})\theta)-\sinc(\theta)|
\le |\ep\omega_{\ep}-1|\theta_0
\le \theta_0\ep^2.
\end{equation}
by \eqref{eqn: omega-ep est2}.

For $0 < X_0 < |X| < \pi$, we have $0 < \sinc(X_0) < \sinc(X)$, and so for $|\theta| < \theta_0 < \pi$, we have $\sinc(\theta) > \sinc(\theta_0) > 0$.
Take $0 < \ep < 1/b$ so small that 
\[
\frac{2\theta_0\ep^2}{\tau_{\infty}} + e^{1-2b/\ep}
< \frac{\sinc(\theta_0)}{\tau_{\infty}}
\]
Then 
\begin{equation}\label{eqn: osc int lower bound aux}
\iota_{\ep}[\chi_{\ep}^{\theta}]
\ge \frac{2\sinc(\tau_{\infty}\theta)}{\tau_{\infty}}-\theta_0\ep^2-e^{1-2b/\ep}
> \frac{\sinc(\theta_0)}{\tau_{\infty}}
> 0.
\end{equation}
If $\theta = 0$, then we are done, as here $\G_{\ep}^0 = 0$.

Otherwise, suppose $\theta \ne 0$.
We rewrite the whole integral as
\[
\iota_{\ep}[\chi_{\ep}^{\theta}+e^{b/\ep}(\tau_{\infty}\omega_{\ep}\theta+1)^{-1}\G_{\ep,b}^{\theta}]
= \iota_{\ep}[\chi_{\ep}^{\theta}]+e^{b/\ep}(\tau_{\infty}\omega_{\ep}\theta+1)^{-1}\iota_{\ep}[\G_{\ep,b}^{\theta}].
\]
Then Lombardi's estimate and the estimate \eqref{eqn: useful G Lip} on $(\tau_{\infty}\omega_{\ep}\theta+1)^{-1}\G_{\ep,b}^{\theta}$ imply
\[
e^{b/\ep}(\tau_{\infty}\omega_{\ep}\theta+1)^{-1}|\iota_{\ep}[\G_{\ep}^{\theta}]|
\le \frac{2}{q}\omega_{\ep}^{-1}\frac{\norm{\G}_{\ep}^{\theta}}{\theta}
\le C\ep.
\]
For $\ep$ sufficiently small, we therefore have 
\[
\iota_{\ep}[\chi_{\ep}^{\theta}]+e^{b/\ep}(\tau_{\infty}\omega_{\ep}\theta+1)^{-1}\iota_{\ep}[\G_{\ep,b}^{\theta}]
\ge \frac{\iota_{\ep}[\chi_{\ep}^{\theta}]}{2}
\ge \frac{\sinc(\theta_0)}{2\tau_{\infty}}
> 0
\]
by \eqref{eqn: osc int lower bound aux}.

\item
{\it{Proof of the Lipschitz estimate \eqref{eqn: I-ep Lip} for $\I_{\ep}$.}}
We will use Lombardi's estimate to control the difference
\[
\I_{\ep}(\theta)-\I_{\ep}(\grave{\theta})
= \frac{1}{2}\int_{-\infty}^{\infty} \big(e^{i(\ep\omega_{\ep})\theta\tau(x)}-e^{i(\ep\omega_{\ep})\grave{\theta}\tau(x)}\big)\sigma(x)e^{2i\omega_{\ep}x} \dx.
\]
The mean value inequality and \eqref{eqn: ep-omega-ep bounds} give
\[
\big|e^{i(\ep\omega_{\ep})\theta\tau(z)}-e^{i(\ep\omega_{\ep})\grave{\theta}\tau(z)}\big|
\le 2|\theta-\grave{\theta}|\tau_b\sup_{0 \le t \le 1} \big|e^{i(\ep\omega_{\ep})\tau(z)((1-t)\theta+t\grave{\theta})}\big|,
\]
where $\tau_b := \tst{\sup_{z \in \U_b} |\tau(z)|}$.
We estimate as in \eqref{eqn: aux est for I-ep-theta map} that
\[
\big|e^{i(\ep\omega_{\ep})\tau(z)((1-t)\theta+t\grave{\theta})}\big|
\le e^{2\theta_0\tan(b)}
\]
and therefore
\[
\sup_{z \in \U_b} e^{q|\re(z)|}\big|\big(e^{i(\ep\omega_{\ep})\theta\tau(z)}-e^{i(\ep\omega_{\ep})\grave{\theta}\tau(z)}\big)\sigma(z)\big|
\le 2\tau_be^{2\theta_0\tan(b)}|\theta-\grave{\theta}|
\]
Lombardi's estimate \eqref{eqn: Lombardi} implies
\[
|\I_{\ep}(\theta)-\I_{\ep}(\grave{\theta})|
\le \left(\frac{2}{q}\right)2\tau_be^{2\theta_0\tan(b)}|\theta-\grave{\theta}|e^{-2b\omega_{\ep}}
\]
and we simplify this further using \eqref{eqn: further simplifying I-theta}.
\qedhere
\end{enumerate}
\end{proof}

\subsection{Estimates on $\J_{\ep,b,-1}^{\alpha,\theta}$ and $\J_{\ep,b,0}^{\alpha,\theta}$}\label{app: J-1 J0}
We develop estimates for the terms $\J_{\ep,b,-1}^{\alpha,\theta}$ defined in \eqref{eqn: J-ep-neg1-alpha-theta} and $\J_{\ep,b,0}^{\alpha,\theta}$ in \eqref{eqn: J-ep-0-alpha-theta}.
Both of these terms are linear combinations of terms of the form
\[
\ep^{r_1}\alpha(\partial_z^{r_2}[\delta_{\ep,b}^{\alpha}]\circ\Tsf_{\ep\theta})\theta^s\varsigma,
\]
where $r_1$, $r_2$, $s \ge 0$ and $\varsigma \in \H_{q,b}$ for all $0 < q < q_{\sigma}$ and $0 < b < b_{\sigma}$, with $\delta_{\ep,b}^{\alpha}$ defined in \eqref{eqn: delta}, and so we develop the following estimates.

\begin{lemma}\label{lem: ultimate fake quadratic prep}
Let $0 < b < b_{\sigma}$, $0 < q_{\sigma}$, and $\theta_0 > 0$.
Let $\varsigma \in \H_{(q+3q_{\sigma})/4,b}$ and $s \ge 0$.
For each $r \ge 0$, there is $C_r > 0$ such that the following hold.

\begin{enumerate}[label={\bf(\roman*)}, ref={(\roman*)}]

\item
{}
[Mapping estimate]
If $0 < \ep < \ep_{\per}$, $|\alpha| < a_{\per}$, and $|\theta| < \theta_0$, then
\begin{equation}\label{eqn: J map}
\norm{\alpha(\partial_z^r[\delta_{\ep,b}^{\alpha}]\circ\Tsf_{\ep\theta})\theta^s\varsigma}_{q,b}
\le C_r\ep^{-r}|\alpha|^2|\theta|^s.
\end{equation}

\item\label{part: ultimate fake quadratic prep2}
{}
[Lipschitz estimate in amplitude]
If $0 < \ep < \ep_{\per}$, $0 \le |\alpha|$, $|\grave{\alpha}| < a_{\per}$, and $|\theta| < \theta_0$, then
\begin{equation}\label{eqn: J Lip alpha}
\norm{\alpha(\partial_z^r[\delta_{\ep,b}^{\alpha}]\circ\Tsf_{\ep\theta})\theta^s\varsigma -\grave{\alpha}(\partial_z^r[\delta_{\ep,b}^{\grave{\alpha}}]\circ\Tsf_{\ep\theta})\theta^s\varsigma}_{(q+q_{\sigma})/2,b}
\le C_r\ep^{-r}|\theta|^s(|\alpha|+|\grave{\alpha}|)|\alpha-\grave{\alpha}|.
\end{equation}

\item
{}
[Lipschitz estimate in phase shift]
If $0 < \ep < \ep_{\per}$, $|\alpha| < a_{\per}$, and $0 \le |\theta|$, $|\grave{\theta}| < \theta_0$, then
\begin{equation}\label{eqn: J Lip theta}
\norm{\alpha(\partial_z^r[\delta_{\ep,b}^{\alpha}]\circ\Tsf_{\ep\theta})\theta^s\varsigma -\alpha(\partial_z^r[\delta_{\ep,b}^{\alpha}]\circ\Tsf_{\ep\grave{\theta}})\grave{\theta}^s\varsigma}_{q,b}
\le C_r\ep^{-r}|\alpha|^2|\theta-\grave{\theta}|.
\end{equation}
\end{enumerate}
\end{lemma}

\begin{proof}

\begin{enumerate}[label={\bf(\roman*)}]

\item
Since $\delta_{\ep,b}^{0,0} = 0$, this is a consequence of the Lipschitz estimate \eqref{eqn: J Lip alpha} with $\grave{\alpha} = 0$.

\item
Rewrite
\[
\alpha(\partial_z^r[\delta_{\ep,b}^{\alpha}]\circ\Tsf_{\ep\theta})\theta^sf-\grave{\alpha}(\partial_z^r[\delta_{\ep,b}^{\grave{\alpha}}]\circ\Tsf_{\ep\theta})\theta^sf
= \Delta_1 + \Delta_2,
\]
where
\[
\Delta_1
:= (\alpha-\grave{\alpha})(\partial_z^r[\delta_{\ep,b}^{\alpha}] \circ \Tsf_{\ep\theta})\theta^sf
\]
and
\[
\Delta_2
:= \grave{\alpha}\big((\partial_z^r[\delta_{\ep,b}^{\alpha}] \circ \Tsf_{\ep\theta})-(\partial_z^r[\delta_{\ep,b}^{\grave{\alpha}}]\circ\Tsf_{\ep\theta})\big)\theta^sf.
\]
We use the mapping estimate \eqref{eqn: fake quadratic map} to obtain the pointwise estimate
\[
e^{(q+q_{\sigma})|\re(z)|/2}|\Delta_1(z)|
\le C_r\ep^{-r}|\theta|^s|\alpha||\alpha-\grave{\alpha}|e^{(q+q_{\sigma})|\re(z)|/2}(|z|+1)|\varsigma (z)|
\]
for $z \in \U_b$, and so
\[
\norm{\Delta_1}_{(q+q_{\sigma})/2,b}
\le C_r\ep^{-r}|\theta|^s|\alpha||\alpha-\grave{\alpha}|\big(\norm{\varsigma}_{q,b}+\norm{f\varsigma}_{q,b}\big),
\]
where $f(z) := z$.
The decay-borrowing estimate \eqref{eqn: decay-borrowing est} gives
\begin{equation}\label{eqn: actual decay borrowing1}
\norm{f\varsigma}_{(q+q_{\sigma})/2,b}
\le \norm{f}_{(q+q_{\sigma})/2-(q+3q_{\sigma})/4,b}\norm{\varsigma}_{(q+3q_{\sigma})/4,b}.
\end{equation}
Since $(q+q_{\sigma})/2-(q+3q_{\sigma})/4 = (q-q_{\sigma})/4$ and $0 < q < q_{\sigma}$, by Lemma \ref{lem: Hqb norm of z} we have
\begin{equation}\label{eqn: actual decay borrowing2}
\norm{f}_{(q-q_{\sigma})/4,b}
\le \frac{4}{e(q_{\sigma}-q)}+b
\le \max\{4e^{-1},b\}\left(\frac{1}{q_{\sigma}-q}+1\right).
\end{equation}
It follows that 
\[
\norm{\Delta_1}_{(q+q_{\sigma})/2,b}
\le C_r\ep^{-r}\left(\frac{1}{q_{\sigma}-q}+1\right)|\theta|^s|\alpha||\alpha-\grave{\alpha}|.
\]
We estimate $\Delta_2$ in exactly the same way using the Lipschitz estimate \eqref{eqn: fake quadratic Lip alpha}.

\item
Rewrite
\[
\alpha(\partial_z^r[\delta_{\ep,b}^{\alpha}]\circ\Tsf_{\ep\theta})\theta^s\varsigma -\alpha(\partial_z^r[\delta_{\ep,b}^{\alpha}]\circ\Tsf_{\ep\grave{\theta}})\grave{\theta}^s\varsigma
= \Delta_1 + \Delta_2,
\]
where
\[
\Delta_1
:= \alpha\big((\partial_z^r[\delta_{\ep,b}^{\alpha}] \circ \Tsf_{\ep\theta})-(\partial_z^r[\delta_{\ep,b}^{\alpha}]\circ \Tsf_{\ep\grave{\theta}})\big)\theta^s\varsigma
\]
and
\[
\Delta_2
:= \alpha(\partial_z^r[\delta_{\ep,b}^{\alpha}] \circ \Tsf_{\ep\grave{\theta}})(\theta-\grave{\theta}^s)\varsigma.
\]
We obtain the pointwise estimate
\[
e^{q|\re(z)|}|\Delta_1(z)|
\le C_r\ep^{-r}\theta_0^s|\alpha|^2|\theta-\grave{\theta}|e^{q|\re(z)|}(|z|+1)|\varsigma(z)|
\]
from the Lipschitz estimate \eqref{eqn: fake quadratic Lip theta}, and so
\[
\norm{\Delta_1}_{q,b}
\le C_r\ep^{-r}|\alpha|^2|\theta-\grave{\theta}|\big(\norm{\varsigma}_{q,b}+\norm{f\varsigma}_{q,b}\big),
\]
where, as before, $f(z) = z$.
Then we apply the decay-borrowing estimate \eqref{eqn: decay-borrowing est} as in the previous part.
The estimate for $\Delta_2$ follows from the mapping estimate \eqref{eqn: fake quadratic map} and decay borrowing one more time.
\qedhere
\end{enumerate}
\end{proof}

Now we are ready to estimate the terms $\J_{\ep,b,-1}^{\alpha,\theta}$ and $\J_{\ep,b,0}^{\alpha,\theta}$.
The individual terms that sum to form $\J_{\ep,b,-1}^{\alpha,\theta}$ have the form
\[
\ep^r(\partial_z^{r+1}[\delta_{\ep,b}^{\alpha}]\circ\Tsf_{\ep\theta})\theta^s\varsigma, \ s \ge 1, \ \varsigma \in \H_{q,b} \text{ for } 0 < q < q_{\sigma}, \ 0 < b < b_{\sigma},
\]
and so the estimates on $\J_{\ep,b,-1}^{\alpha,\theta}$ will contain an $\O(\ep^{-1}\theta)$ factor.
When $\theta \ne 0$, this is worse than the estimates on $\J_{\ep,b,0}^{\alpha,\theta}$, whose individual terms have the form
\[
\ep^r(\partial_z^r[\delta_{\ep,b}^{\alpha}]\circ\Tsf_{\ep\theta})\theta^s\varsigma
\]
with $s \ge 1$ for all but one of those terms.
So, we mostly subsume the better estimates for $\J_{\ep,b,0}^{\alpha,\theta}$ into the worse ones for $\J_{\ep,b,-1}^{\alpha,\theta}$ and use Lemma \ref{lem: ultimate fake quadratic prep} to obtain the following.

\begin{lemma}\label{lem: ultimate fake quadratic}
Let $0 < b < b_{\sigma}$, $0 < q < q_{\sigma}$, and $\theta_0 > 0$.
There exists $C > 0$ such that the following hold.

\begin{enumerate}[label={\bf(\roman*)}]

\item
{}
[Mapping estimate]
If $0 < \ep < \ep_{\per}$, $|\alpha| < a_{\per}$, and $|\theta| < \theta_0$, then
\begin{equation}\label{eqn: ultimate fake quadratic map}
\norm{\J_{\ep,b,-1}^{\alpha,\theta}}_{q,b}+\norm{\J_{\ep,b,0}^{\alpha,\theta}}_{q,b}
\le C\ep^{-1}|\alpha|^2.
\end{equation}

\item
{}
[Lipschitz estimate in amplitude]
If $0 < \ep < \ep_{\per}$, $0 \le |\alpha|$, $|\grave{\alpha}| < a_{\per}$, and $|\theta| < \theta_0$, then
\begin{equation}\label{eqn: ultimate fake quadratic Lip alpha}
\norm{\J_{\ep,b,-1}^{\alpha,\theta}-\J_{\ep,b,-1}^{\grave{\alpha},\theta}}_{(q+q_{\sigma})/2,b}+\norm{\J_{\ep,b,0}^{\alpha,\theta}-\J_{\ep,b,0}^{\grave{\alpha},\theta}}_{(q+q_{\sigma})/2,b}
\le C\ep^{-1}(|\alpha|+|\grave{\alpha}|)|\alpha-\grave{\alpha}|.
\end{equation}

\item
{}
[Lipschitz estimate in phase shift]
If $0 < \ep < \ep_{\per}$, $|\alpha| < a_{\per}$, and $0 \le |\theta|$, $|\grave{\theta}| < \theta_0$, then
\begin{equation}\label{eqn: ultimate fake quadratic Lip theta}
\norm{\J_{\ep,b,-1}^{\alpha,\theta}-\J_{\ep,b,-1}^{\alpha,\grave{\theta}}}_{q,b}+\norm{\J_{\ep,b,0}^{\alpha,\theta}-\J_{\ep,b,0}^{\alpha,\grave{\theta}}}_{q,b}
\le C\ep^{-1}|\alpha|^2|\theta-\grave{\theta}|.
\end{equation}
\end{enumerate}
\end{lemma}

\subsection{Estimates on $\J_{\ep,b,1}^{\alpha,\theta}$}\label{app: J1}
We develop estimates for the term $\J_{\ep,b,1}^{\alpha,\theta}$ defined in \eqref{eqn: J-ep-1-alpha-theta}.
This term is a linear combination of 
\begin{equation}\label{eqn: that extra term in J1}
(\ep\omega_{\ep})\big(\sqrt{1+4\ep^2}-1\big)\omega_{\ep}\big(\alpha{e}^{-b/\ep}\chi_{\ep}^{\theta}\big)(\theta\tau_{\infty})
\end{equation}
and terms of the form
\begin{equation}\label{eqn: J-ep-theta3 terms}
\ep^{r_1+r_2}\alpha(\partial_z^{r_1}[\varphi_{\ep,b}^{\alpha,0}]\circ\Tsf_{\ep\theta})\theta^s\varsigma
\end{equation}
where $r_1$, $r_2 \ge 1$, $s \ge 0$, and $\varsigma \in \H_{q,b}$ for all $0 < q < q_{\sigma}$ and $0 < b < b_{\sigma}$.
We will retain at least one factor of $\ep$ in our estimates on these terms, as the following lemma shows for the latter terms.

\begin{lemma}\label{lem: ultimate J1 prep}
Let $0 < b < b_{\sigma}$, $0 < q < q_{\sigma}$, and $\theta_0 > 0$.
Let $\varsigma \in \H_{(q+3q_{\sigma})/4,b}$ and $s \ge 0$.
For each $r \ge 0$, there is $C_r > 0$ such that the following hold.

\begin{enumerate}[label={\bf(\roman*)}, ref={(\roman*)}]

\item
{}
[Mapping estimate]
If $0 < \ep < \ep_{\per}$, $|\alpha| < a_{\per}$, and $|\theta| < \theta_0$, then
\begin{equation}\label{eqn: J1 map prep}
\norm{\alpha(\partial_z^r[\varphi_{\ep,b}^{\alpha,0}]\circ\Tsf_{\ep\theta})\theta^s\varsigma}_{q,b}
\le C_r\ep^{-r}|\alpha|.
\end{equation}

\item\label{part: ultimate J1 prep 2}
{}
[Lipschitz estimate in amplitude]
If $0 < \ep < \ep_{\per}$, $0 \le |\alpha|$, $|\grave{\alpha}| < a_{\per}$, and $|\theta| < \theta_0$, then
\begin{equation}\label{eqn: J1 Lip alpha prep}
\norm{\alpha(\partial_z^r[\varphi_{\ep,b}^{\alpha,0}]\circ\Tsf_{\ep\theta})\theta^s\varsigma -\grave{\alpha}(\partial_z^r[\varphi_{\ep,b}^{\grave{\alpha},0}]\circ\Tsf_{\ep\theta})\theta^s\varsigma}_{(q+q_{\sigma})/2,b}
\le C_r\ep^{-r}|\alpha-\grave{\alpha}|.
\end{equation}

\item
{}
[Lipschitz estimate in phase shift]
If $0 < \ep < \ep_{\per}$, $|\alpha| < a_{\per}$, and $0 \le |\theta|$, $|\grave{\theta}| < \theta_0$, then
\begin{equation}\label{eqn: J1 Lip theta prep}
\norm{\alpha(\partial_z^r[\varphi_{\ep,b}^{\alpha,0}]\circ\Tsf_{\ep\theta})\theta^s\varsigma -\alpha(\partial_z^r[\varphi_{\ep,b}^{\alpha,0}]\circ\Tsf_{\ep\grave{\theta}})\grave{\theta}^s\varsigma}_{q,b}
\le C_r\ep^{-r}|\alpha||\theta-\grave{\theta}|.
\end{equation}
\end{enumerate}
\end{lemma}

\begin{proof}

\begin{enumerate}[label={\bf(\roman*)}]

\item
This follows from the product estimate \eqref{eqn: product est} and the mapping estimate \eqref{eqn: varphi map ultimate}.

\item
Rewrite
\[
\alpha(\partial_z^r[\varphi{\ep,b}^{\alpha,0}]\circ\Tsf_{\ep\theta})\theta^s\varsigma -\grave{\alpha}(\partial_z^r[\varphi_{\ep,b}^{\grave{\alpha},0}]\circ\Tsf_{\ep\theta})\theta^s\varsigma
= \Delta_1 + \Delta_2,
\]
where
\[
\Delta_1
:= (\alpha-\grave{\alpha})(\partial_z^r[\varphi_{\ep,b}^{\alpha,0}]\circ\Tsf_{\ep\theta})\theta^s\varsigma
\]
and
\[
\Delta_2
:= \grave{\alpha}\big((\partial_z^r[\varphi_{\ep,b}^{\alpha,0}]\circ\Tsf_{\ep\theta})-(\partial_z^r[\varphi_{\ep,b}^{\grave{\alpha},0}]\circ\Tsf_{\ep\theta})\big)\theta^s\varsigma.
\]
We estimate
\[
\norm{\Delta_1}_{(q+q_{\sigma})/2,b}
\le C_r|\alpha-\grave{\alpha}|\ep^{-r}\theta_0^s\norm{\varsigma}_{(q+q_{\sigma})/2,b}
\]
by the mapping estimate \eqref{eqn: varphi map ultimate} and the product estimate \eqref{eqn: product est}.
Next, the Lipschitz estimate \eqref{eqn: varphi Lip alpha ultimate} gives the pointwise estimate
\[
e^{(q+q_{\sigma})|\re(z)|/2}|\Delta_2(z)|
\le C_r|\grave{\alpha}|\ep^{-r}|\alpha-\grave{\alpha}|\theta_0^se^{(q+q_{\sigma})|\re(z)|/2}(1+|z|)|\varsigma(z)|,
\]
and so
\[
\norm{\Delta_2}_{q,b}
\le C_r\ep^{-r}|\grave{\alpha}|\theta_0^s\big(\norm{\varsigma}_{(q+q_{\sigma})/2,b}+\norm{f\varsigma}_{(q+q_{\sigma})/2,b}\big),
\]
where $f(z) := z$.
The decay-borrowing estimates \eqref{eqn: actual decay borrowing1} and \eqref{eqn: actual decay borrowing2} ensure $\norm{f\varsigma}_{(q+q_{\sigma})/2,b} < \infty$.
We note that the estimate on $\Delta_1$ does not have the extra factor of $\grave{\alpha}$ that the estimate on $\Delta_2$ enjoys.

\item
Rewrite
\[
\alpha(\partial_z^r[\varphi_{\ep,b}^{\alpha,0}]\circ\Tsf_{\ep\theta})\theta^s\varsigma -\alpha(\partial_z^r[\varphi_{\ep,b}^{\alpha,0}]\circ\Tsf_{\ep\grave{\theta}})\grave{\theta}^s\varsigma
= \Delta_1 + \Delta_2,
\]
where
\[
\Delta_1
:= \alpha\big((\partial_z^r[\varphi_{\ep,b}^{\alpha,0}]\circ\Tsf_{\ep\theta})-(\partial_z^r[\varphi_{\ep,b}^{\alpha,0}]\circ\Tsf_{\ep\grave{\theta}})\big)\theta^s\varsigma
\]
and
\[
\Delta_2
:= \alpha(\partial_z^r[\varphi_{\ep,b}^{\alpha,0}]\circ\Tsf_{\ep\grave{\theta}})(\theta^s-\grave{\theta}^s)\varsigma.
\]
We estimate
\[
\norm{\Delta_1}_{q,b}
\le C_r|\alpha|\ep^{-r}|\theta-\grave{\theta}|\theta_0^s\norm{\varsigma}_{q,b}
\]
using the Lipschitz estimate \eqref{eqn: varphi Lip theta ultimate} and the product estimate \eqref{eqn: product est} and 
\[
\norm{\Delta_2}_{q,b}
\le C_r\ep^{-r}|\alpha|s|\theta-\grave{\theta}|\theta_0^{s-1}\norm{\varsigma}_{q,b}
\]
using the mapping estimate \eqref{eqn: varphi map ultimate} and the product estimate.
\qedhere
\end{enumerate}
\end{proof}

We then have the following estimates on $\J_{\ep,b,1}$.
They follow from Lemma \ref{lem: ultimate J1 prep} applied to the terms \eqref{eqn: J-ep-theta3 terms} and from the estimates \eqref{eqn: chi-ep-theta norm solo} and \eqref{eqn: chi-ep-theta Lip theta} applied to the term \eqref{eqn: that extra term in J1}, along with the bound
\[
\omega_{\ep}\big|\sqrt{1+4\ep^2}-1\big|
\le C\ep,
\]
which is a consequence of the Taylor expansion \eqref{eqn: sqrt Taylor} of the square root.

\begin{lemma}\label{lem: ultimate J1}
Let $0 < b < b_{\sigma}$, $0 < q < q_{\sigma}$, and $\theta_0 > 0$.
There exists $C > 0$ such that the following hold.

\begin{enumerate}[label={\bf(\roman*)}]

\item
{}
[Mapping estimate]
If $0 < \ep < \ep_{\per}$, $|\alpha| < a_{\per}$, and $|\theta| < \theta_0$, then
\begin{equation}\label{eqn: ultimate J1 map}
\norm{\J_{\ep,b,1}^{\alpha,\theta}}_{q,b}
\le C\ep|\alpha|.
\end{equation}

\item
{}
[Lipschitz estimate in amplitude]
If $0 < \ep < \ep_{\per}$, $0 \le |\alpha|$, $|\grave{\alpha}| < a_{\per}$, and $|\theta| < \theta_0$, then
\begin{equation}\label{eqn: ultimate J1 Lip alpha}
\norm{\J_{\ep,b,1}^{\alpha,\theta}-\J_{\ep,b,1}^{\grave{\alpha},\theta}}_{(q+q_{\sigma})/2,b}
\le C\ep|\alpha-\grave{\alpha}|.
\end{equation}

\item
{}
[Lipschitz estimate in phase shift]
If $0 < \ep < \ep_{\per}$, $|\alpha| < a_{\per}$, and $0 \le |\theta|$, $|\grave{\theta}| < \theta_0$, then
\begin{equation}\label{eqn: ultimate J1 Lip theta}
\norm{\J_{\ep,b,1}^{\alpha,\theta}-\J_{\ep,b,1}^{\alpha,\grave{\theta}}}_{q,b}
\le C\ep|\alpha||\theta-\grave{\theta}|.
\end{equation}
\end{enumerate}
\end{lemma}

\subsection{Estimates on $\rhs_{q,b}^{\ep,0}$}
We develop estimates for the term $\rhs_{q,b}^{\ep,0}$ defined in \eqref{eqn: rhs-ep-0}.
As we point out in the proof, these estimates are not as sharp as they could be, but for simplicity we have dominated some terms by others that work just as well in our eventual applications.
The proof of part \ref{part: ultimate rhs0 2} below contains our most overt and delicate use of decay-borrowing; unlike in part \ref{part: ultimate fake quadratic prep2} of Lemma \ref{lem: ultimate fake quadratic prep} and part \ref{part: ultimate J1 prep 2} of Lemma \ref{lem: ultimate J1 prep}, we need to borrow against the decay of the localized remainder $\eta$, which does not decay quite as fast as the leading order perturbation $\upsilon_{\ep,q}$.
This leads to the mismatch in decay rates in \eqref{eqn: ultimate rhs0 Lip alpha}.

\begin{lemma}\label{lem: ultimate rhs0}
Let $0 < b < b_{\sigma}$ and $0 < q < q_{\sigma}$.
There exists $C > 0$ such that the following hold.

\begin{enumerate}[label={\bf(\roman*)}, ref={(\roman*)}]

\item
{}
[Mapping estimate]
If $0 < \ep < \ep_{\per}$, $|\alpha| < a_{\per}$, $|\theta| < \theta_0$, and $\eta \in \H_{q,b}$ then
\begin{equation}\label{eqn: ultimate rhs0 map}
\norm{\rhs_{q,b}^{\ep,0}(\eta,\alpha,\theta)}_{q,b}
\le C\big(\ep^4+ |\alpha|^2+\norm{\eta}_{q,b}^2\big).
\end{equation}

\item\label{part: ultimate rhs0 2}
{}
[Lipschitz estimate in amplitude]
If $0 < q_1 < q_2 \le (q+q_{\sigma})/2$, $0 < \ep < \ep_{\per}$, $0 \le |\alpha|$, $|\grave{\alpha}| < a_{\per}$, $|\theta| < \theta_0$, and $\eta$, $\grave{\eta} \in \H_{q_2,b}$, then
\begin{equation}\label{eqn: ultimate rhs0 Lip alpha}
\norm{\rhs_{q,b}^{\ep,0}(\eta,\alpha,\theta)-\rhs_{q,b}^{\ep,0}(\grave{\eta},\grave{\alpha},\theta)}_{q_1,b}
\le \frac{C}{q_2-q_1}\big(\ep^2+|\alpha|+|\grave{\alpha}|+\norm{\eta}_{q_2,b}+\norm{\grave{\eta}}_{q_2,b}\big)\big(|\alpha-\grave{\alpha}|+\norm{\eta-\grave{\eta}}_{q_1,b}\big).
\end{equation}

\item
{}
[Lipschitz estimate in phase shift]
If $0 < \ep < \ep_{\per}$, $|\alpha| < a_{\per}$, $0 \le |\theta|$, $|\grave{\theta}| < \theta_0$, and $\eta$, $\grave{\eta} \in \H_{q,b}$, then
\begin{multline}\label{eqn: ultimate rhs0 Lip theta}
\norm{\rhs_{q,b}^{\ep,0}(\eta,\alpha,\theta)-\rhs_{q,b}^{\ep,0}(\grave{\eta},\alpha,\grave{\theta})}_{q,b}
\le C\big(\ep^2+|\alpha|+\norm{\eta}_{q,b}+\norm{\grave{\eta}}_{q,b}\big)\norm{\eta-\grave{\eta}}_{q,b} \\
+ C\big(\ep^2|\alpha|+|\alpha|\norm{\eta}_{q,b}+|\alpha|\norm{\grave{\eta}}_{q,b}\big)|\theta-\grave{\theta}|.
\end{multline}
\end{enumerate}
\end{lemma}

\begin{proof}

\begin{enumerate}[label={\bf(\roman*)}]

\item
We have
\[
\norm{\rhs_{q,b}^{\ep,0}(\eta,\alpha,\theta)}_{q,b}
\le \ep^2\norm{\upsilon_{\ep,q}^{(4)}}_{q,b} 
+ 2\norm{\upsilon_{\ep,q}\eta}_{q,b} 
+ 2|\alpha|\norm{\varphi_{\ep,b}^{\alpha,\theta}\upsilon_{\ep,q}}_{q,b}
+ 2|\alpha|\norm{\varphi_{\ep,b}^{\alpha,\theta}\eta}_{q,b} 
+ \norm{\eta^2}_{q,b}.
\]
The estimates \eqref{eqn: upsilon-ep} on $\upsilon_{\ep,q}$ and $\upsilon_{\ep,q}^{(4)}$ in $\E_{(q+3q_{\sigma})/4,b}$ combined with the embedding estimate \eqref{eqn: embed est}, the mapping estimate \eqref{eqn: varphi map ultimate} on $\varphi_{\ep,b}^{\alpha,\theta}$, and the product estimates \eqref{eqn: product est} and \eqref{eqn: product est2} in give
\[
\norm{\upsilon_{\ep,q}\eta}_{q,b} \le C\ep^2\norm{\eta}_{q,b},
\quad
\norm{\varphi_{\ep,b}^{\alpha,\theta}\upsilon_{\ep,q}}_{q,b} \le C\ep^2,
\quad
\norm{\varphi_{\ep,b}^{\alpha,\theta}\eta}_{q,b} \le C\norm{\eta}_{q,b},
\quad\text{and}\quad
\norm{\eta^2}_{q,b} \le \norm{\eta}_{q,b}^2.
\]
The estimate \eqref{eqn: ultimate rhs0 map} that we choose to record here is slightly worse than what we actually obtain; for example, we have dominated $\ep^2\norm{\eta}_{q,b} \le (\ep^4 + \norm{\eta}_{q,b}^2)/2$.
The sharper estimates from the calculations above would ultimately make no difference in our application of this lemma.

\item
Rewrite
\[
-\frac{1}{2}(\rhs_{q,b}^{\ep,0}(\eta,\alpha,\theta)-\rhs_{q,b}^{\ep,0}(\grave{\eta},\grave{\alpha},\theta))
= \sum_{j=1}^6 \Delta_j + \frac{1}{2}(\eta+\grave{\eta})(\eta-\grave{\eta}),
\]
where
\[
\Delta_1
:= \upsilon_{\ep,q}(\eta-\grave{\eta}),
\]
\[
\Delta_2
:= (\alpha-\grave{\alpha})\varphi_{\ep,b}^{\alpha,\theta}\upsilon_{\ep,q},
\]
\[
\Delta_3
:= \grave{\alpha}(\varphi_{\ep,b}^{\alpha,\theta}-\varphi_{\ep,b}^{\grave{\alpha},\theta})\upsilon_{\ep,q},
\]
\[
\Delta_4
:= (\alpha-\grave{\alpha})\varphi_{\ep,b}^{\alpha,\theta}\eta,
\]
\[
\Delta_5
:= \grave{\alpha}(\varphi_{\ep,b}^{\alpha,\theta}-\varphi_{\ep,b}^{\grave{\alpha},\theta})\eta,
\]
and
\[
\Delta_6
:= \grave{\alpha}\varphi_{\ep,b}^{\grave{\alpha},\theta}(\eta-\grave{\eta}).
\]
It is straightforward to estimate the terms $\Delta_1$ and $\Delta_2$.
We use the product estimate \eqref{eqn: product est2}, the embedding estimate \eqref{eqn: embed est}, and the estimates \eqref{eqn: upsilon-ep} on $\upsilon_{\ep,q}$ to bound
\[
\norm{\Delta_1}_{q_1,b}
\le \norm{\upsilon_{\ep,q}}_{q_1,b}\norm{\eta-\grave{\eta}}_{q_1,b}
\le \norm{\upsilon_{\ep,q}}_{(q+3q_{\sigma})/4,b}\norm{\eta-\grave{\eta}}_{q_1,b}
\le C\ep^2\norm{\eta-\grave{\eta}}_{q_1,b}.
\]
We use the product estimate \eqref{eqn: product est}, the mapping estimate \eqref{eqn: varphi map ultimate} on $\varphi_{\ep,b}^{\alpha,\theta}$, the embedding estimate \eqref{eqn: embed est}, and the estimates \eqref{eqn: upsilon-ep} on $\upsilon_{\ep,q}$ to bound
\[
\norm{\Delta_2}_{q_1,b}
\le |\alpha-\grave{\alpha}|\norm{\varphi_{\ep,b}^{\alpha,\theta}}_{0,b}\norm{\upsilon_{\ep,q}}_{q_1,b}
\le |\alpha-\grave{\alpha}|\norm{\varphi_{\ep,b}^{\alpha,\theta}}_{0,b}\norm{\upsilon_{\ep,q}}_{(q+3q_{\sigma})/4,b}
\le C\ep^2|\alpha-\grave{\alpha}|.
\]

Estimating the term $\Delta_3$ is slightly more involved.
The Lipschitz estimate \eqref{eqn: varphi Lip alpha ultimate} gives the pointwise bound
\[
e^{q_1|\re(z)|}|\Delta_3(z)|
\le C|\grave{\alpha}||\alpha-\grave{\alpha}|e^{q_1|\re(z)|}(1+|z|)|\upsilon_{\ep,q}(z)|,
\]
and so
\[
\norm{\Delta_3}_{q_1,b}
\le C|\grave{\alpha}||\alpha-\grave{\alpha}|\big(\norm{\upsilon_{\ep,q}}_{q_1,b}+\norm{f\upsilon_{\ep,q}}_{q_1,b}\big),
\]
where $f(z) := z$.
The embedding estimate \eqref{eqn: embed est}, the decay-borrowing estimate \eqref{eqn: decay-borrowing est}, and the estimates \eqref{eqn: upsilon-ep} on $\upsilon_{\ep,q}$ give
\[
\norm{f\upsilon_{\ep,q}}_{q_1,b}
\le \norm{f\upsilon_{\ep,q}}_{(q+q_{\sigma})/2,b}
\le \norm{f}_{(q+q_{\sigma})/2-(q+3q_{\sigma})/4,b}\norm{\upsilon_{\ep,q}}_{(q+3q_{\sigma})/4,b}
\le C\ep^2\norm{f}_{(q-q_{\sigma})/4,b},
\]
where $\norm{f}_{(q-q_{\sigma})/4,b} < \infty$ by Lemma \ref{lem: Hqb norm of z}.
We conclude
\begin{equation}\label{eqn: ep2 grave Lip alpha}
\norm{\Delta_3}_{q_2,b}
\le C\ep^2|\grave{\alpha}||\alpha-\grave{\alpha}|.
\end{equation}

The estimates on the final three terms are similar to the ones above, except now we have standalone factors of $\eta$, which decays more slowly than $\upsilon_{\ep,q}$.
We use the mapping estimate \eqref{eqn: varphi map ultimate}, the product estimate \eqref{eqn: product est}, and the embedding estimate \eqref{eqn: embed est} to bound
\[
\norm{\Delta_4}_{q_2,b}
\le C|\alpha-\grave{\alpha}|\norm{\eta}_{q_2,b}
\le C|\alpha-\grave{\alpha}|\norm{\eta}_{q_1,b}.
\]
The same argument gives
\[
\norm{\Delta_6}_{q_2,b}
\le C|\grave{\alpha}|\norm{\eta-\grave{\eta}}_{q_2,b}.
\]

Finally, the Lipschitz estimate \eqref{eqn: varphi Lip alpha ultimate} gives the pointwise bound
\[
e^{q_1|\re(z)|}|\Delta_5(z)|
\le C|\grave{\alpha}||\alpha-\grave{\alpha}|e^{q_1|\re(z)|}|(1+|z|)|\eta(z)|,
\]
and so
\[
\norm{\Delta_5}_{q_1,b}
\le C|\grave{\alpha}||\alpha-\grave{\alpha}|\big(\norm{\eta}_{q_1,b}+\norm{f\eta}_{q_1,b}\big).
\]
where $f(z) = z$ as before.
The decay-borrowing estimate \eqref{eqn: decay-borrowing est} gives
\[
\norm{f\eta}_{q_1,b}
\le \norm{f}_{q_1-q_2,b}\norm{\eta}_{q_2,b}
\le \left(\frac{1}{q_2-q_1}+b\right)\norm{\eta}_{q_2,b},
\]
where the last estimate follows from Lemma \ref{lem: Hqb norm of z}.
We further bound
\[
\frac{1}{q_2-q_1}+b
\le \max\{b,1\}\left(\frac{1}{q_2-q_1}+1\right)
\]
to conclude
\[
\norm{\Delta_5}_{q_1,b}
\le C\left(\frac{1}{q_2-q_1}+1\right)|\grave{\alpha}||\alpha-\grave{\alpha}|\norm{\eta}_{q_2,b}.
\]
Last, since $0 < q_1 < q_2 < (q+q_{\sigma})/2$, we have $0 < q_2-q_1 < ((q+q_{\sigma})/2+1)-1$, from which it follows that 
\[
\frac{1}{q_2-q_1}+1
< \left(\frac{q+q_{\sigma}}{2}+1\right)\frac{1}{q_2-q_1}.
\]
Thus
\[
\norm{\Delta_5}_{q_1,b}
\le \frac{C}{q_2-q_1}|\grave{\alpha}||\alpha-\grave{\alpha}|\norm{\eta}_{q_2,b}.
\]
We emphasize that the constant $C$ here is independent of $q_1$ and $q_2$.

Finally, the product and embedding estimates \eqref{eqn: product est2} and \eqref{eqn: embed est} in $\H_{q,b}$ give
\begin{equation}\label{eqn: eta sq Lip}
\norm{(\eta+\grave{\eta})(\eta-\grave{\eta})}_{q_1,b}
\le \norm{\eta+\grave{\eta}}_{q_1,b}\norm{\eta-\grave{\eta}}_{q_1,b}
\le \big(\norm{\eta}_{q_2,b}+\norm{\grave{\eta}}_{q_2,b}\big)\norm{\eta-\grave{\eta}}_{q_1,b}.
\end{equation}

Again, the estimates in the final recorded result \eqref{eqn: ultimate rhs0 Lip alpha} are slightly worse than they could be.
For example, from \eqref{eqn: ep2 grave Lip alpha}, we have dominated $\ep^2|\grave{\alpha}| \le (\ep^4+|\grave{\alpha}|^2)/2 \le \ep^2+|\grave{\alpha}|/2$, assuming, as we always do, that $0 < \ep$, $|\grave{\alpha}| < 1$.

\item
Rewrite
\[
-\frac{1}{2}(\rhs_{q,b}^{\ep,0}(\eta,\alpha,\theta)-\rhs_{q,b}^{\ep,0}(\grave{\eta},\alpha,\grave{\theta}))
= \sum_{j=1}^4 \Delta_j + \frac{1}{2}(\eta+\grave{\eta})(\eta-\grave{\eta}),
\]
where
\[
\Delta_1
:= \upsilon_{\ep,q}(\eta-\grave{\eta}),
\]
\[
\Delta_2
:= \alpha(\varphi_{\ep,b}^{\alpha,\theta}-\varphi_{\ep,b}^{\alpha,\grave{\theta}})\upsilon_{\ep,q},
\]
\[
\Delta_3
:= \alpha(\varphi_{\ep,b}^{\alpha,\theta}-\varphi_{\ep,b}^{\alpha,\grave{\theta}})\eta,
\]
and
\[
\Delta_4
:= \alpha\varphi_{\ep,b}^{\alpha,\grave{\theta}}(\eta-\grave{\eta}).
\]
We use the product estimate \eqref{eqn: product est2} and the estimates \eqref{eqn: upsilon-ep} on $\upsilon_{\ep,q}$ to bound
\[
\norm{\Delta_1}_{q,b}
\le C\ep^2\norm{\eta-\grave{\eta}}_{q,b}.
\]
The Lipschitz estimate \eqref{eqn: varphi Lip theta ultimate} and the product estimate \eqref{eqn: product est} in $\H_{q,b}$ give
\[
\norm{\Delta_2}_{q,b}
\le C\ep^2|\alpha||\theta-\grave{\theta}|
\quadword{and}
\norm{\Delta_3}_{q,b}
\le C|\alpha|\norm{\eta}_{q,b}|\theta-\grave{\theta}|
\]
while the mapping estimate \eqref{eqn: varphi map ultimate} and the product estimate again give
\[
\norm{\Delta_4}_{q,b}
\le C|\alpha|\norm{\eta-\grave{\eta}}_{q,b}.
\]
We estimate $(\eta+\grave{\eta})(\eta-\grave{\eta})$ as in \eqref{eqn: eta sq Lip}.
\qedhere
\end{enumerate}
\end{proof}

\subsection{Estimates on $\rhs_{q,b}^{\ep}$}
We develop estimates on the term $\rhs_{q,b}^{\ep}$ defined in \eqref{eqn: rhs-ep}.
These follow from the previously established estimates on the four terms constituting $\rhs_{q,b}^{\ep}$ as proved in Lemma \ref{lem: ultimate rhs0} for $\rhs_{q,b}^{\ep,0}$, Lemma \ref{lem: ultimate fake quadratic} for $\J_{\ep,b,-1}$ and $\J_{\ep,b,0}$, and Lemma \ref{lem: ultimate J1} for $\J_{\ep,b,1}$.

\begin{lemma}
Let $0 < b < b_{\sigma}$, $0 < q < q_{\sigma}$, and $\theta_0 > 0$.
There exists $C > 0$ such that the following hold.

\begin{enumerate}[label={\bf(\roman*)}]

\item
{}
[Mapping estimate]
If $0 < \ep < \ep_{\per}$, $|\alpha| < a_{\per}$, $|\theta| < \theta_0$, and $\eta \in \H_{q,b}$, then
\begin{equation}\label{eqn: ultimate R map}
\norm{\rhs_{q,b}^{\ep}(\eta,\alpha,\theta)}_{q,b}
\le C\big(\ep^4 + \ep|\alpha| + \ep^{-1}|\alpha|^2+\norm{\eta}_{q,b}^2\big).
\end{equation}

\item
{}
[Lipschitz estimate in amplitude]
If $0 < q_1 < q_2 \le (q+q_{\sigma})/2$, $0 < \ep < \ep_{\per}$, $0 \le |\alpha|$, $|\grave{\alpha}| < a_{\per}$, $|\theta| < \theta_0$, and $\eta$, $\grave{\eta} \in \H_{q_2,b}$, then
\begin{equation}\label{eqn: ultimate R Lip alpha}
\norm{\rhs_{q,b}^{\ep}(\eta,\alpha,\theta)-\rhs_{q,b}^{\ep}(\grave{\eta},\grave{\alpha},\theta)}_{q_1,b}
\le \frac{C}{q_2-q_1}\big(\ep + \ep^{-1}(|\alpha|+|\grave{\alpha}|) +\norm{\eta}_{q_2,b}+\norm{\grave{\eta}}_{q_2,b}\big)\big(|\alpha-\grave{\alpha}| + \norm{\eta-\grave{\eta}}_{q_1,b}\big).
\end{equation}

\item
{}
[Lipschitz estimate in phase shift]
If $0 < \ep < \ep_{\per}$, $|\alpha| < a_{\per}$, $0 \le |\theta|$, $|\grave{\theta}| < \theta_0$, and $\eta$, $\grave{\eta} \in \H_{q,b}$, then
\begin{multline}\label{eqn: ultimate R Lip theta}
\norm{\rhs_{q,b}^{\ep}(\eta,\alpha,\theta)-\rhs_{q,b}^{\ep}(\grave{\eta},\alpha,\grave{\theta})}_{q,b}
\le C\big(\ep^2+|\alpha|+\norm{\eta}_{q,b}+\norm{\grave{\eta}}_{q,b}\big)\norm{\eta-\grave{\eta}}_{q,b} \\
+ C\big(\ep^{-1}|\alpha|^2+\ep|\alpha|+|\alpha|\norm{\eta}_{q,b}+|\alpha|\norm{\grave{\eta}}_{q,b}\big)|\theta-\grave{\theta}|.
\end{multline}

\end{enumerate}
\end{lemma}

Perhaps the only delicate part of the proof is the development of the estimate \eqref{eqn: ultimate R Lip alpha}, on which we briefly comment.
The Lipschitz estimates in $\alpha$ for $\J_{\ep,b,-1}$ and $\J_{\ep,b,0}$ in \eqref{eqn: ultimate fake quadratic Lip alpha} and for $\J_{\ep,b,1}$ in \eqref{eqn: ultimate J1 Lip alpha} are measured with decay rate $(q+q_{\sigma})/2$.
Here in \eqref{eqn: ultimate R Lip alpha} we measure those estimates with decay rate $q_1 < (q+q_{\sigma})/2$, and so we just need to use the embedding estimate \eqref{eqn: embed est} on the prior Lipschitz estimates.

\subsection{The proof of the estimate \eqref{eqn: amp PS asymp}}\label{app: amp PS asymp}
Suppose that $\eta \in \E_{q,b}$ and $\alpha$, $\theta \in \R$ satisfy the selection mechanism \eqref{eqn: sel mech} with $\norm{\eta}_{q,b} \le C_0\ep^4$.
We use the expansion \eqref{eqn: iota-chi-ep-theta exp} of $\iota_{\ep}[\chi_{\ep}^{\theta}]$ to rewrite this as
\begin{multline*}
\alpha\left(\frac{\theta}{\ep}+1\right)\sinc(\theta)
= \frac{\tau_{\infty}}{2}e^{b/\ep}\iota_{\ep}[\rhs_{q,b}^{\ep}(\eta,\alpha,\theta)-\Sigma\eta-\alpha\G_{\ep,b}^{\theta}] 
-\frac{\tau_{\infty}}{2}\alpha(\tau_{\infty}\omega_{\ep}\theta+1)\I_{\ep}(\theta) \\
- \alpha(\tau_{\infty}\omega_{\ep}\theta+1)\big(\sinc((\ep\omega_{\ep})\theta)-\sinc(\theta)\big) 
- \alpha\left(\omega_{\ep}-\frac{1}{\ep}\right)\sin(\theta).
\end{multline*}
We then use Lombardi's estimate \eqref{eqn: Lombardi}, the mapping estimate \eqref{eqn: ultimate R map} on $\rhs_{q,b}$, the mapping estimate \eqref{eqn: ultimate G map} on $\G_{\ep}^{\theta}$, the estimates \eqref{eqn: ep-omega-ep bounds}, \eqref{eqn: omega-ep est}, and \eqref{eqn: omega-ep est2} on $\omega_{\ep}$, the estimate \eqref{eqn: I-ep-theta map} on $\I_{\ep}(\theta)$, and the Lipschitz estimate \eqref{eqn: useful sinc Lip} on $\sinc(\cdot)$ to bound the right side of this expression by 
\[
C\big(\ep^4+\ep\alpha+\ep^{-1}\alpha^2\big).
\]
This proves the estimate \eqref{eqn: amp PS asymp}.

There is one small subtlety here: the definition of $\G_{\ep,b}^{\theta}$ in \eqref{eqn: G-ep-theta} reveals that $\G_{\ep}^{\theta} \in \E_{q,(b+b_{\sigma})/2}$.
Consequently, we can use Lombardi's estimate to get the sharper bound
\[
e^{b/\ep}|\iota_{\ep}[\alpha\G_{\ep,b}^{\theta}]|
\le e^{b/\ep}\left(\frac{2}{q}\right)e^{-(b+b_{\sigma})/2\ep}|\alpha|\norm{\G_{\ep,b}^{\theta}}_{q,(b+b_{\sigma})/2}
\le C\ep|\alpha|
\]
for sufficiently small $\ep$.
This keeps the overall estimates in line with the $\O(\ep\alpha)$ term that comes from the mapping estimate \eqref{eqn: ultimate R map} on $\rhs_{q,b}$.

\section{Operator Analysis and Estimates}\label{app: op ests}

The primary (though not sole) goal of this appendix is to prove the invertibility of the operators $1+\B_{\ep}^{-1}\P_{\ep,b}(\theta)\Sigma$ in Theorem \ref{thm: KdV perturbation amp} and $1+\B_{\ep}^{-1}\Pi_{\ep,b}(\theta)\Sigma$ in Theorem \ref{thm: KdV perturbation PS}.
This requires careful analysis of and estimates on the many operators that together build these two central ones.

\subsection{The operator $\Sigma$}
Recall from \eqref{eqn: Sigma} that 
\[
\Sigma{f}
:= 2\sigma{f},
\]
where $\sigma \in \E_{q,b}$ for $0 < q < q_{\sigma} = 1$ and $0 < b < b_{\sigma} = \pi$.
The following is a consequence of the product estimate \eqref{eqn: product est2} and the decay enhancing estimate \eqref{eqn: decay-enhancing est}.

\begin{lemma}\label{lem: Sigma}
Let $0 < b < b_{\sigma}$, $0 < q$, $\grave{q} < q_{\sigma}$.
There is $C > 0$ such that
\begin{equation}\label{eqn: Sigma est}
\norm{\Sigma}_{\H_{q,b} \to \H_{q,b}}
\le C\norm{f}_{q,b}
\quadword{and}
\norm{\Sigma}_{\H_{q,b} \to \H_{q+\grave{q},b}}
\le C.
\end{equation}
\end{lemma}

\subsection{The operator $\V_{\omega}$}\label{app: V-omega}
Let $\omega$, $b$, $q > 0$.
Recall from \eqref{eqn: V-omega} that 
\begin{equation}\label{eqn: V-omega app1}
(\V_{\omega}f)(z)
:= \int_{[z,\infty)} \sin(\omega(z-w))f(w) \dw.
\end{equation}
We first show that when $f$ meets the solvability condition, we have a ``reflection'' formula for $\V_{\omega}$, which will be helpful in various subsequent steps.

\begin{lemma}
Let $b$, $q > 0$ and $\omega \in \R$.
Let $f \in \H_{q,b}$ with $\medint_{-\infty}^{\infty} f(x)e^{\pm{i}\omega{x}} \dx = 0$.
Then
\begin{equation}\label{eqn: V flip aux}
\int_{[z,\infty)} \sin(\omega(z-w))f(w) \dw
= -\int_{(-\infty,z]} \sin(\omega(z-w))f(w) \dw
\end{equation}
for all $z \in \U_b$.
\end{lemma}

\begin{proof}
We prove this for real $z$, from which the identity \eqref{eqn: V flip aux} follows by analyticity.
Let $z= x \in \R$.
A trigonometric addition formula gives
\[
\int_x^{\infty} \sin(\omega(x-s))f(s)\ds
= \sin(\omega{x})\int_x^{\infty} \cos(\omega{s})f(s) \ds
- \cos(\omega{x})\int_x^{\infty} \sin(\omega{s}f(s) \ds.
\]
Adding zero, this becomes
\begin{multline*}
\int_x^{\infty} \sin(\omega(x-s))f(s)\ds
= \sin(\omega{x})\int_{-\infty}^{\infty} \cos(\omega{s})f(s) \ds - \sin(\omega{x})\int_{-\infty}^x \cos(\omega{s})f(s) \ds \\
- \cos(\omega{x})\int_{-\infty}^{\infty} \sin(\omega{s})f(s) \ds + \cos(\omega{x})\int_{-\infty}^x \sin(\omega{s})f(s) \ds.
\end{multline*}
Since $\medint_{-\infty}^{\infty} f(x)e^{\pm{i}\omega{x}} \dx = 0$ (here it is important to include both cases of $\pm{i}\omega$) if and only if 
\[
\int_{-\infty}^{\infty} \cos(\omega{s})f(s) \ds
= \int_{-\infty}^{\infty} \sin(\omega{s})f(s) \ds
= 0,
\]
we find
\[
\int_x^{\infty} \sin(\omega(x-s))f(s)\ds
= -\sin(\omega{x})\int_{-\infty}^x \cos(\omega{s})f(s) \ds + \cos(\omega{x})\int_{-\infty}^x \sin(\omega{s})f(s) \ds.
\]
The same trigonometric addition formula yields
\[
\int_x^{\infty} \sin(\omega(x-s))f(s) \ds
= \int_{-\infty}^x \sin(\omega(s-x))f(s) \ds,
\]
and this is \eqref{eqn: V flip aux} for $x \in \R$.
\end{proof}

\begin{theorem}\label{thm: V-omega}
Let $b$, $q > 0$ and $\omega \in \R$.

\begin{enumerate}[label={\bf(\roman*)}]

\item
If $f \in \H_{q,b}$ with $\medint_{-\infty}^{\infty} f(x)e^{i\omega{x}} \dx = 0$, then $\V_{\omega}f \in \H_{q,b}$ with 
\begin{equation}\label{eqn: V-omega est}
\norm{\V_{\omega}{f}}_{q,b}
\le \frac{\norm{f}_{q,b}}{q}.
\end{equation}

\item
If $f \in \E_{q,b}$ with $\medint_{-\infty}^{\infty} f(x)e^{i\omega{x}} \dx = 0$, then $\V_{\omega}{f} \in \E_{q,b}$.
\end{enumerate}
\end{theorem}

\begin{proof}
\begin{enumerate}[label={\bf(\roman*)}]

\item
Let $\re(z) \ge 0$.
Using the formula $(\V_{\omega}f)(z) = \medint_{[z,\infty)} \sin(\omega(z-w))f(w) \dw$, we have
\begin{multline*}
e^{q|\re(z)|}|(\V_{\omega}{f})(z)|
\le e^{q\re(z)}\int_{\re(z)}^{\infty} \big|\sin\big(\omega(z-(s+i\im(z))\big)f(s+i\im(z))\big| \ds \\
\le e^{q\re(z)}\norm{f}_{q,b}\int_{\re(z)}^{\infty} |\sin(\omega(\re(z)-s))|e^{-q|s|} \dt
\le e^{q\re(z)}\norm{f}_{q,b}\int_{\re(z)}^{\infty} e^{-q|s|} \ds.
\end{multline*}
Here the parametrization of the line integral defining $\V_{\omega}f$ resulted in the real argument $\omega(\re(z)-s))$ for the sine, thus $|\sin(\omega(\re(z)-s))| \le 1$.
This cancelation prevents the exponentially large estimate that the sine otherwise has on all of $\U_b$; see \cite[App.\@ 7.E]{lombardi} for a related discussion of how the convolution kernel yields better estimates than one might expect.

Since $\re(z) \ge 0$, we have
\begin{equation}\label{eqn: V-omega aux constant sign}
\int_{\re(z)}^{\infty} e^{-q|s|} \ds
= \int_{\re(z)}^{\infty} e^{-qs} \ds
= \frac{e^{-q\re(z)}}{q}.
\end{equation}
Thus 
\[
e^{q|\re(z)|}|(\V_{\omega}{f})(z)|
\le e^{q\re(z)}\norm{f}_{q,b}\frac{e^{-q\re(z)}}{q}
= \frac{\norm{f}_{q,b}}{q}.
\]

Here it was helpful that the interval of integration $[\re(z),\infty)$ was contained in $[0,\infty)$, which led to \eqref{eqn: V-omega aux constant sign}.
For $\re(z) < 0$, we use the same style of estimates on the formula 
\[
(\V_{\omega}f)(z) 
= \int_{(-\infty,z]} \sin(\omega(z-w))f(w) \dw
\]
from \eqref{eqn: V flip aux}.
With this version of $\V_{\omega}f$, we integrate over $(-\infty,\re(z)]$, which is contained in $(-\infty,0]$, and that leads to an estimate similar to \eqref{eqn: V-omega aux constant sign}.

\item
We prove that $(\V_{\omega}f)(-x) = (\V_{\omega}f)(x)$ for $x \in \R$.
By analyticity, this proves the evenness of $\V_{\omega}f$ on $\U_b$.
For $x \in \R$, we compute
\begin{align*}
(\V_{\omega}f)(-x)
&= \int_{-x}^{\infty} \sin(\omega(-x-s))f(s) \ds \\
&= \int_{-\infty}^x \sin(\omega(-x+t))f(-t) \dt \text{ by substitution } \\
&= -\int_{-\infty}^x \sin(\omega(x-t))f(t) \dt \text{ since } f \text{ is even } \\
&= (\V_{\omega}f)(x) \text{ by \eqref{eqn: V flip aux}}.
\qedhere
\end{align*}
\end{enumerate}
\end{proof}

\subsection{The operator $\L_{\mu}$}\label{app: L-mu}
Let $b$, $q > 0$, $\mu \in \R$, and $f \in \H_{q,b}$.
Recall that from \eqref{eqn: L-mu} that
\[
\L_{\mu}f
= e^{-\mu{z}}\int_{(-\infty,z]} e^{\mu{w}}f(w) \dw
+ e^{\mu{z}}\int_{[z,\infty)} e^{-\mu{w}}f(w) \dw.
\]
We first show that the two terms above in the definition of $\L_{\mu}$ are ``reflections'' of each other.

\begin{lemma}
Let $b$, $q > 0$ and $\mu \in \R$.
Then
\[
R\L_{\mu}^-
= \L_{\mu}^+R,
\]
where $R$ is the reflection operator defined in \eqref{eqn: reflection}.
\end{lemma}

\begin{proof}
We start with $x \in \R$ and compute
\[
(R\L_{\mu}^-f)(x)
= e^{\mu{x}}\int_{-\infty}^x e^{\mu{s}}f(s) \ds
= e^{\mu{x}}\int_{-x}^{\infty} e^{-\mu{s}}f(-s) \ds
= (\L_{\mu}^+Rf)(x).
\]
Since $\L_{\mu}^{\pm}f$ are analytic, the result follows on $\U_b$.
\end{proof}

A useful consequence of this lemma is that 
\begin{equation}\label{eqn: L refl}
\L_{\mu}^+f
= \L_{\mu}^+R^2f
= R\L_{\mu}^-Rf,
\end{equation}
and so we can extract results for $\L_{\mu}^+$ from knowledge of $\L_{\mu}^-$.
In particular, we have
\begin{equation}\label{eqn: L-mu and R}
\L_{\mu}
= \L_{\mu}^-+\L_{\mu}^+
= \L_{\mu}^-+R\L_{\mu}^-R
\end{equation}

\begin{theorem}\label{thm: L-mu mapping}
Let $\mu$, $b$, $q > 0$.

\begin{enumerate}[label={\bf(\roman*)},ref={(\roman*)}]

\item
If $0 < q < \mu$, then
\begin{equation}\label{eqn: L-mu est proto}
\norm{\L_{\mu}f}_{\H_{q,b} \to \H_{q,b}}
\le \frac{4}{\mu+q}+\frac{2}{\mu-q}.
\end{equation}

\item
If $0 < q < 1$, there are $C$, $\ep_{\L} > 0$ such that if $0 < \ep < \ep_{\L}$, then
\begin{equation}\label{eqn: L-mu est}
\norm{\L_{\mu_{\ep}}}_{\H_{q,b} \to \H_{q,b}}
\le C.
\end{equation}

\item
If $f \in \E_{q,b}$, then $\L_{\mu}f \in \E_{q,b}$.
\end{enumerate}
\end{theorem}

\begin{proof}

\begin{enumerate}[label={\bf(\roman*)}]

\item
First we show that 
\begin{equation}\label{eqn: L-mu-minus aux}
\norm{\L_{\mu}^-f}_{q,b}
\le \left(\frac{2}{\mu+q}+\frac{1}{\mu-q}\right)\norm{f}_{q,b}
\end{equation}
for any $f \in \H_{q,b}$.
The estimate for all of $\L_{\mu}$ then follows from the identity \eqref{eqn: L refl} and the formula \eqref{eqn: L-mu and R}.

Let $z \in \U_b$, so
\[
e^{q|\re(z)|}|(\L_{\mu}^-f)(z)|
= e^{q|\re(z)|}|e^{-\mu{z}}|\left|\int_{(-\infty,z]} e^{\mu{w}}f(w) \dw\right|.
\]
Assume first that $\re(z) \le 0$.
We compute
\[
e^{q|\re(z)|}|e^{-\mu{z}}|
= e^{-(\mu+q)\re(z)},
\]
and we estimate
\begin{align*}
\left|\int_{(-\infty,z]} e^{\mu{w}}f(w) \dw\right|
&= \left|\int_{-\infty}^{\re(z)} e^{\mu(s+i\im(z))}f(s+i\im(z)) \ds\right| \\
&\le \int_{-\infty}^{\re(z)} e^{\mu{s}}e^{q|\re(s+i\im(z))|}\norm{f}_{q,b} \ds \\
&= \norm{f}_{q,b}\int_{-\infty}^{\re(z)} e^{(\mu+q)s} \ds \\
&= \norm{f}_{q,b}\frac{e^{(\mu+q)\re(z)}}{\mu+q}.
\end{align*}
For $\re(z) \le 0$, we conclude
\[
e^{q|\re(z)|}|(\L_{\mu}^-f)(z)|
\le e^{-(\mu+q)\re(z)}\norm{f}_{q,b}\frac{e^{(\mu+q)\re(z)}}{\mu+q}
= \frac{\norm{f}_{q,b}}{\mu+q}.
\]

For $\re(z) > 0$, we have instead
\[
e^{q|\re(z)|}|e^{-\mu{z}}|
= e^{(q-\mu)\re(z)}
\]
and 
\begin{align*}
\left|\int_{(-\infty,z]} e^{\mu{w}}f(w) \dw\right|
&\le \int_{-\infty}^0 e^{\mu{s}}|f(s+i\im(z))| \ds + \int_0^{\re(z)} e^{\mu{s}}|f(s+i\im(z))| \ds \\
&\le \left(\int_{-\infty}^0 e^{(\mu+q)s} \ds + \int_0^{\re(z)} e^{(\mu-q)s} \ds\right)\norm{f}_{q,b} \\
&= \left(\frac{1}{\mu+q}+\frac{e^{(\mu-q)\re(z)}-1}{\mu-q}\right)\norm{f}_{q,b}.
\end{align*}
Consequently, for $\re(z) > 0$, we conclude
\begin{multline*}
e^{q|\re(z)|}|(\L_{\mu}^-f)(z)|
\le e^{(q-\mu)\re(z)}\left(\frac{1}{\mu+q}+\frac{e^{(\mu-q)\re(z)}-1}{\mu-q}\right)\norm{f}_{q,b} \\
= \left(\frac{e^{(q-\mu)\re(z)}}{\mu+q}+\frac{1-e^{(q-\mu)\re(z)}}{\mu-q}\right)\norm{f}_{q,b}.
\end{multline*}
Here it is important that $0 < q < \mu$ and $\re(z) > 0$, as we obtain $0 < e^{(q-\mu)\re(z)} < 1$, and so
\[
e^{q|\re(z)|}|(\L_{\mu}^-f)(z)|
\le \left(\frac{1}{\mu+q}+\frac{1}{\mu-q}\right)\norm{f}_{q,b}.
\]
This establishes the estimate \eqref{eqn: L-mu-minus aux}.

\item
The estimate \eqref{eqn: mu-ep lim} on $\mu_{\ep}$ provides $\ep_{\L} > 0$ such that if $0 < \ep < \ep_{\L}$, then $(1+q)/2 < \mu_{\ep} < 1$.
Then $0 < (1-q)/2 < \mu_{\ep}-q$, and so \eqref{eqn: L-mu est} follows from \eqref{eqn: L-mu est proto}.

\item
Suppose that $f \in \E_{q,b}$, so $Rf = f$.
We use the formula \eqref{eqn: L-mu and R} to compute
\[
R\L_{\mu}f
= R\L_{\mu}^-f + \L_{\mu}^-Rf
= R\L_{\mu}^-Rf + \L_{\mu}^-f
= \L_{\mu}f,
\]
and so $\L_{\mu}f$ is even.
\qedhere
\end{enumerate}
\end{proof}

Next, we develop a precise quantitative relationship between $\L_{\mu}$ and $\L_{\grave{\mu}}$.
This will be the key to proving the operator norm convergence $\norm{\L_{\mu_{\ep}} - \L_1}_{\E_{\grave{q},b}\to\E_{q,b}} \to 0$ as $\ep \to 0$ when $0 < q < \grave{q} < 1$, which is a desired auxiliary result for our proofs of the convergence of $1+\B_{\ep}^{-1}\P_{\ep,b}(\theta)\Sigma$ and $1+\B_{\ep}^{-1}\Pi_{\ep,b}(\theta)\Sigma$ to $1+\L_1\Sigma$.
Because we are only interested in proving this convergence for even functions, we pose the following auxiliary result only for $z \in \C$ with nonnegative real part.

\begin{lemma}\label{lem: L-mu perturbation}
There is $C > 0$ such that if $\mu$, $\grave{\mu}$, $q$, $\grave{q}$, $b > 0$ with $0 < q < \grave{q} < \min\{\mu,\grave{\mu}\}$, $f \in \H_{\grave{q},b}$, and $z \in \U_b$ with $\re(z) \ge 0$, then
\begin{equation}\label{eqn: L-mu perturbation}
\begin{aligned}
e^{q\re(z)}|((\L_{\mu}-\L_{\grave{\mu}})f)(z)|
&\le \frac{C}{\mu-\grave{q}}|z|\big(e^{(q-\min\{\mu,\grave{\mu}\})\re(z)} + e^{((q-\grave{q})+|\mu-\grave{\mu}|)\re(z)}\big)|\mu-\grave{\mu}|\norm{f}_{\grave{q},b} \\
&+ C\left(\frac{1}{\grave{q}}\left(\frac{1}{\grave{q}}+b\right)+\frac{1}{\mu-\grave{q}}\left(\frac{1}{\mu-\grave{q}}+b\right)\right)|\mu-\grave{\mu}|\norm{f}_{\grave{q},b} \\
&+ \frac{C}{\mu-\grave{q}}\left(\frac{1}{\mu-\grave{q}}+b\right)e^{((q-\grave{q})+|\mu-\grave{\mu}|)\re(z)}|\mu-\grave{\mu}|\norm{f}_{\grave{q},b}.
\end{aligned}
\end{equation}
\end{lemma}

\begin{proof}
Since 
\[
\L_{\mu}-\L_{\grave{\mu}}
= (\L_{\mu}^- - \L_{\grave{\mu}}^-)
+ R(\L_{\mu}^- - \L_{\grave{\mu}}^-)R
\]
by \eqref{eqn: L-mu and R}, we just estimate $\L_{\mu}^- - \L_{\grave{\mu}}^-$.

\begin{enumerate}[label={\bf\arabic*.}]

\item
{\it{Preliminary estimates.}}
For $w \in \C$ and $\mu$, $\grave{\mu} > 0$, the mean value theorem gives
\[
|e^{\mu{w}}-e^{\grave{\mu}w}|
\le |w||\mu-\grave{\mu}|\sup_{0 \le t \le 1} |e^{w((1-t)\grave{\mu}+t\mu)}|
= |w||\mu-\grave{\mu}|\sup_{0 \le t \le 1} e^{\re(w)((1-t)\grave{\mu}+t\mu)}.
\]
Since
\[
\min\{\mu,\grave{\mu}\}
\le (1-t)\grave{\mu}+t\mu
\le \max\{\mu,\grave{\mu}\},
\]
if $\re(w) > 0$, then 
\begin{equation}\label{eqn: exp Lip +}
|e^{\mu{w}}-e^{\grave{\mu}w}|
\le |w||\mu-\grave{\mu}|e^{\max\{\mu,\grave{\mu}\}\re(w)},
\end{equation}
while if $\re(w) < 0$, then
\begin{equation}\label{eqn: exp Lip -}
|e^{\mu{w}}-e^{\grave{\mu}w}|
\le |w||\mu-\grave{\mu}|e^{\min\{\mu,\grave{\mu}\}\re(w)}.
\end{equation}

We will also need the bounds
\begin{equation}\label{eqn: mu mu min max}
\mu - \min\{\mu,\grave{\mu}\} 
\le |\mu-\grave{\mu}|
\quadword{and}
\max\{\mu,\grave{\mu}\}-\mu
\le |\mu-\grave{\mu}|.
\end{equation}

\item
{\it{Estimates for $\L_{\mu}^--\L_{\grave{\mu}}^-$.}}
We bound
\[
e^{q|\re(z)|}\big|(\big(\L_{\mu}^- -\L_{\grave{\mu}}^-)f\big)(z)| 
\le \Delta_1+\Delta_2,
\]
where now
\[
\Delta_1
:= e^{q|\re(z)|}|e^{-\mu{z}}-e^{-\grave{\mu}z}|\int_{-\infty}^{\re(z)} \big|e^{\mu(s+i\im(z))}f(s+i\im(z))\big| \ds
\]
and
\[
\Delta_2
:= e^{q|\re(z)|}|e^{-\grave{\mu}z}|\int_{-\infty}^{\re(z)} \big|e^{\mu(s+i\im(z))}-e^{\grave{\mu}(s+i\im(z))}\big||f(s+i\im(z))| \ds.
\]

\item
{\it{Estimates on $\Delta_1$.}}
Since $\re(z) \ge 0$, we estimate the integral in $\Delta_1$ as
\begin{multline*}
\int_{-\infty}^{\re(z)} \big|e^{\mu(s+i\im(z))}f(s+i\im(z))\big| \ds
\le \int_{-\infty}^0 e^{\mu{s}}e^{\grave{q}s}\norm{f}_{\grave{q},b} \ds + \int_0^{\re(z)} e^{\mu{s}}e^{-\grave{q}s}\norm{f}_{\grave{q},b} \ds \\
\le \left(\frac{1}{\mu+\grave{q}} + \frac{e^{(\mu-\grave{q})\re(z)}+1}{\mu-\grave{q}}\right)\norm{f}_{\grave{q},b}.
\end{multline*}
Then we use \eqref{eqn: exp Lip -} with $w = -z$ and $\re(z) \ge 0$ to estimate the exponential difference in $\Delta_1$ and conclude
\begin{align*}
\Delta_1
&\le e^{q\re(z)}
|z|
|\mu-\grave{\mu}|
e^{-\min\{\mu,\grave{\mu}\}\re(z)}
\left(\frac{1}{\mu+\grave{q}} + \frac{e^{(\mu-\grave{q})\re(z)}+1}{\mu-\grave{q}}\right)\norm{f}_{\grave{q},b} \\
\\
&\le |z|\left(
\frac{e^{(q-\min\{\mu,\grave{\mu}\})\re(z)}}{\mu+\grave{q}}
+ \frac{e^{((q-\grave{q})+(\mu-\min\{\mu,\grave{\mu}\}))\re(z)}+e^{(q-\min\{\mu,\grave{\mu}\})\re(z)}}{\mu-\grave{q}}
\right)
|\mu-\grave{\mu}|\norm{f}_{\grave{q},b}.
\end{align*}
When we use \eqref{eqn: mu mu min max}, this simplifies slightly to
\begin{equation}\label{eqn: Delta1 L-mu mu conv}
\Delta_1
\le \left(\frac{1}{\mu+\grave{q}}+\frac{1}{\mu-\grave{q}}\right)|z|\big(e^{(q-\min\{\mu,\grave{\mu}\})\re(z)} + e^{((q-\grave{q})+|\mu-\grave{\mu}|)\re(z)}\big)|\mu-\grave{\mu}|\norm{f}_{\grave{q},b}.
\end{equation}

\item
{\it{Estimates on $\Delta_2$.}}
For $\Delta_2$, we first compute the prefactor on the integral as 
\[
e^{q|\re(z)|}|e^{-\mu{z}}|
= e^{q\re(z)}e^{-\mu\re(z)}
= e^{(q-\mu)\re(z)},
\]
and so 
\[
\Delta_2
\le \Delta_{21}+\Delta_{22},
\]
where
\[
\Delta_{21}
:= e^{(q-\mu)\re(z)}\int_{-\infty}^0 \big|e^{\mu(s+i\im(z))}-e^{\grave{\mu}(s+i\im(z))}\big||f(s+i\im(z))| \ds
\]
and
\[
\Delta_{22}
:= e^{(q-\mu)\re(z)}\int_0^{\re(z)} \big|e^{\mu(s+i\im(z))}-e^{\grave{\mu}(s+i\im(z))}\big||f(s+i\im(z))| \ds.
\]
We estimate the difference of exponentials in $\Delta_{21}$ using \eqref{eqn: exp Lip -} with $w = s+i\im(z)$ and $\re(w) = s \le 0$ to find
\begin{equation}\label{eqn: Delta21 L-mu mu conv}
\begin{aligned}
\Delta_{21}
&\le e^{(q-\mu)\re(z)}|\mu-\grave{\mu}|\int_{-\infty}^0 |s+i\im(z)|e^{\min\{\mu,\grave{\mu}\}s}e^{\grave{q}s}\norm{f}_{\grave{q},b} \ds \\
&\le e^{(q-\mu)\re(z)}\norm{f}_{\grave{q},b}|\mu-\grave{\mu}|\int_{-\infty}^0 (|s|+b)e^{\grave{q}s}\ds \text{ since } 0 < e^{\min\{\mu,\grave{\mu}\}s} \le 1 \text{ for } s \le 0 \\
&= \frac{1}{\grave{q}}\left(\frac{1}{\grave{q}}+b\right)e^{(q-\mu)\re(z)}|\mu-\grave{\mu}|\norm{f}_{\grave{q},b} \\
&\le \frac{1}{\grave{q}}\left(\frac{1}{\grave{q}}+b\right)|\mu-\grave{\mu}|\norm{f}_{\grave{q},b} \text{ since } q-\mu < 0 \text{ and } \re(z) \ge 0.
\end{aligned}
\end{equation}

We estimate the difference of exponentials in $\Delta_{22}$ using \eqref{eqn: exp Lip +} again with $w = s+i\im(z)$ and $\re(w) = s \ge 0$ to find
\begin{multline*}
\Delta_{22}
\le e^{(q-\mu)\re(z)}|\mu-\grave{\mu}|\int_0^{\re(z)} |s+i\im(z)|e^{\max\{\mu,\grave{\mu}\}s}e^{-\grave{q}s}\norm{f}_{\grave{q},b} \ds \\
\le e^{(q-\mu)\re(z)}|\mu-\grave{\mu}|\norm{f}_{\grave{q},b}\int_0^{\re(z)} (s+b)e^{(\max\{\mu,\grave{\mu}\}-\grave{q})s} \ds,
\end{multline*}
where
\begin{multline*}
\int_0^{\re(z)} (s+b)e^{(\max\{\mu,\grave{\mu}\}-\grave{q})s} \ds 
= \frac{(\max\{\mu,\grave{\mu}\}-\grave{q})\re(z)e^{(\max\{\mu,\grave{\mu}\}-\grave{q})\re(z)}+e^{(\max\{\mu,\grave{\mu}\}-\grave{q})\re(z)}+1}{(\max\{\mu,\grave{\mu}\}-\grave{q})^2} \\
+ \frac{b(e^{(\max\{\mu,\grave{\mu}\}-\grave{q})\re(z)}+1)}{\max\{\mu,\grave{\mu}\}-\grave{q}}.
\end{multline*}
When we multiply by $e^{(q-\mu)\re(z)} \le e^{(q-\min\{\mu,\grave{\mu}\})\re(z)}$ and use \eqref{eqn: mu mu min max}, we obtain
\begin{multline}\label{eqn: Delta22 L-mu mu conv}
\Delta_{22}
\le \frac{1}{\max\{\mu,\grave{\mu}\}-\grave{q}}\left(\frac{1}{\max\{\mu,\grave{\mu}\}-\grave{q}}+b\right)\big(e^{((q-\grave{q})+|\mu-\grave{\mu}|)\re(z)}+e^{(q-\min\{\mu,\grave{\mu}\})\re(z)}\big)|\mu-\grave{\mu}|\norm{f}_{\grave{q},b} \\
+ \frac{|z|e^{((q-\grave{q})+|\mu-\grave{\mu}|)\re(z)}}{\max\{\mu,\grave{\mu}\}-\grave{q}}.
\end{multline}

\item
{\it{Conclusions.}}
We collect the estimates on $\Delta_1$ from \eqref{eqn: Delta1 L-mu mu conv} and $\Delta_2$ from \eqref{eqn: Delta21 L-mu mu conv} and \eqref{eqn: Delta22 L-mu mu conv} to obtain the desired bounds.
In the process, we use estimates like
\[
\max\left\{\frac{1}{\mu+\grave{q}}, \frac{1}{\max\{\mu,\grave{\mu}\}-\grave{q}}\right\} \le \frac{1}{\mu-\grave{q}}
\quadword{and}
e^{(q-\min\{\mu,\grave{\mu}\})\re(z)} \le 1 \text{ for } \re(z) \ge 0
\]
to simplify the bounds.
\qedhere
\end{enumerate}
\end{proof}

\begin{theorem}\label{thm: L-mu op norm}
Let $b > 0$ and $0 < q < \grave{q} < 1$.
There are $C$, $\ep_{\L} > 0$ such that if $0 < \ep < \ep_{\L}$, then
\begin{equation}\label{eqn: L-mu op norm perturb}
\norm{\L_{\mu_{\ep}}-\L_1}_{\E_{\grave{q},b} \to \E_{q,b}}
\le C\ep^2.
\end{equation}
\end{theorem}

\begin{proof}
We use Lemma \ref{lem: L-mu perturbation} with $\mu = \mu_{\ep}$, $\grave{\mu} = 1$, and $q$, $\grave{q}$, and $b$ as given here.
Since $\tst{\lim_{\ep \to 0} \mu_{\ep} = 1}$ by Lemma \ref{lem: mu-ep} and since $0 < \grave{q} < 1$, there is $\ep_{\grave{q}} > 0$ such that if $0 < \ep < \ep_{\grave{q}}$, then $(\grave{q}+1)/2 < \mu_{\ep}$.
And since $\min\{\mu_{\ep},1\} = \mu_{\ep}$ by Lemma \ref{lem: mu-ep}, we really have $0 < q < \grave{q} < \min\{\mu_{\ep},1\}$, as Lemma \ref{lem: L-mu perturbation} demands.

Now we consider the three terms in \eqref{eqn: L-mu perturbation}.
Our choices above give
\[
\mu_{\ep}-\grave{q}
> \frac{\grave{q}+1}{2}-\grave{q}
= \frac{1-\grave{q}}{2}
> 0,
\]
and $|\mu_{\ep}-1| < \ep^2$ by Lemma \ref{lem: mu-ep}.
Consequently, there is $C > 0$ such that if $0 < \ep < \ep_{\grave{q}}$, $f \in \E_{q,b}$, and $\re(z) \ge 0$, then
\begin{multline}\label{eqn: L-mu op norm perturb aux1}
e^{q\re(z)}\big|\big((\L_{\mu_{\ep}}-\L_1)f\big)(z)\big|
\le C|z|\big(e^{(q-\mu_{\ep})\re(z)} + e^{((q-\grave{q})+\ep^2)\re(z)}\big)\ep^2\norm{f}_{\grave{q},b} 
+ C\ep^2\norm{f}_{\grave{q},b} \\
+ Ce^{((q-\grave{q})+\ep^2)\re(z)}\ep^2\norm{f}_{\grave{q},b}.
\end{multline}

Next, for $0 < \ep < \ep_{\grave{q}}$, we have
\[
q-\mu_{\ep}
< q - \frac{\grave{q}+1}{2}
= \frac{(q-\grave{q})+(q-1)}{2}
< 0.
\]
This ensures that 
\begin{equation}\label{eqn: L-mu op norm perturb aux2}
\sup_{\substack{z \in \U_b \\ 0 < \ep < \ep_{\grave{q}}}} |z|e^{(q-\mu_{\ep})|\re(z)|}
< \infty.
\end{equation}
If we take 
\[
\ep_{\L}
:= \min\left\{\ep_{\grave{q}}, \sqrt{\frac{\grave{q}-q}{2}}\right\},
\]
then
\begin{equation}\label{eqn: L-mu op norm perturb aux3}
\sup_{\substack{z \in \U_b \\ 0 < \ep < \ep_{\L}}} e^{((q-\grave{q})+\ep^2)\re(z)}
< \infty.
\end{equation}

We combine \eqref{eqn: L-mu op norm perturb aux1}, \eqref{eqn: L-mu op norm perturb aux2}, and \eqref{eqn: L-mu op norm perturb aux3} to conclude that 
\[
e^{q\re(z)}\big|\big((\L_{\mu_{\ep}}-\L_1)f\big)(z)\big|
\le C\ep^2\norm{f}_{\grave{q},b}
\]
for $0 < \ep < \ep_{\L}$, $f \in \E_{\grave{q},b}$, and $z \in \U_b$ with $\re(z) \ge 0$.
Since $(\L_{\mu_{\ep}}-\L_1)f$ is even, the estimate \eqref{eqn: L-mu op norm perturb} follows.
\end{proof}

Finally, we estimate the action of just $\L_1$ on a particular kind of function.
This is just a careful integration by parts argument.
A related technique appears in \cite[App.\@ A.11]{faver-wright}.

\begin{lemma}\label{lem: ibp aux for L1}
Let $0 < q < q_{\sigma}$, $0 < b < b_{\sigma}$, and $0 < \theta_0 < \pi$.
There are $C$, $\ep_{\L_1} > 0$ such that 
\begin{equation}\label{eqn: ibp aux for L1}
\norm{\L_1\big(\cos(\omega\Tsf_{\ep\theta}(\cdot))f\big)}_{q,b} 
\le Ce^{b|\omega|}|\omega|^{-1}\big(\norm{f}_{q,b}+\norm{f'}_{q,b}\big)
\end{equation}
for all $\omega \ne 0$, $0 < \ep < \ep_{\L_1}$, $|\theta| < \theta_0$, and $f \in \E_{q,b}^1$.
\end{lemma}

\begin{proof}
Let 
\[
\ep_{\L_1}
:= \frac{1}{2\theta_0\tau_{\infty}\norm{\sigma}_{0,b}}.
\]
Then for $0 < \ep < \ep_{\L_1}$, $z \in \U_b$, and $|\theta| < \theta_0$, we have
\begin{equation}\label{eqn: Tsf deriv lower bound}
|\Tsf_{\ep\theta}'(z)|
= |1+\ep\theta\tau_{\infty}\sigma(z)|
\ge 1-\ep_{\L_1}\theta_0\tau_{\infty}\norm{\sigma}_{0,b}
= \frac{1}{2}.
\end{equation}
We can then integrate by parts to compute first
\begin{align*}
\big(\L_1^-(e^{\pm{i}\omega\Tsf_{\theta}(\cdot)}f)\big)(x)
&= e^{-x}\int_{-\infty}^x \frac{e^sf(s)}{\pm{i\omega}\Tsf_{\ep\theta}'(s)}[\pm{i}\omega\Tsf_{\ep\theta}'(s)e^{\pm{i\omega}\Tsf_{\ep\theta}(s)}] \ds \\
\\
&= \left(\pm\frac{1}{i\omega}\right)\frac{f(x)e^{\pm{i}\omega\Tsf_{\ep\theta}(x)}}{\Tsf_{\ep\theta}'(x)} \\
&\mp\left(\frac{1}{i\omega}\right)e^{-x}\int_{-\infty}^x e^s\left[f(s)\left(\frac{\Tsf_{\ep\theta}'(s)-\Tsf_{\theta}''(s))}{(\Tsf_{\ep\theta}'(s))^2}\right)\right]e^{\pm{i}\omega\Tsf_{\ep\theta}(s)} \ds \\
&\mp\left(\frac{1}{i\omega}\right)e^{-x}\int_{-\infty}^x e^s\left[f'(s)\left(\frac{1}{\Tsf_{\ep\theta}'(s)}\right)\right]e^{\pm{i}\omega\Tsf_{\ep\theta}(s)}  \ds.
\end{align*}
Put
\[
Q_{1,\ep}^{\theta}(z) := \frac{\Tsf_{\ep\theta}'(z)-\Tsf_{\ep\theta}''(z))}{(\Tsf_{\ep\theta}'(z))^2}
\quadword{and}
Q_{2,\ep}^{\theta}(z) := \frac{1}{\Tsf_{\ep\theta}'(z)}.
\]
to conclude
\begin{multline*}
\big(\L_1^-\big(\cos(\omega\Tsf_{\ep\theta}(\cdot))f\big)\big)(x)
= \frac{\big(\L_1^-(e^{i\omega\Tsf_{\ep\theta}(\cdot)}f)\big)(x)+\big(\L_1^-(e^{-i\omega\Tsf_{\ep\theta}(\cdot)}f)\big)(x)}{2} \\
= \frac{1}{\omega}\sin(\omega\Tsf_{\ep\theta}(x))Q_{2,\ep}^{\theta}(x)f(x) 
-\frac{1}{\omega}\big(\L_1^-\big(\sin(\omega\Tsf_{\ep\theta}(\cdot))Q_{1,\ep}^{\theta}f\big)\big)(x) 
-\frac{1}{\omega}\big(\L_1^-\big(\sin(\omega\Tsf_{\theta}(\cdot))Q_{2,\ep}^{\theta}f'\big)\big)(x).
\end{multline*}
Since this result is true for all $x \in \R$, by analyticity we have
\[
\L_1^-\big(\cos(\omega\Tsf_{\ep\theta}(\cdot))f\big)
= \frac{1}{\omega}\sin(\omega\Tsf_{\ep\theta}(\cdot))Q_{2,\ep}^{\theta}f
-\frac{1}{\omega}\L_1^-\big(\sin(\omega\Tsf_{\ep\theta}(\cdot))Q_{1,\ep}^{\theta}f\big)
-\frac{1}{\omega}\L_1^-\big(\sin(\omega\Tsf_{\ep\theta}(\cdot))Q_{2,\ep}^{\theta}f'\big)
\]
on all of $\U_b$.

It follows from the estimate \eqref{eqn: Tsf deriv lower bound} on $\Tsf_{\ep\theta}'$ that for $0 < \ep < \ep_{\L_1}$ and $0 \le \theta < \theta_0$, we have $Q_{j,\ep}^{\theta} \in \H_{q,b}$ with
\[
\sup_{\substack{0 < \ep < \ep_{\L_1} \\ 0 \le \theta < \theta_0}} \norm{Q_{j,\ep}^{\theta}}_{q,b}
< \infty.
\]
The estimate \eqref{eqn: L-mu-minus aux} on $\L_1^-$, the product estimate \eqref{eqn: product est} in $\H_{q,b}$, and the bound \eqref{eqn: sin cos Hqb} then give
\[
e^{q|\re(z)|}\big|\big(\L_1^-\big(\cos(\omega\Tsf_{\ep\theta}(\cdot))f\big)\big)(z)\big|
\le \frac{Ce^{b|\omega|}}{|\omega|}(\norm{f}_{q,b}+\norm{f'}_{q,b})
\]
for $\re(z) \ge 0$.
It follows from the formula \eqref{eqn: L-mu and R} that 
\[
e^{q|\re(z)|}\big|\big(\L_1\big(\cos(\omega\Tsf_{\ep\theta}(\cdot))f\big)\big)(z)\big|
\le \frac{Ce^{b|\omega|}}{|\omega|}(\norm{f}_{q,b}+\norm{f'}_{q,b})
\]
for $\re(z) \ge 0$.
The estimate \eqref{eqn: ibp aux for L1} then holds because $\L_1\big(\cos(\omega\Tsf_{\ep\theta}(\cdot))f\big)$ is even when $f \in \E_{q,b}$.
\end{proof}

\subsection{The operator $\K_0$}
Recall from \eqref{eqn: K0} that 
\[
\K_0f
= f''-f+2\sigma{f}
= f''-f+\Sigma{f}.
\] 
We prove here that $\K_0 \colon \E_{q,b} \to \E_{q,b}$ is invertible for $0 < q < q_{\sigma}$ and $0 < b < b_{\sigma}$.
Lombardi achieves a similar result in \cite[Lem.\@ 6.4.1]{lombardi} by treating $\K_0$ as a classical differential operator and exploiting its fundamental solution set; in keeping with the style here, our treatment is more functional-analytic and involves factoring $\K_0$ as the product of the invertible operator $\partial_z^2-1$ and $1 + (\partial_z^2-1)^{-1}\Sigma$, which we show to be an injective compact perturbation of the identity and therefore invertible.

We first control the kernel of $\K_0$ with the following result.
This is a variation on \cite[Lem.\@ D.5.1]{faver-dissertation}, which we include here for completeness; see also the proof of \cite[Lem.\@ 3.1, part (b)]{amick-toland} and \cite[Lem.\@ 13]{amick-kirchgassner}.

\begin{lemma}\label{lem: K0 kernel}
Any bounded solution $f \in \Cal^2(\R)$ of $f''-f+2\sigma{f} = 0$ is a multiple of $\sigma'$.
\end{lemma}

\begin{proof}
Since $\sigma$ satisfies $\sigma''-\sigma+\sigma^2=0$, certainly $f=\sigma'$ satisfies $f''-f+2\sigma{f} = 0$.
If a bounded solution to $f''-f+2\sigma{f} = 0$ is not a multiple of $\sigma'$, then this differential equation has two independent solutions.
We can use those solutions to construct a bounded solution $f$ to the initial value problem
\[
\begin{cases}
f''-f+2\sigma{f} = 0 \\
f(x_0) = 1 \\
f'(x_0) = 1,
\end{cases}
\]
where $x_0 > 0$ is large enough that $|\sigma(x)| \le 1/4$ for $x \ge x_0$.
We show that this solution $f$ is actually unbounded, a contradiction.

Specifically, we bound $f$ from below as 
\begin{equation}\label{eqn: est for K0 kernel}
f(x) 
\ge 1+x-x_0
\end{equation}
for $x \ge x_0$.
We first prove this estimate on any interval $[x_0,x_1)$ on which $f$ is positive; such an $x_1 > x_0$ exists by continuity, since $f(x_0) = 1$.
Then we show that we can take $x_1=\infty$.

Here is the proof of the estimate \eqref{eqn: est for K0 kernel}.
Since $f(x) > 0$ for $x \in [x_0,x_1)$ and $|\sigma(x)| < 1/4$ for $x \ge x_0$, we have
\[
f''(x)
= f(x)-2\sigma(x)f(x)
\ge \frac{f(x)}{2}
> 0.
\]
Then $f'$ is increasing on $[x_0,x_1)$, so 
\[
f(x)
= f(x_0)+\int_{x_0}^x f'(s) \ds
\ge f(x_0)+f'(x_0)(x-x_0)
= 1+x-x_0.
\]

Now we prove that we can take $x_1 = \infty$.
Otherwise, if there is $x > x_0$ such that $f(x) \le 0$, then by continuity $f(\grave{x}) = 0$ for some $\grave{x}  > x_0$.
Put
\[
x_2
:= \inf\set{x > x_0}{f(x) = 0}.
\]
Approximating $x_2$ by a sequence of zeros of $f$, we have $f(x_2) = 0$ by continuity.
But, also by continuity, $f(x) > 0$ for $x \in [x_0,x_2)$, so the estimate \eqref{eqn: est for K0 kernel} implies $f(x) \ge 1+x-x_0$ for $x \in [x_0,x_2)$.
Then $f(x_2) \ge 1+x_2-x_0 > 0$, a contradiction.
\end{proof}

Likewise, our proof that $\K_0$ is invertible on $\E_{q,b}$ is quite similar to the proof of \cite[Prop.\@ D.5.2]{faver-dissertation}.

\begin{theorem}\label{thm: K0 inv}
Let $0 < q < q_{\sigma}$ and $0 < b < b_{\sigma}$.
The operators $1+\L_1\Sigma$, $\K_0 \colon \E_{q,b} \to \E_{q,b}$ are invertible.
\end{theorem}

\begin{proof}
First, the operator $\partial_z^2 -1 \colon \E_{q,b} \to \E_{q,b}$ is invertible with inverse $(\partial_z^2-1)^{-1} = \L_1$: it is injective by the elementary theory of constant-coefficient differential equations, and a direct calculation shows $(\partial_z^2-1)\L_1f = f$, so it is surjective.
Then we may factor
\[
\K_0
= \partial_z^2-1+\Sigma
= (\partial_z^2-1)(1+\L_1\Sigma).
\]
It therefore suffices to show that $1+\L_1\Sigma \colon \E_{q,b}\to \E_{q,b}$ is invertible.
We do this by appealing to the Fredholm alternative.

First, $1+\L_1\Sigma$ is injective, as if $(1+\L_1\Sigma)f = 0$ for $f \in \E_{q,b}$, then $\K_0f = 0$, and so $f = 0$ by Lemma \ref{lem: K0 kernel}.
Next, we claim that $\L_1\Sigma \colon \E_{q,b} \to \E_{q,b}$ is compact.
Assuming this to be true, the Fredholm alternative guarantees that $1+\L_1\Sigma$ is surjective, thus invertible, and so $\K_0$ is the product of two invertible operators.

We conclude by showing that $\L_1\Sigma$ is compact on $\E_{q,b}$.
If $f \in \E_{q,b}$, then $\Sigma{f} \in \E_{q+\grave{q},b}$ for any $0 < \grave{q} < q_{\sigma}$ by the decay-enhancing estimate \eqref{eqn: decay-enhancing est}.
We then have $\L_1\Sigma{f} \in \E_{q+\grave{q},b}^2$, for if $g \in \E_{q+\grave{q},b}$, then $\L_1g \in \E_{q+\grave{q},b}$ and so $(\L_1g)'' = \L_1g + g \in \E_{q+\grave{q},b}$.
The interpolation result from Lemma \ref{lem: interpolation} then guarantees that $\L_1g \in \E_{q+\grave{q},b}^2$.
Since $\E_{q+\grave{q},b}^2$ is compactly embedded in $\E_{q,b}$ by Theorem \ref{thm: compact embedding}, we conclude that $\L_1\Sigma \colon \E_{q,b} \to \E_{q,b}$ is compact.
\end{proof}

\subsection{The operators $\P_{\ep,b}(\theta)$ and $\Pi_{\ep,b}(\theta)$}
Recall from \eqref{eqn: P-ep-theta} that
\[
\P_{\ep,b}(\theta)f
= f-\frac{\iota_{\ep}[f]}{\iota_{\ep}[\chi_{\ep}^{\theta}+e^{b/\ep}(\tau_{\infty}\omega_{\ep}\theta+1)^{-1}\G_{\ep,b}^{\theta}]}\big(\chi_{\ep}^{\theta}+e^{b/\ep}(\tau_{\infty}\omega_{\ep}\theta+1)^{-1}\G_{\ep,b}^{\theta}\big).
\]
In Appendix \ref{app: chi-ep-theta} we cautioned that $\chi_{\ep}^{\theta}$ is, by itself, exponentially large, but in practice it appears with friends that render it more manageable.
Here is one such instance.

\begin{theorem}\label{thm: P-ep-theta}
Let $0 < b < b_{\sigma}$, $0 < q < q_{\sigma}$, and $0 < \theta_0 < \pi$.
There exists $C > 0$ such that if $0 < \ep < \ep_{\chi}$, $0 \le \theta < \theta_0$, and $f \in \H_{q,b}$, then
\begin{equation}\label{eqn: P-ep-theta map}
\norm{\P_{\ep,b}(\theta)}_{\E_{q,b} \to \E_{q,b}}
\le C.
\end{equation}
\end{theorem}

\begin{proof}
We estimate
\[
\norm{\P_{\ep,b}(\theta)f}_{q,b}
\le \norm{f}_{q,b} + \frac{|\iota_{\ep}[f]|}{|\iota_{\ep}[\chi_{\ep}^{\theta}+e^{b/\ep}(\tau_{\infty}\omega_{\ep}\theta+1)^{-1}\G_{\ep,b}^{\theta}]|}\norm{\chi_{\ep}^{\theta}+e^{b/\ep}(\tau_{\infty}\omega_{\ep}\theta+1)^{-1}\G_{\ep,b}^{\theta}}_{q,b}.
\]
Lombardi's estimate \eqref{eqn: Lombardi} on $\iota_{\ep}[f]$, the lower bound \eqref{eqn: osc int lower bound2} on $|\iota_{\ep}[\chi_{\ep}^{\theta}+e^{b/\ep}(\tau_{\infty}\omega_{\ep}\theta+1)^{-1}\G_{\ep,b}^{\theta}]|$, and the upper bound \eqref{eqn: chi-ep-theta norm} on $\norm{\chi_{\ep}^{\theta}+e^{b/\ep}(\tau_{\infty}\omega_{\ep}\theta+1)^{-1}\G_{\ep,b}^{\theta}}_{q,b}$ combine to give
\[
\frac{|\iota_{\ep}[f]|}{|\iota_{\ep}[\chi_{\ep}^{\theta}+e^{b/\ep}(\tau_{\infty}\omega_{\ep}\theta+1)^{-1}\G_{\ep,b}^{\theta}]|}\norm{\chi_{\ep}^{\theta}+e^{b/\ep}(\tau_{\infty}\omega_{\ep}\theta+1)^{-1}\G_{\ep,b}^{\theta}}_{q,b}
\le C\norm{f}_{q,b}.
\qedhere
\]
\end{proof}

Recall from \eqref{eqn: Pi-ep-theta} that 
\[
\Pi_{\ep,b}(\theta)f
:= f - \frac{\iota_{\ep}[f]}{\iota_{\ep}[\chi_{\ep}^{\theta}]}\chi_{\ep}^{\theta}.
\]
We estimate $\Pi_{\ep,b}(\theta)$ mostly in the same way as $\P_{\ep,b}(\theta)$, except now we also need Lipschitz estimates in $\theta$.

\begin{theorem}
Let $0 < b < b_{\sigma}$, $0 < q < q_{\sigma}$, and $0 < \theta_0 < \pi$.
There exists $C > 0$ such that the following hold.

\begin{enumerate}[label={\bf(\roman*)}]

\item
If $0 < \ep < \ep_{\chi}$, $|\theta| < \theta_0$, and $f \in \H_{q,b}$, then
\begin{equation}\label{eqn: Pi-ep-theta map}
\norm{\Pi_{\ep,b}(\theta)}_{\E_{q,b} \to \E_{q,b}}
\le C.
\end{equation}

\item
If $0 < \ep < \ep_{\chi}$, $0 \le |\theta|$, $|\grave{\theta}| < \theta_0$, and $f \in \H_{q,b}$, then
\begin{equation}\label{eqn: Pi-ep-theta Lip}
\norm{\Pi_{\ep,b}(\theta)-\Pi_{\ep,b}(\grave{\theta})}_{\E_{q,b} \to \E_{q,b}}
\le C|\theta-\grave{\theta}|.
\end{equation}
\end{enumerate}
\end{theorem}

\begin{proof}

\begin{enumerate}[label={\bf(\roman*)}]

\item
The reasoning here is the same as in the proof of Theorem \ref{thm: P-ep-theta}, except now we delete the term $e^{b/\ep}(\tau_{\infty}\omega_{\ep}\theta+1)^{-1}\G_{\ep,b}^{\theta}$ and use the estimates \eqref{eqn: chi-ep-theta norm solo} to bound $\norm{\chi_{\ep}^{\theta}}_{q,b}$ from above and \eqref{eqn: osc int lower bound1} to bound $|\iota_{\ep}[\chi_{\ep}^{\theta}]|$ from below.

\item
We rewrite
\[
\big(\P_{\ep,b}(\theta)-\P_{\ep,b}(\grave{\theta})\big)f
= \iota_{\ep}[f]\left(\frac{\big(\iota_{\ep}[\chi_{\ep}^{\grave{\theta}}-\chi_{\ep}^{\theta}]\chi_{\ep}^{\theta}\big) 
+\big(\iota_{\ep}[\chi_{\ep}^{\theta}](\chi_{\ep}^{\theta}-\chi_{\ep}^{\grave{\theta}})\big)}{\iota_{\ep}[\chi_{\ep}^{\theta}]\iota_{\ep}[\chi_{\ep}^{\grave{\theta}}]}\right).
\]
Lombardi's estimate \eqref{eqn: Lombardi} on $\iota_{\ep}[f]$, the lower bound \eqref{eqn: osc int lower bound1} on $|\iota_{\ep}[\chi_{\ep}^{\theta}]|$, and the upper bound \eqref{eqn: chi-ep-theta norm solo} on $\norm{\chi_{\ep}^{\theta}}_{q,b}$ combine to give
\begin{equation}\label{eqn: Pi-ep-theta Lip aux}
\norm{\big(\P_{\ep,b}(\theta)-\P_{\ep,b}(\grave{\theta})\big)f}_{q,b}
\le Ce^{-2b/\ep}\norm{f}_{q,b}\norm{\chi_{\ep}^{\theta}}_{q,b}\norm{\chi_{\ep}^{\theta}-\chi_{\ep}^{\grave{\theta}}}_{q,b}
\le Ce^{-b/\ep}\norm{f}_{q,b}\norm{\chi_{\ep}^{\theta}-\chi_{\ep}^{\grave{\theta}}}_{q,b}.
\end{equation}
We use this, the Lipschitz estimate \eqref{eqn: chi-ep-theta Lip theta} on $\chi_{\ep}^{\theta}$ in $\theta$, and the remaining factor of $e^{-b/\ep}$ in \eqref{eqn: Pi-ep-theta Lip aux} to conclude the desired estimate \eqref{eqn: Pi-ep-theta Lip}.
\qedhere
\end{enumerate}
\end{proof}

\subsection{The operators $\B_{\ep}^{-1}$, $1+\B_{\ep}^{-1}\P_{\ep,b}(\theta)\Sigma$, and $1+\B_{\ep}^{-1}\Pi_{\ep,b}(\theta)\Sigma$}\label{app: KdV perturbation}
Let $\ep > 0$, $0 < b < b_{\sigma}$, $0 < q < q_{\sigma}$, and $f \in \E_{q,b} \cap \ker(\iota_{\ep})$.
Recall from \eqref{eqn: B-ep-inv} that
\[
\B_{\ep}^{-1}f
= -\frac{1}{\omega_{\ep}\sqrt{1+4\ep^2}}\V_{\omega_{\ep}}f
+ \frac{1}{\mu_{\ep}\sqrt{1+4\ep^2}}\L_{\mu_{\ep}}f,
\]
where we have more recently studied $\V_{\omega_{\ep}}$ in Appendix \ref{app: V-omega} and $\L_{\mu_{\ep}}$ in Appendix \ref{app: L-mu}.
The estimates \eqref{eqn: V-omega est} on $\V_{\omega_{\ep}}$ and \eqref{eqn: L-mu est} on $\L_{\mu_{\ep}}$ imply the following.

\begin{lemma}\label{lem: B-ep-inv}
Let $0 < q < q_{\sigma}$ and $0 < b < b_{\sigma}$.
Take $\ep_{\L} > 0$ from Lemma \ref{thm: L-mu mapping}.
There is $C > 0$ such that if $0 < \ep < \ep_{\L}$ and $f \in \E_{q,b} \cap \ker(\iota_{\ep})$, then
\begin{equation}\label{eqn: B-ep-inv est}
\norm{\B_{\ep}^{-1}f}_{q,b}
\le C\norm{f}_{q,b}.
\end{equation}
\end{lemma}

Now we prove the invertibility of the operators $1+\B_{\ep}^{-1}\P_{\ep,b}(\theta)$ and $1+\B_{\ep}^{-1}\Pi_{\ep,b}(\theta)$.
The proofs are largely the same, so we give the one for $1+\B_{\ep}^{-1}\P_{\ep,b}(\theta)$ in more detail, as $\P_{\ep,b}(\theta)$ is, arguably, more complicated than $\Pi_{\ep,b}(\theta)$.

\begin{theorem}\label{thm: KdV perturbation amp}
Let $0 < b < b_{\sigma}$, $0 < q < q_{\sigma}$, and $0 < \theta_0 < \pi$.
There are $C > 0$ and $\ep_{\B,1} \in (0,\ep_{\chi})$ such that if $0 < \ep < \ep_{\B,1}$ and $0 \le \theta < \theta_0$, then $1+\B_{\ep}^{-1}\P_{\ep,b}(\theta)\Sigma$ is invertible on $\E_{q,b}$ and 
\begin{equation}\label{eqn: lin norm inv}
\norm{(1+\B_{\ep}^{-1}\P_{\ep,b}(\theta)\Sigma)^{-1}}_{\E_{q,b} \to \E_{q,b}} 
\le C.
\end{equation}
\end{theorem}

\begin{proof}
We rewrite
\begin{equation}\label{eqn: Delta1 Delta2 for B perturb}
(1+\B_{\ep}^{-1}\P_{\ep,b}(\theta)\Sigma)f-(1+\L_1\Sigma)f
= \Delta_1+\Delta_2,
\end{equation}
where
\[
\Delta_1 = (\B_{\ep}^{-1}-\L_1)\P_{\ep,b}(\theta)\Sigma{f}
\quadword{and}
\Delta_2 = \L_1(\P_{\ep,b}(\theta)-1)\Sigma{f}.
\]
We further expand
\[
\Delta_1
= \Delta_{11}+\Delta_{12}+\Delta_{13},
\]
where
\[
\Delta_{11}
= \frac{1}{\omega_{\ep}\sqrt{1+4\ep^2}}\V_{\omega_{\ep}}\P_{\ep,b}(\theta)\Sigma{f},
\]
\[
\Delta_{12}
= \left(1-\frac{1}{\mu_{\ep}\sqrt{1+\ep^2}}\right)\L_1\P_{\ep,b}(\theta)\Sigma{f},
\]
and
\[
\Delta_{13}
= \frac{1}{\mu_{\ep}\sqrt{1+4\ep^2}}(\L_1-\L_{\mu_{\ep}})\P_{\ep,b}(\theta)\Sigma{f},
\]

We estimate
\begin{equation}\label{eqn: 11 aux for KdV perturb}
\norm{\Delta_{11}}_{q,b}
\le C\ep\norm{f}_{q,b}
\end{equation}
using the estimate \eqref{eqn: ep-omega-ep bounds} on $\omega_{\ep}$, the estimate \eqref{eqn: V-omega est} on $\V_{\omega_{\ep}}$, the estimate \eqref{eqn: P-ep-theta map} on $\P_{\ep,b}(\theta)$, and the estimate \eqref{eqn: Sigma est} on $\Sigma$.
Here we are able to invoke \eqref{eqn: V-omega est} because $\iota_{\ep}[\P_{\ep,b}(\theta)\Sigma{f}] = 0$.
We likewise estimate
\begin{equation}\label{eqn: 12 aux for KdV perturb}
\norm{\Delta_{12}}_{q,b}
\le C\ep^2\norm{f}_{q,b}
\end{equation}
using now the estimate \eqref{eqn: mu-ep lim} on $|\mu_{\ep}-1|$ and the estimate \eqref{eqn: L-mu est proto} on $\L_1$, the latter with $\mu=1$.

The estimate for $\Delta_{13}$ is more complicated.
To invoke the various other estimates that allow us to control $\Delta_{13}$, we first put
\[
\ep_{\B,0}
:= \min\{1,\ep_{\per},\ep_{\L},\ep_{\chi},\ep_{\L_1}\}
\]

Suppose that $0 < q < \grave{q} < q_{\sigma}$.
If $f \in \E_{q,b}$, then $\Sigma{f} \in \E_{\grave{q},b}$ by the decay-enhancing estimate \eqref{eqn: decay-enhancing est}, and so $\P_{\ep,b}(\theta)\Sigma{f} \in \E_{\grave{q},b}$ as well.
Then we may apply Theorem \ref{thm: L-mu op norm} and the previous estimates on $\P_{\ep,b}(\theta)$ and $\Sigma$ to bound
\[
\norm{(\L_1-\L_{\mu_{\ep}})\P_{\ep,b}(\theta)\Sigma{f}}_{\grave{q},b}
\le C\ep^2\norm{\P_{\ep,b}(\theta)\Sigma{f}}_{\grave{q},b}
\le C\ep^2\norm{\Sigma{f}}_{\grave{q},b}
\le C\ep^2\norm{f}_{q,b},
\]
from which it follows that 
\begin{equation}\label{eqn: 13 aux for KdV perturb}
\norm{\Delta_{13}}_{q,b}
\le C\ep^2\norm{f}_{q,b}
\end{equation}
when $0 < \ep < \ep_{\B,0}$.

To estimate $\Delta_2$ from \eqref{eqn: Delta1 Delta2 for B perturb}, we first caution that we do not try to control $\P_{\ep,b}(\theta)-1$, since $\P_{\ep,b}(\theta)$ by itself is not a small perturbation of the identity.
Rather, we compute
\begin{equation}\label{eqn: B0 inv P -1}
\L_1(\P_{\ep,b}(\theta)-1)f
= -\frac{\iota_{\ep}[f]}{\iota_{\ep}[\chi_{\ep}^{\theta}+e^{b/\ep}(\tau_{\infty}\omega_{\ep}\theta+1)^{-1}\G_{\ep,b}^{\theta}]}\L_1\big(\chi_{\ep}^{\theta}+e^{b/\ep}(\tau_{\infty}\omega_{\ep}\theta+1)^{-1}\G_{\ep,b}^{\theta}\big).
\end{equation}
Lombardi's estimate \eqref{eqn: Lombardi} on $\iota_{\ep}[f]$, the lower bound \eqref{eqn: iota-ep-chi amp} on $\chi_{\ep}^{\theta}+e^{b/\ep}(\tau_{\infty}\omega_{\ep}\theta+1)^{-1}\G_{\ep,b}^{\theta}$, and the definition of $\chi_{\ep}^{\theta}$ in \eqref{eqn: chi-ep-theta} then give
\[
\norm{\L_1(\P_{\ep,b}(\theta)-1)f}_{q,b}
\le Ce^{-b/\ep}\norm{f}_{q,b}\big(2\norm{\L_1\big(\cos(\omega_{\ep}\Tsf_{\ep\theta}(\cdot))\sigma\big)}_{q,b}+(\tau_{\infty}\omega_{\ep}\theta+1)^{-1}\norm{\L_1\G_{\ep,b}^{\theta}}_{q,b}\big).
\]
Lemma \ref{lem: ibp aux for L1} gives
\[
\norm{\L_1\big(\cos(\omega_{\ep}\Tsf_{\ep\theta}(\cdot))\sigma\big)}_{q,b}
\le Ce^{b\omega_{\ep}}\omega_{\ep}^{-1}(\norm{\sigma}_{q,b}+\norm{\sigma'}_{q,b}),
\]
and we can estimate 
\[
e^{-b/\ep}e^{b\omega_{\ep}}
\le C
\] 
using \eqref{eqn: omega-ep est}.
Last, we bound
\[
(\tau_{\infty}\omega_{\ep}\theta+1)^{-1}\norm{\L_1\G_{\ep,b}^{\theta}}_{q,b}
\le C\ep
\]
from \eqref{eqn: useful G Lip}.
This gives
\begin{equation}\label{eqn: 2 aux for KdV perturb}
\norm{\Delta_2}_{q,b}
\le C\ep.
\end{equation}

We combine the estimates \eqref{eqn: 11 aux for KdV perturb}, \eqref{eqn: 12 aux for KdV perturb}, \eqref{eqn: 13 aux for KdV perturb}, and \eqref{eqn: 2 aux for KdV perturb} to conclude the existence of $C > 0$ such that if $0 < \ep < \ep_{\B,0}$, then
\[
\norm{\big((1+\B_{\ep}^{-1}\P_{\ep,b}(\theta)\Sigma)-(1+\L_1\Sigma)\big)f}_{q,b}
\le C\ep\norm{f}_{q,b}.
\]
Since $1+\L_1\Sigma$ is invertible on $\E_{q,b}$ by Theorem \ref{thm: K0 inv}, there is $\ep_{\B,1} \in (0,\ep_{\B,0})$ such that if $0 < \ep < \ep_{\B,1}$, then $1+\B_{\ep}^{-1}\P_{\ep,b}(\theta)\Sigma$ is invertible with the uniform estimate \eqref{eqn: lin norm inv} on its inverse.
\end{proof}

\begin{theorem}\label{thm: KdV perturbation PS}
Let $0 < b < b_{\sigma}$, $0 < q < q_{\sigma}$, and $0 < \theta_0 < \pi$.
There are $C$, $\ep_{\B,2} > 0$ such that the following hold.

\begin{enumerate}[label={\bf(\roman*)}]

\item
If $0 < \ep < \ep_{\B,2}$ and $|\theta| < \theta_0$, then $1+\B_{\ep}^{-1}\P_{\ep,b}(\theta)\Sigma$ is invertible on $\E_{q,b}$ and 
\begin{equation}\label{eqn: lin norm inv PS}
\norm{(1+\B_{\ep}^{-1}\Pi_{\ep,b}(\theta)\Sigma)^{-1}}_{\E_{q,b} \to \E_{q,b}} 
\le C.
\end{equation}

\item
If $0 < \ep < \ep_{\B,2}$ and $0 \le |\theta|$, $|\grave{\theta}| < \theta_0$, then
\begin{equation}\label{eqn: lin Lip PS}
\norm{(1+\B_{\ep}^{-1}\Pi_{\ep,b}(\theta)\Sigma)^{-1}-(1+\B_{\ep}^{-1}\Pi_{\ep,b}(\grave{\theta})\Sigma)^{-1}}_{\E_{q,b} \to \E_{q,b}} 
\le C|\theta-\grave{\theta}|.
\end{equation}
\end{enumerate}
\end{theorem}

\begin{proof}
\begin{enumerate}[label={\bf(\roman*)}]

\item
The reasoning here is the same as in the proof of Theorem \ref{thm: KdV perturbation amp}, except we replace all appearances of $\P_{\ep,b}(\theta)$ with $\Pi_{\ep,b}(\theta)$ and use the mapping estimate \eqref{eqn: Pi-ep-theta map} on $\Pi_{\ep,b}(\theta)$ as needed.
Also, in the calculation of $\L_1(\Pi_{\ep,b}(\theta)-1)f$, as in \eqref{eqn: B0 inv P -1}, we delete all appearances of $e^{b/\ep}(\tau_{\infty}\omega_{\ep}\theta+1)^{-1}\G_{\ep,b}^{\theta}$. 

\item
Abbreviate $\T_{\ep,b}(\theta) := 1+\B_{\ep}^{-1}\Pi_{\ep,b}(\theta)\Sigma$.
We have
\[
\T_{\ep,b}(\theta)^{-1}-\T_{\ep,b}(\grave{\theta})^{-1}
= \T_{\ep,b}(\theta)^{-1}\big(\T_{\ep,b}(\grave{\theta})-\T_{\ep,b}(\theta)\big)\T_{\ep,b}(\grave{\theta})^{-1}
\]
The uniform bounds on $\T_{\ep,b}(\theta)^{-1}$ and $\T_{\ep,b}(\grave{\theta})$ from \eqref{eqn: lin norm inv PS} and the Lipschitz estimate \eqref{eqn: Pi-ep-theta Lip} on $\P_{\ep,b}(\theta)-\P_{\ep,b}(\grave{\theta})$ then yield the uniform Lipschitz estimate \eqref{eqn: lin Lip PS}.
\qedhere
\end{enumerate}
\end{proof}
 
\section{The Solution to the Amplitude Selection Problem}

\subsection{The proof of Lemma \ref{lem: ultimate amp lemma}}\label{app: ultimate amp lemma}

\subsubsection{The proof of the mapping estimate \eqref{eqn: amp sel map est}}\label{app: amp sel map est}
Let
\[
\ep_{\map,0}
:= \min\{1,\ep_{\per},\ep_{\B,1}\}
\]
and take $0 < \ep < \ep_{\map,0}$.
The definition of $\norm{\cdot}_{(q+q_{\sigma})/2,b}$ in \eqref{eqn: Xqb} and of $\F_{q,b}^{\ep,\theta}$ in \eqref{eqn: F-ep-b} give
\[
\norm{\F_{q,b}^{\ep,\theta}(\eta,\alpha)}_{(q+q_{\sigma})/2,b}
= \norm{\Ncal_{q,b}^{\ep,\theta}(\eta,\alpha)}_{(q+q_{\sigma})/2,b} + (\tau_{\infty}\omega_{\ep}\theta+1)^{-1}|\A_{q,b}^{\ep,\theta}(\eta,\alpha)|.
\]
We estimate $\norm{\Ncal_{q,b}^{\ep,\theta}(\eta,\alpha)}_{(q+q_{\sigma})/2,b}$ using the definition of $\Ncal_{q,b}^{\ep,\theta}$ in \eqref{eqn: Ncal-ep amp}, the estimate \eqref{eqn: lin norm inv} on $(1+\B_{\ep}^{-1}\P_{\ep,b}(\theta)\Sigma)^{-1}$, the estimate \eqref{eqn: B-ep-inv est} on $\B_{\ep}^{-1}$, the estimate \eqref{eqn: P-ep-theta map} on $\P_{\ep,b}(\theta)$, and the mapping estimate \eqref{eqn: ultimate R map} on $\rhs_{q,b}^{\ep}(\eta,\alpha,\theta)$ in $\E_{(q+q_{\sigma})/2,b}$ to obtain
\begin{equation}\label{eqn: Ncal-eP-ep-theta map est aux amp}
\norm{\Ncal_{q,b}^{\ep,\theta}(\eta,\alpha)}_{(q+q_{\sigma})/2,b}
\le C\big(\ep^4 + \ep|\alpha| + (1+\ep^{-1})|\alpha|^2+\norm{\eta}_{(q+q_{\sigma})/2,b}^2\big).
\end{equation}
We estimate $\A_{q,b}^{\ep,\theta}(\eta,\alpha)$ using the definition of $\A_{q,b}^{\ep,\theta}$ in \eqref{eqn: A-ep}, Lombardi's estimate \eqref{eqn: Lombardi}, the mapping estimate \eqref{eqn: ultimate R map} on $\rhs_{q,b}^{\ep}(\eta,\alpha,\theta)$, the estimate \eqref{eqn: Sigma est} on $\Sigma$, the estimate \eqref{eqn: Ncal-eP-ep-theta map est aux amp} above on $\Ncal_{q,b}^{\ep,\theta}(\eta,\alpha)$,and the lower bound \eqref{eqn: osc int lower bound2} on $\iota_{\ep}[\chi_{\ep}^{\theta}+(\tau_{\infty}\omega_{\ep}\theta+1)^{-1}\G_{\ep,b}^{\theta}]$ to obtain
\begin{equation}\label{eqn: A-eP-ep-theta map aux}
|\A_{q,b}^{\ep,\theta}(\eta,\alpha)|
\le C\big(\ep^4 + \ep|\alpha| + (1+\ep^{-1})|\alpha|^2+\norm{\eta}_{(q+q_{\sigma})/2,b}^2\big).
\end{equation}
The positivity of $\tau_{\infty}\omega_{\ep}\theta+1$ ensures that $(\tau_{\infty}\omega_{\ep}\theta+1)^{-1}|\A_{q,b}^{\ep,\theta}(\eta,\alpha)|$ has the same estimate.
Then
\[
\norm{\F_{q,b}^{\ep,\theta}(\eta,\alpha)}_{(q+q_{\sigma})/2,b} 
\le C_{\star}\big(\ep^4 + \ep|\alpha| + (1+\ep^{-1})|\alpha|^2+\norm{\eta}_{(q+q_{\sigma})/2,b}^2\big)
\]
for some $C_{\star} > 0$.
We single out the constant $C_{\star}$ here for future reference, and for contrast with other constants.

Put $C_0 := 2C_{\star}$ to conclude that if $\norm{(\eta,\alpha)}_{(q+q_{\sigma})/2,b} \le C_0\ep^4$, then (since $0 < \ep < 1$)
\begin{align*}
\ep^4 + \ep|\alpha| + (1+\ep^{-1})|\alpha|^2+\norm{\eta}_{q,b}^2
&\le \ep^4+C_{\star}\ep^5+4C_{\star}^2\ep^8+4C_{\star}^2\ep^7+4C_{\star}^2\ep^8 \\
&= \ep^4(1+C_{\star}\ep+4C_{\star}^2\ep^4+4C_{\star}^2\ep^3+4C_{\star}^2\ep^4) \\
&\le \ep^4(1+(C_{\star}+12C_{\star}^2)\ep),
\end{align*}
and so
\[
\norm{\F_{q,b}^{\ep,\theta}(\eta,\alpha)}_{(q+q_{\sigma})/2,b} 
\le \ep^4(C_{\star}+C_{\star}^2(1+12C_{\star})\ep).
\]
Take 
\begin{equation}\label{eqn: ep-map amp}
\ep_{\map}
:= \min\left\{\ep_{\map,0},\frac{1}{C_{\star}(1+12C_{\star})}\right\}
\end{equation}
to conclude that if $0 < \ep < \ep_{\map}$ and $\norm{(\eta,\alpha)}_{(q+q_{\sigma})/2,b} \le C_0\ep^4$, then
\[
\norm{\F_{q,b}^{\ep,\theta}(\eta,\alpha)}_{(q+q_{\sigma})/2,b} 
\le C_0\ep^4,
\]
and this is the mapping estimate \eqref{eqn: amp sel map est}.

\subsubsection{The proof of the contraction estimate \eqref{eqn: amp sel contr est}}
Suppose for now that $0 < \ep < \ep_{\map}$ as defined in \eqref{eqn: ep-map amp} above.
The definition of $\norm{\cdot}_{q,b}$ in \eqref{eqn: Xqb} and of $\F_{q,b}^{\ep,\theta}$ in \eqref{eqn: F-ep-b} give
\[
\norm{\F_{q,b}^{\ep,\theta}(\eta,\alpha)-\F_{q,b}^{\ep,\theta}(\grave{\eta},\grave{\alpha})}_{q,b}
= \norm{\Ncal_{q,b}^{\ep,\theta}(\eta,\alpha)-\Ncal_{q,b}^{\ep,\theta}(\grave{\eta},\grave{\alpha})}_{q,b} + |\A_{q,b}^{\ep,\theta}(\eta,\alpha)-\A_{q,b}^{\ep,\theta}(\grave{\eta},\grave{\alpha})|.
\]
We estimate the difference in $\Ncal_{q,b}^{\ep,\theta}$ using the definition of $\Ncal_{q,b}^{\ep,\theta}$ in \eqref{eqn: Ncal-ep amp}, the estimate \eqref{eqn: lin norm inv} on $(1+\B_{\ep}^{-1}\P_{\ep,b}(\theta)\Sigma)^{-1}$, the estimate \eqref{eqn: B-ep-inv est} on $\B_{\ep}^{-1}$, the estimate \eqref{eqn: P-ep-theta map} on $\P_{\ep,b}(\theta)$, and the Lipschitz estimate \eqref{eqn: ultimate R Lip alpha} on $\rhs_{q,b}^{\ep}(\eta,\alpha,\theta)$ with $q_1 = q$ and $q_2 = (q+q_{\sigma})/2$ to obtain
\begin{multline}\label{eqn: Ncal-ep Lip est aux amp1}
\norm{\Ncal_{q,b}^{\ep,\theta}(\eta,\alpha)-\Ncal_{q,b}^{\ep,\theta}(\grave{\eta},\grave{\alpha})}_{q,b} \\
\le \frac{C}{q_{\sigma}-q}\big(\ep+(1+\ep^{-1})(|\alpha|+|\grave{\alpha}|)+\norm{\eta}_{(q+q_{\sigma})/2,b}+\norm{\grave{\eta}}_{(q+q_{\sigma})/2,b}\big)\norm{(\eta,\alpha)-(\grave{\eta},\grave{\alpha})}_{q,b}.
\end{multline}
We estimate the difference in $\A_{q,b}^{\ep,\theta}$ using the definition of $\A_{q,b}^{\ep,\theta}$ in \eqref{eqn: A-ep}, Lombardi's estimate \eqref{eqn: Lombardi}, the Lipschitz estimate \eqref{eqn: ultimate R Lip alpha} on $\rhs_{q,b}^{\ep}(\eta,\alpha,\theta)$ with $q_1 = q$ and $q_2 = (q+q_{\sigma})/2$, the estimate \eqref{eqn: Sigma est} on $\Sigma$, the estimate \eqref{eqn: Ncal-ep Lip est aux amp1} above on $\Ncal_{q,b}^{\ep,\theta}(\eta,\alpha)$, the lower bound \eqref{eqn: osc int lower bound2} on $\iota_{\ep}[\chi_{\ep}^{\theta}+(\tau_{\infty}\omega_{\ep}\theta+1)^{-1}\G_{\ep,b}^{\theta}]$, and the positivity of $\tau_{\infty}\omega_{\ep}\theta+1$ to obtain
\begin{multline*}
(\tau_{\infty}\omega_{\ep}\theta+1)^{-1}|\A_{q,b}^{\ep,\theta}(\eta,\alpha)-\A_{q,b}^{\ep,\theta}(\grave{\eta},\grave{\alpha})| \\
\le \frac{C}{q_{\sigma}-q}\big(\ep+(1+\ep^{-1})(|\alpha|+|\grave{\alpha}|)+\norm{\eta}_{(q+q_{\sigma})/2,b}+\norm{\grave{\eta}}_{(q+q_{\sigma})/2,b}\big)\norm{(\eta,\alpha)-(\grave{\eta},\grave{\alpha})}_{q,b}.
\end{multline*}
We conclude
\begin{equation}\label{eqn: Ncal-ep Lip est aux amp2}
\norm{\F_{q,b}^{\ep,\theta}(\eta,\alpha)-\F_{q,b}^{\ep,\theta}(\grave{\eta},\grave{\alpha})}_{q,b} 
\le C_{\F}\big(\ep+(1+\ep^{-1})(|\alpha|+|\grave{\alpha}|)+\norm{\eta}_{(q+q_{\sigma})/2,b}+\norm{\grave{\eta}}_{(q+q_{\sigma})/2,b}\big)\norm{(\eta,\alpha)-(\grave{\eta},\grave{\alpha})}_{q,b}
\end{equation}
for some $C_{\F} > 0$.
We single out the constant $C_{\F}$ here for future reference, and for contrast with other constants.

Then if $\norm{(\eta,\alpha)}_{q,b}$, $\norm{(\grave{\eta},\grave{\alpha})}_{q,b} \le C_0\ep^4$ with $C_0$ defined above, we have (since $0 < \ep < 1$)
\[
\norm{\F_{q,b}^{\ep,\theta}(\eta,\alpha)-\F_{q,b}^{\ep,\theta}(\grave{\eta},\grave{\alpha})}_{q,b} 
\le  C_{\F}\ep(1+6C_0)\norm{(\eta,\alpha)-(\grave{\eta},\grave{\alpha})}_{q,b}.
\]
Take 
\begin{equation}\label{eqn: ep-Lip amp}
\ep_{\Lip} 
:= \min\left\{\ep_{\map},\frac{1}{2C_{\F}(1+6C_0)}\right\}
\end{equation}
to conclude that if $0 < \ep < \ep_{\Lip}$ and $\norm{(\eta,\alpha)}_{q,b}$, $\norm{(\grave{\eta},\grave{\alpha})}_{q,b} \le C_0\ep^4$, then
\[
\norm{\F_{q,b}^{\ep,\theta}(\eta,\alpha)-\F_{q,b}^{\ep,\theta}(\grave{\eta},\grave{\alpha})}_{q,b} 
\le \frac{1}{2}\norm{(\eta,\alpha)-(\grave{\eta},\grave{\alpha})}_{q,b},
\]
and this is the contraction estimate \eqref{eqn: amp sel contr est}.

\subsubsection{The proof of the Lipschitz estimate \eqref{eqn: amp sel Lip est}}
Suppose that $0 < \ep < \ep_{\Lip}$ as defined in \eqref{eqn: ep-Lip amp} above.
We reason as in the proof of the contraction estimate above, except in the Lipschitz estimate from \eqref{eqn: ultimate R Lip alpha} we now take $q_1 = q/2$ and $q_2 = q$, to conclude
\begin{multline*}
\norm{\F_{q,b}^{\ep,\theta}(\eta,\alpha)-\F_{q,b}^{\ep,\theta}(\grave{\eta},\grave{\alpha})}_{q/2,b} 
\le \tilde{C}_{\F}\big(\ep+(1+\ep^{-1})(|\alpha|+|\grave{\alpha}|)+\norm{\eta}_{q,b}+\norm{\grave{\eta}}_{q,b}\big)\norm{(\eta,\alpha)-(\grave{\eta},\grave{\alpha})}_{q/2,b}.
\end{multline*}
This is the same as \eqref{eqn: Ncal-ep Lip est aux amp2}, except now possibly $\tilde{C}_{\F} \ne C_{\F}$.
If $\norm{(\eta,\alpha)}_{\X_{(q+q_0)/2}}$, $\norm{(\grave{\eta},\grave{\alpha})}_{\X_{(q+q_0)/2}} \le C_0\ep^4$, then
\[
\norm{\F_{q,b}^{\ep,\theta}(\eta,\alpha)-\F_{q,b}^{\ep,\theta}(\grave{\eta},\grave{\alpha})}_{q/2,b} 
\le \tilde{C}_{\F}\ep\big(1+6C_0\big)\norm{(\eta,\alpha)-(\grave{\eta},\grave{\alpha})}_{q/2,b}.
\]
Take 
\[
\ep_{\star}
:= \min\left\{\ep_{\Lip},\frac{1}{2\tilde{C}_{\F}(1+6C_0)}\right\}
\] 
to conclude that if $0 < \ep < \ep_{\star}$ and $\norm{(\eta,\alpha)}_{q,b}$, $\norm{(\grave{\eta},\grave{\alpha})}_{q,b} \le C_0\ep^4$, then
\[
\norm{\F_{q,b}^{\ep,\theta}(\eta,\alpha)-\F_{q,b}^{\ep,\theta}(\grave{\eta},\grave{\alpha})}_{q/2,b} 
\le \frac{1}{2}\norm{(\eta,\alpha)-(\grave{\eta},\grave{\alpha})}_{q/2,b},
\]
and this is the Lipschitz estimate \eqref{eqn: amp sel Lip est}.

\subsection{The proof of the additional estimate \eqref{eqn: extra small amp est} for $\alpha_{\ep}$}\label{app: extra small amp est}
We have $\alpha_{\ep} = (\tau_{\infty}\omega_{\ep}\theta+1)^{-1}\A_{q,b}^{\ep,\theta}(\eta_{\ep},\alpha_{\ep})$, where $\norm{(\eta_{\ep},\alpha_{\ep})}_{q,b} \le C_0\ep^4$.
We redo the mapping estimates at the start of Appendix \ref{app: amp sel map est} in $\E_{q,b}$, not $\E_{(q+q_{\sigma})/2,b}$ to arrive at the analogue of \eqref{eqn: A-eP-ep-theta map aux}:
\[
(\tau_{\infty}\omega_{\ep}\theta+1)^{-1}|\A_{q,b}^{\ep,\theta}(\eta_{\ep},\alpha_{\ep},\theta)|
\le C\ep\big(\ep^4+\ep|\alpha_{\ep}|+(1+\ep^{-1})|\alpha_{\ep}|^2+\norm{\eta_{\ep}}_{q,b}^2\big)
\le C\ep^5.
\]
Here the extra factor of $\ep$ is due to the positivity of $\tau_{\infty}\omega_{\ep}\theta+1$ and the estimate
\[
\frac{1}{\tau_{\infty}\omega_{\ep}\theta+1}
\le \frac{1}{\tau_{\infty}\omega_{\ep}\theta}
\le C\ep.
\]

\subsection{The proof of Theorem \ref{thm: main amp}}\label{app: proof of main amp}
Let $0 < q < q_{\sigma}$, $0 < b < b_{\sigma}$, and $0 \le \theta < \pi$.
Theorem \ref{thm: amp sel} gives $\ep_{\star} > 0$ such that for $0 < \ep < \ep_{\star}$, there are $\eta_{\ep} \in \E_{q,b}$ and $\alpha_{\ep} \in \R$ such that 
\begin{equation}\label{eqn: u main amp app}
u(z)
= \sigma(z)+\upsilon_{\ep,q}(z) + \alpha_{\ep}\varphi_{\ep,b}^{\alpha_{\ep},\theta}(z) + \eta_{\ep}(z)
\end{equation}
solves \eqref{eqn: the problem}, where
\[
\varphi_{\ep,b}^{\alpha_{\ep},\theta}(z)
= \cos(\Omega_{\ep,b}^{\alpha_{\ep}}z) + \sum_{n=1}^{\infty} \alpha_{\ep}^ne^{-(n+1)b/\ep}\psi_{n,\ep}(\Omega_{\ep,b}^{\alpha_{\ep}}z)
\]
Put 
\[
\Upsilon_{\ep} 
:= \upsilon_{\ep,q} + \eta_{\ep},
\]
\[
\Omega_{\ep} 
:= \Omega_{\ep,b}^{\alpha_{\ep}},
\]
and
\[
\Phi_{\ep}(z) 
:= \cos(\Omega_{\ep}z) + \sum_{n=1}^{\infty} \alpha_{\ep}^ne^{-nb/\ep}\psi_{n,\ep}(\Omega_{\ep}z),
\]
so \eqref{eqn: u main amp app} has the form \eqref{eqn: main amp u}.

The second estimate in \eqref{eqn: main amp per} for $\Omega_{\ep}$ is \eqref{eqn: Omega-ep-alpha map2}.
The estimate \eqref{eqn: main amp amp} for $\alpha_{\ep}$ is \eqref{eqn: amp sel est} and \eqref{eqn: extra small amp est}.
The estimate \eqref{eqn: main amp error} for $\Upsilon_{\ep}$ is also \eqref{eqn: amp sel est}.

Now we estimate as in the proof of part \ref{part: ultimate psi lemma 1} of Lemma \ref{lem: ultimate psi lemma} that if $X \in \R$, then
\[
|\psi_{n,\ep}(X)|
\le \sum_{k=1}^{n+1} |c_{n,k,\ep}||\cos(kX)|
\le \sum_{k=1}^{n+1} |c_{n,k,\ep}|
\le \sqrt{n+1}\norm{\psi_{n,\ep}}_{L_{\per}^2}
\le \frac{\sqrt{n+1}}{a_{\per,0}^n}.
\]
Then
\[
\sum_{n=1}^{\infty} |\alpha_{\ep}^ne^{-nb/\ep}\psi_{n,\ep}(\Omega_{\ep}x)|
\le \frac{|\alpha_{\ep}|}{a_{\per,0}}\sum_{n=1}^{\infty} \left(\frac{|\alpha_{\ep}|}{a_{\per,0}}\right)^{n-1}\sqrt{n+1}e^{-nb/\ep}
\le \frac{|\alpha_{\ep}|}{a_{\per,0}}\sum_{n=1}^{\infty} \sqrt{n+1}(e^{-b/\ep})^n
\]
and this last series converges by the ratio test.
This proves that 
\[
|\Phi_{\ep}(z)-\cos(\Omega_{\ep}z)|
\le C|\alpha_{\ep}|,
\]
which is the first estimate in \eqref{eqn: main amp per}.
We emphasize that because we consider this estimate over $\R$, it is much better than our prior struggles with periodic estimates on the strips $\U_b$.
This is the same as Lombardi's observation stated before \cite[Lem.\@ 7.3.7]{lombardi} that the periodics are ``exponentially bounded'' on $\R$ but ``simply bounded'' on strips.

Last, we discuss how the ``shift-reflected'' identity \eqref{eqn: per shift refl} permits the restriction in Theorem \ref{thm: main amp} to phase shifts in $[0,\pi)$.
Suppose that $\pi \le \theta < 2\pi$.
Then for $0 < \ep < \ep_{\star}$ and $x \in \R$, we have
\begin{multline*}
\varphi_{\ep,b}^{\alpha_{\ep},0}(x+\ep\theta)
= \phi_{\ep}^{\alpha_{\ep}}(x+\ep\theta)
= \phi_{\ep}^{\alpha_{\ep}}\left(x+\ep(\theta-\pi)+(\ep\omega_{\ep}^a-1)\frac{\pi}{\omega_{\ep}^a} + \frac{\pi}{\omega_{\ep}^a}\right) \\
= \phi_{\ep}^{-\alpha_{\ep}}\left(x+\ep(\theta-\pi)+(\ep\omega_{\ep}^a-1)\frac{\pi}{\omega_{\ep}^a}\right)
= \varphi_{\ep,b}^{-\alpha_{\ep},0}\left(x+\ep(\theta-\pi)+(\ep\omega_{\ep}^a-1)\frac{\pi}{\omega_{\ep}^a}\right)
\end{multline*}
By parts \ref{part: omega-ep-a bound} and \ref{part: ep omega-ep-a minus 1} of Lemma \ref{lem: ultimate omega lemma}, there is $C > 0$ such that 
\[
\left|(\ep\omega_{\ep}^a-1)\frac{\pi}{\omega_{\ep}^a}\right|
\le C\ep^2
\]
for all $\ep$ and $a$.
Then we use the mapping estimate \eqref{eqn: varphi map ultimate} with $r=1$ to control
\[
\left|\varphi_{\ep,b}^{-\alpha_{\ep},0}\left(x+\ep(\theta-\pi)+(\ep\omega_{\ep}^a-1)\frac{\pi}{\omega_{\ep}^a}\right)-\varphi_{\ep,b}^{-\alpha_{\ep},0}(x+\ep(\theta-\pi))\right|
\le C\ep^{-1}\left|(\ep\omega_{\ep}^a-1)\frac{\pi}{\omega_{\ep}^a}\right|
\le C\ep.
\]
So, if we wanted to construct a nanopteron asymptotic to a periodic ripple with a phase shift of $\theta \in [\pi,2\pi)$, that would be the same as constructing a nanopteron asymptotic to a periodic ripple with a phase shift of $\pi-\theta \in [0,\pi)$.
Amick and Toland encountered the same symmetry in their periodics \cite[pp.\@ 39, 45]{amick-toland}.

\section{The Solution to the Phase Shift Selection Problem}

\subsection{Estimates on $\M_{q,b}^{\ep}$}\label{app: M-ep}
We prove mapping and Lipschitz estimates for the nonlinear functional $\M_{q,b}^{\ep}$ defined in \eqref{eqn: M-ep}.

\begin{lemma}
Let $0 < b < b_{\sigma}$, $0 < q < q_{\sigma}$, $0 < \theta_0 < \pi$, and $0 < A_0 < A_1 < 1$.
There are $C_{\M}$, $\ep_{\M} > 0$ such that the following hold.

\begin{enumerate}[label={\bf(\roman*)}]

\item
If $0 < \ep < \ep_{\M}$, $A_0 < A < A_1$, $|\theta| < \theta_0$, and $\eta \in \E_{q,b}$, then
\begin{equation}\label{eqn: M-eP-ep-theta map}
|\M_{q,b}^{\ep}(\eta,A\ep^4,\theta)|
< C_{\M}\big(\ep+\ep^{-3}\norm{\eta}_{q,b}^2+\ep^{-3}\norm{\eta}_{q,b}\big).
\end{equation}

\item
If $0 < \ep < \ep_{\M}$, $A_0 < A < A_1$, $0 \le |\theta|$, $|\grave{\theta}| < \theta_0$, and $\eta$, $\grave{\eta} \in \E_{q,b}$, then
\begin{equation}\label{eqn: M-ep Lip}
|\M_{q,b}^{\ep}(\eta,A\ep^4,\theta)-\M_{q,b}^{\ep}(\grave{\eta},A\ep^4,\grave{\theta})|
< C_{\M}\big(\ep^{-3}+\norm{\eta}_{q,b}+\norm{\grave{\eta}}_{q,b}\big)\norm{\eta-\grave{\eta}}_{q,b}
+ C_{\M}\big(\ep+\norm{\eta}_{q,b}+\norm{\grave{\eta}}_{q,b}\big)|\theta-\grave{\theta}|.
\end{equation}
\end{enumerate}
\end{lemma}

\begin{proof}
\begin{enumerate}[label={\bf(\roman*)}]

\item
Let
\[
\ep_{\M}
:= \min\{1,\ep_{\per},\ep_{\chi},\ep_{\B,2}, \ep_{\L}\}
\]
and suppose $0 < \ep < \ep_{\M}$.
The estimate \eqref{eqn: I-ep-theta map} on $\I_{\ep}(\theta)$ and the expansion \eqref{eqn: iota-chi-ep-theta exp} of $\iota_{\ep}[\chi_{\ep}^{\theta}]$ give
\[
\frac{(\ep\omega_{\ep})\theta\I_{\ep}(\theta)}{2}
+ \frac{\ep\iota_{\ep}[\chi_{\ep}^{\theta}]}{2}
\le C\ep.
\]
Next, Lombardi's estimate \eqref{eqn: Lombardi}, the mapping estimate \eqref{eqn: ultimate R map} on $\rhs_{q,b}^{\ep}$, the mapping estimate \eqref{eqn: ultimate G map} on $\G_{\ep,b}^{\theta}$, and the estimate \eqref{eqn: Sigma est} on $\Sigma$ imply
\[
\left|\frac{\ep{e}^{b/\ep}\iota_{\ep}[\rhs_{q,b}^{\ep}(\eta,\alpha,\theta)-\alpha\G_{\ep,b}^{\theta}-\Sigma\eta]}{2\alpha}\right|
\le C\ep\left(\frac{\ep^4}{|\alpha|}+\ep+(1+\ep^{-1})|\alpha|+\frac{\norm{\eta}_{q,b}^2}{|\alpha|}\right).
\]
Taking $\alpha = A\ep^4$ with $A_0 < A < A_1$, we conclude \eqref{eqn: M-eP-ep-theta map}.

\item
The Lipschitz estimate \eqref{eqn: I-ep Lip} for $\I_{\ep}$, the bound \eqref{eqn: I-ep-theta map} on $\I_{\ep}$, and the expansion \eqref{eqn: iota-chi-ep-theta exp} of $\iota_{\ep}[\chi_{\ep}^{\theta}]$ give 
\[
\left|\left(\frac{(\ep\omega_{\ep})\theta\I_{\ep}(\theta)}{2}
+ \frac{\ep\iota_{\ep}[\chi_{\ep}^{\theta}]}{2}\right)
-\left(\frac{(\ep\omega_{\ep})\grave{\theta}\I_{\ep}(\grave{\theta})}{2}
+ \frac{\ep\iota_{\ep}[\chi_{\ep}^{\grave{\theta}}]}{2}\right)\right|
\le C\ep|\theta-\grave{\theta}|.
\]
Next, Lombardi's estimate, the Lipschitz estimate \eqref{eqn: ultimate R Lip theta} on $\rhs_{q,b}^{\ep}$, and the Lipschitz estimate \eqref{eqn: ultimate G Lip theta} on $\G_{\ep,b}^{\theta}$ imply
\begin{multline*}
\frac{\ep{e}^{b/\ep}|\iota_{\ep}[\rhs_{q,b}^{\ep}(\eta,\alpha,\theta)-\alpha\G_{\ep,b}^{\theta}-\Sigma\eta]-\iota_{\ep}[\rhs_{q,b}^{\ep}(\grave{\eta},\alpha,\grave{\theta})-\alpha\G_{\ep,b}^{\grave{\theta}}-\Sigma\grave{\eta}]|}{2\alpha} \\
\le C\ep\left(\frac{\ep^2}{|\alpha|}+1+\frac{\norm{\eta}_{q,b}}{|\alpha|}+\frac{\norm{\grave{\eta}}_{q,b}}{|\alpha|}\right)\norm{\eta-\grave{\eta}}_{q,b}
+ C\ep\big(\ep^{-1}|\alpha|+\ep+\norm{\eta}_{q,b}+\norm{\grave{\eta}}_{q,b}\big)|\theta-\grave{\theta}| \\
+ C\ep|\theta-\grave{\theta}|
+ C\frac{\ep}{|\alpha|}\norm{\eta-\grave{\eta}}_{q,b}.
\end{multline*}
Taking $\alpha = A\ep^4$ with $A_0 < A < A_1$, we conclude \eqref{eqn: M-ep Lip}.
We emphasize that the very coarse $\O(\ep^{-3})$ estimate here is due to the difference $\Sigma(\eta-\grave{\eta})$.
\qedhere
\end{enumerate}
\end{proof}

\subsection{The proof of Theorem \ref{thm: PS exists}}\label{app: PS exists}
It follows from the mapping estimate \eqref{eqn: M-eP-ep-theta map} on $\M_{q,b}^{\ep}$ that if $0 < \ep < \ep_{\M}$, $A_0 < A < A_1$, $|\theta| < \pi/2$, and $\eta \in \E_{q,b}$ with $\norm{\eta}_{q,b} \le C_0\ep^4$, then
\[
|\M_{q,b}^{\ep}(\eta,A\ep^4,\theta)|
< C\ep
\]
for some $C > 0$, where this $C$ depends on $C_0$.
Now take $0 < \ep_{\Theta,0} < \ep_{\M}$ small enough such that if $0 < \ep < \ep_{\Theta,0}$, then 
\begin{equation}\label{eqn: ep-Theta-0}
0 < C\ep < \frac{1}{2\sqrt{2}}
\quadword{and}
0 < \frac{1}{2\ep\omega_{\ep}} < 1.
\end{equation}
the latter possible from \eqref{eqn: omega-ep est2}.
Then $\pi/4(\ep\omega_{\ep}) < \pi/2$ and so
\[
\sin\left(\frac{\pi}{4}\right) + \M_{q,b}^{\ep}\left(\eta,A\ep^4,\frac{\pi}{4(\ep\omega_{\ep})}\right) > 0
\quadword{and}
\sin\left(-\frac{\pi}{4}\right) + \M_{q,b}^{\ep}\left(\eta,A\ep^4,-\frac{\pi}{4(\ep\omega_{\ep})}\right) < 0.
\]
The intermediate value theorem then provides $\Theta_{q,b}^{\ep,A}(\eta) \in [-\pi/4(\ep\omega_{\ep}),\pi/4(\ep\omega_{\ep})]$ such that \eqref{eqn: sel mech PS solved} holds.
This is analogous to Lombardi's proof of \cite[Prop.\@ 7.3.20]{lombardi}.

\subsection{Estimates on $\Theta_{q,b}^{\ep,A}$}
We prove mapping and Lipschitz estimates for the phase shift $\Theta_{q,b}^{\ep,A}$ constructed in Theorem \ref{thm: PS exists}.

\begin{lemma}\label{lem: Theta map and Lip}
Let $0 < A_0 < A_1 < 1$.
There is $C_{\Theta} > 0$ such that for all $C_0 > 0$, there exists $\ep_{\Theta,1} \in (0,\ep_{\Theta,0})$ such that the following hold.

\begin{enumerate}[label={\bf(\roman*)}]

\item
If $0 < \ep < \ep_{\Theta,1}$, $\eta \in \E_{q,b}$ with $\norm{\eta}_{q,b} \le C_0\ep^4$, and $A_0 < A < A_1$, then
\begin{equation}\label{eqn: Theta-eP-ep-theta map}
|\Theta_{q,b}^{\ep,A}(\eta)|
\le C_{\Theta}\big(\ep+\ep^{-3}\norm{\eta}_{q,b}^2+\ep^{-3}\norm{\eta}_{q,b}\big).
\end{equation}

\item
If $0 < \ep < \ep_{\Theta,1}$, $\eta$, $\grave{\eta} \in \E_{q,b}$ with $0 \le \norm{\eta}_{q,b}$, $\norm{\grave{\eta}}_{q,b} < C_0\ep^4$, and $A_0 < A < A_1$, then
\begin{equation}\label{eqn: Theta-ep Lip}
|\Theta_{q,b}^{\ep,A}(\eta)-\Theta_{q,b}^{\ep,A}(\grave{\eta})|
\le C_{\Theta}\ep^{-3}\norm{\eta-\grave{\eta}}_{q,b}.
\end{equation}
\end{enumerate}
\end{lemma}

\begin{proof}

\begin{enumerate}[label={\bf(\roman*)}]

\item
Let $0 < \ep < \ep_{\Theta,0}$.
If $\Theta_{q,b}^{\ep,A}(\eta) = 0$, then \eqref{eqn: Theta-eP-ep-theta map} is certainly true.
Otherwise, suppose $\Theta_{q,b}^{\ep,A}(\eta) \ne 0$ and rewrite \eqref{eqn: sel mech PS solved} as
\[
(\ep\omega_{\ep})\Theta_{q,b}^{\ep,A}(\eta)\sinc((\ep\omega_{\ep})\Theta_{q,b}^{\ep,A}(\eta))) + \M_{q,b}^{\ep}(\eta,A\ep^4,\Theta_{q,b}^{\ep,A}(\eta))
= 0
\]
to find
\begin{equation}\label{eqn: Theta is kinda FP}
\Theta_{q,b}^{\ep,A}(\eta)
= -\frac{\M_{q,b}^{\ep}(\eta,A\ep^4,\Theta_{q,b}^{\ep,A}(\eta))}{(\ep\omega_{\ep})\sinc((\ep\omega_{\ep})\Theta_{q,b}^{\ep,A}(\eta)))}.
\end{equation}
Then the estimate \eqref{eqn: ep-Theta-0} on $1/\ep\omega_{\ep}$ and the bounds \eqref{eqn: Theta-ep pi bounds} on $\Theta_{q,b}^{\ep,A}(\eta)$ give
\begin{equation}\label{eqn: Theta-ep denom est}
\big|(\ep\omega_{\ep})\sinc((\ep\omega_{\ep})\Theta_{q,b}^{\ep,A}(\eta)))|\big|
\ge \frac{1}{4}.
\end{equation}
The mapping estimate \eqref{eqn: M-eP-ep-theta map} on $\M_{q,b}^{\ep}$ then implies 
\[
|\Theta_{q,b}^{\ep,A}(\eta)|
\le 4C_{\M}\big(\ep+\ep^{-3}\norm{\eta}_{q,b}^2+\ep^{-3}\norm{\eta}_{q,b}\big),
\]
and this is \eqref{eqn: Theta-eP-ep-theta map}.
We emphasize that the constant $C_{\Theta,\map} := 4C_{\M}$ is independent of $C_0$.

\item
Continue to assume $0 < \ep < \ep_{\Theta,0}$.
Abbreviate $\theta_{\ep}= \Theta_{q,b}^{\ep,A}(\eta)$ and $\grave{\theta}_{\ep} = \Theta_{q,b}^{\ep,A}(\grave{\eta})$.
If $\theta_{\ep}= \grave{\theta}_{\ep}$, then \eqref{eqn: Theta-ep Lip} is certainly true.
Otherwise, for $\theta_{\ep}\ne \grave{\theta}_{\ep}$, since 
\[
\sin((\ep\omega_{\ep})\theta_{\ep})+\M_{q,b}^{\ep}(\eta,A\ep^4,\theta_{\ep})
= \sin((\ep\omega_{\ep})\grave{\theta}_{\ep})+\M_{q,b}^{\ep}(\eta,A\ep^4,\grave{\theta}_{\ep})
= 0,
\]
we have
\[
(\ep\omega_{\ep})(\theta_{\ep}-\grave{\theta}_{\ep})\left(\frac{\sin((\ep\omega_{\ep})\theta_{\ep})-\sin((\ep\omega_{\ep})\grave{\theta}_{\ep})}{(\ep\omega_{\ep})(\theta_{\ep}-\grave{\theta}_{\ep})}\right)
= \M_{q,b}^{\ep}(\eta,A\ep^4,\grave{\theta}_{\ep})-\M_{q,b}^{\ep}(\eta,A\ep^4,\theta_{\ep}).
\]
The mean value theorem allows us to express
\[
\frac{\sin((\ep\omega_{\ep})\theta_{\ep})-\sin((\ep\omega_{\ep})\grave{\theta}_{\ep})}{(\ep\omega_{\ep})(\theta_{\ep}-\grave{\theta}_{\ep})}
= \cos((\ep\omega_{\ep})((1-t)\theta_{\ep}+t\grave{\theta}_{\ep}))
\]
for some $t \in [0,1]$.
It follows from the estimate \eqref{eqn: omega-ep est2} on $\ep\omega_{\ep}$ and the bound \eqref{eqn: Theta-ep pi bounds} on $\theta_{\ep}$ and $\grave{\theta}_{\ep}$ that $|(\ep\omega_{\ep})((1-t)\theta_{\ep}+t\grave{\theta}_{\ep})| < \pi/4$ for $\ep > 0$ suitably small, and so $|\cos((\ep\omega_{\ep})((1-t)\theta_{\ep}+t\grave{\theta}_{\ep}))| > 1/2$.
We conclude
\[
\theta_{\ep}-\grave{\theta}_{\ep}
= \frac{ \M_{q,b}^{\ep}(\eta,A\ep^4,\grave{\theta}_{\ep})-\M_{q,b}^{\ep}(\eta,A\ep^4,\theta_{\ep})}{(\ep\omega_{\ep})\cos((\ep\omega_{\ep})((1-t)\theta_{\ep}+t\grave{\theta}_{\ep}))}.
\]
The Lipschitz estimate \eqref{eqn: M-ep Lip} then gives
\begin{equation}\label{eqn: Lip theta aux}
|\theta_{\ep}-\grave{\theta}_{\ep}|
\le C_{\M}\big(\ep^{-3}+2C_0\ep^4\big)\norm{\eta-\grave{\eta}}_{q,b} 
+ C_{\M}\big(\ep + 2C_0\ep^4\big)|\theta_{\ep}-\grave{\theta}_{\ep}|.
\end{equation}
We single out the constant $C_{\M}$ here for future reference, and for contrast with other constants.

Now take 
\[
\ep_{\Theta,1}
:= \min\left\{1,\ep_{\Theta,0},\frac{1}{2C_0},\frac{1}{4C_{\M}}\right\}.
\]
Then
\[
\ep^{-3}+2C_0\ep^4
\le \ep^{-3}+1
\le 2\ep^{-3}
\quadword{and}
C_{\M}(\ep+2C_0\ep^4)
\le \frac{1}{2},
\]
from which \eqref{eqn: Lip theta aux} rearranges into
\[
|\theta_{\ep}-\grave{\theta}_{\ep}|
\le 4C_{\M}\ep^{-3}\norm{\eta-\grave{\eta}}_{q,b},
\]
and this is \eqref{eqn: Theta-ep Lip} with the same choice of $C_{\Theta}$ as above.
Again, this constant $C_{\Theta}$ is independent of $C_0$, although the threshold $\ep_{\Theta,1}$ certainly depends on $C_0$.
\qedhere
\end{enumerate}
\end{proof}

Since $\Theta_{q,b}^{\ep,A}(\eta)$ does satisfy the fixed-point equation \eqref{eqn: Theta is kinda FP}, we might still wonder (contrary to our prior claims) if we could solve for the phase shift via a quantitative contraction mapping argument joint with $\eta$, instead of by the intermediate value theorem construction.
After all, for $\eta = \O(\ep^4)$, we do obtain from \eqref{eqn: Theta-eP-ep-theta map} a good $\O(\ep)$ mapping estimate on the phase shift.
However, the Lipschitz estimates for $\eta$ as given by the fixed-point problem \eqref{eqn: eta eqn PS proto1} are only $\O(\ep^2)$; this is ultimately a consequence of the Lipschitz estimate \eqref{eqn: ultimate R Lip theta} on $\rhs_{q,b}^{\ep}$.
This $\O(\ep^2)$ Lipschitz estimate cannot contend with the $\O(\ep^{-3})$ Lipschitz estimate from \eqref{eqn: Theta-ep Lip}.
Conversely, when we substitute $\theta = \theta_{\ep}(\eta,\alpha)$ into \eqref{eqn: eta eqn PS proto1} to obtain the actual fixed-point problem \eqref{eqn: eta eqn PS} for $\eta$, the Lipschitz estimates in $\theta$ that this problem inherits from \eqref{eqn: ultimate R Lip theta} on $\rhs_{q,b}^{\ep}$, as well as the $\O(\ep^4)$ factor that comes with $\G_{\ep,b}^{\theta}$, are just enough to overcome the doom of the $\O(\ep^{-3})$ Lipschitz estimate from \eqref{eqn: Theta-ep Lip}.

\subsection{The proof of Lemma \ref{lem: ultimate PS lemma}}\label{app: ultimate PS lemma}

\subsubsection{The proof of the mapping estimate \eqref{eqn: eta eqn PS-ep-theta map PS}}
The following result resolves our dilemma, posed before the statement of Lemma \ref{lem: ultimate PS lemma}, about the selection of the constant $C_0$ that controls the size of $\eta$.

\begin{lemma}\label{lem: weird aux for PS map}
Let $C > 0$ and $0 < A_0 < A_1 < 1$.
There is $C_{\star} > 0$ such that if $0 < \ep < \ep_{\B,2}$, $\eta \in \E_{q,b}$, $A_0 < A < A_1$, and
\begin{equation}\label{eqn: weird theta aux for PS map}
|\theta| 
\le \min\left\{C\big(\ep+\ep^{-3}\norm{\eta}_{q,b}^2+\ep^{-3}\norm{\eta}_{q,b}\big),\frac{\pi}{2}\right\},
\end{equation}
then 
\begin{equation}\label{eqn: weird aux for PS map}
\norm{(1+\B_{\ep}^{-1}\Pi_{\ep,b}(\theta)\Sigma)^{-1}\B_{\ep}^{-1}\Pi_{\ep,b}(\theta)\big(\rhs_{q,b}^{\ep}(\eta,A\ep^4,\theta)-A\ep^4\G_{\ep}^{\theta}\big)}_{q,b}
\le C_{\star}\big(\ep^4+\norm{\eta}_{q,b}^2+\ep\norm{\eta}_{q,b}\big).
\end{equation}
\end{lemma}

\begin{proof}
We first use the estimate \eqref{eqn: lin norm inv PS} on $(1+\B_{\ep}^{-1}\Pi_{\ep,b}(\theta)\Sigma)^{-1}$, the estimate \eqref{eqn: B-ep-inv est} on $\B_{\ep}^{-1}$, and the estimate \eqref{eqn: Pi-ep-theta map} on $\Pi_{\ep,b}(\theta)$, to bound, for $|\theta| < \pi/2$,
\begin{multline}\label{eqn: PS map aux}
\norm{(1+\B_{\ep}^{-1}\Pi_{\ep,b}(\theta)\Sigma)^{-1}\B_{\ep}^{-1}\Pi_{\ep,b}(\theta)\big(\rhs_{q,b}^{\ep}(\eta,A\ep^4,\theta)-A\ep^4\G_{\ep}^{\theta}\big)}_{q,b} \\
\le C_1\norm{\rhs_{q,b}^{\ep}(\eta,A\ep^4,\theta)}_{q,b} + C_1\ep^4\norm{\G_{\ep}^{\theta}}_{q,b}
\end{multline}
for some $C_1 > 0$.
Then we use the mapping estimates \eqref{eqn: ultimate R map} on $\rhs_{q,b}^{\ep}$ and \eqref{eqn: ultimate G map} on $\G_{\ep}^{\theta}$ to bound
\[
C_1\norm{\rhs_{q,b}^{\ep}(\eta,A\ep^4,\theta)}_{q,b} 
\le C_2\big(\ep^4+A\ep^5+A^2\ep^7+\norm{\eta}_{q,b}^2\big)
\]
for some $C_2 > 0$ and
\[
C_1\ep^4\norm{\G_{\ep}^{\theta}}{q,b}
\le C_3\ep^4|\theta|
\le C_3C\ep^4\big(\ep+\ep^{-3}\norm{\eta}_{q,b}^2+\ep^{-3}\norm{\eta}_{q,b}\big)
\]
for some $C_3 > 0$.
Here the undecorated constant $C$ comes from \eqref{eqn: weird theta aux for PS map}.
The estimate \eqref{eqn: weird aux for PS map} then follows because $0 < \ep$, $A < 1$.
\end{proof}

Now we can prove the mapping estimate \eqref{eqn: eta eqn PS-ep-theta map PS}.
Take $C_{\star}$ from Lemma \ref{lem: weird aux for PS map} where $C = C_{\Theta}$ from Lemma \ref{lem: Theta map and Lip}.
Let $C_0 = 2C_{\star}$.
Take $\ep_{\Theta,1}$ from Lemma \ref{lem: Theta map and Lip} and put
\begin{equation}\label{eqn: ep-Theta-2}
\ep_{\Theta,2}
:= \min\left\{\ep_{\Theta,1},\frac{1}{C_0^2+C_0}\right\}.
\end{equation}
Assume $0 < \ep < \ep_{\Theta,2}$, $\eta \in \E_{q,b}$ with $\norm{\eta}_{q,b} \le C_0\ep^4$, and $A_0 < A < A_1$.
Lemma \ref{lem: Theta map and Lip} ensures that $\theta = \Theta_{q,b}^{\ep,A}(\eta)$, with $\Theta_{q,b}^{\ep,A}(\eta) \in \R$ from Theorem \ref{thm: PS exists}, satisfies \eqref{eqn: weird theta aux for PS map} with $C = C_{\Theta}$, and so from \eqref{eqn: weird aux for PS map}, the map $\F_{q,b}^{\ep,A}$ defined in \eqref{eqn: eta eqn PS} satisfies
\[
\norm{\F_{q,b}^{\ep,A}(\eta)}_{q,b}
\le C_{\star}\big(\ep^4+\norm{\eta}_{q,b}^2+\ep\norm{\eta}_{q,b}\big)
\le C_{\star}(\ep^4+C_0^2\ep^8+C_0\ep^5)
\le \frac{C_0}{2}\ep^4(1+(C_0^2+C_0)\ep)
\le C_0\ep^4.
\]
This is the mapping estimate \eqref{eqn: eta eqn PS-ep-theta map PS}.

\subsubsection{The proof of the contraction estimate \eqref{eqn: eta eqn PS Lip PS}}
Suppose for now that $0 < \ep < \ep_{\Theta,1}$ as defined in \eqref{eqn: ep-Theta-2} above.
Let $A_0 < A < A_1$ and suppose that $\eta$, $\grave{\eta} \in \E_{q,b}$ with $\norm{\eta}_{q,b}$, $\norm{\grave{\eta}}_{q,b} \le C_0\ep^4$ with $C_0$ defined above.
We use the mapping and Lipschitz estimates \eqref{eqn: lin norm inv PS} and \eqref{eqn: lin Lip PS} on $(1+\B_{\ep}^{-1}\Pi_{\ep,b}(\theta)\Sigma)^{-1}$, the estimate \eqref{eqn: B-ep-inv est} on $\B_{\ep}^{-1}$, and the mapping and Lipschitz estimate \eqref{eqn: P-ep-theta map} and \eqref{eqn: Pi-ep-theta Lip} on $\Pi_{\ep,b}(\theta)$ to bound
\[
\norm{\F_{q,b}^{\ep,A}(\eta)-\F_{q,b}^{\ep,A}(\grave{\eta})}_{q,b}
\le C(\Delta_1+\Delta_2+\Delta_3),
\]
where
\[
\Delta_1
:= |\Theta_{q,b}^{\ep,A}(\eta)-\Theta_{q,b}^{\ep,A}(\grave{\eta})|\big(\norm{\rhs_{q,b}^{\ep}(\eta,A\ep^4,\Theta_{q,b}^{\ep,A}(\eta)}_{q,b} + \ep^4\norm{\G_{\ep}^{\Theta_{q,b}^{\ep,A}(\eta)}}_{q,b}\big),
\]
\[
\Delta_2
:= \norm{\rhs_{q,b}^{\ep}(\eta,A\ep^4,\Theta_{q,b}^{\ep,A}(\eta))-\rhs_{q,b}^{\ep}(\grave{\eta},A\ep^4,\Theta_{q,b}^{\ep,A}(\grave{\eta}))}_{q,b},
\]
and
\[
\Delta_3
:= \ep^4\norm{\G_{\ep}^{\Theta_{q,b}^{\ep,A}(\eta)}-\G_{\ep}^{\Theta_{q,b}^{\ep,A}(\grave{\eta})}}_{q,b}.
\]
We use the mapping estimates \eqref{eqn: ultimate R map} on $\rhs_{q,b}^{\ep}$ and \eqref{eqn: ultimate G map} on $\G_{\ep}^{\theta}$ to bound
\[
\Delta_1
\le C\ep^4|\Theta_{q,b}^{\ep,A}(\eta)-\Theta_{q,b}^{\ep,A}(\grave{\eta})|.
\]
Then we use the Lipschitz estimates \eqref{eqn: ultimate R Lip theta} on $\rhs_{q,b}^{\ep}$ and \eqref{eqn: ultimate G Lip theta} on $\G_{\ep}^{\theta}$ to bound
\begin{multline*}
\Delta_2
\le C\big(\ep^2+A\ep^4+\norm{\eta}_{q,b}+\norm{\grave{\eta}}_{q,b}\big)\norm{\eta-\grave{\eta}}_{q,b} \\
+ C\big(A^2\ep^7+A\ep^5+A\ep^4\norm{\eta}_{q,b}+A\ep^4\norm{\grave{\eta}}_{q,b}\big)|\Theta_{q,b}^{\ep,A}(\eta)-\Theta_{q,b}^{\ep,A}(\grave{\eta})|
\end{multline*}
and
\[
\Delta_3
\le C\ep^4|\Theta_{q,b}^{\ep,A}(\eta)-\Theta_{q,b}^{\ep,A}(\grave{\eta})|.
\]
The estimate for $\Delta_2$ simplifies to
\[
\Delta_2
\le C\ep\norm{\eta-\grave{\eta}}_{q,b} + C\ep^4|\Theta_{q,b}^{\ep,A}(\eta)-\Theta_{q,b}^{\ep,A}(\grave{\eta})|
\]
when we use $0 < \ep$, $A < 1$ and $\norm{\eta}_{q,b}$, $\norm{\grave{\eta}}_{q,b} \le C_0\ep^4$.

We combine the estimates on $\Delta_1$, $\Delta_2$, and $\Delta_3$ to obtain
\begin{equation}\label{eqn: eta eqn PS Lip PS pre}
\norm{\F_{q,b}^{\ep,A}(\eta)-\F_{q,b}^{\ep,A}(\grave{\eta})}_{q,b}
\le C\ep\norm{\eta-\grave{\eta}}_{q,b} + C\ep^4|\Theta_{q,b}^{\ep,A}(\eta)-\Theta_{q,b}^{\ep,A}(\grave{\eta})|
\le C\ep\norm{\eta-\grave{\eta}}_{q,b}
\end{equation}
when we use the Lipschitz estimate \eqref{eqn: Theta-ep Lip} on $\theta_{\ep}$.
We emphasize that the $\O(\ep^4)$ factor that $|\Theta_{q,b}^{\ep,A}(\eta)-\Theta_{q,b}^{\ep,A}(\grave{\eta})|$ gains here save the final Lipschitz estimates from the threatening $\O(\ep^{-3})$ factor
that $|\Theta_{q,b}^{\ep,A}(\eta)-\Theta_{q,b}^{\ep,A}(\grave{\eta})|$ brings from \eqref{eqn: Theta-ep Lip}.
Take $\ep_{\star} := \min\{\ep_{\Theta,1},1/2C\}$ to turn \eqref{eqn: eta eqn PS Lip PS pre} into the contraction estimate \eqref{eqn: eta eqn PS Lip PS}.

\section*{Acknowledgments}

I gratefully acknowledge support from the National Science Foundation through grant DMS-2405535.

\bibliographystyle{siam}
\bibliography{KdV_PS_nanopterons_bib}{}

\end{document}